\documentclass[10pt]{amsart}
\usepackage{kotex, amsmath, amsfonts, amssymb, amsthm, epstopdf, latexsym, enumerate, mathrsfs}
\usepackage[a4paper]{geometry}
\usepackage{tikz-cd}
\usepackage[onehalfspacing]{setspace}
\usepackage{braket}
\usepackage{blkarray}
\usepackage{mathtools}
\usepackage{lscape}
\usepackage{xstring}
\usepackage{graphicx}
\usepackage{bbm}
\usepackage{subcaption}
\graphicspath{ {images/} }
\displaywidth=\textwidth
\displayindent=-\leftskip

\makeatletter
\newcommand*{\doublerightarrow}[2]{\mathrel{
		\settowidth{\@tempdima}{$\scriptstyle#1$}
		\settowidth{\@tempdimb}{$\scriptstyle#2$}
		\ifdim\@tempdimb>\@tempdima \@tempdima=\@tempdimb\fi
		\mathop{\vcenter{
				\offinterlineskip\ialign{\hbox to\dimexpr\@tempdima+1em{##}\cr
					\rightarrowfill\cr\noalign{\kern.5ex}
					\rightarrowfill\cr}}}\limits^{\!#1}_{\!#2}}}
\newcommand*{\triplerightarrow}[1]{\mathrel{
		\settowidth{\@tempdima}{$\scriptstyle#1$}
		\mathop{\vcenter{
				\offinterlineskip\ialign{\hbox to\dimexpr\@tempdima+1em{##}\cr
					\rightarrowfill\cr\noalign{\kern.5ex}
					\rightarrowfill\cr\noalign{\kern.5ex}
					\rightarrowfill\cr}}}\limits^{\!#1}}}
\makeatother

\usepackage{etoolbox}
\makeatletter
\patchcmd\start@gather{$$}{%
	$$%
	\displaywidth=\textwidth
	\displayindent=-\leftskip
}{}{\errmessage{Cannot patch \string\start@gather}}
\patchcmd\start@align{$$}{%
	$$%
	\displaywidth=\textwidth
	\displayindent=-\leftskip
}{}{\errmessage{Cannot patch \string\start@align}}
\patchcmd\start@multline{$$}{%
	$$%
	\displaywidth=\textwidth
	\displayindent=-\leftskip
}{}{\errmessage{Cannot patch \string\start@multline}}
\patchcmd\mathdisplay{$$}{%
	$$%
	\displaywidth=\textwidth
	\displayindent=-\leftskip
}{}{\errmessage{Cannot patch \string\mathdisplay}}
\makeatother

\theoremstyle{plain}
\newtheorem{thm}{Theorem}[section]
\newtheorem{lemma}[thm]{Lemma}
\newtheorem{prop}[thm]{Proposition}
\newtheorem*{cor}{Corollary}

\theoremstyle{definition}
\newtheorem{defn}{Definition}[section]

\newtheorem{example}{Example}[section]

\theoremstyle{remark}
\newtheorem*{remark}{Remark}

\def\resp{respectively}

\newcommand{\deq}{\coloneqq}
\newcommand{\Aoo}{\caa_\infty}

\newcommand{\super}{{\bzz/2\bzz}}
\newcommand{\tth}{^\text{th}}
\newcommand{\field}{{\mathbbm{k}}}

\renewcommand{\deg}[1]{{\left|#1\right|}}

\newcommand{\Tw}{\operatorname{Tw}}

\newcommand{\caa}{\mathcal{A}}

\newcommand{\ccc}{\mathcal{C}}
\newcommand{\cdd}{\mathcal{D}}

\newcommand{\cff}{\mathcal{F}}

\newcommand{\ctt}{\mathcal{T}}

\newcommand{\Z}{\mathbb{Z}}

\newcommand{\bpp}{\mathbb{P}}

\newcommand{\bzz}{\mathbb{Z}}

\newcommand{\fm}{\mathfrak{m}}

\newcommand{\ho}{\operatorname{Hom}}

\newcommand{\id}{\operatorname{id}}

\newcommand{\im}{\operatorname{Im}}

\newcommand{\can}{{\operatorname{can}}}
\newcommand{\Sp}{{\operatorname{Sp}}}
\newcommand{\Ord}{{\operatorname{Ord}}}

\newcommand{\half}[1]{\left(\frac{1}{2}\right)^{#1}}

\newcommand{\con}{{\operatorname{con}}}
\newcommand{\disc}{{\operatorname{disc}}}
\newcommand{\comp}{{\operatorname{comp}}}
\newcommand{\thick}{{\operatorname{thick}}}
\newcommand{\Con}{{\ccc\operatorname{on}}}
\newcommand{\Disc}{\cdd{\operatorname{isc}}}
\newcommand{\Comp}{\ccc{\operatorname{omp}}}
\newcommand{\Thick}{\ctt{\operatorname{hick}}}

\def\resp{respectively}

\def\ora{\overrightarrow}
\title{Topological Fukaya category of tagged arcs}
\author[Cho]{Cheol-Hyun Cho}
\address{Department of Mathematical Sciences, Research Institute in Mathematics\\ Seoul National University\\ Gwanak-gu\\Seoul \\ South Korea}
\email{chocheol@snu.ac.kr}
\author[Kim]{Kyoungmo Kim}
\address{Department of Mathematical Sciences\\ Seoul National University\\ Gwanak-gu\\Seoul \\ South Korea}
\email{kyoungmo@snu.ac.kr}

\date{}

\begin{document}

	\maketitle
	
\begin{abstract}
A tagged arc on a surface is introduced by Fomin, Shapiro, and Thurston to study cluster theory on marked surfaces.
Given a tagged arc system on a graded marked surface, we define its $\bzz$-graded $\Aoo$-category, generalizing the construction of
Haiden, Katzarkov, and Kontsevich for arc systems. When a tagged arc system arises from a non-trivial involution on a marked surface, we show that this $\Aoo$-category is quasi-isomorphic to the invariant part of the topological Fukaya category under the involution. In particular, this identifies tagged
arcs with non-geometric idempotents of Fukaya category.
\end{abstract}
		
	\tableofcontents

\section{Introduction}

Fomin, Shapiro, and Thurston introduce the concept of tagged arcs in \cite{FST08} to study cluster theory on marked surfaces. A tagged arc is an arc with additional data of tagging at the endpoints lying on interior marked points. Here, a tagging is a choice of an element of $\super$. An ideal triangulation of a marked surface by tagged arcs and flips of triangulations (given by replacing a diagonal by the other one) relate the combinatorial geometry of marked surfaces with cluster algebras. In particular, tagged arcs solve the problem of flipping self-folded triangles.

On the other hand, gentle algebra is a special-type of associative algebra introduced in \cite{AS87}. Its derived category has been intensively studied both in algebraic and
in geometric perspectives. Algebraically, indecomposable objects in the derived categories of gentle algebras are studied in \cite{BM03}, \cite{BMM03}, \cite{BD05}, \cite{BD17Arx}, and morphisms between them are studied in \cite{ALP16}. A combinatorial derived invariant of gentle algebra is introduced in \cite{AG08}. Also, in \cite{SZ03}, they show that the class of gentle algebras is closed under derived equivalence. In \cite{assem2010gentle}, gentle algebras are realized as Jacobian algebras of ideal triangulations of boundary-marked surfaces.
 
Surprising relations between geometry of surfaces, in particular its Fukaya category, and gentle algebras have been found.
In \cite{Bocklandt16} and \cite{HKK17}, Bocklandt and Haiden-Katzarkov-Kontsevich introduce topological Fukaya category of surface, which is a topological version of  (partially)  wrapped Fukaya category in \cite{FSS09}, \cite{AS10}, \cite{Auroux10}, \cite{Auroux10_2}, \cite{Sylvan19}.  Topological Fukaya category gives an $\Aoo$-enhancement of derived category of graded homologically smooth and proper gentle algebra. In \cite{OPS18Arx}, Opper, Plamondon, and Schroll give a geometric model for derived category of any ungraded gentle algebra. Also in \cite{BC21}, Baur and Sim\~{o}es give a geometric model for module category of finite dimensional gentle algebra.  These geometric models provide geometric interpretations of many algebraic properties. 
In \cite{LP20}, Lekili and Polishchuk give the converse construction of \cite{HKK17} and find a geometric interpretation of the derived invariant in \cite{AG08}. Also, in \cite{APS23}, 
Amiot, Plamondon, and Schroll use the model of \cite{OPS18Arx} to interpret the invariant for any gentle algebra.

Skew-gentle algebra is introduced by Gei\ss \; and de~la Pe\~{n}a  \cite{GP99} as a generalization of a gentle algebra.
Labardini-Fragoso associates a quiver with potential to an ideal triangulation by tagged arcs of a marked surface in \cite{Labardini09}.
Gei\ss, Labardini-Fragoso, and Schr\"{o}er \cite{GLS16} and Qiu and Zhou \cite{QZ17} show that the Jacobian algebra of an appropriate triangulation is a skew-gentle algebra.

It is known that skew-gentle algebras are related to $\super$-orbifold surfaces. Suppose that a marked surface and an associated arc system for a gentle algebra are equipped with
a $\super$-action. In \cite{AP21}, \cite{AB22}, \cite{LSV22}, \cite{Amiot23}, they use $\super$-equivariant theory to study indecomposable objects and derived-equivalences of skew-gentle algebras. In \cite{QZ17}, \cite{QZZ22Arx}, they introduce geometric models for the intersection numbers between tagged arcs and study the derived categories of (graded) skew-gentle algebras. 
The space of stability conditions on the latter are shown to be isomorphic to that of quadratic differentials in \cite{LQW23Arx}. Here, interior marked points for tagged arcs correspond to $\super$-orbifold points in the former, and we will not distinguish them from now on.
 
Therefore, it is natural to conjecture that the derived category of a skew-gentle algebra should be related to a partially wrapped Fukaya category of the corresponding $\super$-orbifold surface  (see \cite{LSV22}). In this paper and the sequel, we give an answer for this conjecture. Let us explain this in more detail.
Consider a graded marked $\super$-orbifold surface $(S, M, O, \eta)$.

Our first main observation is the precise relationship between tagged arcs on $S$ and the Fukaya category of $\super$-orbifold surface $S$.
Recall that tagged arcs are defined in a combinatorial way and have been lacking direct geometric interpretations (cf. \cite{QZZ22Arx}). We find that arcs ending at $\super$-orbifold points in the orbifold Fukaya category admits non-trivial idempotents. These types of non-geometric idempotents do not seem to appear in the smooth Fukaya categories.
We will explain the correspondence between idempotents and tagged arcs. This explains why tagged arcs are relevant to the study of skew-gentle algebras in light of the above conjecture.

Next, we define topological Fukaya categories of tagged arc systems of $\super$-orbifold surfaces, generalizing the construction of Haiden-Katzarkov-Kontsevich \cite{HKK17} for arc systems on smooth surfaces. Namely, we define morphisms between tagged arcs, and also define $\Aoo$-operations on morphisms in a combinatorial way.  

Contrary to smooth cases, when two tagged arcs meet at an interior marking, morphisms exist only in one direction (depending on the grading and the tagging).
We call this morphism an interior morphism, and call the imaginary non-existing morphism in the other direction a non-morphism.
Intuitively, as tagged arcs correspond to idempotents, a non-trivial morphism between arcs does not necessarily induce a non-trivial morphism between their idempotents. 

Another new feature is a composition sequence. Recall that a disc sequence in \cite{HKK17} provides contributions of ``holomorphic discs'' to $\Aoo$-operations.
First, we can generalize the notion of disc sequence in our setting by allowing foldings along arcs to interior marked points. Furthermore, if one of the corner of a disc is a
non-morphism, then the complementary morphism can be regarded as an output of an $\Aoo$-operation, and we call this a composition sequence. 
There is one more non-trivial operation, called $\fm^{\thick}$, when the operation involves two different tagged arcs of the same underlying arc. 
\begin{thm}\label{thm:Intro1}
	Let $(S, M, O, \eta)$ be a graded marked $\super$-orbifold surface and $\Gamma$ be a tagged arc system.  We can define a $\Z$-graded unital $\Aoo$-category  $\cff_\Gamma(S, M, O, \eta)$. 
\end{thm}
Different choices of tagged arc systems on $(S, M, O, \eta)$ result in Morita equivalent $\Aoo$-categories.
\begin{thm}\label{thm:Intro2}
	Let $(S, M, O, \eta)$ be a graded marked $\super$-orbifold surface and $\Gamma_1, \Gamma_2$ be tagged arc systems on it. Then, two $\Aoo$-categories $\cff_{\Gamma_1}(S, M, O, \eta)$ and $\cff_{\Gamma_2}(S, M, O, \eta)$ are Morita equivalent.
\end{thm}

A special class of tagged arc systems, called {\em involutive tagged arc system},  can be obtained from the arc system of the $\super$-covering. 
These are defined as follows. For a $\super$-orbifold surface $S$, we have the $2$-to-$1$ branched covering given by a smooth surface $\tilde{S}$ with the deck transformation $\iota$ satisfying $\iota^2 = \id_{\tilde{S}}$.
Suppose that we have a marking $\tilde{M}$ and a grading $\eta$ which are preserved by $\iota$, and denote by $O$ the set of $\iota$-fixed points and by $M$ the induced marking on $S$.
Then, $(S,M,O,\eta)$ is a graded $\super$-orbifold surface. Let us choose an arc system $\tilde{\Gamma}$ on $\tilde{S}$ which is invariant under $\iota$.
We associate an involutive tagged arc system $\Gamma$ on $(S,M,O,\eta)$ corresponding to  $\tilde{\Gamma}$, which can be defined combinatorially as in Definition \ref{defn:InvolutiveCondition}.

There are three  main reasons to consider involutive tagged arc systems.
First, this allows us to relate tagged arcs and non-geometric idempotents of the Fukaya category, as we have explained earlier.
Second, we will prove the Morita invariance of the $\Aoo$-category of tagged arc system in Theorem \ref{thm:Intro2} using
the Morita invariance of the arc system in the cover (via Theorem \ref{thm:Intro3} and \ref{thm:Intro4}).
And third, involutive tagged arc systems will correspond to skew-gentle algebras.

Let us give more detailed explanations on each of these reasons. First, any involutive tagged arc system has a special property that there exists at most one underlying arc that meets a given interior marking.
If two tagged arcs meet at an interior (orbifold) marking, then they intersect transversely in the two-fold covering. As arcs in the arc system are disjoint, this
does not happen in the cover. In particular, there are no interior morphisms for the involutive tagged arc system. Hence, the special feature of interior morphisms and non-morphisms cannot be seen directly. Nonetheless, the derived Fukaya category has additional objects in the form of twisted complexes and interior morphisms arise in this setting.
In Section \ref{sec:tagide}, we will illustrate the geometric origin of our definition of interior morphisms between tagged arcs,
by computing morphisms between the corresponding idempotents.

Second, an involutive tagged arc system defines an  $\Aoo$-category  
by Theorem \ref{thm:Intro1}. We can give another definition by taking the $\super$-invariant part of the topological Fukaya category of the arc system in the cover.
From \cite{HKK17}, we have a $\Z$-graded $\Aoo$-category  $\cff_{\tilde{\Gamma}}(\tilde{S}, \tilde{M}, \eta)$. This admits a strict $\super$-action from $\iota$ and
its $\super$-invariant part, after taking idempotent completion, defines the $\Aoo$-category $\cff_{\tilde{\Gamma}}(\tilde{S}, \iota, \tilde{M}, \eta)$.
In this $\Aoo$-category, the $\iota$-invariant arc $\gamma$ has two non-geometric idempotents, hence idempotent completion is necessary.
We prove that the Morita equivalent class of the $\Aoo$-category is independent of the choice of $\iota$-invariant arc system on $\tilde{S}$.
\begin{thm}\label{thm:Intro3}
	Let $(\tilde{S}, \iota, \tilde{M}, \eta)$ be an involutive graded marked surface and $\tilde{\Gamma}_1,\tilde{\Gamma}_2$ be $\iota$-invariant arc systems. Then, two $\Aoo$-categories $\cff_{\tilde{\Gamma}_1}(\tilde{S}, \iota, \tilde{M}, \eta)$ and $\cff_{\tilde{\Gamma}_2}(\tilde{S}, \iota, \tilde{M}, \eta)$ are Morita equivalent.
\end{thm}

We show that these two constructions of $\Aoo$-categories are quasi-isomorphic to each other.
\begin{thm}\label{thm:Intro4}
	Let $(S, M, O, \eta)$ be a graded marked $\super$-orbifold surface and $(\tilde{S}, \iota, \tilde{M}, \eta)$ be its double cover.   
	\begin{enumerate}		
		\item There is a one-to-one correspondence between the following two sets: 
			$$\{\text{Involutive tagged arc systems on $(S, M, O, \eta)$}\} \rightarrow \{\text{$\iota$-invariant arc systems on $(\tilde{S}, \iota, \tilde{M}, \eta)$}\}.$$
		\item For an involutive tagged arc system $\Gamma$ of $(S, M, O, \eta)$ and the corresponding $\iota$-invariant arc system $\tilde{\Gamma}$ of $(\tilde{S}, \iota, \tilde{M}, \eta)$, there is a quasi-equivalence between two $\Aoo$-categories $\cff_{\Gamma}(S, M, O, \eta)$ and $\cff_{\tilde{\Gamma}}(\tilde{S}, \iota, \tilde{M}, \eta)$.
		\item Any tagged arc system can be transformed into a Morita equivalent involutive tagged arc system, by a sequence of  adding or subtracting tagged arcs 
		(such that intermediate ones are  also tagged arc systems).
	\end{enumerate}
\end{thm}

In the sequel, we will find that different choices of tagged arc systems result in different derived-tame algebras.
Being a skew-gentle algebra is not a derived invariant notion. Namely, there are other types of algebras that are not skew-gentle themselves, but derived equivalent to some skew-gentle algebras. Our $\Aoo$-category of a tagged arc system provides an efficient tool to study these algebras.
In particular, involutive tagged arc systems correspond to skew-gentle algebras. As a corollary, we know algebras corresponding general tagged arc systems are all derived equivalent to skew-gentle algebra. For example, we reprove that derived-tame cyclic Nakayama algebra in \cite{BGV21} and generalized gentle algebras in \cite{Fonseca22} are derived equivalent to skew-gentle algebras. In addition, we prove the algebras introduced in \cite[Definition 5.7]{BD18Arx} are derived-equivalent to skew-gentle algebras.

Before we finish the introduction, let us illustrate our construction for the case of $D_n$-quiver.
Namely, we consider tagged arc systems on a disc with one interior marking, such that the endomorphism algebras of the direct sum of chosen tagged arcs in the $\Aoo$-category are path algebras of $D_n$-quivers.
\begin{example}
	$$\begin{tikzcd} 
		&  &  & \delta^+ \\
		\alpha \arrow[r] &\beta \arrow[r] & \gamma \arrow[ru] \arrow[rd] & \\
		&  &  & \delta^-
	\end{tikzcd} \quad 
	\begin{tikzcd} 
		&  &  & \epsilon \arrow[ld] \\
		\alpha \arrow[r] &\beta \arrow[r] & \gamma \arrow[rd] & \\
		&  &  & \delta^+
	\end{tikzcd}$$
	Let us consider a disc with four boundary markings and one interior marking. Its double cover is also a disc and we can take a $\super$-invariant line field on it.
	 Let $\Gamma_a, \Gamma_b$, and $\Gamma_c$ be the tagged arc systems as illustrated in Figure \ref{fig:IntroExample}(a), (b), and (c), \resp. 
		\def\mySize{0.3\linewidth}
	\begin{figure}[h!]
		\centering
		\begin{subfigure}[b]{\mySize}
			\includegraphics[width=\linewidth]{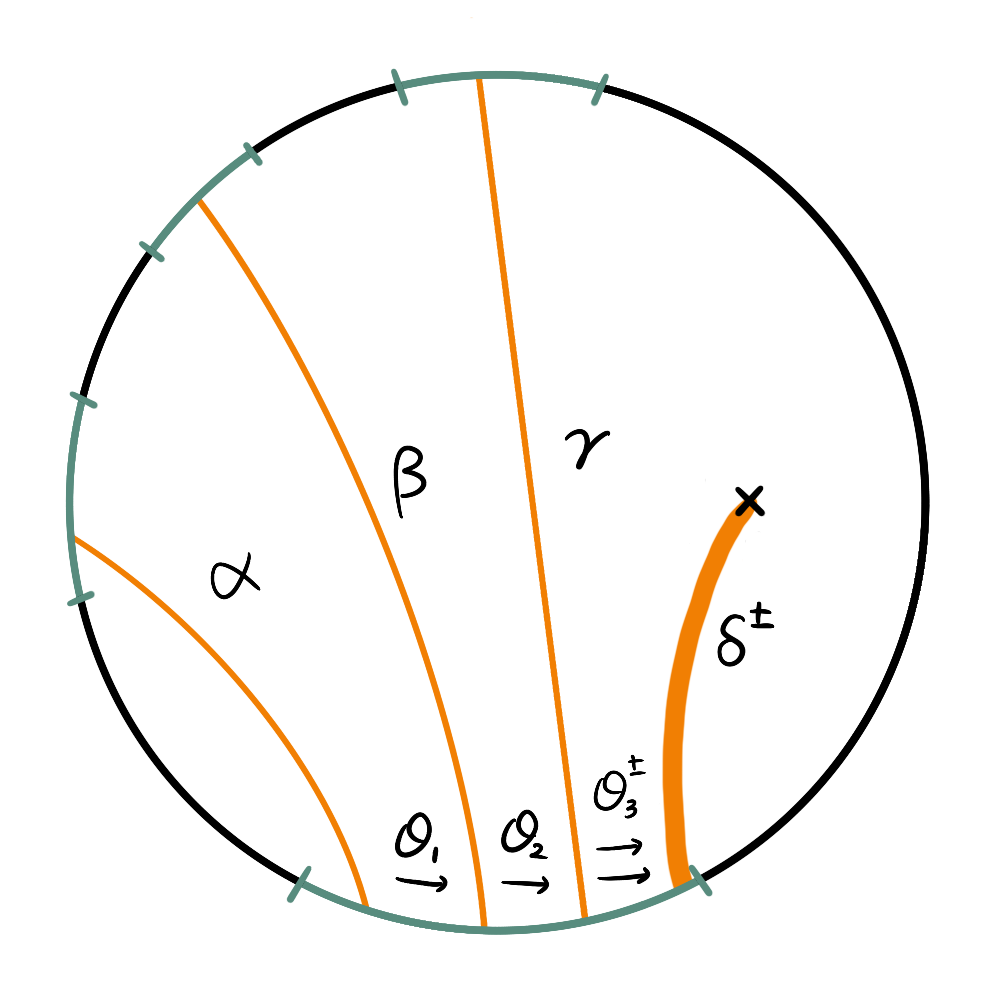}
			\caption{}
		\end{subfigure}
		\begin{subfigure}[b]{\mySize}
			\includegraphics[width=\linewidth]{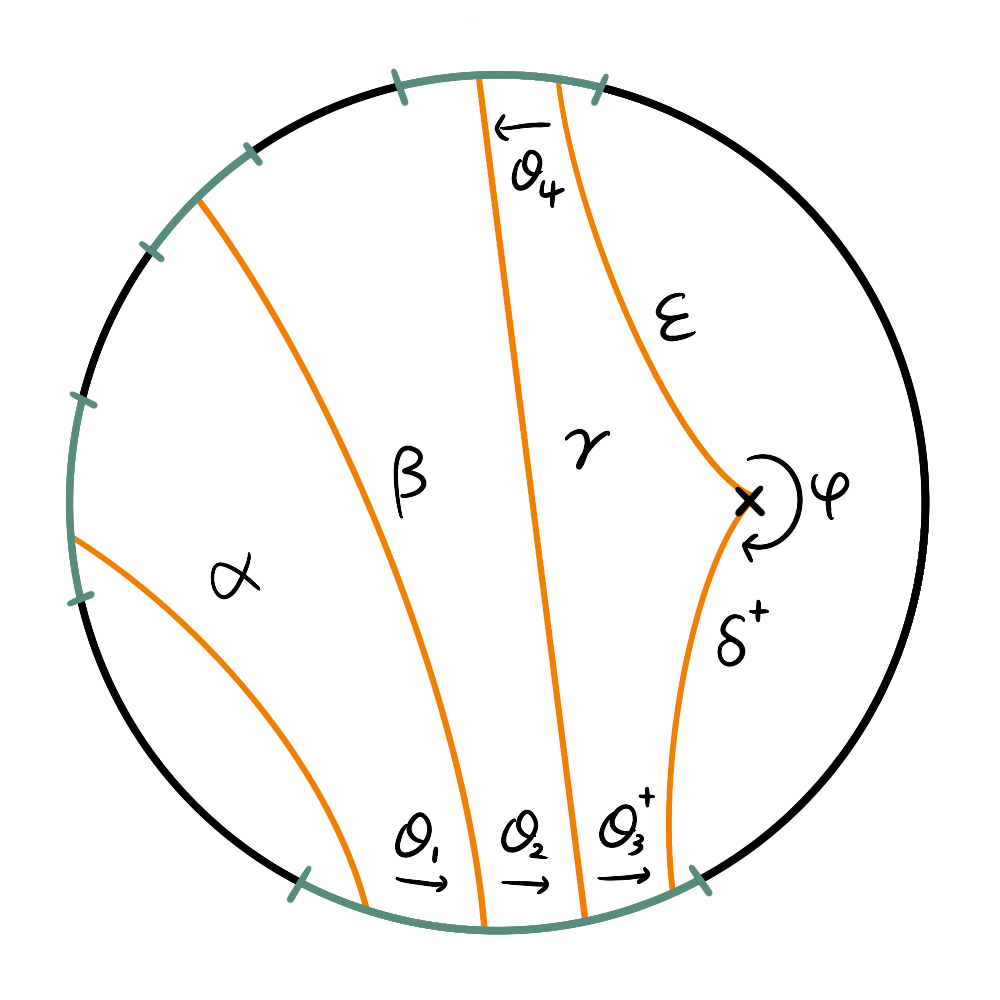}
			\caption{}
		\end{subfigure}
		\begin{subfigure}[b]{\mySize}
			\includegraphics[width=\linewidth]{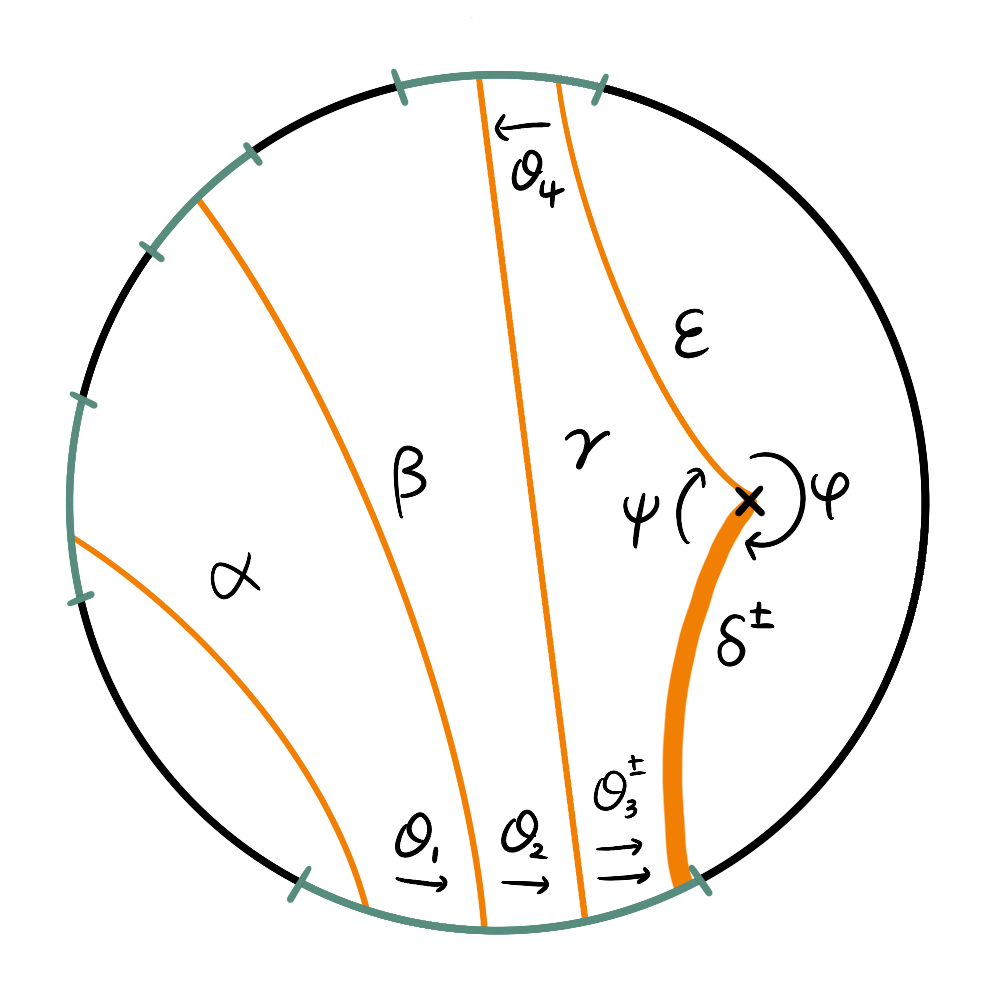}
			\caption{}
		\end{subfigure}
		\caption{Disc with three arc systems}
		\label{fig:IntroExample}
	\end{figure}
	 \begin{itemize}
	 \item[(a)]  $\Gamma_a$ consists of 5 tagged arcs where the pair $(\delta^+, \delta^-)$ is a thick pair. Namely $\delta^+$ and $\delta^-$ have the same underlying arc, but have opposite taggings at the interior marked point. There is a morphism $\theta_3^+$ (resp. $\theta_3^-$) from $\gamma$ to $\delta^+$ (resp. $\delta^-$). Morphisms can be concatenated via $\Aoo$-operation $\fm_2^\con$,
	 and the corresponding endomorphism algebra is the path algebra of the first $D_n$-quiver above. This is a skew-gentle quiver.
	 \item[(b)] $\Gamma_b$ consists of 5 tagged arcs as in the figure. We can choose a line field and grading so that we have an interior morphism $\phi$ from the arc $\epsilon$ to $\delta^+$. (There is no morphism from $\delta^+$ to $\epsilon$, and we say it is a non-morphism.) Three arcs $(\epsilon,\gamma,\delta^+)$ bound a disc but the corner at the interior marking is a non-morphism. In this case, we have a new $\Aoo$-operation $\fm^\comp$, called a composition sequence, and we have  $\fm_2^\comp(\theta_4, \theta_3^+) = \phi$. The corresponding endomorphism algebra is the path algebra of the second $D_n$-quiver above. Note that $\phi$ represents the path from $\epsilon$ to $\delta^+$. This second quiver is not skew-gentle.
	\item[(c)] $\Gamma_c$ consists of 6 tagged arcs. Now there are two interior morphisms $\psi\in \ho(\delta^-,\epsilon)$ and $\phi \in \ho(\epsilon, \delta^+)$.
	There is a disc sequence  $(\theta_3^-, \psi, \theta_4)$ which defines $\fm_3(\theta_3^-, \psi, \theta_4) = \frac{1}{2}e_{\gamma}$.

	\end{itemize}
Both $\Gamma_a$ and $\Gamma_b$ are Morita equivalent to $\Gamma_c$ from  Proposition \ref{prop:MoritaEquivalency}. Thus, we can deduce via geometry that two quivers above are derived equivalent.  Note that $\Aoo$-categories for (a) and (b) are formal and (c) is not formal.
\end{example}

Upon circulation of this work, we have learned that there is a similar work in preparation by Barmeier, Schroll, and Wang.

\section*{Organization}

In Section $2$, we recall the concept of topological Fukaya category following \cite{HKK17}. Then, we generalize the notion to $\super$-orbifold surface in Section $3$. In this section, we introduce tagged arc system, interior morphism, and disc and composition sequences. In Section $4$, we define $\Aoo$-category associated with tagged arc system and topological Fukaya category associated with $\super$-orbifold surface by introducing Theorem \ref{thm:Intro1}. In Section $5$, we prove Theorem \ref{thm:Intro2} and the first part of Theorem \ref{thm:Intro4}. In Section $6$, we define topological Fukaya category in involutive setting and prove Theorem \ref{thm:Intro3} and the second and third parts of Theorem \ref{thm:Intro4}. In Section $7$, we prove Theorem \ref{thm:Intro1}.

\section*{Convention}

We fix an algebraically closed field $\field$ with characteristic $0$. Any $\Aoo$-category in this paper is unital. We use the sign convention in \cite{Seidel08}, but we use the reverse order for composition. That is, for objects $X_0, \dots, X_n$, the $\Aoo$-structure map is a map $$\fm_n : \ho(X_0, X_1) \otimes \dots \otimes \ho(X_{n-1}, X_n) \rightarrow \ho(X_0,X_n).$$

\section*{Acknowledgment}
This work started as an attempt to understand the geometry behind the Burban and Drozd's theory of non-commutative curves in \cite{BD18Arx}, and we would like to thank Igor Burban for the explanations  and encouragements. We would like to thank Sibylle Schroll for explaining her work on skew-gentle algebras. We  also thank Kyungmin Rho for 
sharing his ideas and helpful discussions on this topic and other joint works.
\section{Topological Fukaya categories of surfaces}\label{section:TFCofSurfaces}

In this section, we recall the concept of topological Fukaya category mainly following \cite{HKK17}. Section \ref{subsection:MarkedSurfaces} and \ref{subsection:Grading} contain basic definitions and properties of graded boundary-marked surfaces and graded curves. Then, in Section \ref{subsection:TFCofSurfaces}, we review the notion of topological Fukaya category. For more details, see \cite{HKK17}.

\subsection{Boundary-marked surfaces}\label{subsection:MarkedSurfaces}

Let us consider a compact oriented surface $S$ with nonempty boundary $\partial S$. We set an orientation on $S$ \textbf{clockwise}. Also, we use the induced orientation on $\partial S$ as the clockwise rotation of inward vector field.

A {\em boundary marking} on $S$ is an orientation preserving embedding $M : \coprod_{i=1}^r I_i \rightarrow S$ from a disjoint union of finite number of unit intervals with the canonical orientation. We call both the restriction $M_i = M|_{I_i}$ and its image $\im(M_i)$ a {\em marking}. The pair $(S, M)$ is called a {\em boundary-marked surface} if each boundary component of $S$ contains at least one marking.

A {\em curve} on $(S, M)$ is an immersion $\gamma : (I, \partial I) \rightarrow (S, M)$. In particular, $M_i$ itself and a restriction of it are curves and we call them {\em boundary curves}. We say two boundary curves $\theta_1$ and $\theta_2$ are {\em concatenable} when $\theta_1(1) = \theta_2(0)$, and denote the {\em concatenation} by $\theta_1 \bullet \theta_2$.

We say two curves are {\em isotopic} to each other when they are homotopic as maps from $(I, \partial I)$ to $(S, M)$. An {\em arc} is an embedded curve $\gamma : I \rightarrow S$ such that $\gamma$ intersects $M$ transversely, $\gamma^{-1}(M) = \partial I$, and $\gamma$ is not isotopic to a boundary curve. By abuse of notation, we also denote by $\gamma$ the image of the arc $\gamma$. An {\em arc system} is a collection of arcs $\Gamma = \{\gamma_1, \dots, \gamma_n\}$ satisfying the following two conditions.
\begin{itemize}
	\item For each pair $i\neq j$, $\gamma_i \cap \gamma_j = \emptyset$ and $\gamma_i \not\simeq \gamma_j$.
	\item Each component of $S \setminus \bigcup_{i=1}^n\gamma_i$ is a topological disc with at most one unmarked boundary component.
\end{itemize}

\begin{example}\label{example:TheDisc}
	Let us consider the $2$-dimensional disc $D$ with $n$-boundary markings $M_n$ and an arc system consisting of $n$-arcs which are isotopic to unmarked boundary components. The case of $n=4$ is illustrated in Figure \ref{fig:DiscWithArcSystem}(a). Such a disc with arc system can always be made into the canonical configuration, as in Figure \ref{fig:DiscWithArcSystem}(b), up to isotopy. Although it is not an arc system in the rigorous sense as each arc is on the boundary, it is useful to regard it as an boundary-marked surface when we define $\Aoo$-structure for topological Fukaya category. Let us denote the canonical $n$-boundary marking, the arcs, and the boundary morphisms by $M^\can_n$, $\delta_1, \dots, \delta_n$, and $\theta^\can_1, \dots, \theta^\can_n$, \resp.
	\begin{figure}[h!]
		\centering
		\begin{subfigure}[b]{0.3\linewidth}
			\includegraphics[width=\linewidth]{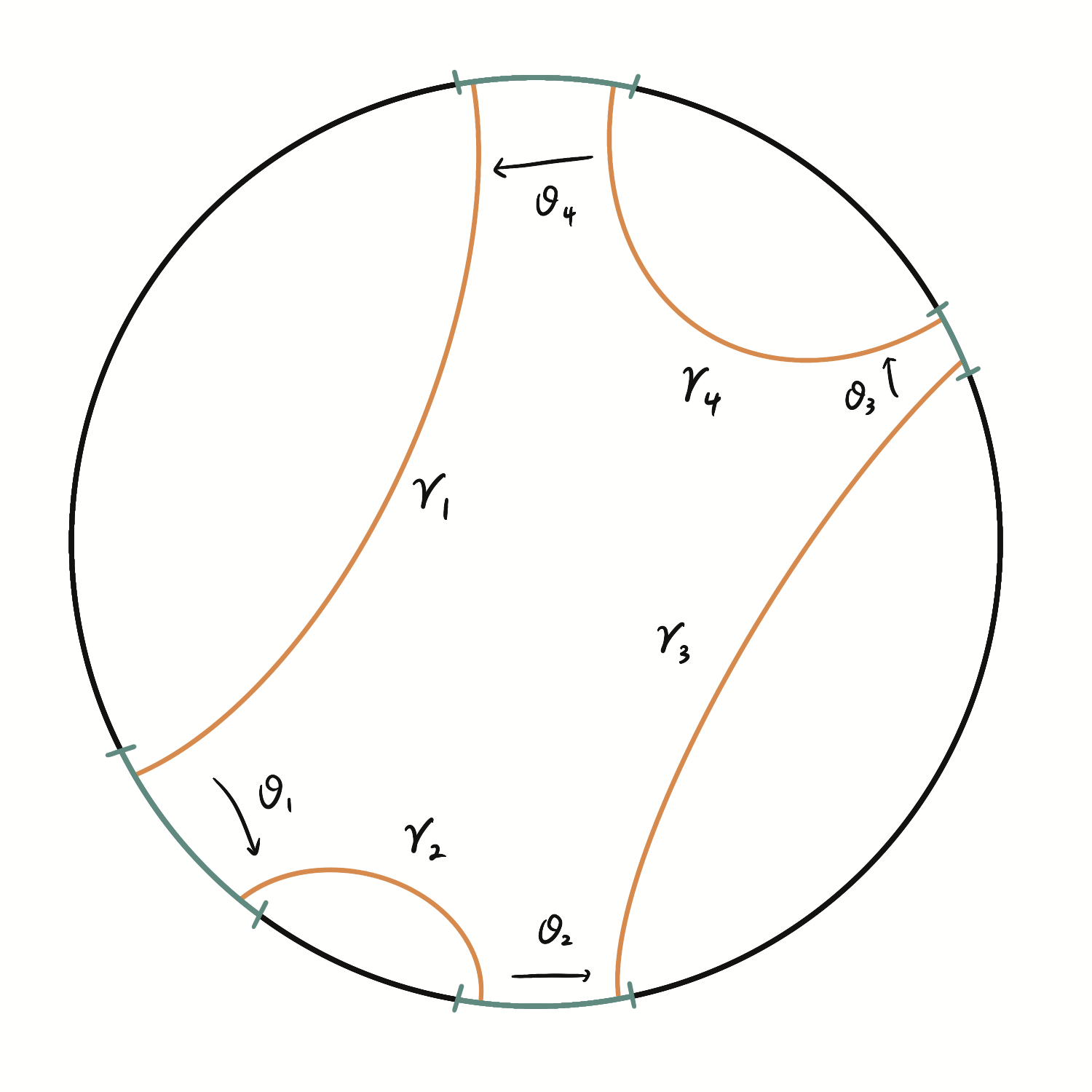}
			\caption{Non-canonical arc system}
		\end{subfigure}
		\begin{subfigure}[b]{0.3\linewidth}
			\includegraphics[width=\linewidth]{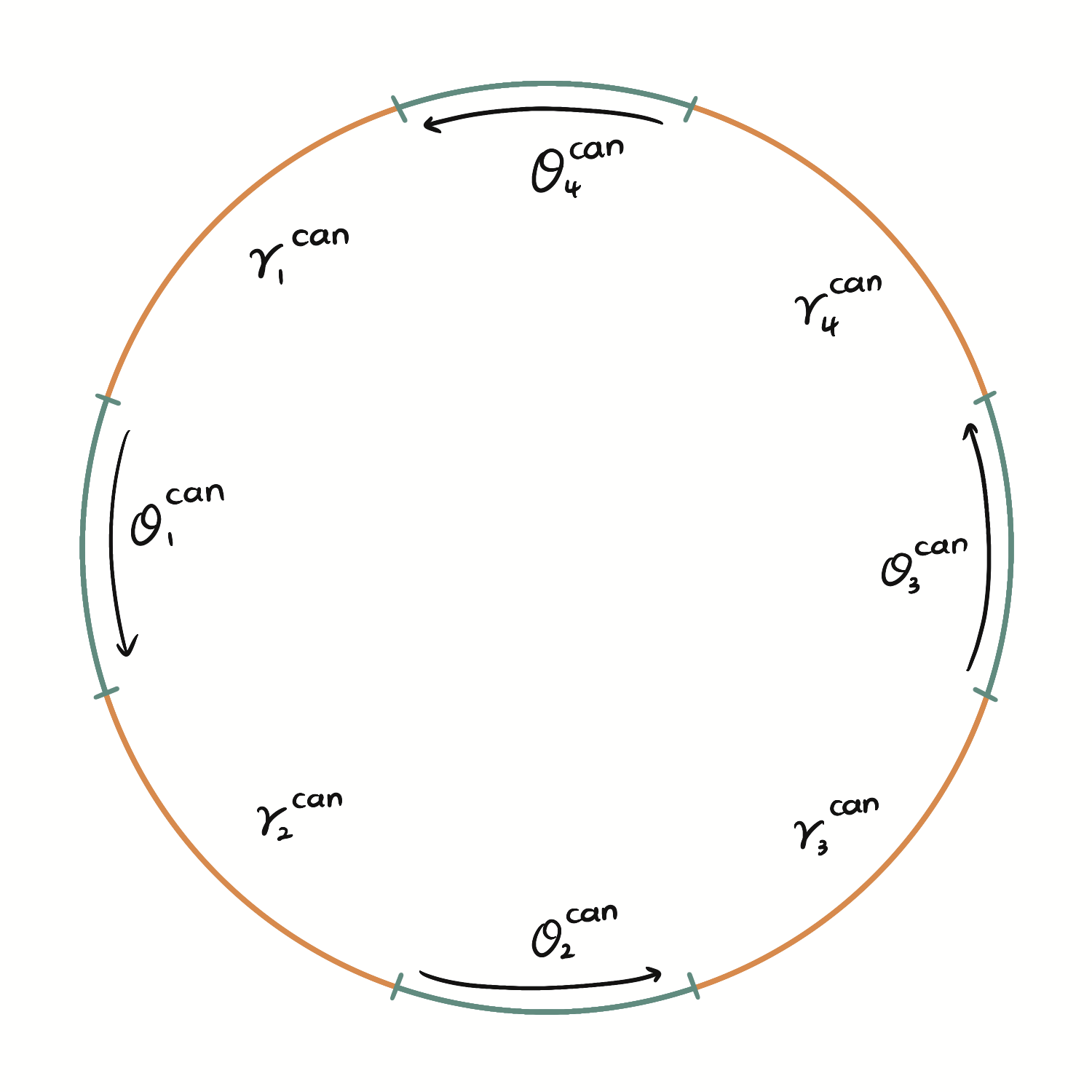}
			\caption{Canonical arc system}
		\end{subfigure}
	\caption{Disc with arc systems}
	\label{fig:DiscWithArcSystem}
	\end{figure}
\end{example}

Let $(S_1, M_1)$ and $(S_2, M_2)$ be two boundary-marked surfaces. A {\em morphism} from $(S_1, M_1)$ to $(S_2, M_2)$ is an orientation preserving immersion $f : S_1 \rightarrow S_2$ such that $f(M_1) \subseteq M_2$. For a curve $\gamma : (I, \partial I) \rightarrow (S_1, M_1)$, the morphism $f$ induces the curve $f_*\gamma$ on $(S_2, M_2)$. Also $f$ sends a boundary morphism $\theta$ on $M_1$ to the boundary morphism $f_*\theta$ on $M_2$. Now let us recall the definition of disc sequence.
\begin{defn}\label{defn:DiscSequence}
	Let $(S, M)$ be a boundary-marked surface and $\gamma_1, \dots, \gamma_n$ be arcs on it. A {\em disc sequence} is a sequence $(\theta_1, \dots, \theta_n)$ of boundary morphisms $\theta_i : \gamma_i \rightarrow \gamma_{i+1}$, where $\gamma_{n+1} \deq \gamma_1$, such that there is a morphism from $(D, M^\can_n)$ to $(S, M)$ satisfies the following two conditions.
	\begin{itemize}
		\item For each $i$, $f_*\gamma^{\can}_i = \gamma_i$.
		\item For each $i$, $f_*(\theta^{\can}_i) = \theta_i$.
	\end{itemize}
\end{defn}

\subsection{Gradings}\label{subsection:Grading}
The {\em projective tangent bundle} $\bpp(TS)$ is a fiber bundle on $S$ whose fiber at $p$ is the projective space $\bpp(T_pS)$. The orientation on $S$ induces an orientation on $\bpp(T_pS)$ for each $p \in S$. So we have the preferred generator $\omega_p$ of $\pi_1(\bpp(T_pS)) = \bzz$. For distinct $x, y \in \bpp(T_pS)$, we denote the shortest clockwise path from $x$ to $y$ by $\kappa^x_y$.

A {\em line field} is a smooth global section of $\bpp(TS)$. Two line fields $\eta$ and $\eta'$ are said to be {\em equivalent} if there is a path from $\eta$ to $\eta'$ in the space of line fields. For example, any line fields on a contractible surface are all equivalent. A {\em graded boundary-marked surface} is a triple $(S, M, \eta)$ of a surface $S$, a boundary marking $M$, and a line field $\eta$ on $S$.

Let $\gamma : I \rightarrow S$ be a curve and consider the pullback bundle $\gamma^*(\bpp(TS))$ on $I$. Then, we have two sections $\gamma^*(\eta) \deq \eta \circ \gamma$ and $\dot{\gamma}$. A {\em grading} of $\gamma$ is a homotopy class of paths $\hat{\gamma}$ from $\gamma^*(\eta)$ to $\dot{\gamma}$ in the space of sections of $\gamma^*(\bpp(TS))$. A {\em graded curve} is a pair $(\gamma, \hat{\gamma})$. We often write the graded curve by $\gamma$ omitting the grading $\hat{\gamma}$. When the curve is an arc, we call it a {\em graded arc}.

Let $(\alpha, \hat{\alpha})$ and $(\beta, \hat{\beta})$ be graded curves intersecting at $p = \alpha(t) = \beta(s)$ transversely for $t, s\in I$. Then, the {\em intersection index} is given by the winding number of the following loop.
$$\hat{\alpha}(t) \cdot \kappa^{\dot{\alpha}(t)}_{\dot{\beta}(s)} \cdot (\hat{\beta}(s))^{-1} \in \pi_1(\bpp(T_pS), \eta_p) \cong \bzz.$$
The following lemma is immediate from the definition so we omit the proof.
\begin{lemma}\label{lemma:SumOfIntersectionIndex}
	Let $(\gamma, \hat{\gamma})$ be the third graded curve such that $\gamma(u) = p$ for some $u\in I$. Also assume that $\dot{\alpha}(s), \dot{\beta}(t)$, and $\dot{\gamma}(u)$ are all distinct. Then,
	$$i_p(\alpha, \beta) + i_p(\beta, \gamma) = \begin{cases} i_p(\alpha, \gamma) &\text{ if $(\dot\alpha(s), \dot\beta(t), \dot\gamma(u))$ is clockwise}, \\ i_p(\alpha, \gamma) + 1 &\text{ if $(\dot\alpha(s), \dot\beta(t), \dot\gamma(u))$ is counterclockwise.} \end{cases}$$
\end{lemma}

A {\em shift} of a graded arc $(\alpha, \hat{\alpha})$ is given by composing $\omega_p$ with $\hat{\alpha}$ pointwisely. More precisely, for an integer $n\in \bzz$, the grading of the $n\tth$-shift $\alpha[n]$ is given by $\hat{\alpha} \cdot \omega^n$, where $(\hat{\alpha} \cdot \omega)_p = \hat{\alpha}(p) \cdot (\omega_p)^n$. For another graded curve $\beta$ intersecting with $\alpha$ transversely at $p$, we have $$i_p(\alpha[n], \beta[m]) = i_p(\alpha, \beta) + n - m.$$

Now let us recall the notion of boundary morphism. Let $\alpha$ and $\beta$ be graded arcs in $(S, M, \eta)$ with endpoints $p$ and $q$, \resp, on the same marking so that there is a boundary path $\theta^p_q$. We call $\theta^p_q$ a {\em boundary morphism} from $\alpha$ to $\beta$ and its {\em degree} is given by $$\deg{\theta^p_q} \deq i_p(\alpha, \theta^p_q) - i_q(\beta, \theta^p_q).$$
Note that the degree $\deg{\theta^p_q}$ is independent of grading of $\theta^p_q$. We denote by $\Theta(\alpha, \beta)$ the set of all boundary morphisms from $\alpha$ to $\beta$.

\begin{lemma}\label{lemma:DegreeConditionForDiscSequence}
	Let $(\theta_1, \dots, \theta_n)$ be a disc sequence. Then, we have $$\deg{\theta_1} + \dots + \deg{\theta_n} = n-2.$$
\end{lemma}
\begin{proof}
	Let $f : (D, M^\can_n) \rightarrow (S, M)$ be an immersion associated with the disc sequence. Then, the pull-backed section $f^*\eta$ defines a line field on the disc $D$. Also, one can prove that intersection index doesn't change if we perturb the line field or arcs isotopically. Since any bundle on $D$ is trivial, we may assume the line field on $D$ is given by foliation of parallel straight lines. We may further assume the arcs are straight lines. Then, the degree of boundary morphisms are nothing but the inner angles between lines divided by $\pi$. So the sum of degrees is the sum of inner angles of $n$-gon divided by $\pi$, which is exactly $(n-2)$.
\end{proof}

A {\em graded arc system} on $(S, M, \eta)$ is a collection of graded arcs $\Gamma = \{\gamma_1, \dots, \gamma_n\}$ such that the underlying arcs form an arc system on $(S, M)$. If it has no disc sequence, then we say the system {\em formal}. More precisely, formal arc system is defined as follows.
\begin{defn}\label{defn:FormalArcSystem}
	Let $\Gamma = \{\gamma_1, \dots, \gamma_n\}$ be a graded arc system on $(S, M, \eta)$. We say $\Gamma$ is {\em formal} if each disc component of $$S \setminus \bigcup_{i=1}^n\gamma_i$$ has exactly one unmarked boundary component.
\end{defn}

\subsection{Topological Fukaya categories of graded boundary-marked surfaces}\label{subsection:TFCofSurfaces}

A graded arc system $\Gamma$ on a graded boundary-marked surface $(S, M, \eta)$ gives rise to an $\Aoo$-category $\cff_\Gamma(S, M, \eta)$. It turns out that the Morita equivalence class of it is independent of the choice of $\Gamma$. This gives the {\em topological Fukaya category} of $(S, M, \eta)$.

First let us recall the definition of the associated $\Aoo$-category.
\begin{defn}\cite[Section 3.3]{HKK17}
	Let $(S, M, \eta)$ be a graded boundary-marked surface and $\Gamma$ be a graded arc system. Then, an $\Aoo$-category $\cff_\Gamma(S, M, \eta)$ consists of the following data.
	\begin{itemize}
		\item The set of objects is $\Gamma$.
		\item The basis for morphism space consists of boundary morphisms and the unit. More precisely, for two graded arcs $\alpha, \beta \in \Gamma$,
		$$\ho_{\cff_\Gamma(S, M, \eta)}(\alpha, \beta) \deq \begin{cases}  \field\left<\Theta(\alpha, \alpha)\right>\oplus \field\left<e_{\alpha}\right> &\text{if $\alpha = \beta$}, \\\field\left<\Theta(\alpha, \beta)\right> &\text{if $\alpha \neq \beta$.} \end{cases}$$
		\item The differential $\fm_1$ is set to be zero.
		\item For two concatenable boundary morphisms $\theta_1$ and $\theta_2$, $\fm_2(\theta_1, \theta_2) \deq (-1)^{\deg{\theta_1}}\theta_1\bullet \theta_2$.
		\item Let $(\theta_1, \dots, \theta_n)$ be a disc sequence with $\theta_i : \gamma_i \rightarrow \gamma_{i+1}$, where $\gamma_{n+1}$ is set to be $\gamma_1$. Also, let $\phi : \alpha \rightarrow \gamma_1$ and $\psi : \gamma_n \rightarrow \beta$ be boundary morphisms concatenable with $\theta_1$ and $\theta_n$, \resp. Then,		
		\begin{align*}
			\fm_n(\theta_1, \dots, \theta_n) &\deq e_{\gamma_1}, \\
			\fm_n(\phi \bullet \theta_1, \dots, \theta_n) &\deq (-1)^{\deg{\phi}}\phi, \\
			\fm_n(\theta_1, \dots, \theta_n \bullet \psi) &\deq \psi.
		\end{align*}
		\item The other $\fm_n$'s are all zero.
	\end{itemize}
\end{defn}

\begin{figure}[h!]
	\centering
	\begin{subfigure}[b]{0.3\linewidth}
		\includegraphics[width=\linewidth]{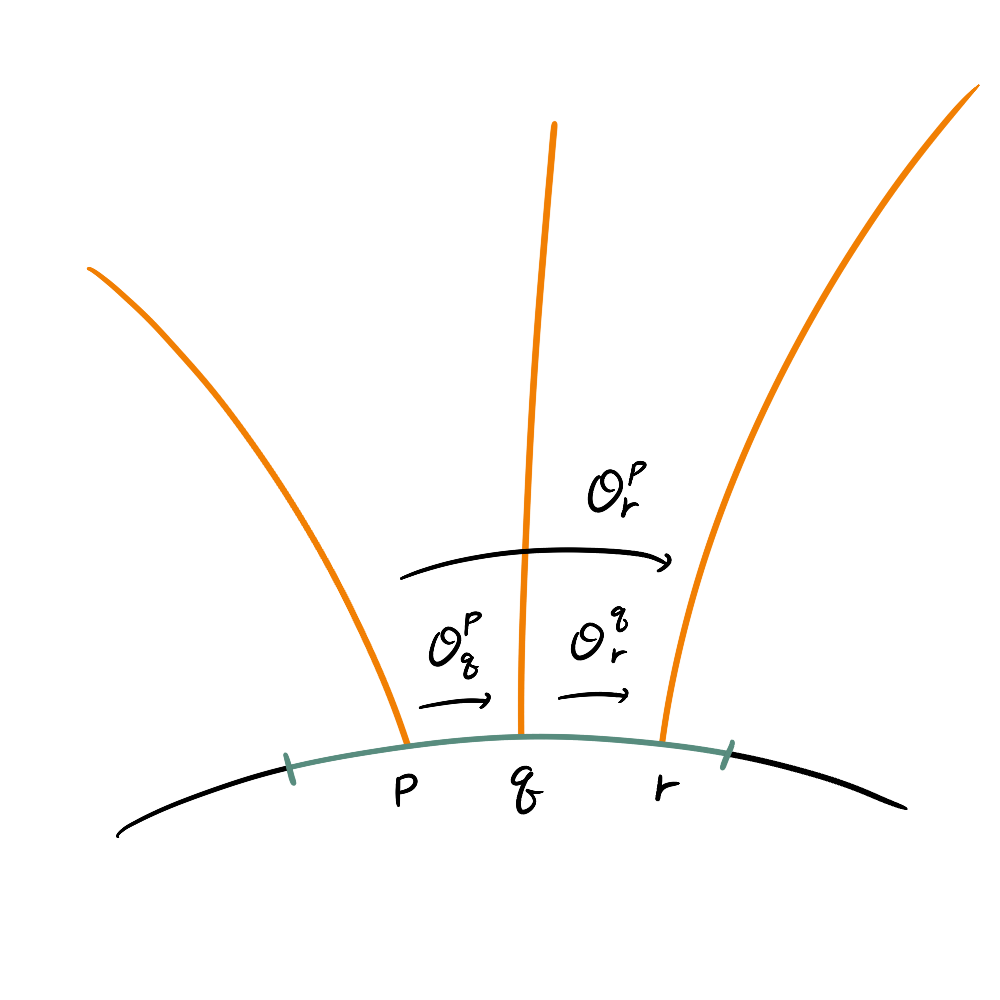}
		\caption{Composition of boundary morphisms}
	\end{subfigure}
	\begin{subfigure}[b]{0.3\linewidth}
		\includegraphics[width=\linewidth]{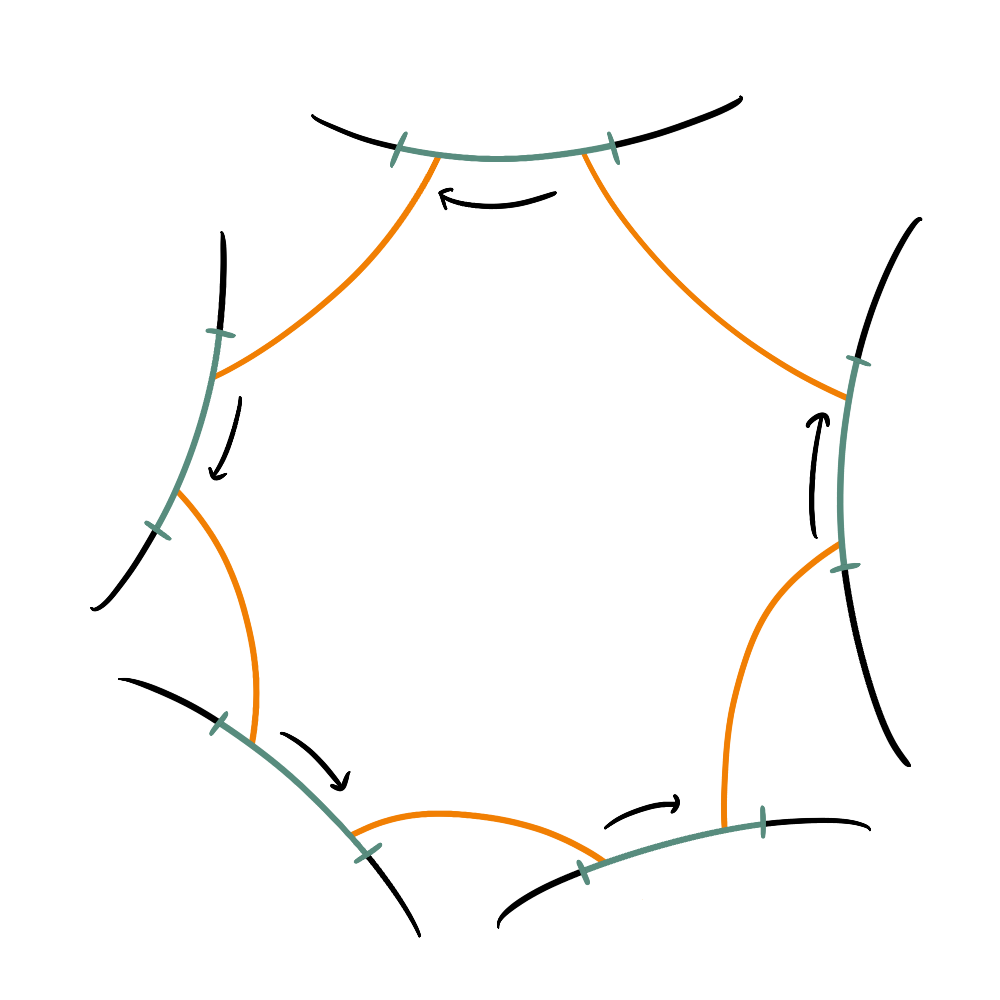}
		\caption{Disc sequence}
	\end{subfigure}
	\caption{$\Aoo$-structure for topological Fukaya category of surface}
	\label{fig:AooForSurface}
\end{figure}

\begin{thm}\cite[Lemma 3.2 and Proposition 3.3]{HKK17}\label{Theorem:MoritaEquivalencyOfTopologicalFukayaCategoryForSurfaces}
	Let $(S, M, \eta)$ be a graded boundary-marked surface. Then the following hold.
	\begin{enumerate}
		\item For two graded arc system $\Gamma_1 \subseteq \Gamma_2$, the inclusion functor $$\cff_{\Gamma_1}(S, M, \eta) \rightarrow \cff_{\Gamma_2}(S, M, \eta)$$ induces a quasi-equivalence $\Aoo$-functor $$\Tw(\cff_{\Gamma_1}(S, M, \eta)) \rightarrow \Tw(\cff_{\Gamma_2}(S, M, \eta)).$$
		\item The Morita equivalence class of $\cff_\Gamma(S, M, \eta)$ independents of $\Gamma$.
	\end{enumerate}
\end{thm}

\begin{defn}\label{definition:TopologicalFukayaCategory}
	The {\em topological Fukaya category} of a graded boundary-marked surface $(S, M, \eta)$ is the $\Aoo$-category $$\cff(S, M, \eta) \deq \Tw(\cff_\Gamma(S, M, \eta))$$ for a graded arc system $\Gamma$.
\end{defn}

\section{Tagged arc systems}\label{section:TFCofOrbiSurfaces}
In this section, we introduce the tagged arc system and  generalize Sections \ref{subsection:MarkedSurfaces} and \ref{subsection:Grading} to this setting.
The concept of tagged arc is first introduced in \cite{FST08}. It is a graded curve each of whose endpoints is either a boundary marking or an interior marking,
with a choice of an element in $\super = \{0, 1\}$ at each interior marking endpoint. 

\subsection{Graded marked $\super$-orbi-surfaces and tagged arcs}\label{subsection:TaggedArcs}
\begin{defn}
	A {\em marked surface} is a triple $(S, M, O)$, where $(S, M)$ is a boundary-marked surface and $O$ is a finite subset of $\operatorname{Int}(S)$ whose elements are called {\em interior markings}.
\end{defn}
We define arcs on a marked surface as follows.
\begin{defn}
	An {\em arc} is an embedding $\gamma : (I, \partial I) \rightarrow (S, M \cup O)$ such that $\gamma^{-1}(M \cup O) = \partial I$ and it is not isotopic to a boundary path nor a constant path. We denote the set of interior endpoints of $\gamma$ by $O(\gamma)$ and define the {\em interior number} $\nu(\gamma)$ as the cardinality of $O(\gamma)$.
For an interior marking $p \in O(\gamma)$, we say $\gamma$ is {\em hanging} at $p$. 
\end{defn}

In order to define the notion of  grading to a marked surface and arcs, we regard the surface $S$ as an orbifold surface, by equipping 
the points in $O$ with $\super$-orbifold structure. Namely, a neighborhood of any point in $O$ in $S$ admits a uniformizing cover with $\super$-action.
Thus $(S,O)$ is  a $\super$-orbifold surface with boundary. 
It is well-known that $(S,O)$ admits a 2-fold branched covering $$\pi : \tilde{S} \rightarrow S $$
by a smooth surface $\tilde{S}$ with boundary whose ramification locus is $O$. 
We fix $\tilde{S}$ and $\pi$ from now on.

Denote by $\iota : \tilde{S} \rightarrow \tilde{S}$ the deck transformation.
We have $\iota^2=\id_{\tilde{S}}$ and hence $\iota$ is an involution. 
The inverse image $\pi^{-1}(M)$ is an invariant boundary marking on $\tilde{S}$ and we denote it by $\tilde{M}$.
Let  $\eta$ be an $\iota$-invariant line field on $\tilde{S}$.
It is not hard to see that such $\eta$ exists: in a uniformizing cover at $\super$-points, line fields may be given by horizontal or vertical ones.
We call these data, abbreviated by the tuple $(S, O, M, \eta)$, a {\em graded marked $\super$-orbi-surface}, or more simply, a {\em graded marked orbi-surface}.

We define a {\em grading} on an arc $\gamma$ as a grading $\hat{\gamma}$ on the lift  $\tilde{\gamma}$ that is compatible with $\super$-action.
Let us explain what this means for each type of arc.
\begin{itemize}
	\item $\nu(\gamma) = 0$. Then, $\gamma$ does not pass through orbifold points in $O$. It has two disjoint lifts $\tilde{\gamma}$ and $\iota_*\tilde{\gamma}$, which are arcs on $(\tilde{S}, \tilde{M})$. Note that $\tilde{\gamma}$ is gradable since it is an arc, this induces a grading on $\iota_*\tilde{\gamma}$.
	\item  $\nu(\gamma) = 1$. Then, $\pi^{-1}(\gamma)$ is a one dimensional $\super$-invariant submanifold on $\tilde{S}$ with boundary on $\tilde{M}$. We choose a parametrization $\tilde{\gamma} : [0,2] \rightarrow \tilde{S}$ so that
	$\iota \circ \tilde{\gamma}(t) = \tilde{\gamma}(2-t)$ and hence $\pi \circ \tilde{\gamma}(1) \in O(\gamma)$.
	Also, $\tilde{\gamma}$ and $\iota_*\tilde{\gamma}$ can be graded as before.
	 	\item  $\nu(\gamma) = 2$. Then, $\pi^{-1}(\gamma)$ is a $\super$-invariant circle on $\tilde{S}$.
	We choose a parametrization $\tilde{\gamma}:\mathbb{R}/2\mathbb{Z} \rightarrow \tilde{S}$ so that
	$\iota \circ \tilde{\gamma}(t) = \tilde{\gamma}(2-t)$. 
\begin{lemma}
	Let $\gamma$ be an arc on $S$ with $\nu(\gamma) = 2$. Then, its lift $\tilde{\gamma}$ is gradable.
\end{lemma}
\begin{proof}
	To show the loop is gradable, we have to find a homotopy between $\dot{\tilde{\gamma}}$ and $\eta$ along $\tilde{\gamma}$. Since $\tilde{\gamma}(2-t) = (\iota \circ \tilde{\gamma})(t)$ and $\eta$ is $\iota$-invariant, we have $$\dot{\tilde{\gamma}}(2-t) = \iota_*\dot{\tilde{\gamma}}(t), \quad \eta_{\tilde{\gamma}(2-t)} = \iota_*\eta_{\tilde{\gamma}(t)}.$$ Let $\alpha$ be any homotopy from $\dot{\tilde{\gamma}}|_{[0, 1]}$ to $\eta|_{[0, 1]}$. Then, define $\hat{\alpha}$ so that $\hat{\alpha}|_{[0, 1]} = \alpha$ and $\hat{\alpha}(2-t) = \hat{\alpha}(t)$. This gives a homotopy from $\dot{\tilde{\gamma}}$ to $\eta$.
\end{proof}
\end{itemize}
For each arc $\gamma$, we fix a lift $\tilde{\gamma}$ as above.
The parametrization of $\tilde{\gamma}$ provides an orientation of $\tilde{\gamma}$. 
Note that in the cases of $\nu(\gamma) =1,2$, $\tilde{\gamma}$ and $\iota \circ \tilde{\gamma}$ has the same underlying curve but have opposite orientations.
 
Now we define tagging of a graded arc as a choice of an element of $\super$ at each interior endpoint.
\begin{defn}
	Let $(\gamma, \hat{\gamma})$ be a graded arc on $S$.  A {\em tagging} on $\gamma$ is a function $\tau : O(\gamma) \rightarrow \super$. 
	If $\tau(p) =1$ (resp. $\tau(p)=0$),  $\; \gamma$ is said to be  {\em notched} (resp. {\em plain}) at $p$. 
		A {\em tagged arc} is a triple $(\gamma, \hat{\gamma}, \tau)$. For a tagged arc $(\gamma, \hat{\gamma}, \tau)$ and an integer $n \in \bzz$, its {\em $n\tth$-shift} is $$(\gamma, \hat{\gamma}, \tau)[n] \deq (\gamma, \hat{\gamma}[n], \tau+n).$$
\end{defn}
We often simply write a tagged arc as $\gamma$ omitting the grading $\hat{\gamma}$ or the tagging $\tau$.
A notched arc is indicated by a small stick near $p$ as in other literature (see Figure \ref{fig:InteriorMorphism}).

We will show in Section \ref{section:TFCofInvolutiveSurfaces} that tagging can be understood geometrically as a choice of 
{\em idempotent} in the Fukaya category of $\super$-orbifold surfaces.

\subsection{Interior morphisms}\label{subsection:InteriorMorphisms}
We define morphisms between tagged arcs intersecting at the interior marking.
In the standard Floer theory, if two curves $L_1$ and $L_2$ transversely intersect at a point $p$, then 
this gives rise to two generators $\alpha_p \in CF(L_1,L_2), \overline{\alpha}_p \in CF(L_2,L_1)$.

The distinctive feature in our tagged setting will be that when two tagged arcs intersect at the interior marking $p$,
among two seemingly possible generators at $p$, only one of them will be declared to be a morphism, and this choice
depends on the degree of the intersection and the data of tagging.
A geometric reason for this phenomenon will be explained in Section \ref{section:TFCofInvolutiveSurfaces}:
roughly speaking the data of tagging is a choice of an idempotent, and thus morphisms only exist between the correct pair of
idempotents. Also, in our case, exactly one of the two possible generators has the correct pair. 

In this section, we give its combinatorial definition using the degree and tagging data as follows.
Let $(\alpha,\sigma)$ and $(\beta,\tau)$ be tagged arcs  hanging at $p \in O$. 
Their lifts $\tilde{\alpha}$ and $\tilde{\beta}$ intersect at $\tilde{p} = \pi^{-1}(p) \in \tilde{S}$ giving two morphisms:
$$p^\alpha_\beta \in CF(\tilde{\alpha},\tilde{\beta}), p_\alpha^\beta \in CF(\tilde{\beta},\tilde{\alpha}).$$
The intersection index 
 $i_{\tilde{p}}(\tilde{\alpha}, \tilde{\beta})$ from the gradings is independent of the choice of lifts, and we denote it as $i_p(\alpha, \beta)$.
 
We will discard the morphism  $p^\alpha_\beta$ from $(\alpha,\sigma)$ to $(\beta,\tau)$ if 
$$ \sigma(p) -  \tau(p) \neq i_p(\alpha,\beta).$$
Then, one may observe that exactly one out of $\{p^\alpha_\beta,  p_\alpha^\beta\}$ is
discarded from $i_p(\alpha, \beta) = 1 - i_p(\beta,\alpha)$.
  
\begin{defn}\label{defn:intmor}
	Let $(S, O, M, \eta)$ be a graded marked orbi-surface and $(\alpha, \sigma), (\beta, \tau)$ be tagged arcs on $S$ hanging at $p \in O$. When $\sigma(p)  -\tau(p)=i_p(\alpha, \beta)$, we say $p$ defines an {\em interior morphism} from $(\alpha, \sigma)$ to $(\beta, \tau)$ and we denote it by $p^{\alpha, \sigma}_{\beta, \tau}$. Its {\em degree} is given by $$\deg{p^{\alpha, \sigma}_{\beta, \tau}} \deq i_p(\alpha, \beta).$$
\end{defn}

Here, we give an example. Figure \ref{fig:InteriorMorphism}(b) illustrates possible choices of taggings of $\alpha$ and $\beta$ and the corresponding interior morphisms. The opposite morphism has been discarded, and will not be considered as a morphism. (In this example, as in Figure \ref{fig:InteriorMorphism}(a), we assume that $\eta$ is an oriented line field  of $\tilde{S}$, thus the parity of $i_p(\alpha,\beta)$ agrees with signed intersection number  of $\tilde{\alpha}$ and $\tilde{\beta}$ at $p$.)

\begin{figure}[h!]
	\centering
	\begin{subfigure}[b]{0.4\linewidth}
		\includegraphics[width=\linewidth]{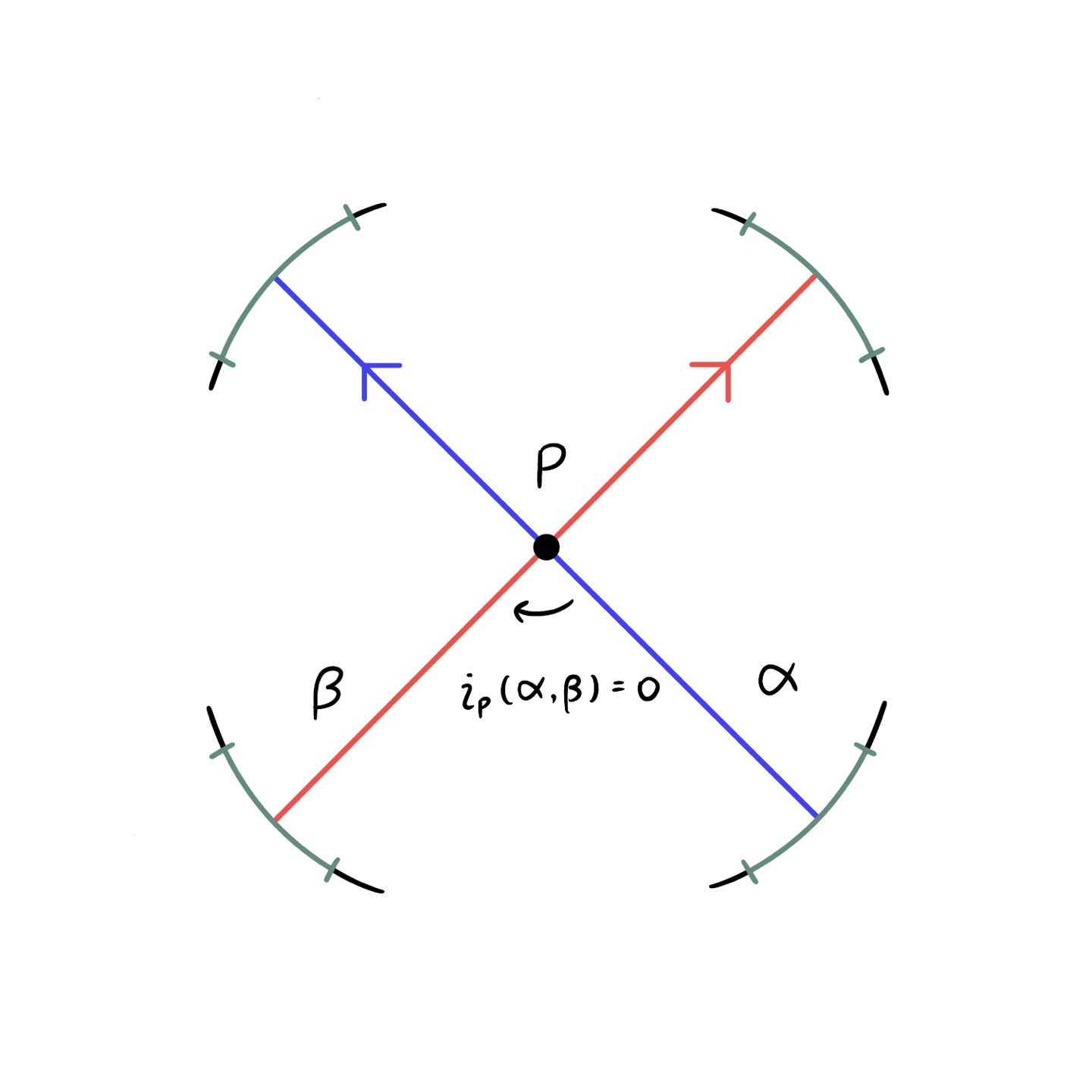}
		\caption{Intersection morphism on the covering}
	\end{subfigure}
	\begin{subfigure}[b]{0.4\linewidth}
		\includegraphics[width=\linewidth]{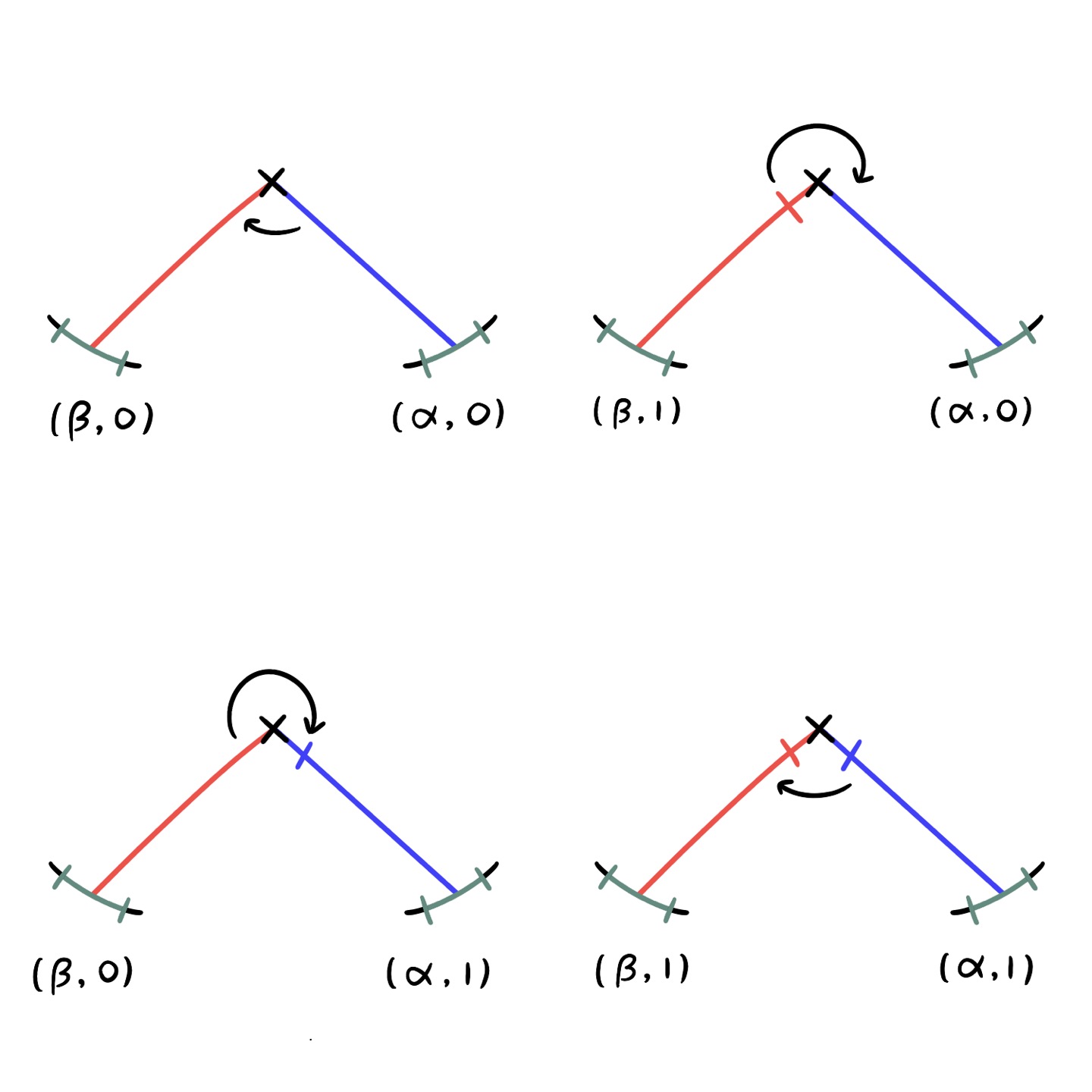}
		\caption{Interior morphisms on the orbi-surface}
	\end{subfigure}
	\caption{Interior morphisms}
	\label{fig:InteriorMorphism}
\end{figure}

\begin{lemma}\label{lemma:InteriorMorphismProperties}
	Suppose that an interior marking $p$ defines an interior morphism from $(\alpha, \sigma)$ to $(\beta, \tau)$. Then the following hold.
	\begin{enumerate}
		\item The intersection $p$ does not define an interior morphism from $(\beta, \tau)$ to $(\alpha, \sigma)$.
		\item For any $m, n \in \bzz$, the intersection $p$ defines an interior morphism from $(\alpha, \sigma)[m]$ to $(\beta, \tau)[n]$.
	\end{enumerate}
\end{lemma}
\begin{proof}
	These follow from $i_p(\beta, \alpha) = 1-i_p(\alpha, \beta)$ and $i_p(\alpha[m], \beta[n]) = i_p(\alpha, \beta) + m - n$. 
\end{proof}

\subsection{Disc and composition sequences}\label{subsection:DiscAncCompositionSequences}
For the $\Aoo$-structure of tagged arcs, we introduce the notion of disc and composition sequences.
Let us first adapt the notion of disc sequence  for tagged arcs.
We will allow two new phenomena. One is to allow interior morphisms between two tagged arcs as a 
part of a disc sequence (see Figure \ref{fig:AooForOrbiSurface}(a)). Also, we allow some of the boundary arcs of a disc to be folded and mapped to the double of tagged arcs (see Figure \ref{fig:AooForOrbiSurface}(b)).
Let us formally define it using the canonical arc system of the disc $(D, M^\can_n)$ in Example \ref{example:TheDisc}.
\begin{figure}[h!]
	\centering
	\begin{subfigure}[b]{0.3\linewidth}
		\includegraphics[width=\linewidth]{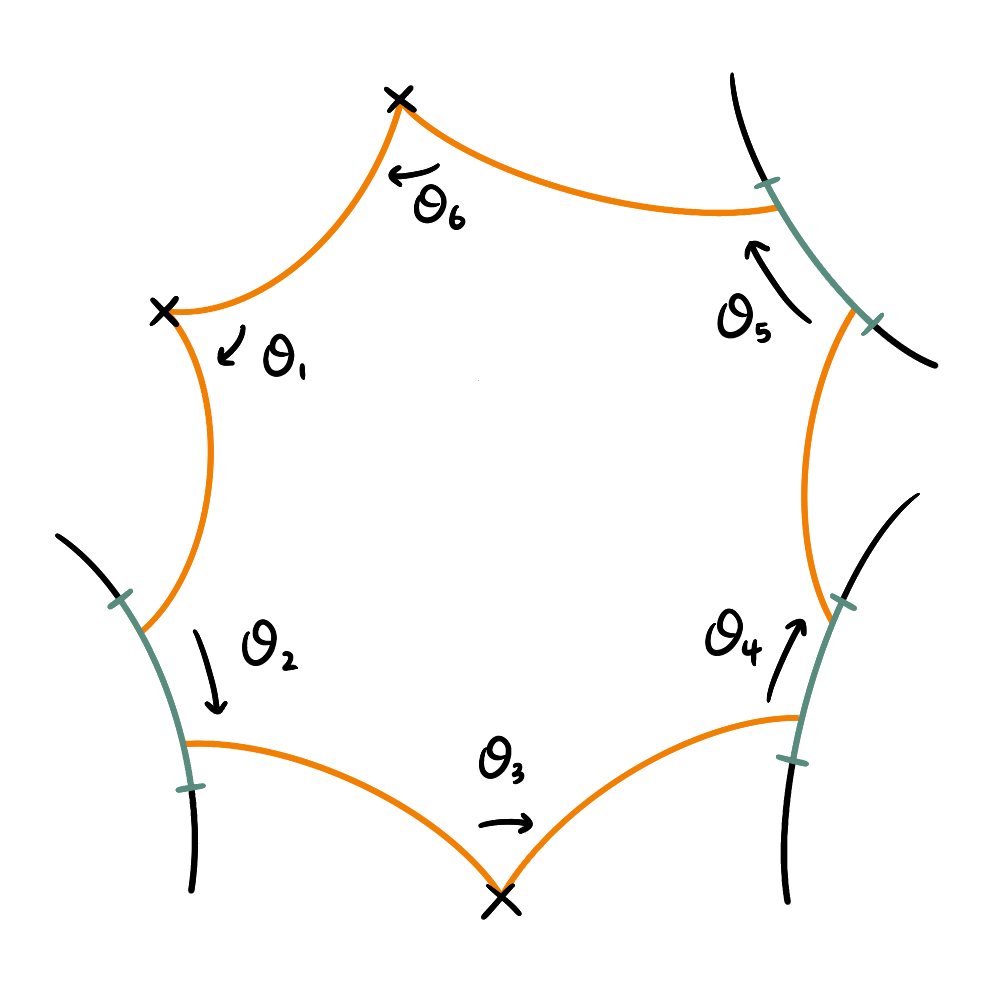}
		\caption{Disc sequence}
		\label{fig:AooForOrbiSurface:DiscSequence}
	\end{subfigure}
	\begin{subfigure}[b]{0.3\linewidth}
		\includegraphics[width=\linewidth]{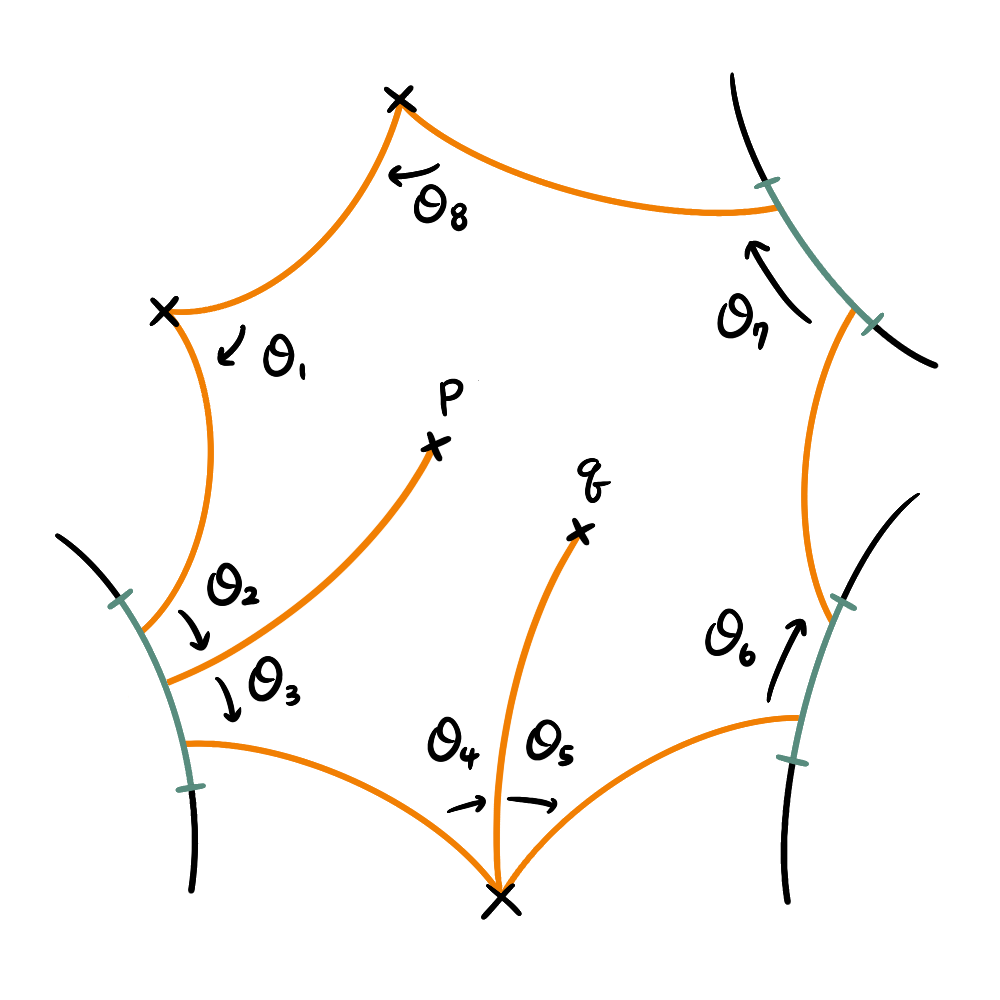}
		\caption{Folded disc sequence}
		\label{fig:AooForOrbiSurface:FoldedDiscSequence}
	\end{subfigure}
	\begin{subfigure}[b]{0.3\linewidth}
		\includegraphics[width=\linewidth]{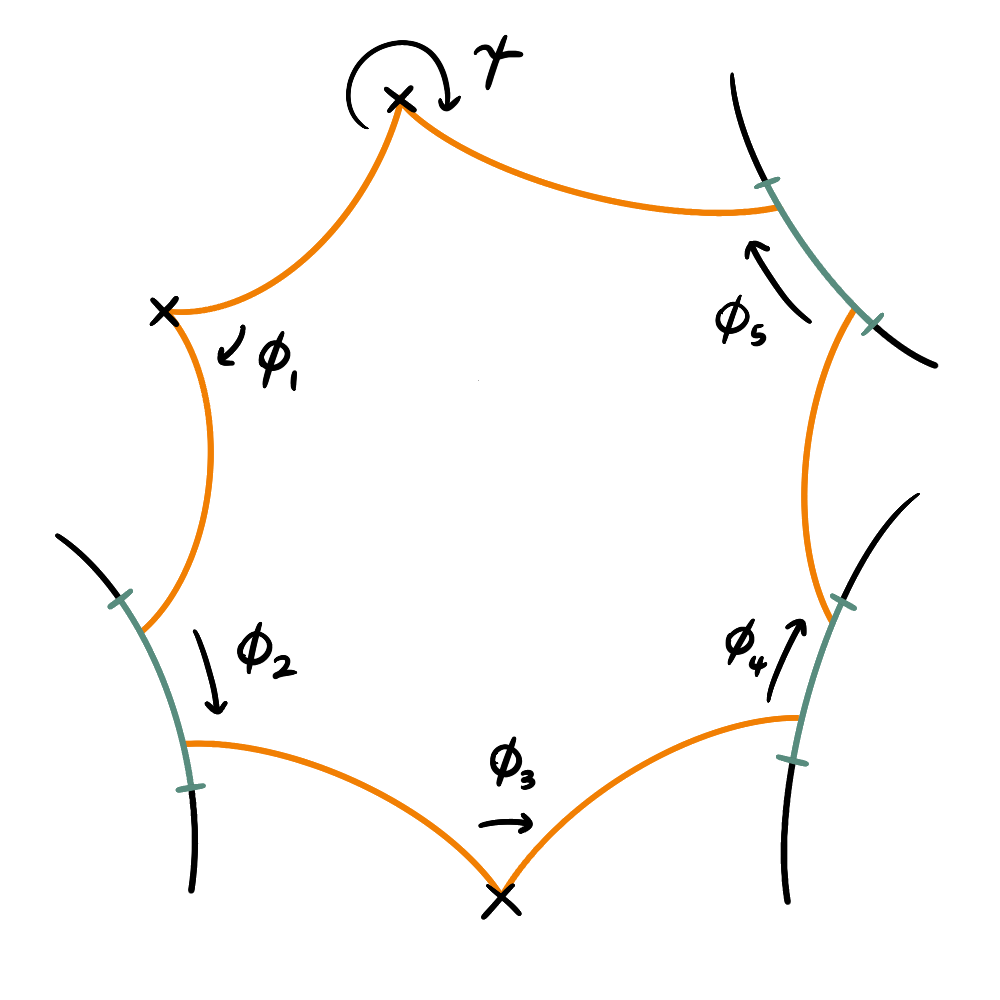}
		\caption{Composition sequence}
		\label{fig:AooForOrbiSurface:CompositionSequence}
	\end{subfigure}
	\caption{$\Aoo$-structure with tagged arcs}
	\label{fig:AooForOrbiSurface}
\end{figure}

\begin{defn}
	Let $(S, O, M, \eta)$ be a graded marked orbi-surface and $\gamma_1, \dots, \gamma_n$ be tagged arcs. Also, let $\phi_i : \gamma_i \rightarrow \gamma_{i+1}$ be boundary or interior morphisms. We call the sequence $(\phi_1, \dots, \phi_n)$ a {\em disc sequence} if there is an orientation preserving immersion (on the interior) $f : (D, M^\can_n) \rightarrow (S, M \cup O)$ that satisfies the following two conditions.
	\begin{itemize}
		\item For each $i$, $f (\gamma_i^\can)$ equals $\gamma_i$ (up to reparametrization of the domain) or the double cover of $\gamma_i$ branched at one point in the middle (in the latter case, we say that $(\theta_{i-1},\theta_i)$ is folded).
		 \item For each $i$, $f_*(\theta^\can_i) = \phi_i$.
	\end{itemize}
We will set $\gamma_{n+1} \deq \gamma_1$ and $\theta_{n+1} \deq \theta_1$ for notational convenience.
\end{defn}
In Figure \ref{fig:AooForOrbiSurface}(b), it is folded between $\theta_2$ and $\theta_3$ (along $\gamma_3$) and between $\theta_4$ and $\theta_5$ (along $\gamma_4$). Note that $\theta_2$ and $\theta_3$ are concatenable, and so are $\theta_4$ and $\theta_5$. These are the only folded pairs in the disc sequence.

Now, we also need a new notion, called a composition sequence (see Figure \ref{fig:AooForOrbiSurface}(c)).
We defined the notion of interior morphism between two tagged arcs hanging at the same marking. (The opposite of an interior morphism was not a morphism.)
A composition sequence is similar to a disc sequence, but occurs when the output of the disc operation is such an
interior morphism. (Hence, the opposite of this output is a non-morphism.)
\begin{defn}
	Let $(S, O, M, \eta)$ be a graded marked orbi-surface and $\gamma_1, \dots \gamma_n, \gamma_{n+1}$ be tagged arcs. Also, let $\phi_i : \gamma_i \rightarrow \gamma_{i+1}$ be boundary or interior morphisms and $\psi : \gamma_1 \rightarrow \gamma_{n+1}$ be an interior morphism at an orbifold point $p$. We call the sequence $(\phi_1, \dots, \phi_n ; \psi)$ a {\em composition sequence} if there is an orientation preserving immersion (on the interior) $f : (D, M^\can_{n+1}) \rightarrow (S, M \cup O)$ that satisfies the following two conditions.
	\begin{itemize}
		\item For each $i$,	$f (\gamma_i^\can)$ equals $\gamma_i$ (up to reparametrization of the domain) or the double cover of $\gamma_i$ branched at one point in the middle.
		\item For each $i=1, \dots, n$, $f_*(\theta^\can_i) = \phi_i$ and $f_*(\theta^\can_{n+1}) = p$.
	\end{itemize}
	We say, for $i=1, \dots, n-1$, the pair $(\phi_i, \phi_{i+1})$ are {\em folded} if they are concatenable. 
	If the disc is folded along $\gamma_1$ (resp. $\gamma_n$), then we will say $(\psi, \phi_1)$ (resp. $(\phi_n, \psi)$) is {\em folded} in this case.
\end{defn}
\begin{remark}
Note that if $(\psi, \phi_1)$ is folded, then $\psi$ has a decomposition $\phi_1 \bullet \theta$ for some nonzero interior morphism $\theta$. Similarly, if $(\phi_n, \psi)$ is folded, then $\psi$ has a decomposition $\theta \bullet \phi_n$.
\end{remark}
For a composition sequence $(\phi_1, \dots, \phi_n ; \psi)$, $\psi$ is uniquely determined by $(\phi_1, \dots, \phi_n)$. So we sometimes omit $\psi$. We call $\psi$ the {\em value} of the composition sequence. 
The following can be proved as in Lemma \ref{lemma:DegreeConditionForDiscSequence}. We will use this in Section \ref{sec:Aooness} repeatedly.
\begin{lemma}\label{Lemma:DegreeFormulae}
	The following hold.
	\begin{enumerate}
		\item Let $(\phi_1, \dots, \phi_n)$ be a disc sequence. Then, $\deg{\phi_1} + \dots + \deg{\phi_n} = n-2$.
		\item Let $(\phi_1, \dots, \phi_n ; \psi)$ be a composition sequence. Then, $\deg{\phi_1} + \dots + \deg{\phi_n} = \deg{\psi} + n-2$.
	\end{enumerate}
\end{lemma}

\subsection{Tagged arc systems}\label{subsection:TaggedArcSystems}
Let us generalize the notion of arc system for tagged arcs. First, we define an arc system using the underlying arcs. With grading and tagging, we will define a pre-tagged arc system. Then, we will define a {\em tagged arc system} as a pre-tagged arc system satisfying {\em thick}, {\em good}, and {\em full} conditions.

\begin{defn}\label{defn:ArcSystemForOrbiSuefaces}
	Let $(S, M, O, \eta)$ be a graded marked orbi-surface. An {\em arc system} is a collection $\Gamma = \{\gamma_1, \dots, \gamma_n\}$ of arcs satisfying the following three conditions (see Figure \ref{fig:ArcSystem}).
	\begin{itemize}
		\item For each pair $i\neq j$, $\gamma_i \cap \gamma_j \subset O$ and $\gamma_i \not\simeq \gamma_j$. If they meet at $O$, then they are transversal.
		\item Each component of $S \setminus \bigcup_{i=1}^n \gamma_i$ is a topological disc satisfying one of the following conditions.
		\begin{itemize}
			\item It is a disc with at most one unmarked boundary and no interior marking.
			\item It is a disc with at most one interior marking and no unmarked boundary.
		\end{itemize}
		\item $\bigcup_{\gamma \in \Gamma}\gamma$ is a forest. That is, it is a disjoint union of trees.
		\item Let $\gamma \in \Gamma$ be a tagged arc with $\nu(\gamma) = 2$. Then, there is $p\in O(\gamma)$ and a disc component $D \in S \setminus \bigcup_{i=1}^n\gamma_i$ such that $p \in \overline{D}$ and $D$ has an unmarked boundary component.
	\end{itemize}
\end{defn}

\begin{remark}
	The last condition guarantees that there are no disc sequences or composition sequences containing a tagged arc of interior number $2$ in their interior.
\end{remark}

\begin{figure}[h!]
	\centering
	\begin{subfigure}[b]{0.3\linewidth}
		\includegraphics[width=\linewidth]{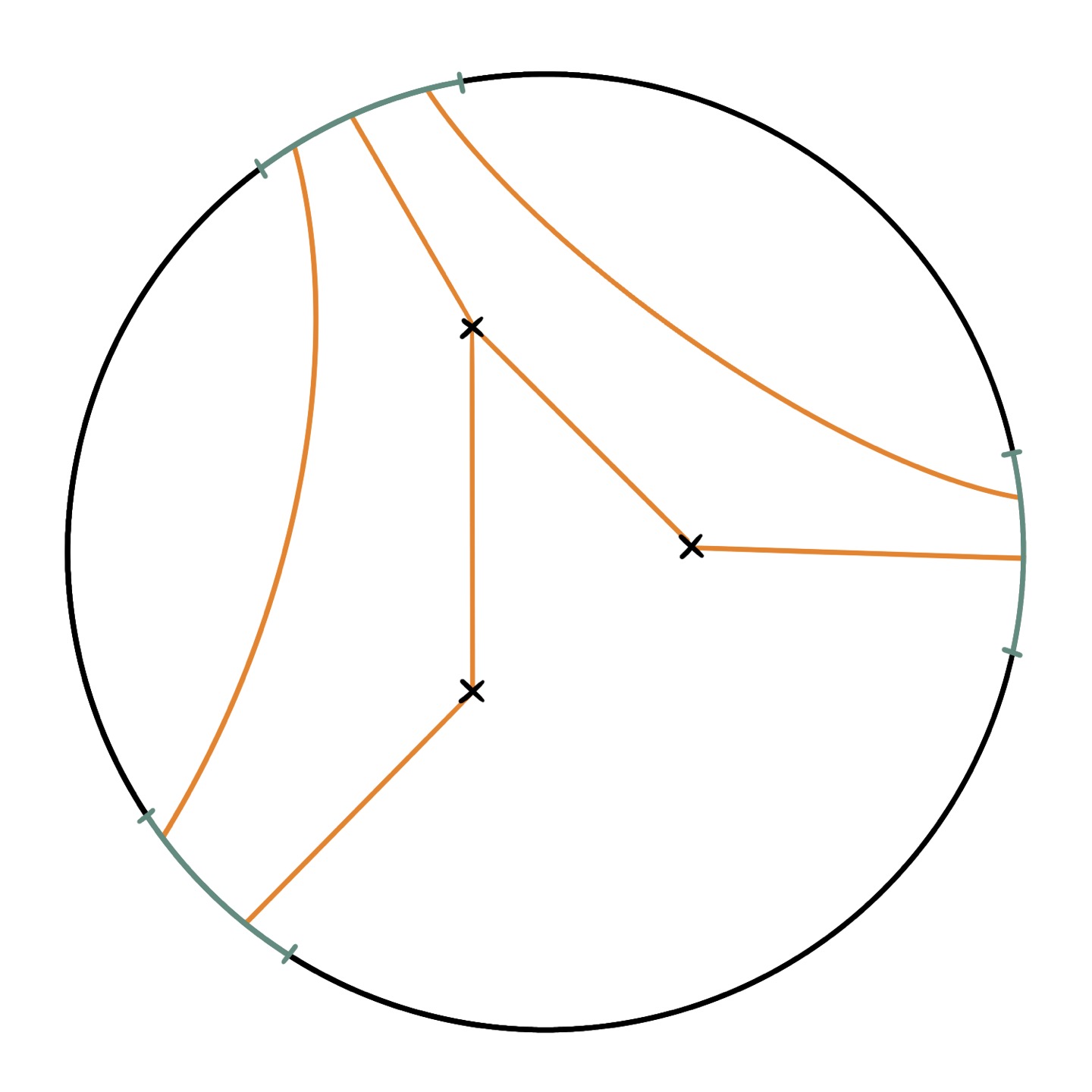}
		\caption{Valid arc system}
	\end{subfigure}
	\begin{subfigure}[b]{0.3\linewidth}
		\includegraphics[width=\linewidth]{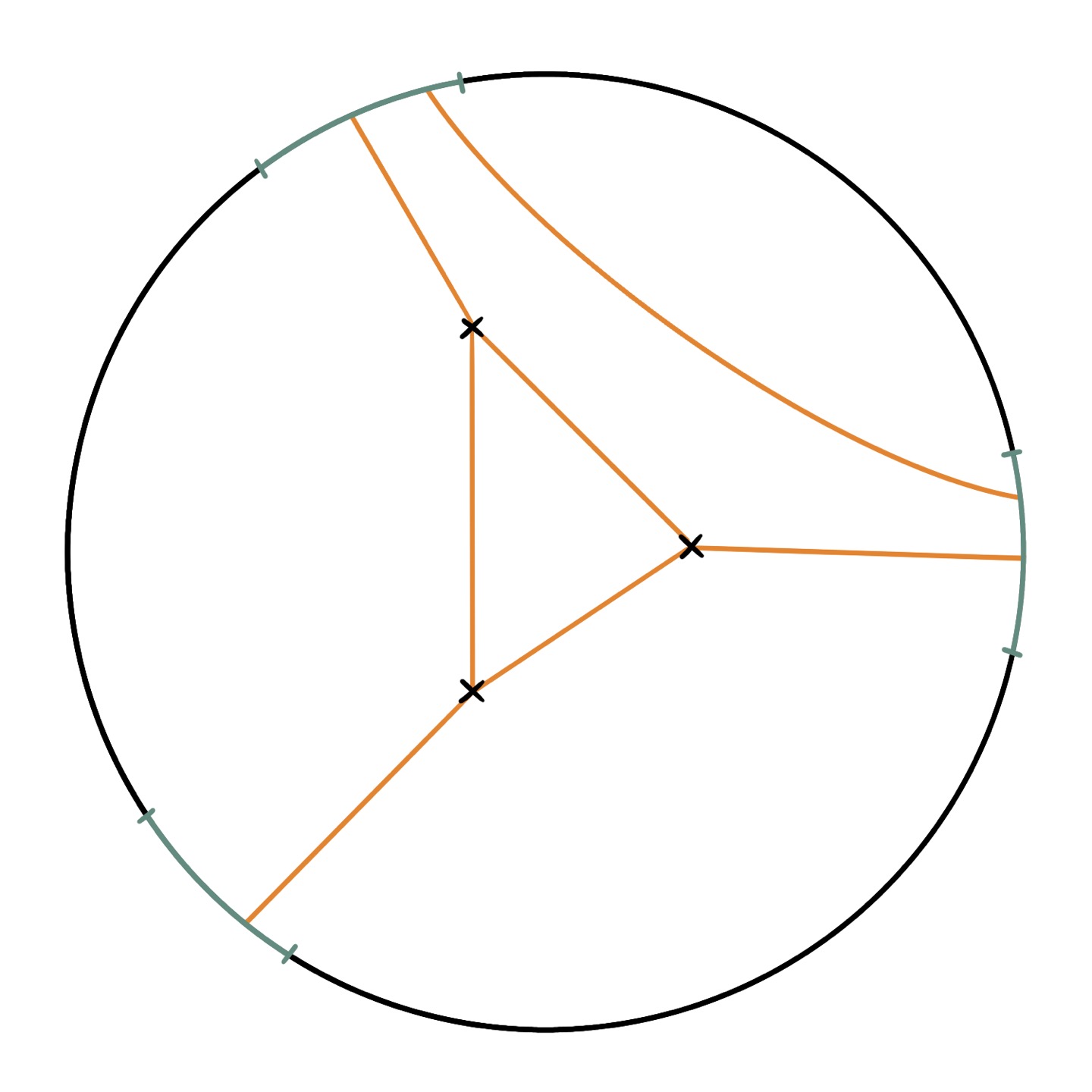}
		\caption{Violate the third condition}
	\end{subfigure}
	\begin{subfigure}[b]{0.3\linewidth}
		\includegraphics[width=\linewidth]{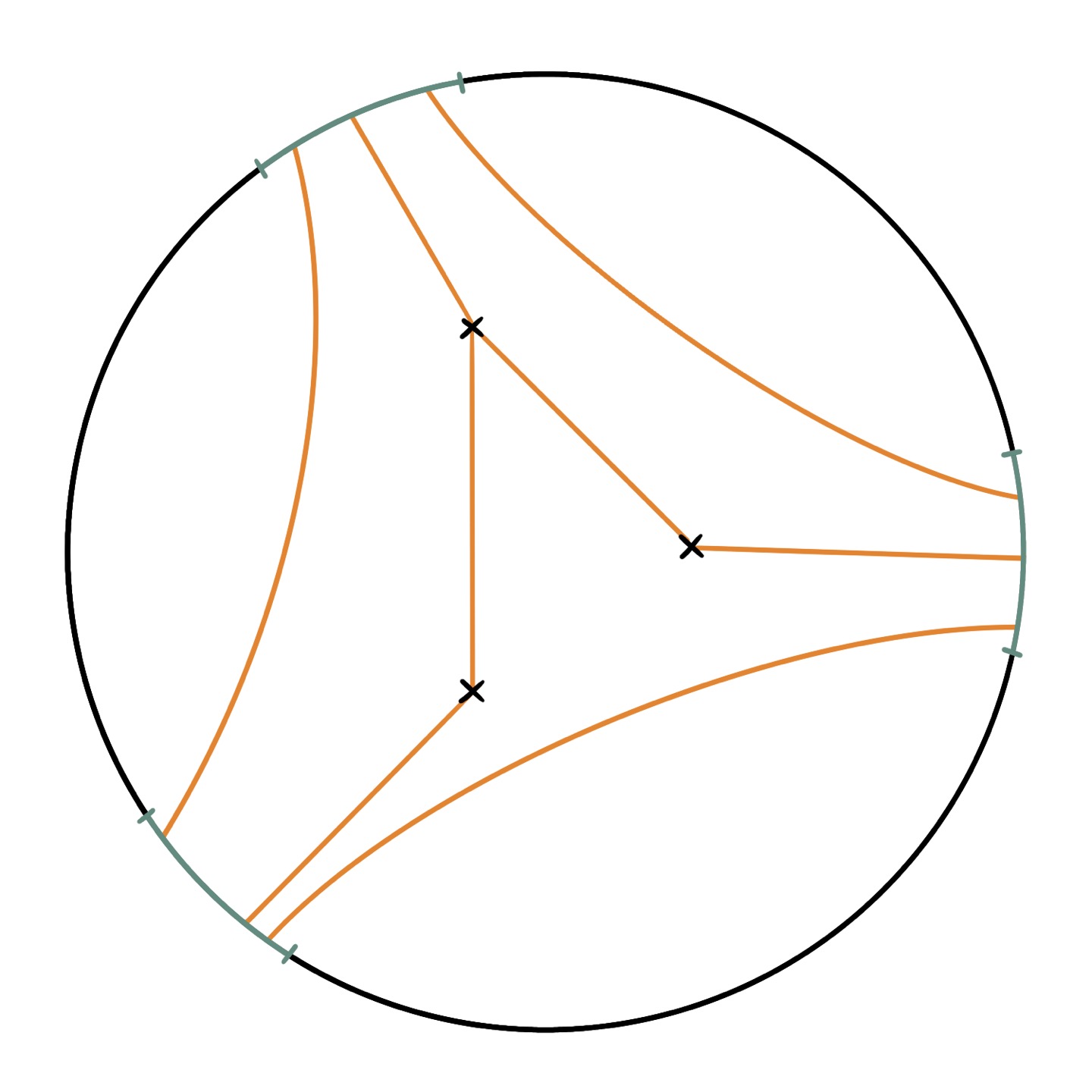}
		\caption{Violate the fourth condition}
	\end{subfigure}
	\caption{Examples of arc and non-arc systems}
	\label{fig:ArcSystem}
\end{figure}

\begin{defn}
	For a collection of tagged arcs $\Gamma = \{(\gamma_1, \tau_1), \dots, (\gamma_n, \tau_n)\}$, we denote by $\underline{\Gamma}$ the set of underlying arcs. We call $\Gamma$ a {\em pre-tagged arc system} if $\underline{\Gamma}$ is an arc system.
\end{defn}

We sometimes allow two different taggings of  a single underlying arc in a pre-tagged arc system. The precise condition is formulated as a thick pair and the thick condition.
\begin{defn}\label{defn:ThickPair}
	Let $(\alpha, \sigma)$ and $(\beta, \tau)$ be two tagged arcs. We say the pair $((\alpha, \sigma), (\beta, \tau))$ a {\em thick pair} if  they have the same underlying arc $\gamma$ with $\nu(\gamma) \geq 1$ and satisfies one of the following condition.
	\begin{itemize}
		\item If $\nu(\gamma) = 1$, $(\beta,\tau)$ can be obtained by 
		changing the tagging of $(\alpha,\sigma)$ and by taking an overall shift.
		(i.e. $(\beta, \tau) = (\alpha, -\sigma)[n]$. )
		\item If  $\nu(\gamma) = 2$, $(\beta,\tau)$ can be obtained by changing the tagging of $(\alpha,\sigma)$ at one of its end point and by taking an overall shift.
	\end{itemize}
	 In this case, let $p$ be the endpoint  where the tagging is changed. We will say the underlying arc is {\em thick} at $p$.
\end{defn}

\begin{defn}\label{defn:ThickCondition}
	We say a pre-tagged arc system $\Gamma$ is {\em thick} if the following conditions hold. 
	\begin{itemize}
		\item  Given an underlying arc $\gamma \in \underline{\Gamma}$, there are at most two tagged arcs in $\Gamma$ whose underlying arc is $\gamma$. If there are two, then they form a thick pair. (Hence $\gamma$ is thick.)
		\item If $\gamma$ is thick at $p$ and $\nu(\gamma) = 2$, then $\gamma$ is the only arc hanging at $p$.
	\end{itemize}
\end{defn}
The last condition is imposed to have simpler  $\Aoo$-operations. In particular, we can work with
a nice  tagged arc system, which will be introduced soon.

Let $\Gamma$ be a pre-tagged arc system satisfying thick condition. Now let us
explain the {\em good} condition. This concerns the ordering of tagged arcs hanging at a single point $p \in O$.
\begin{defn}
	For an orbifold point $p \in O$, we take a sequence of tagged arcs hanging at $p$:
	$$\left((\gamma_1, \tau_1), \dots, (\gamma_n, \tau_n)\right)$$
	for any $n \geq 2$.
	Assume that none of them are thick and that they are ordered clockwise at $p$. 
	We say the sequence is {\em good} if $p$ defines the interior morphism from $(\gamma_i, \tau_i)$ to $(\gamma_{i+1}, \tau_{i+1})$ for all $i=1, \dots, n-1$.
\end{defn}
Figure \ref{fig:GoodAndBadTriple}(a) illustrates a good sequence of length 3. Here, the arrows indicate interior morphisms.
Note that interior morphisms in a good sequence can be {\em concatenated}. Figure \ref{fig:GoodAndBadTriple}(b) illustrates a non-good sequence.
\begin{lemma}
	If $\left((\gamma_1, \tau_1), \dots, (\gamma_n, \tau_n)\right)$ is good, then
	 $p$ defines the interior morphism from  $(\gamma_i, \tau_i)$ to $(\gamma_{j}, \tau_{j})$
	 for any $1 \leq i < j \leq n$.
\end{lemma}
\begin{proof}
	It is enough to show the lemma for $n=3$. Since $p$ defines $p^{\gamma_1, \tau_1}_{\gamma_2, \tau_2}$ and $p^{\gamma_2, \tau_2}_{\gamma_3, \tau_3}$, we have $$\tau_2(p) = \tau_1(p) + i_p(\gamma_1, \gamma_2), \quad \tau_3(p) = \tau_2(p) + i_p(\gamma_2, \gamma_3).$$ Using Lemma \ref{lemma:SumOfIntersectionIndex}, we have $i_p(\gamma_1, \gamma_2) + i_p(\gamma_2, \gamma_3) = i_p(\gamma_1, \gamma_3)$. So we get $\tau_3(p) = \tau_1(p) + i_p(\gamma_1, \gamma_3)$, which implies $p$ defines $p^{\gamma_1, \tau_1}_{\gamma_3, \tau_3}$.
\end{proof}
Thus a subsequence of a good sequence is good as well.
Later we will use this concatenation to define $\fm_2$ multiplication of $\Aoo$ structure,
so let us make the following definition.
\begin{defn}
	If $\left((\alpha, \rho), (\beta, \sigma), (\gamma, \tau)\right)$ is a good triple hanging at $p \in O$, we say $p^{\alpha, \rho}_{\beta, \sigma}$ and $p^{\beta, \sigma}_{\gamma, \tau}$ are {\em concatenable} and define the concatenation 
$$p^{\alpha, \rho}_{\beta, \sigma} \bullet p^{\beta, \sigma}_{\gamma, \tau} \deq p^{\alpha, \rho}_{\gamma, \tau}.$$
\end{defn}

One can also check that if the triple is not good so that $p$ defines interior morphisms $p^{\alpha, \rho}_{\gamma, \tau}$ and $p^{\beta, \sigma}_{\alpha, \rho}$, then $p$ defines the interior morphism $p^{\gamma, \tau}_{\beta, \sigma}$ using similar arguments (see Figure \ref{fig:GoodAndBadTriple} right).

\begin{figure}[h!]
	\centering
	\begin{subfigure}[b]{0.4\linewidth}
		\includegraphics[width=\linewidth]{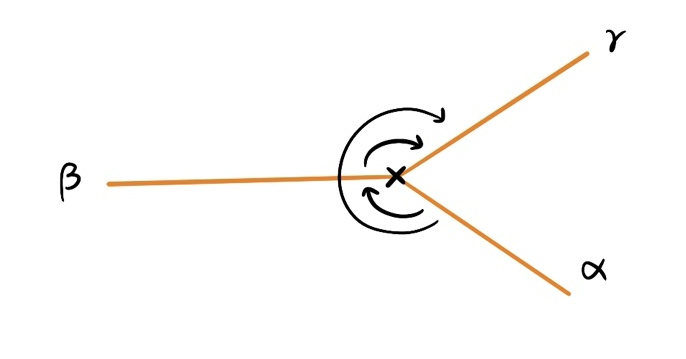}
		\caption{}
	\end{subfigure}
	\begin{subfigure}[b]{0.4\linewidth}
		\includegraphics[width=\linewidth]{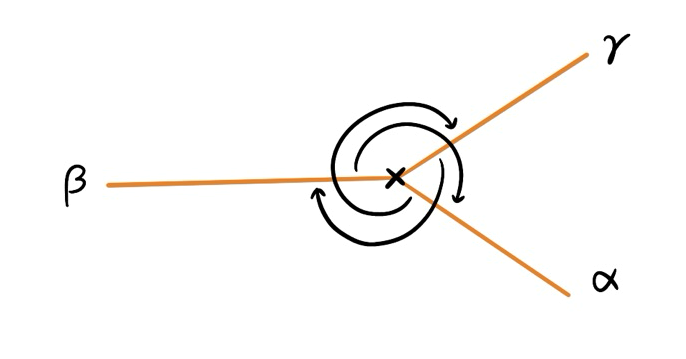}
		\caption{}
	\end{subfigure}
	\caption{Good triple and not-good triple}
	\label{fig:GoodAndBadTriple}
\end{figure}
 
 Now, we are ready to define the good condition of pre-tagged arc system.
 For an interior marking $p \in O$, let $\Gamma_p$ be the set of tagged arcs hanging at $p$. It has a cyclic ordering by clockwise orientation. We will consider  {\em linear}  subordering of certain subsets of $\Gamma_p$ and require them to form good sequences.
\begin{defn}\label{defn:GoodCondition}
	Let $\Gamma$ be a thick pre-tagged arc system and $p\in O$ be an interior marking. We say $\Gamma$ is {\em good} at $p$ if one of the following holds. Let $\Gamma_p = \{(\gamma_{i_1}, \tau_{i_1}), \dots, (\gamma_{i_r}, \tau_{i_r})\}$ be the set of tagged arcs hanging at $p$, which is ordered clockwise.
	\begin{itemize}
		\item $\Gamma_p$ has at most one thick pair. If it has a thick pair, then $\left((\gamma_{i_1}, \tau_{i_1}), (\gamma_{i_r}, \tau_{i_r})\right)$ is the thick pair.
		\item If $\Gamma_p$ has no thick pair, then there is $1 \leq k \leq r$ such that the following sequence is good: $$\left((\gamma_{i_k}, \tau_{i_k}), \dots, (\gamma_{i_r}, \tau_{i_r}), (\gamma_{i_1}, \tau_{i_1}), \dots (\gamma_{i_{k-1}}, \tau_{i_{k-1}})\right).$$
		\item If $\Gamma_p$ has a thick pair, then following  two sequences are good: $$\left((\gamma_{i_1}, \tau_{i_1}), \dots, (\gamma_{i_{r-1}}, \tau_{i_{r-1}})\right), \quad \left((\gamma_{i_2}, \tau_{i_2}), \dots, (\gamma_{i_r}, \tau_{i_r})\right).$$ 	\end{itemize}
		We say $\Gamma$ is {\em good} if it is good at all interior marking $p \in O$.
\end{defn}

We finally define the full condition.
\begin{defn}\label{defn:FullCondition}
	Let $\Gamma$ be a pre-tagged arc system satisfying thick and good conditions. Then, we say $\Gamma$ is {\em full} if each component of $S\setminus \bigcup_{\gamma \in \underline{\Gamma}}\gamma$ is a topological disc satisfying the following conditions.
	\begin{itemize}
		\item If it has neither an interior marking nor an unmarked boundary, then it defines a disc sequence or a composition sequence. If it is a folded composition sequence, then each folded arc must be thick. If it is a folded disc sequence, then at most one folded arc can be not thick.
		\item If it has an interior marking, then it defines a disc sequence. If it is folded, then each folded arc must be thick.
		\item If it has an unmarked boundary, then it defines a disc sequence after adjoining an arc following the unmarked boundary. If it is folded, then each folded arc must be thick.
	\end{itemize}
\end{defn}

\begin{remark}
	The conditions can be summarized as follows. Let $D$ be a disc component of $S\setminus \bigcup_{\gamma \in \underline{\Gamma}}\gamma$. Let us say an interior marking $p \in \overline{D} \cap O$ {\em inward} if the interior morphism defined by $p$ points to $D$ and {\em outward} otherwise. Then, $D$ must satisfy the following.
	\begin{align*}
		&\#(D \cap O) + \#\{\text{unmarked boundary components of $D$}\}\\
		+& \#\{\text{outward interior markings of $D$}\} + \#\{\text{not thick folded arcs of $D$}\} \leq 1.
	\end{align*}
\end{remark}

Putting together three conditions, we get the following definition of tagged arc system.
\begin{defn}\label{defn:NiceCondition}
	Let $(S, M, O, \eta)$ be a graded marked orbi-surface. A {\em tagged arc system} is a pre-tagged arc system $\Gamma$ which is thick, good, and full. We say $\Gamma$ is {\em nice} if every interior morphism between tagged arcs in $\Gamma$ is of degree $0$.  Here, boundary morphisms are allowed to have arbitrary degrees.
\end{defn}

\begin{lemma}\label{Lemma:MakeBeNice}
	Any tagged arc system is nice up to degree shift. More precisely, let $\Gamma = \{(\gamma_1, \tau_1), \dots, (\gamma_n, \tau_n)\}$ be a tagged arc system. Then, there is a sequence of integers $\overrightarrow{\mu} = (\mu_1, \dots, \mu_n)$ such that the collection $$\Gamma[\overrightarrow{\mu}] \deq \{(\gamma_1, \tau_1)[\mu_1], \dots, (\gamma_n, \tau_n)[\mu_n]\}$$ is a nice tagged arc system.
\end{lemma}
\begin{proof}
	By Lemma \ref{lemma:InteriorMorphismProperties},  $\Gamma[\overrightarrow{\mu}]$ is still a tagged arc system for any sequence of integers $\overrightarrow{\mu}$. Let us explain how to choose an appropriate $\overrightarrow{\mu}$ to make $\Gamma$ nice.
		
	Let $\sim$ be an equivalence relation on $\Gamma$ defined as $(\alpha, \tau) \sim (\beta, \sigma)$ if and only if they are the same or form a thick pair. Let $\Gamma'$ be a set of representatives of $\Gamma/\sim$. Let us define a graded directed graph $G$ as follows. Its vertices are tagged arcs in $\Gamma'$ whose interior number is nonzero and edges are irreducible interior morphisms with the same grading.
	
	We claim that each component of $G$ is a tree. Note that the construction above gives a one-to-one correspondence between components of $G$ and $\Gamma'$. Let us assume $G$ is connected. Also let us denote by $v_G, e_G$, $v_{\Gamma'}$, and $e_{\Gamma'}$ the number of vertices and edges of $G$ and $\Gamma'$, \resp. By the good condition and absence of thick arcs the number of irreducible interior morphisms defined by $p$ is one less then the number of arcs hanging at $p$. Thus we have $$e_G = \sum_{p \in O}\left(\operatorname{val}(p) -1\right) = 2e_\Gamma - v_\Gamma.$$ Together with $$v_G = \text{ the number of arcs in $\Gamma'$} = e_\Gamma,$$ we get $$v_G - e_G = e_G - (2e_G - v_G) = v_\Gamma - e_\Gamma = 1.$$ Here, the last equality comes from the fact that $\Gamma$ is a tree. This proves $G$ is a tree.
	
	Let us take any vertex of $G$ as the root and denote the distance from the root to a vertex $v$ by $d(v)$. We prove the lemma by induction on distance. Suppose that we have shifted degrees of arcs corresponding vertices $v$ such that $d(v) < k$ for some $k \in \bzz_{>0}$ so that edges between them are of degree $0$. Let $\alpha \in \Gamma'$ be an arc with $d(\alpha) = k$. As $G$ is a tree, there is a unique arc $\beta$ with $d(\beta) = k-1$ with an interior morphism $\phi$ between $\alpha$ and $\beta$. By shifting $\beta$ by $\pm\deg{\phi}$, we can make $\phi$ to be of degree $0$. Since there are no edges between vertices of distance $k$, we can shift degree of all arcs of distance $k$ at the same time. By induction on distance, we can make any interior morphisms between arcs in $\Gamma'$ to be of degree $0$.
	
	Let $\alpha$ be a tagged arc in $\Gamma \setminus \Gamma'$ with its thick pair $\beta$. Suppose that $\nu(\alpha) = 1$ and $\alpha$ is hanging at $p$. Then, either $\alpha$ is a source or a sink in $\Gamma_p$. So we can shift $\alpha$ so that degree of interior morphisms defined by $p$ is $0$. Now suppose that $\nu(\alpha) = 2$ and $O(\alpha) = \{p, q\}$. Then, by thick condition either $\Gamma_p = \{\alpha, \beta\}$ or $\Gamma_q = \{\alpha, \beta\}$. So this case is essentially the same with the previous case. This proves the lemma.
\end{proof}

\section{$\Aoo$-categories of tagged arc systems}\label{section:AooCategoryofTaggedArcs}
In this section, we define an $\Aoo$-category $\cff_\Gamma(S, M, O, \eta)$ associated with a nice tagged arc system $\Gamma$ of a graded marked orbi-surface $(S, M, O, \eta)$.
Since any tagged arc system is nice up to degree shift,  a general construction for any tagged arc system can be obtained by adapting the signs 
(see \cite{fu02} or \cite[Section $3$I]{Seidel18}).
 If there are no interior markings (orbifold points), then the resulting $\Aoo$-category is the same as that of \cite{HKK17}. 

Objects of the $\Aoo$-category $\cff_\Gamma(S, M, O, \eta)$ are  tagged arcs.
Morphisms are given by boundary and interior morphisms (together with an identity
morphism). For a composable sequence of morphisms $\overrightarrow{\theta}$, the desired $\Aoo$-operation will be
defined as a sum of several types of operations (for $n \geq 2$ while $\fm_1= \fm_0=0$):
$$\fm_n(\overrightarrow{\theta}) \deq \fm^\con_n(\overrightarrow{\theta}) + \fm^\disc_n(\overrightarrow{\theta}) + \fm^\comp_n(\overrightarrow{\theta}) + \fm^\thick_n(\overrightarrow{\theta}).$$
Actual definition involves delicate signs and weights, so let us start with the concatenation operation $\fm^\con_n$ and introduce relevant notions.
Concatenation operation  $\fm^\con_n$ is non-trivial if and only if $n=2$.
Recall that we denote by using $\bullet$ the geometric morphism obtained by the concatenation of two interior or two boundary morphisms.
We record it as an $\fm^\con_2$ operation as follows.
 \begin{figure}[h!]
	\centering
	\includegraphics[width=0.8\linewidth]{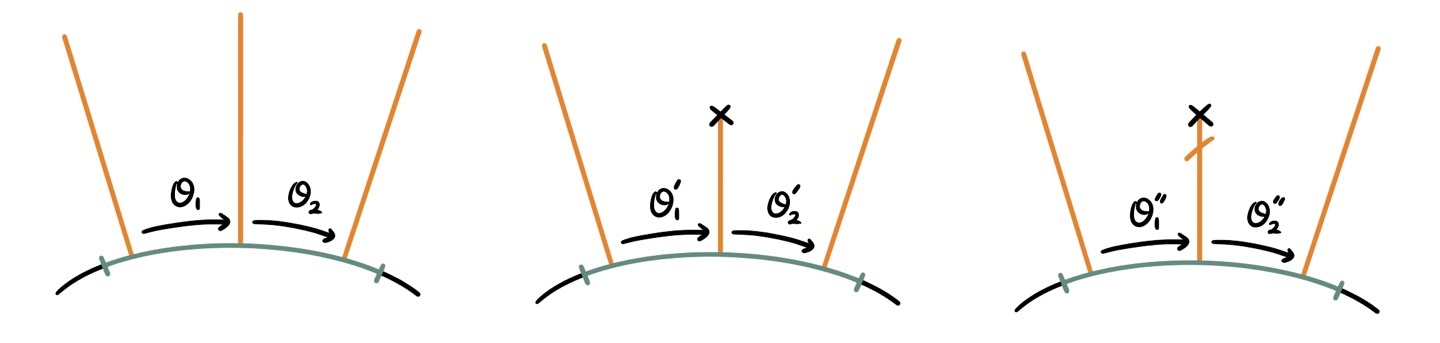}
	\caption{Composition of boundary morphisms}
	\label{fig:tagcomp1}
\end{figure}
Figure \ref{fig:tagcomp1} illustrates some of possible cases for boundary morphisms.
Then, the compositions for the cases of Figure \ref{fig:tagcomp1} will be defined as
\begin{align*}
	\fm_2^\con(\theta_1,\theta_2) &= (-1)^{\deg{\theta_1}} \theta_1 \bullet \theta_2, \\
	\fm_2^\con(\theta_1',\theta_2') &= (-1)^{\deg{\theta'_1}}\frac{1}{2} \theta'_1 \bullet \theta'_2, \\
	\fm_2^\con(\theta_1'',\theta_2'') &= -(-1)^{\deg{\theta''_1}}\frac{1}{2} \theta''_1 \bullet \theta''_2.	
\end{align*}
The same holds regardless of intersection number and tagging of the first or third arc.
\begin{figure}[h!]
	\centering
	\includegraphics[width=0.9\linewidth]{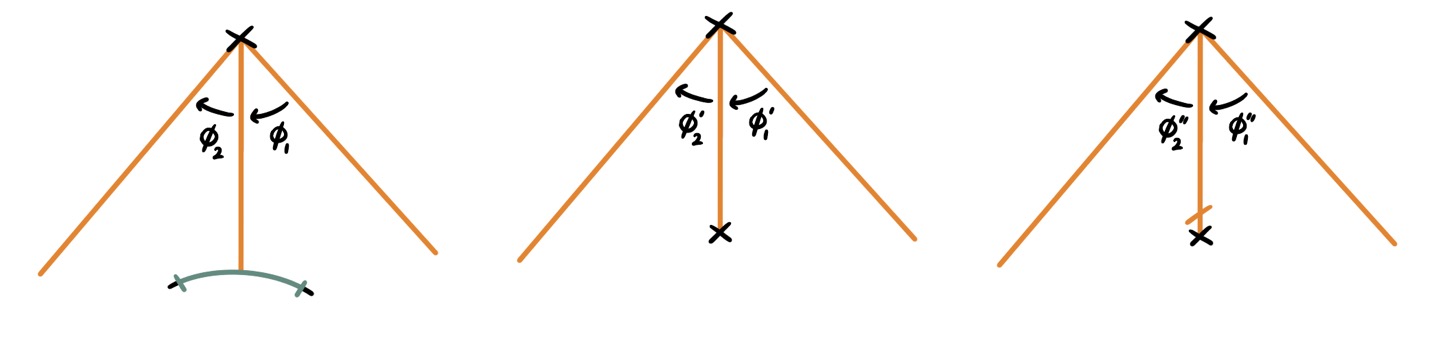}
	\caption{Composition of interior morphisms}
	\label{fig:tagcomp2}
\end{figure}
Similarly, the compositions for interior morphisms in Figure  \ref{fig:tagcomp2} will be defined as  
\begin{align*}
	\fm_2^\con(\phi_1,\phi_2) &=  \phi_1 \bullet \phi_2, \\
	\fm_2^\con(\phi_1',\phi_2') &= \frac{1}{2} \phi_1' \bullet \phi_2', \\
	\fm_2^\con(\phi_1'',\phi_2'') &= - \frac{1}{2} \phi_1'' \bullet \phi_2''.	
\end{align*}
Again, the same holds regardless of tagging at the interior marking defining the morphisms and interior number and tagging of the first and third arcs.
As niceness implies $(-1)^{\deg{\phi_1}}=1$, the formulas are similar to boundary cases.

To handle interior numbers and signs more efficiently, we introduce following notations. 
First, let us call a boundary or interior morphism a {\em basic morphism} (so it is not a unit nor a non-trivial linear combination of different types of morphisms).
\begin{defn}
For a basic morphism $\theta$, we define 
$$ \left[\theta\right] = \begin{cases}  1 &\text{if $\theta$ is an interior morphism,} \\ 0 &\text{otherwise.} \end{cases}$$
 For two basic morphisms $\theta_1 : \gamma_1 \rightarrow \gamma_2$ and $\theta_2 : \gamma_2 \rightarrow \gamma_3$, we define 
 $$\left<\theta_1, \theta_2\right> \deq \begin{cases}1 &\text{if $\theta_1$ and $\theta_2$ are concatenable boundary morphisms and $\nu(\gamma_2) = 1$,} \\ 1 &\text{if $\theta_1$ and $\theta_2$ are concatenable  interior morphisms and $\nu(\gamma_2) = 2$,} \\ 0 &\text{otherwise.} \end{cases}$$ 
 \end{defn} 
Here $ \left<\theta_1, \theta_2\right>$ measures whether disc can be folded between $\theta_1$ and $\theta_2$.
Thus, for the second and third cases of Figure  \ref{fig:tagcomp1} and Figure \ref{fig:tagcomp2}, we have $ \left<\theta_1, \theta_2\right> =1$
whereas for the first case we have $ \left<\theta_1, \theta_2\right> =0$. 

\begin{defn}
For a nice tagged arc system, we define a sign for a basic morphism as follows.
$$ \sigma(\theta) \deq \begin{cases}  0 &\text{if $\theta$ is a boundary morphism,} \\ \tau_1(p) & \text{if $\theta$ is an interior morphism $p^{\alpha, \tau_1}_{\beta, \tau_2}$.}  \end{cases}$$
From the nice condition, if $p^{\alpha, \tau_1}_{\beta, \tau_2}$ is an interior morphism, we have $\tau_1(p)=\tau_2(p)$.

We define a sign for a pair of basic morphisms by the tagging of  the other end point of the possible folded arc in the middle. That is, we define
$$\sigma(\theta_1, \theta_2) \deq \begin{cases} 0 &\text{if $\left<\theta_1, \theta_2\right> =0$,} \\
 \tau(q) &\text{if $\left<\theta_1, \theta_2\right> = 1$ and $q$ is the other interior marking endpoint of the middle arc $(\gamma_2, \tau)$.} \end{cases}$$  
\end{defn}

One can check that the previous concatenation formulas can now be written in a uniform way \eqref{eq:mncon} in Definition \ref{defn:FukayaCategoryForMarkedSurface}.

\begin{defn}
For a disc sequence $\overrightarrow{\theta} = (\theta_1, \dots, \theta_n)$, we define its {\em weight} and {\em sign} as follows.
\begin{align*}
	\Phi(\overrightarrow{\theta}) &\deq \sum_{i=1}^n \big(\left[\theta_i\right] + (1- \left[\theta_i\right])\left<\theta_i, \theta_{i+1}\right> \big), \\
	\Sigma(\overrightarrow{\theta}) &\deq \sum_{i=1}^n  \big( \sigma(\theta_i) ( 1- \left<\theta_i, \theta_{i+1}\right>) \big).
\end{align*}
\end{defn}
For a not-folded disc sequence $\overrightarrow{\theta}$, the weight counts the number of interior morphisms among $(\theta_1,\ldots, \theta_n)$.
For a folded disc sequence $\overrightarrow{\theta}$, the weight counts in addition the number of folded arcs attached to the boundary
(not attached to the interior marking). The sign $\Sigma$ counts the number of interior morphisms between notched arcs minus the number of folded notched arcs among them.

\begin{defn}
For a composition sequence $\overrightarrow{\phi} = (\phi_1, \dots, \phi_n)$ with value $\psi$, we define its {\em weight} and {\em sign} as follows.
\begin{align*}
	\Phi(\overrightarrow{\phi} ; \psi) &\deq \sum_{i=2}^{n - 1}\left[\phi_i\right] + \sum_{i=1}^{n - 1}\big(\left( 1-  \left[\phi_i\right] \right)\left<\phi_i, \phi_{i+1}\right>\big), \\
	\Sigma(\overrightarrow{\phi} ; \psi) &\deq \sigma(\psi, \phi_1) + \sigma(\phi_n, \psi) - \sigma(\phi_1)  +  \sum_{i=1}^{n - 1} \big( \sigma(\phi_i)(1-\left<\phi_i, \phi_{i+1}\right>)\big).
\end{align*}
 Here, if $(\psi, \phi_1)$ is folded, there exist $\theta$ such that $\psi = \phi_1 \bullet \theta$. In addition, we define 
 $$\left<\psi, \phi_1\right> \deq \left<\phi_1, \theta\right>, \quad \sigma(\psi, \phi_1) \deq \sigma(\phi_1, \theta).$$
 We define $\left<\phi_n, \psi\right>$ and $\sigma(\phi_n, \psi)$ in a similar way.  They are defined to be zero if not folded. 
\end{defn}
Now, we are ready to define the desired $A_{\infty}$-category.

\begin{defn}\label{defn:FukayaCategoryForMarkedSurface}
	Let $(S, M, O, \eta)$ be a graded marked orbi-surface and $\Gamma$ be a nice tagged arc system. Then, define an $\Aoo$-category $\cff_\Gamma(S, M, O, \eta)$ as follows.
	\begin{itemize}
		\item The set of objects is $\Gamma$.
		\item The basis for morphism space consists of boundary morphisms, interior morphisms, and the unit. More precisely, for distinct tagged arcs $\alpha, \beta$, the hom space is given as follows. Here, $\Theta(\alpha, \beta)$ is the set of basic morphisms from $\alpha$ to $\beta$.
		 \begin{align*}
			\ho_{\cff_\Gamma(S, M, O, \eta)}(\alpha, \alpha) &\deq \field\left<\Theta(\alpha, \alpha)\right> \oplus \field\left<e_\alpha\right> \\
			\ho_{\cff_\Gamma(S, M, O, \eta)}(\alpha, \beta) &\deq \begin{cases} \field\left<\Theta(\alpha, \beta)\right> &\text{ if $(\alpha, \beta)$ is not a thick pair} \\  0 &\text{ if $(\alpha, \beta)$ is a thick pair} \end{cases}
		\end{align*}
	\item The $\fm_1$ is set to be zero.
	\item For two concatenable boundary or interior morphisms $\theta_1$ and $\theta_2$, 
	\begin{equation}\label{eq:mncon}
	\fm^\con_2(\theta_1, \theta_2) \deq (-1)^{\deg{\theta_1} + \sigma(\theta_1, \theta_2)}\left(\frac{1}{2}\right)^{\left<\theta_1, \theta_2\right>}\theta_1\bullet \theta_2.\end{equation}
	\item Let $\overrightarrow{\phi} = (\phi_1, \dots, \phi_n)$ be a disc sequence with $\phi_i : \gamma_ i \rightarrow \gamma_{i+1}$ (See Figure \ref{fig:AooForOrbiSurface}(a) and (b)).
Let $\theta : \alpha \rightarrow \gamma_1$ and $\psi : \gamma_n \rightarrow \beta$ be boundary or interior morphisms concatenable with $\phi_1$ and $\phi_n$, \resp. Then,
	\begin{align*}
		\fm^\disc_n(\phi_1, \dots, \phi_n) &\deq (-1)^{\Sigma(\overrightarrow{\phi})}\left(\frac{1}{2}\right)^{\Phi(\overrightarrow{\phi}) - \left<\phi_n, \phi_1\right>}e_{\gamma_1}, \\
		\fm^\disc_n(\theta \bullet \phi_1, \dots, \phi_n) &\deq (-1)^{\Sigma(\overrightarrow{\phi}) + \deg{\theta} + \sigma(\theta, \phi_1)}\left(\frac{1}{2}\right)^{\Phi(\overrightarrow{\phi}) - \left<\theta, \phi_1\right>}\theta, \\
		\fm^\disc_n(\phi_1, \dots, \phi_n \bullet \psi) &\deq (-1)^{\Sigma(\overrightarrow{\phi}) + \sigma(\phi_n, \psi)}\left(\frac{1}{2}\right)^{\Phi(\overrightarrow{\phi}) - \left<\phi_n, \psi\right>}\psi.
	\end{align*}
	\item For a composition sequence $\overrightarrow{\phi} = (\phi_1, \dots, \phi_n)$ with value $\psi$ (see Figure \ref{fig:AooForOrbiSurface}(c)),
	$$\fm^\comp_n(\phi_1, \dots, \phi_n) \deq (-1)^{\Sigma(\overrightarrow{\phi} ; \psi)}\left(\frac{1}{2}\right)^{\Phi(\overrightarrow{\phi} ; \psi)}\psi.$$
	\item Let $\theta_i : \gamma_i \rightarrow \gamma_{i+1}$, $i=1, 2, 3$ be such that $\theta_1$ is a boundary morphism, $\theta_2$ and $\theta_3$ are interior morphisms, and $(\gamma_2, \gamma_4)$ is a thick pair. Then, there is another boundary morphism $\theta_1^\vee : \gamma_1 \rightarrow \gamma_4$. Here, we call the pair $(\theta_2, \theta_3)$ and the triple $(\theta_1 ; (\theta_2, \theta_3))$ a thick pair and thick triple, \resp. Then,
	$$\fm^\thick_3(\theta_1, \theta_2, \theta_3) \deq (-1)^{\deg{\theta_1} + 1 + \sigma(\theta_2) + \sigma(\theta_2, \theta_3)}\left(\frac{1}{2}\right)^{1 + \left<\theta_2, \theta_3\right>} \theta_1^\vee.$$
	\item Let $\theta_i : \gamma_i \rightarrow \gamma_{i+1}$, $i=1, 2, 3$ be such that $\theta_3$ is a boundary morphism, $\theta_1$ and $\theta_2$ are interior morphisms, and $(\gamma_1, \gamma_3)$ is a thick pair. Then, there is another boundary morphism ${}^\vee\theta_3 : \gamma_1 \rightarrow \gamma_4$. Here, we call the pair $(\theta_1, \theta_2)$ and the triple $((\theta_1, \theta_2) ; \theta_3)$ a thick pair and thick triple, \resp. Then,
	$$\fm^\thick_3(\theta_1, \theta_2, \theta_3) \deq (-1)^{\sigma(\theta_1) + \sigma(\theta_1, \theta_2)}\left(\frac{1}{2}\right)^{1 + \left<\theta_1, \theta_2\right>} {}^\vee\theta_3.$$
	\item The value of $\fm^\con, \fm^\disc, \fm^\comp$, and $\fm^\thick$ for other inputs are all zero. Then, for a sequence $\overrightarrow{\theta} = (\theta_1, \dots, \theta_n)$, we define $$\fm_n(\overrightarrow{\theta}) \deq \fm^\con_n(\overrightarrow{\theta}) + \fm^\disc_n(\overrightarrow{\theta}) + \fm^\comp_n(\overrightarrow{\theta}) + \fm^\thick_n(\overrightarrow{\theta}).$$
	\end{itemize}	
\end{defn}

\begin{remark}
	Let $(\theta_1; (\theta_2, \theta_3))$ be a thick triple. Then, since $\theta_2$ and $\theta_3$ are interior morphisms, we have $$\deg{\theta_1^\vee} = \deg{\theta_1} + \deg{\theta_2} + \deg{\theta_3} - 1 = \deg{\theta_1} - 1.$$ To force every interior morphism have degree $0$, $\theta_1$ cannot be an interior morphism. This is the reason why we need the second condition of thick condition.	
\end{remark}

Let us unravel the definition in some cases.
Suppose that a disc sequence $\overrightarrow{\phi}= (\phi_1, \dots, \phi_n)$ is not folded. Let $k$ be the number of interior morphisms among them,
and let $k'$ be the number of interior morphisms between notched arcs (recall that due to the nice condition either both $\gamma_i$ and $\gamma_{i+1}$ are notched at $p$ or both are not). Then $\fm_n^\disc$ can be expressed as, when $(\theta, \phi_1)$ and $(\phi_n, \psi)$ are not folded,
	\begin{align*}
		\fm^\disc_n(\phi_1, \dots, \phi_n) &= (-1)^{k'}\left(\frac{1}{2}\right)^k e_{\gamma_1}, \\
		\fm^\disc_n(\theta \bullet \phi_1, \dots, \phi_n) &= (-1)^{k' + |\theta|}\left(\frac{1}{2}\right)^k \theta, \\
		\fm^\disc_n(\phi_1, \dots, \phi_n \bullet \psi) &= (-1)^{k'}\left(\frac{1}{2}\right)^k \psi.
	\end{align*}
Intuitively, $\left(\frac{1}{2}\right)$ factor is a contribution from idempotents, and negative signs are from the notched arcs.

Suppose that a disc sequence $\overrightarrow{\phi}= (\phi_1, \dots, \phi_n)$ has $l_1$ folded arc attached to boundary markings and $l_2$ folded
arc attached to interior markings. Let $k$ be the number of interior morphisms among $(\phi_1, \dots, \phi_n)$.
Let $k'$ be the number of interior morphisms between notched arcs and $k''$ be the number of folding pairs among these notched arcs. 
Then, $\fm_n^\disc$ can be expressed as 
\begin{equation}
\fm^\disc_n(\phi_1, \dots, \phi_n) = (-1)^{k'-k''}\left(\frac{1}{2} \right)^{k+l_1- \left<\phi_n, \phi_1\right>} e_{\gamma_1}.
\end{equation}

Suppose that a composition sequence $\overrightarrow{\phi} = (\phi_1, \dots, \phi_n)$ with value $\psi$ is not folded. 
Let $j$ be the number of interior morphisms among  $(\phi_2, \dots, \phi_{n - 1})$ (namely, we ignore the first and last morphisms $\phi_1$ and $\phi_n$),
and let $j'$ be the number of interior morphisms between notched arcs among $(\phi_2, \dots, \phi_{n-1})$.
Then, $\fm_n^\comp$ can be expressed as 
$$\fm^\comp_n(\phi_1, \dots, \phi_n) = (-1)^{j'}\left(\frac{1}{2}\right)^j \psi.$$

Suppose that a composition sequence $\overrightarrow{\phi} = (\phi_1, \dots, \phi_n)$ with value $\psi$ has a folding but not along $\gamma_1$ nor $\gamma_n$. Then, define $j$ and $j'$ as above. Let $j''$ be the number of folding pairs among notched arcs at interior markings.
 Let $l_1$ (resp. $l_2$) be the number of folded arcs attached to the boundary markings (resp. interior markings).
$$\fm^\comp_n(\phi_1, \dots, \phi_n) = (-1)^{j'-j'' + \sigma(\phi_1) + \sigma(\phi_n)}\left(\frac{1}{2}\right)^{j+l_1} \psi.$$

Although the appearance of $\fm^\thick_3$ might look surprising, it has a geometric explanation.
An arc of interior number $1$ has a lift in the double covering surface, which is preserved under $\super$-action. After suitable perturbations of arcs involved in the operation $\fm^\thick_3$, one can see a disc sequence among them.
However, we will proceed with the above algebraic definition.

These data indeed form an $\Aoo$-category.
\begin{thm}\label{thm:ItIsAooCategory}
	Let $(S, M, O, \eta)$ be a graded marked orbi-surface and $\Gamma$ be a nice tagged arc system in it. Then, the data $\cff_\Gamma(S, M, O, \eta)$ defined in Definition \ref{defn:FukayaCategoryForMarkedSurface} is an $\Aoo$-category.
\end{thm}
Its proof will be given in Section \ref{sec:Aooness}, and we only illustrate the idea here (without sign).

Consider a disc $D$ that defines a disc sequence  
as in Figure \ref{fig:AooForOrbiSurface}(a) or (b). Suppose that there is a tagged arc $\alpha$ on $D$ (in the tagged arc system) that divides $D$ into two pieces $D_1$ and $D_2$.
We will consider how this can be realized as a composition of two $\Aoo$-operations.

We can divide this into three cases, depending the number of ends of $\alpha$ lying on  folded interior markings of $D$.
The first case (i) is when none of the ends of $\alpha$ lie on  folded interior markings of $D$, and 
we will see that  this corresponds to the composition of two disc sequences.
The second case (ii) is when exactly one end lies on folded interior markings $D$,
and we will see that this corresponds to the composition of a disc sequence and a composition sequence.
The last case is when two ends lie on folded interior markings,
but this is excluded by the assumption of our tagged arc system that at least one of $D_1$ and $D_2$ must have an unmarked boundary component
(but this cannot happen because $D$ is a disc sequence).
\begin{itemize}
\item {\bf (Disc splits into two discs)} 
This is the case (i) as illustrated in Figure \ref{fig:DiscSequenceDecomposition}(a).
In this case, the arguments are similar to that of \cite{HKK17}, except that we allow folded discs as in Figure \ref{fig:AooForOrbiSurface}(b).
Suppose the inputs for the $\Aoo$-formula  are 
$$(\phi_1,\cdots, \phi_i, \psi_1, \cdots, \psi_n \bullet \phi_{i+1}, \cdots, \phi_m).$$
When $i > 0$, the splitting as in Figure \ref{fig:DiscSequenceDecomposition}(a) produces the following two canceling terms.
\begin{equation}\label{eq:dd}
\fm_k(\phi_1,\cdots,\phi_i, \fm_n(\psi_1,\cdots,\psi_n \bullet \phi_{i+1}),\cdots,\phi_m), \fm_{m+n-1}(\phi_1,\cdots, \fm_2(\phi_i,  \psi_1),\cdots,\psi_n \bullet \phi_{i+1} ,\cdots,\phi_m)
\end{equation}
In the case of $i=0$, two canceling terms are given by
\begin{equation}\label{eq:dd2} 
\fm_m( \fm_n(\psi_1,\cdots,\psi_n \bullet \phi_{1}),\cdots,\phi_m), \fm_{n}(\psi_1,\cdots, \fm_m( \psi_n \bullet \phi_1,\cdots,\phi_m))
\end{equation}
We will check later that they have the same weights (the same factor of $\frac{1}{2}$) and opposite signs.

\def\mySize{0.3\linewidth}
 \begin{figure}[h!]
 	\centering
 	\begin{subfigure}[b]{\mySize}
 		\includegraphics[width=\linewidth]{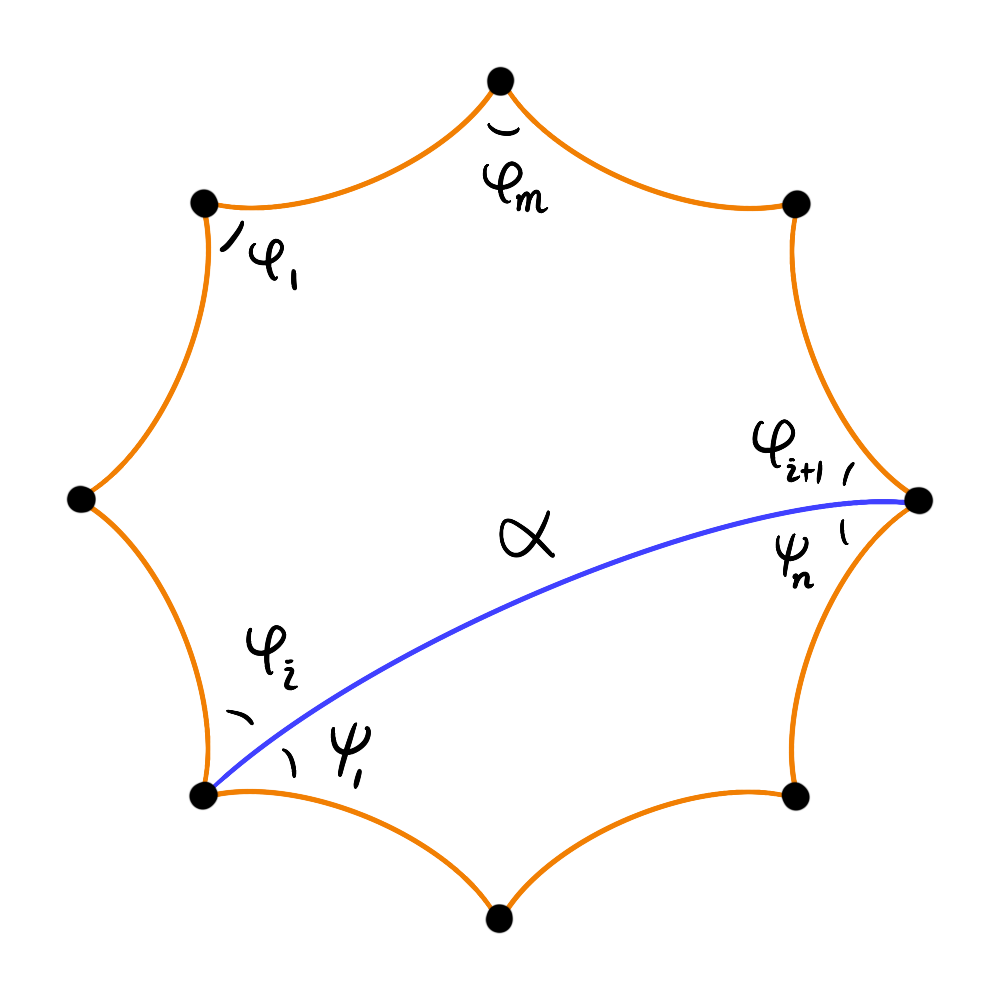}
 		\caption{}
 	\end{subfigure}
 	\begin{subfigure}[b]{\mySize}
 		\includegraphics[width=\linewidth]{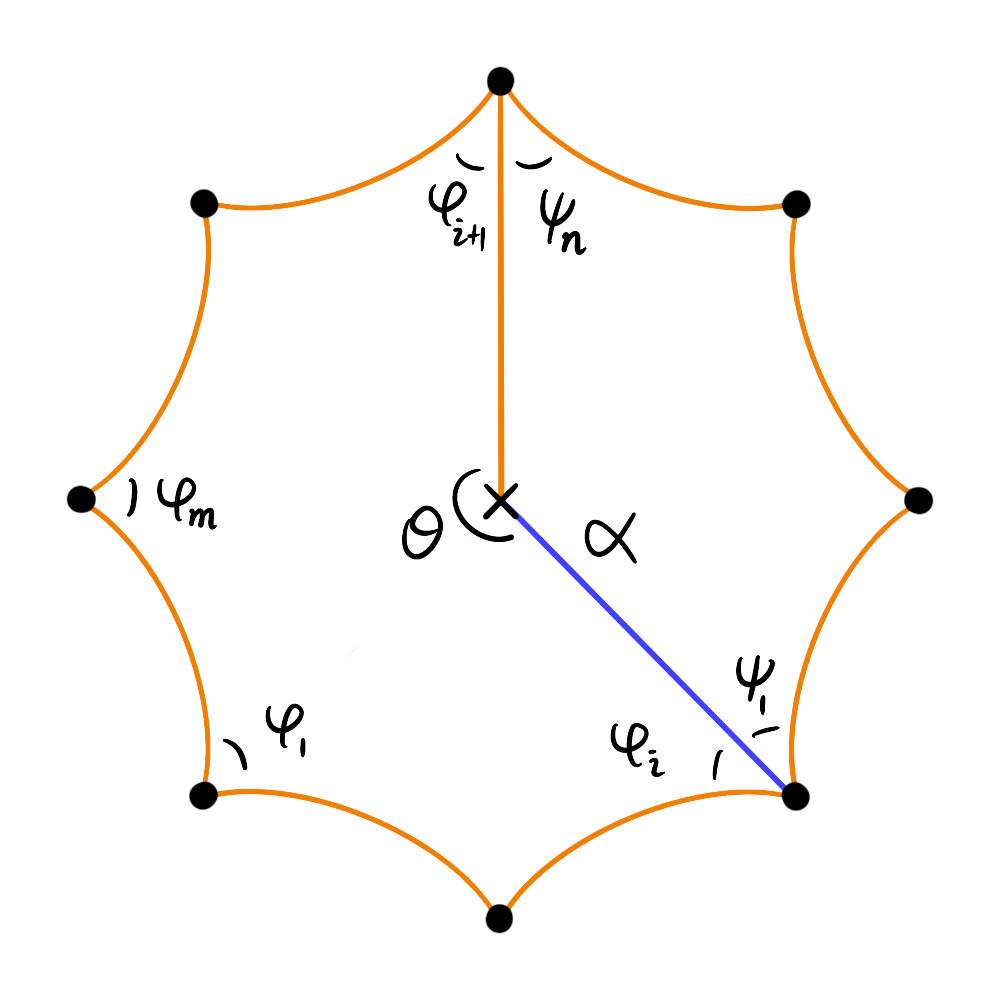}
 		\caption{}
 	\end{subfigure}
	\caption{Two ways of disc sequence splitting}
	\label{fig:DiscSequenceDecomposition}
\end{figure}

\item {\bf (Disc splits into composition and disc sequences)}
Let us discuss the case (ii) as illustrated in Figure \ref{fig:DiscSequenceDecomposition}(b).
This gives rise to the following three canceling terms
Suppose the inputs for the $\Aoo$-formula are 
$$(\phi_1,\cdots, \phi_i, \psi_1, \cdots, \psi_n, \phi_{i+1}, \cdots, \phi_m).$$
Then, we have the following three possible quadratic expressions.
\begin{equation}\label{eq:DCComp}
\fm_{m+1}^\disc(\phi_1,\cdots,\phi_i, \fm_n^\comp(\psi_1,\cdots,\psi_n), \phi_{i+1},\cdots,\phi_m)
\end{equation}
\begin{equation}\label{eq:DCConL}
\fm_{m+1}^\disc(\phi_1,\cdots,\fm_2^\con(\phi_i, \psi_1),\cdots,  \psi_n, \phi_{i+1} ,\cdots,\phi_m)
\end{equation}
\begin{equation}\label{eq:DCConR}
\fm_{m+1}^\disc(\phi_1,\cdots,\phi_i, \psi_1,\cdots, \fm_2^\con(\psi_n ,\phi_{i+1}),\cdots,\phi_m)
\end{equation}
Figure \ref{fig:DiscSequenceDecomposition}(b) represents the expression \eqref{eq:DCComp} and Figure \ref{fig:DiscDecomposition2}(a) and (b) represent the folded disc sequences for the expressions \eqref{eq:DCConL} and \eqref{eq:DCConR}, respectively.

The interior marking (at the output of $\fm_n^\comp$) contributes weight $\frac{1}{2}$ to $\fm_{m+1}^\disc$ operation, but weight $1$ to $\fm_n^\comp$ in the first formula. So the total weight is $\frac{1}{2}$.
In the second and third formula, both $\fm_{m+1}^\disc$ and $\fm_2^\con$ have weight $\frac{1}{2}$, so the total weight is $\frac{1}{4}$. Thus they are canceled in the form $\frac{1}{2} - \frac{1}{4} - \frac{1}{4}$.
Therefore, it is crucial for the $\Aoo$-identities that we have  these weights.
Here, we assume that $1<i<m$. The other cases have to be handled differently.
\def\mySize{0.3\linewidth}
 \begin{figure}[h!]
	\centering
	\begin{subfigure}[b]{\mySize}
		\includegraphics[width=\linewidth]{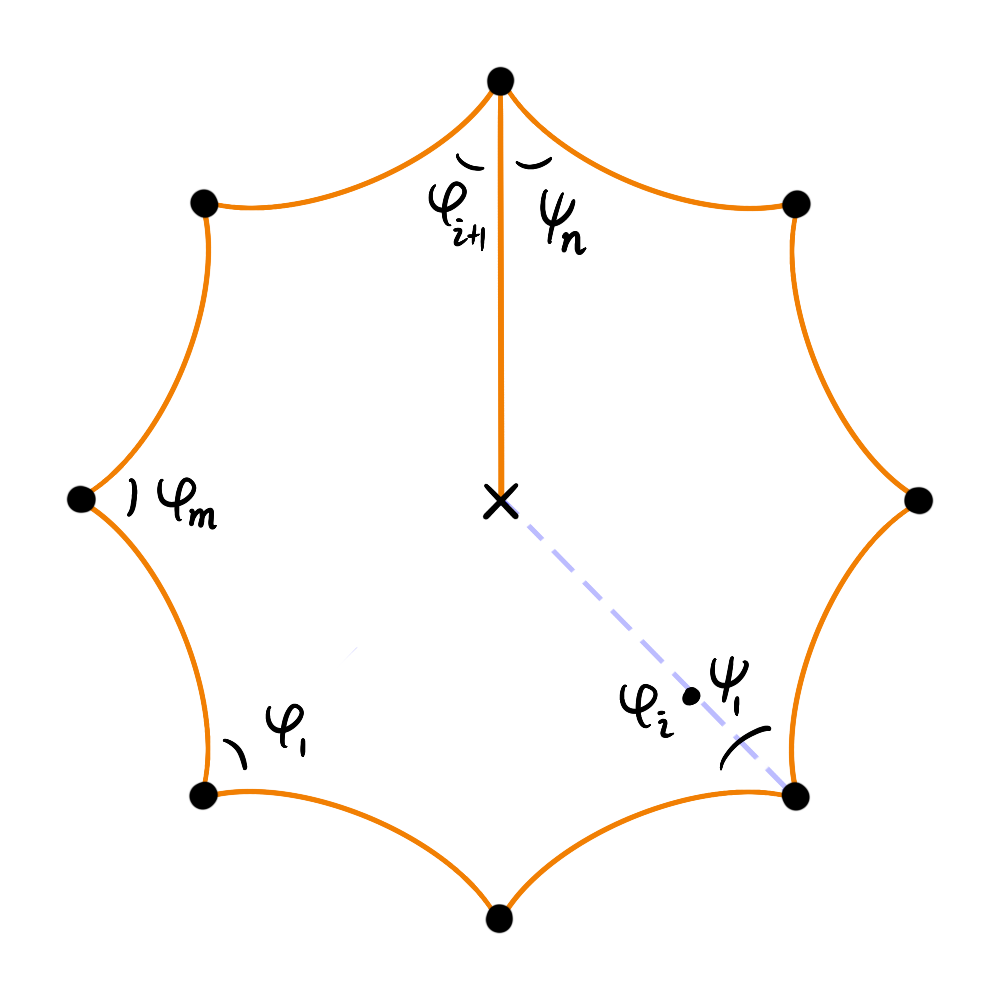}
		\caption{}
	\end{subfigure}
	\begin{subfigure}[b]{\mySize}
		\includegraphics[width=\linewidth]{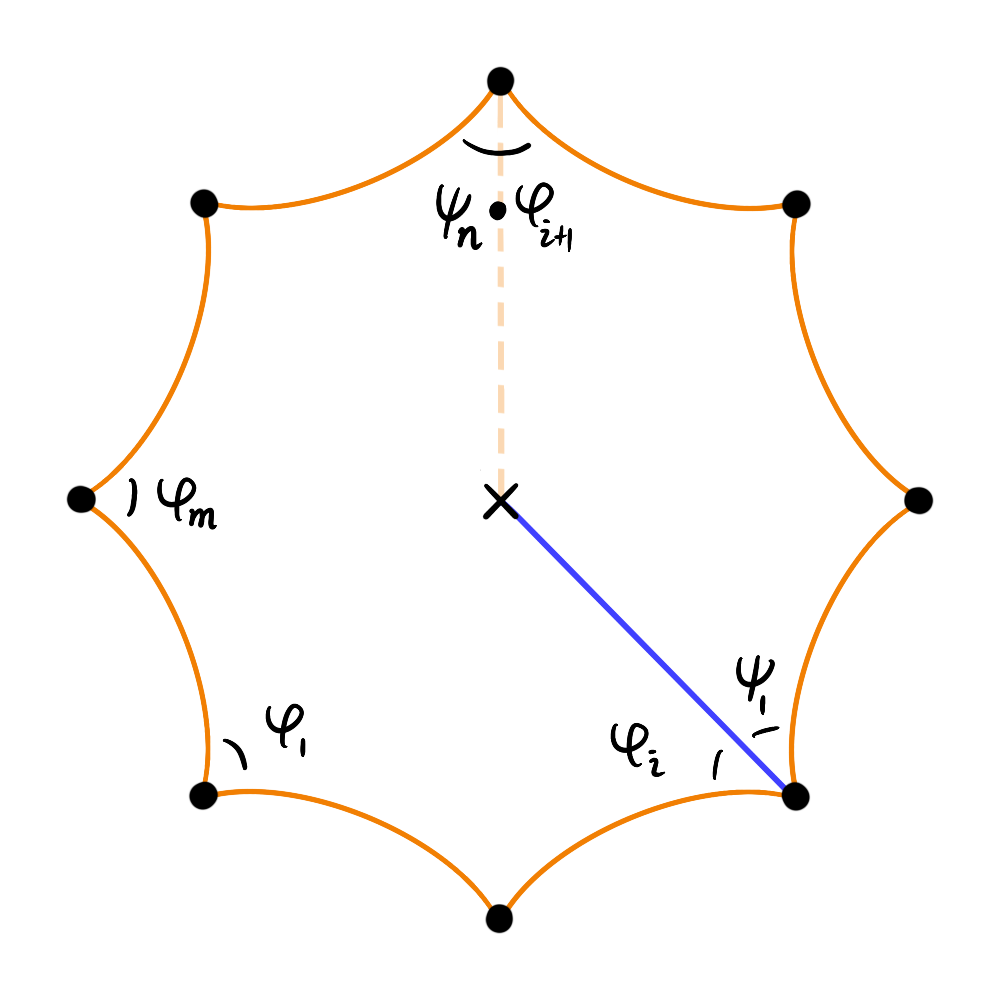}
		\caption{}
	\end{subfigure}
	\caption{Folded discs for $\Aoo$-identity of case (ii)}
	\label{fig:DiscDecomposition2}
\end{figure}
\end{itemize}

Let us also illustrate other properties of our $\Aoo$-identities.

\begin{itemize}
	\item {\bf (Composition sequence splits)}
	The basic idea is similar to the case of disc sequence, but there are more variations depending on whether a splitting involves the first or last input of a composition sequence or not. A composition sequence can be split in three ways : (iii) one disc sequence and one composition sequence, (iv) two composition sequences, and (v) one disc sequence and a disc which is neither a disc sequence nor a composition sequence.
	\def\mySize{0.245\linewidth}
	\begin{figure}[h!]
		\centering
		\begin{subfigure}[b]{\mySize}
			\includegraphics[width=\linewidth]{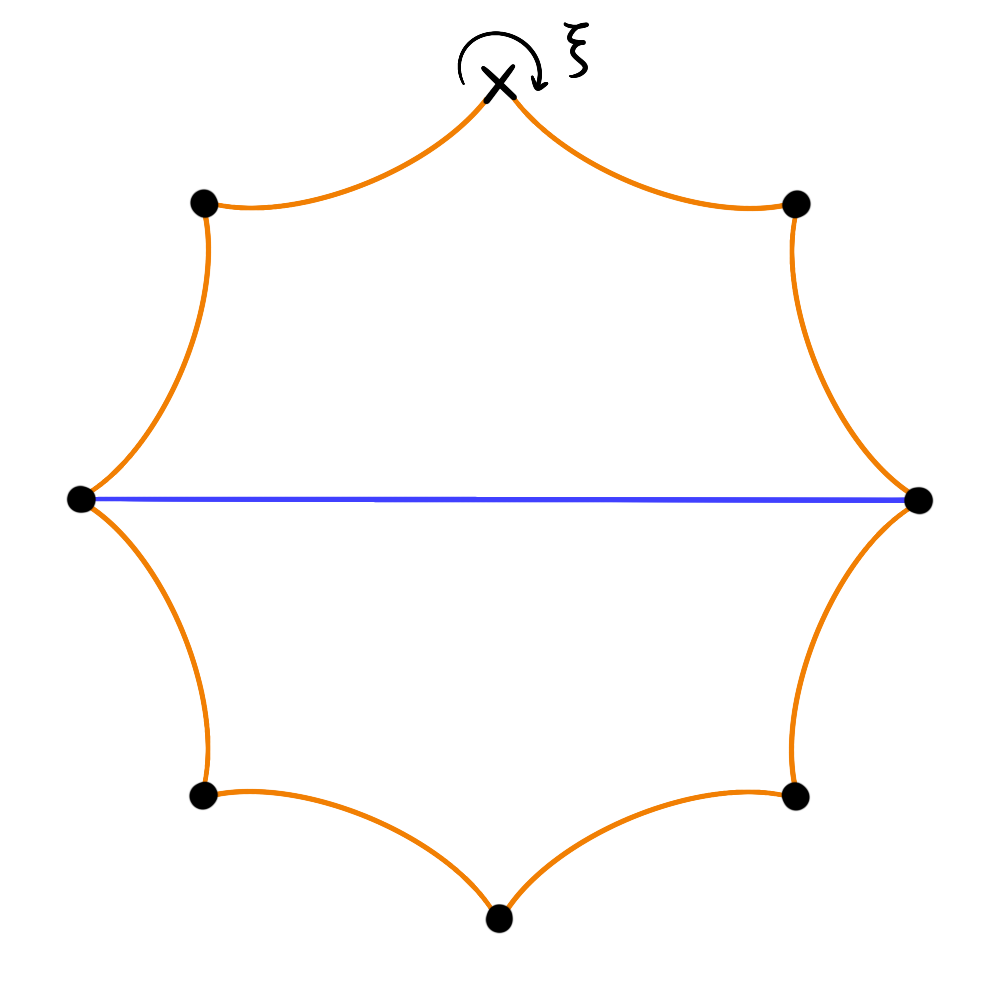}
			\caption{}
		\end{subfigure}
		\begin{subfigure}[b]{\mySize}
			\includegraphics[width=\linewidth]{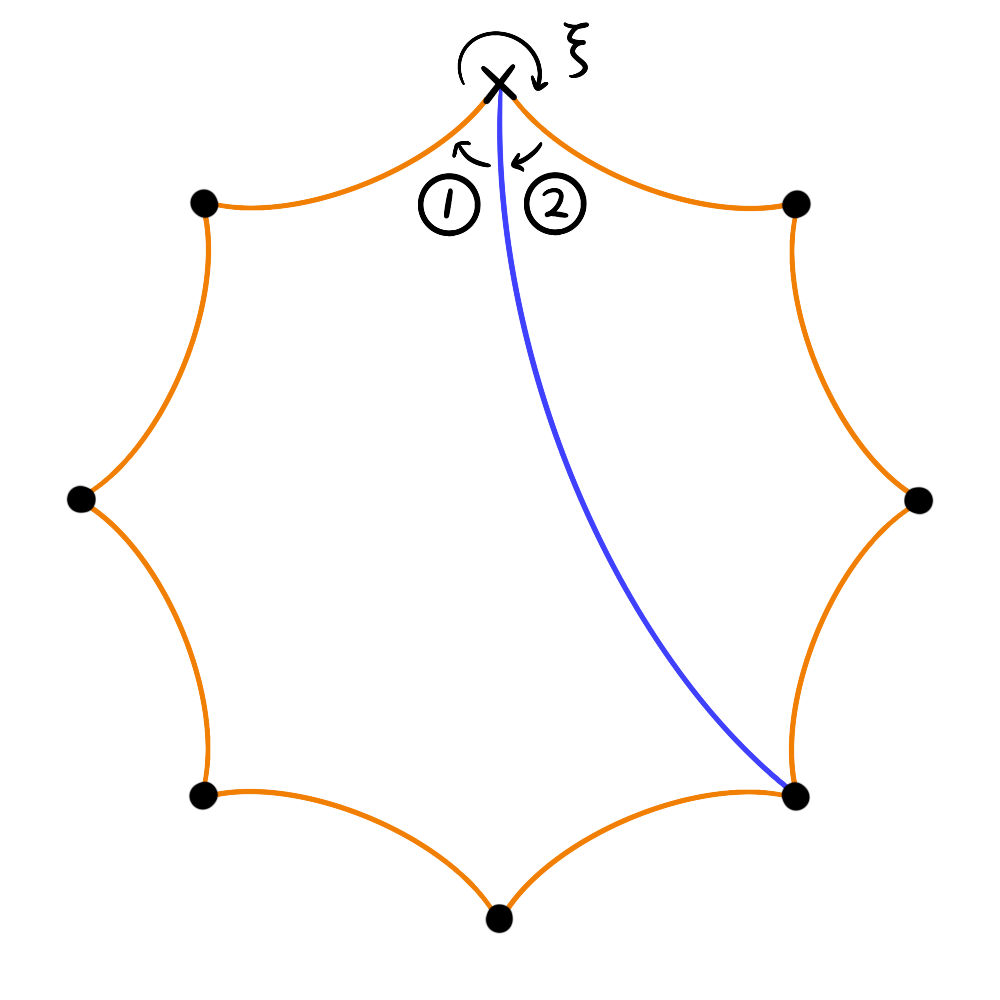}
			\caption{}
		\end{subfigure}
		\begin{subfigure}[b]{\mySize}
			\includegraphics[width=\linewidth]{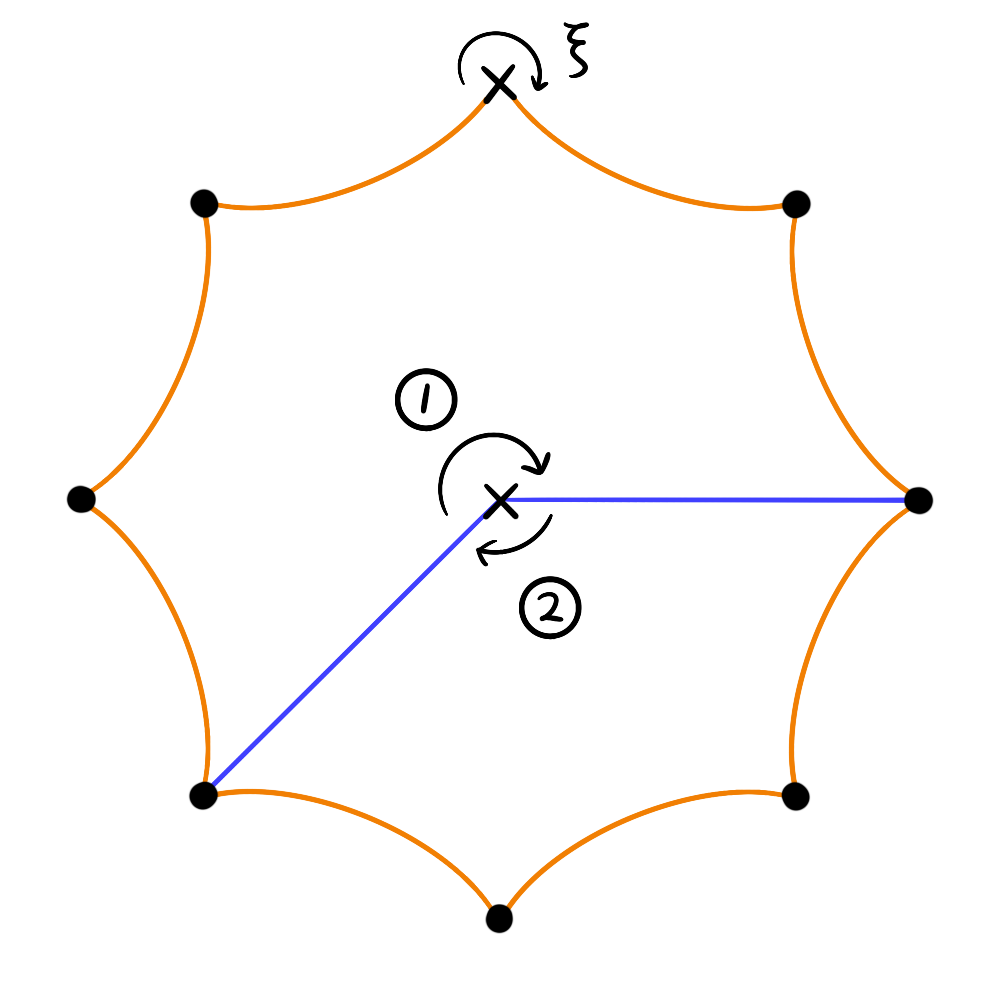}
			\caption{}
		\end{subfigure}
		\begin{subfigure}[b]{\mySize}
			\includegraphics[width=\linewidth]{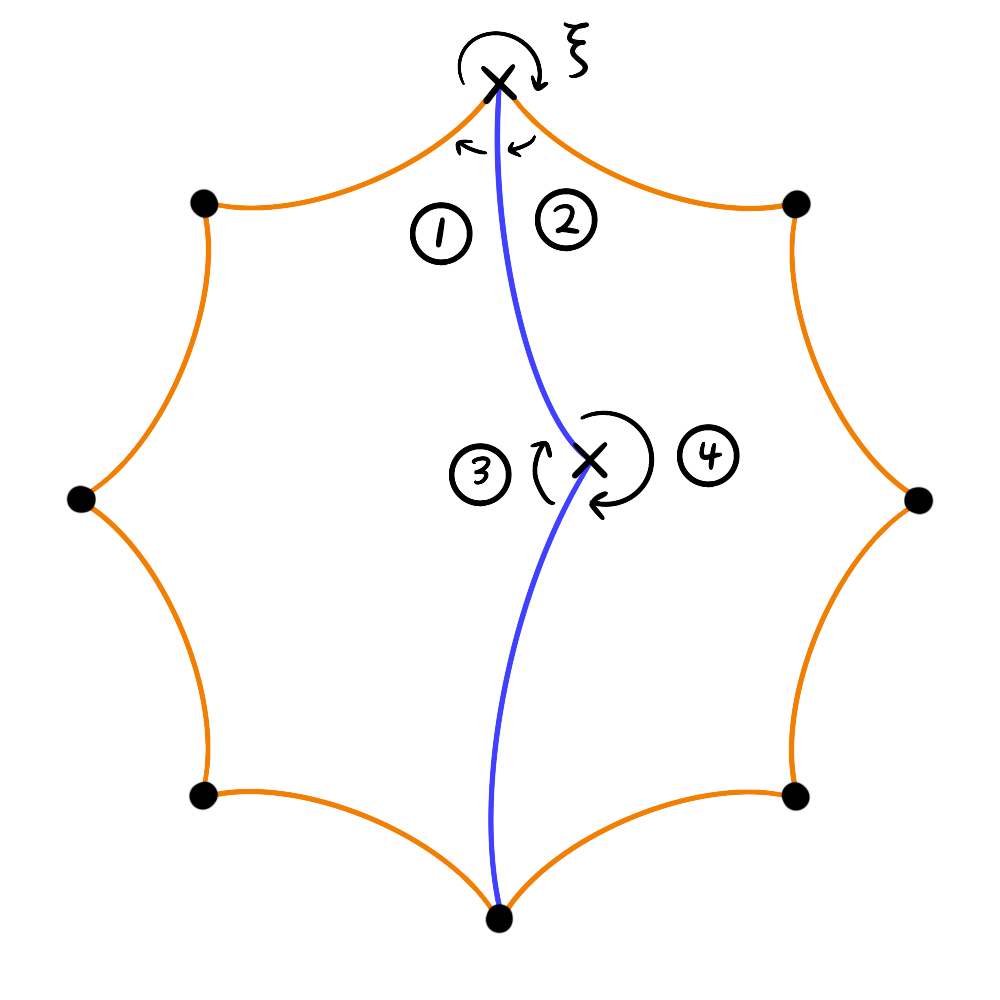}
			\caption{}
		\end{subfigure}
		\caption{Splitting of composition sequences}
		\label{fig:CompDecomposition1}
	\end{figure}
	\begin{lemma}\label{lem:comdecomp}
		A composition sequence splits into the following cases as in Figure \ref{fig:CompDecomposition1}.
		A disc $D$ of a composition sequence with an output $\xi$ decomposes in one of the following way. Let $\alpha$ be the dissecting arc.
		\begin{enumerate}
			\item Figure \ref{fig:CompDecomposition1}(a), when $\alpha$ does not meet neither folded interior markings nor the output $\xi$. The disc $D$ decomposes as (iii).
			\item Figure \ref{fig:CompDecomposition1}(b), when $\alpha$ does not meet folded interior markings but meet the output $\xi$. The disc $D$ decomposes as (iii).
			This divides into two subcases: (b)-1 if {\em $\textcircled{1}$} is an interior morphism (hence {\em $\textcircled{2}$} is a non-morphism), and the case (b)-2 on the other case.
			\item Figure \ref{fig:CompDecomposition1}(c), when $\alpha$ meets a folded interior marking but does not meet the output $\xi$. This has two subcases: (c)-1 If {\em $\textcircled{1}$} is an interior morphism, the disc $D$ decomposes as (iv), (c)-2 if {\em $\textcircled{2}$} is an interior morphism, then the disc $D$ decomposes as (v).
			\item Figure \ref{fig:CompDecomposition1}(d), when $\alpha$ meets both a folded interior marking and the output $\xi$. This divides into four subcases: the disc $D$ decomposes as (v), (iv), (iv), and (v), \resp.
			$$\begin{cases}
			 (d)-1 & \text{{\em $\textcircled{1}$, $\textcircled{3}$} are interior morphisms (hence {\em $\textcircled{2}$, $\textcircled{4}$} are non-morphisms)}, \\
			 (d)-2 & \text{{\em $\textcircled{2}$, $\textcircled{3}$} are interior morphisms (hence {\em $\textcircled{1}$, $\textcircled{4}$} are non-morphisms)}, \\
			 (d)-3 & \text{{\em $\textcircled{1}$, $\textcircled{4}$} are interior morphisms (hence {\em $\textcircled{2}$, $\textcircled{3}$} are non-morphisms)}, \\
			 (d)-4 & \text{{\em $\textcircled{2}$, $\textcircled{4}$} are interior morphisms (hence {\em $\textcircled{1}$, $\textcircled{3}$} are non-morphisms)}.
			 \end{cases}$$
		\end{enumerate}
	\end{lemma}
	In Section \ref{sec:Aooness}, we will spell out the corresponding disc/composition sequences and compute related weights and signs.
	
	\item {\bf (Thick pairs)} Suppose that we are given a disc with two markings and one interior marking $p$. Let $\alpha^+, \alpha^-, \beta$, and $\gamma$ be tagged arcs on the disc as in Figure \ref{fig:ThickAooExample}. Here, $(\alpha^+, \alpha^-)$ is a thick pair and assume $p$ defines $p^{\alpha^+}_\beta$ and $p^\beta_{\alpha^-}$. Then, we have one disc sequence $(p^{\beta}_{\alpha^-}, \theta^-, \phi)$ and one composition sequence $(\theta^+, \phi ; p^{\alpha^+}_\beta)$. In particular, $\fm_2(p^{\alpha^+}_\beta, \fm_3(p^\beta_{\alpha^-}, \theta^-, \phi)) = \pm\frac{1}{2}p^{\alpha^+}_\beta$. Then, there has to be a term cancel this. This is why we need $\fm^\thick$ operation. Indeed, we have $$\fm_2(\fm_3(p^{\alpha^+}_\beta, p^\beta_{\alpha^-}, \theta^-), \phi) = \pm\frac{1}{2}\fm_2(\theta^+, \phi) = \pm\frac{1}{2}p^{\alpha^+}_\beta.$$

	\def\mySize{0.3\linewidth}
	\begin{figure}[h!]
		\centering
		\includegraphics[width=\mySize]{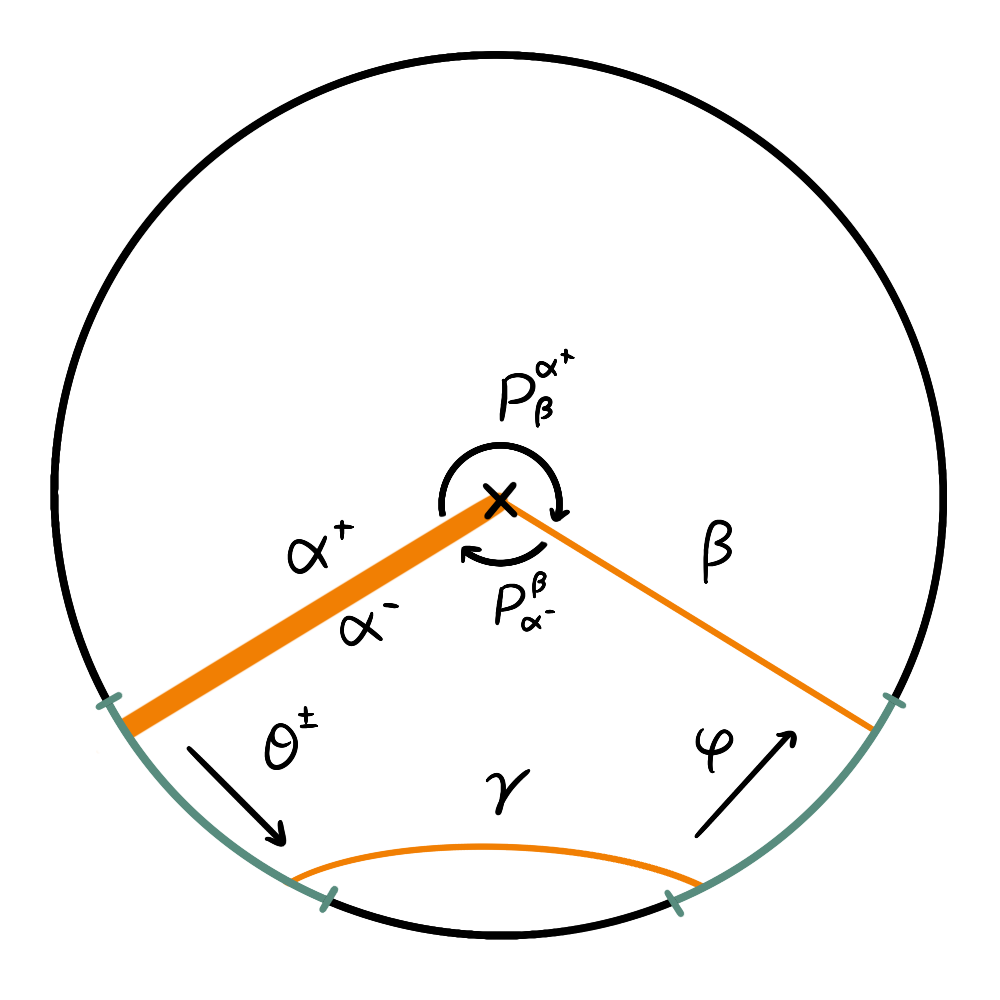}
		\caption{Example of $\Aoo$-identity including $\fm^\thick$}
		\label{fig:ThickAooExample}
	\end{figure}

	\item {\bf (Negative sign for notched interior morphisms)} It is not immediately clear why these extra signs are necessary.
	One important usage is for the generation of the Fukaya category. This will be discussed in complete generality in the next section, so let us only explain the idea by
	a single example. Consider a disc with one interior marking, and a tagged arc system as in Figure \ref{fig:TwistedComplexExample}. Here, we have the thick pair  $(\gamma_2^+,\gamma_2^-)$,
	and morphisms $\theta_1^\pm \in \ho^1(\gamma_1,\gamma_2^\pm), \theta_2^\pm \in \ho^1(\gamma_2^\pm, \gamma_3)$.
	\def\mySize{0.3\linewidth}
	\begin{figure}[h!]
		\centering
		\includegraphics[width=\mySize]{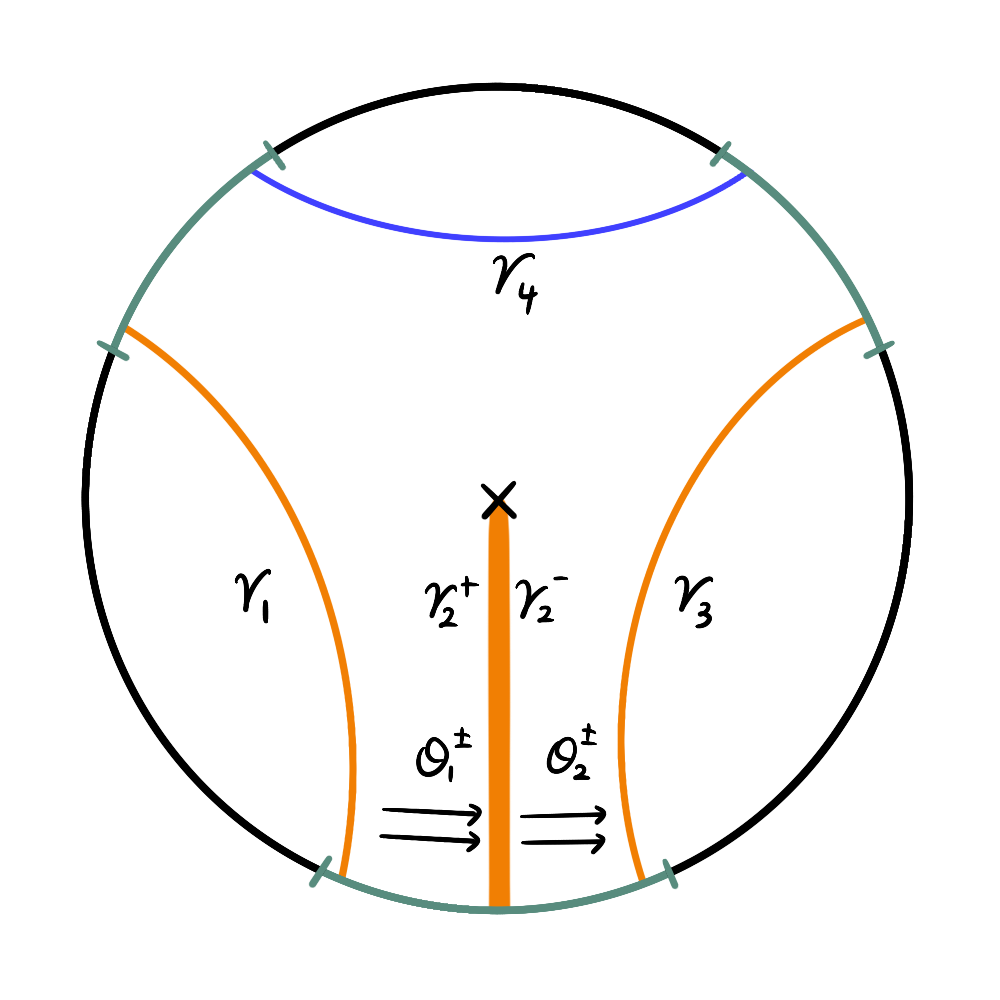}
		\caption{Example of a twisted complex}
		\label{fig:TwistedComplexExample}
	\end{figure}
	\begin{lemma}
	We have a twisted complex, which is quasi-isomorphic to the arc $\gamma_4$,
	$$\gamma_1 \overset{(\theta_1^+,\theta_1^-)}{\longrightarrow} \gamma_2^+ \oplus \gamma_2^- \overset{(\theta_2^+,\theta_2^-)}{\longrightarrow} \gamma_3.$$
	\end{lemma}
	\begin{proof}
		To show that it defines a twisted complex, it is enough to check that the following vanishes:
		$$\fm_2^\con(\theta_1^+,\theta_2^+) + \fm_2^\con(\theta_1^-,\theta_2^-)$$
		By definition, it equals $\frac{1}{2}\theta_3 - \frac{1}{2}\theta_3=0$ and we obtain the claim. Isomorphism to $\gamma_4$ will be proved at the next section in greater generality.
	\end{proof}
\end{itemize}

\section{Morita equivalence}\label{section:MoritaEquivalence}
We have defined an $\Aoo$-category of tagged arc system on a graded marked orbi-surface $(S, M, O, \eta)$.
We will investigate an equivalence relation among different tagged arc systems in this section.
We will show that we can transform any tagged arc system in its equivalence class into a special type, called an {\em involutive} arc system.
\begin{defn}
	 We say two tagged arc systems $\Gamma_1$ and $\Gamma_2$  for a graded marked orbi-surface $(S, M, O, \eta)$ are {\em Morita equivalent} if their associated $\Aoo$-categories $\cff_{\Gamma_1}(S, M, O, \eta)$ and $\cff_{\Gamma_2}(S, M, O, \eta)$ are Morita equivalent. That is, their idempotent completions of triangulated enhancements $$\Pi(\Tw(\cff_{\Gamma_1}(S, M, O, \eta))), \quad \Pi(\Tw(\cff_{\Gamma_2}(S, M, O, \eta)))$$ are quasi-equivalent. In this case, we denote by $\Gamma_1 \simeq \Gamma_2$.
\end{defn}
We consider the case of an inclusion.
\begin{prop}\label{prop:MoritaEquivalency}
	Let $(S, M, O, \eta)$ be a graded marked orbi-surface and $\Gamma$ be a tagged arc system. Suppose that for an arc $\gamma \in \Gamma$, $\Gamma' = \Gamma \setminus \{\gamma\}$ is also a tagged arc system. Then, the natural inclusion functor $$\cff_{\Gamma'}(S, M, O, \eta) \rightarrow \cff_{\Gamma}(S, M, O, \eta)$$ is a Morita equivalence. In particular, $\Gamma$ and $\Gamma'$ are Morita equivalent.
\end{prop}
Given a disc sequence $(\theta_1, \dots, \theta_r)$ for surfaces in Section \ref{section:TFCofSurfaces}, $\gamma_r$ is
isomorphic to the twisted complex 
$$\gamma_1 \xrightarrow{\theta_1} \cdots    \xrightarrow{\theta_{r-2}} \gamma_{r-1}$$
as shown in  \cite[Section 3.3]{HKK17}.
We will show a tagged arc version of this lemma, and the above proposition follows immediately.

We introduce some notations.
Let $\Gamma$ be a nice tagged arc system on $(S, M, O, \eta)$, and $(\theta_1, \dots, \theta_r)$ be a disc sequence   with $\theta_i: \gamma_i \rightarrow \gamma_{i+1}$ for each $i$ such that each folded arc is thick.
If $\gamma_i$ is not folded, we also write $\gamma_i$ as $\gamma_i^0$.
If $\gamma_i$ is folded, then we write $\gamma_i$ as $\gamma_i^+$ and denote by $\gamma_i^-$ its thick pair.
Hence, we introduce the following notation to make folded arc thick.
$$\Delta(\gamma_i) \deq 
	\begin{cases} 
		\gamma^0_i 
		&\text{if $\gamma_i$ is not folded,} \\
		\gamma^+_i \oplus \gamma^-_i[d_i] 
		&\text{if $\gamma_i$ is folded.} 
	\end{cases}$$
	
The shift $d_i \in \mathbb{Z}$ will be defined below.
We introduce the index set $I_i$ as $I_i = \{0\}$ if $\gamma_i$ is not folded, and $I_i =\{+,-\}$ if it is folded.
We also make the basic morphism $\theta_i:\gamma_i \rightarrow \gamma_{i+1}$ into a thick version.
Namely, there are basic morphisms $(\theta_i)^{\sigma_1}_{\sigma_2}: \gamma^{\sigma_1}_i \rightarrow \gamma_{i+1}^{\sigma_2}$
for $\sigma_1 \in I_i, \sigma_2 \in I_{i+1}$. These form the following $|I_{i+1}|\times |I_i|$ matrix $\Delta(\theta_i)$.
\begin{equation}\label{eq:dtheta}
	\Delta(\theta_i) \deq 
	\begin{cases} 
		\begin{bmatrix} 
			(\theta_i)^0_0
		\end{bmatrix} 
		&\text{if both $\gamma_i$ and $\gamma_{i+1}$ are not folded,} \\
		\begin{bmatrix} 
			(\theta_i)^+_0 \\ (\theta_i)^-_0
		\end{bmatrix} 
		&\text{if $\gamma_i$ is not folded and $\gamma_{i+1}$ is folded,} \\ 
		\begin{bmatrix} 
			(\theta_i)^0_+ & (\theta_i)^0_- 
		\end{bmatrix} 
		&\text{if $\gamma_i$ is folded and $\gamma_{i+1}$ is not folded,} \\ 
		\begin{bmatrix} 
			(\theta_i)^+_+ & (\theta_i)^-_+ \\ (\theta_i)^+_- & (\theta_i)^-_- 
		\end{bmatrix} 
		&\text{if both $\gamma_i$ and $\gamma_{i+1}$ are folded.}
	\end{cases}
	\end{equation}
In the definition, we shift $\gamma_i^-$ so that the grading of  $\gamma^+_i$ and $\gamma^-_i[d_i]$ are the same.
This  makes $\Delta(\theta_i)$ homogeneous.

\begin{lemma}\label{Lemma:TwistedComplex}
	The complex $T$ given by 
	$$ \Delta(\gamma_1) \xrightarrow{\Delta(\theta_1)} \Delta(\gamma_2)[s_2] \xrightarrow{\Delta(\theta_2)} \Delta(\gamma_3)[s_3] \rightarrow \dots \rightarrow \Delta(\gamma_{r-1})[s_{r-1}]$$
	is a twisted complex and quasi-isomorphic to $\Delta(\gamma_r)[-| \theta_r |]$.
	Here $s_i = \deg{\Delta(\theta_1)} + \dots + \deg{\Delta(\theta_{i - 1})} - i + 1$.
\end{lemma}
Assuming this lemma, let us give a proof of Proposition \ref{prop:MoritaEquivalency}.
\begin{proof}
	We claim that there is a disc sequence containing $\gamma$ in $\Gamma$ such that any folded arc other than $\gamma$ are all thick. Suppose that $\gamma$ is a boundary of only one disc component $D$. Then, $\gamma$ is folded. Since $\Gamma \setminus \{\gamma\}$ is still a tagged arc system, $D$ has to be a disc sequence such that every folded arc other than $\gamma$ is thick. Now suppose that $\gamma$ is a boundary of two disc components $D_1$ and $D_2$. Then, $D_1 \cup D_2$ has at most one unmarked boundary, folded arc which is not thick, or outward interior marking. Therefore, either $D_1$ or $D_2$ defines a disc sequence with the desired property.
	
	Then, by Lemma \ref{Lemma:TwistedComplex}, the object $\Delta(\gamma)$ is quasi-isomorphic to a twisted complex consisting of tagged arcs of $\Gamma'$. Thus $\gamma$ is in the category $\Pi\Tw(\cff_{\Gamma'}(S, M, O, \eta))$, which proves the inclusion functor from $\cff_{\Gamma'}(S, M, O, \eta)$ to $\cff_{\Gamma}(S, M, O, \eta)$ is a Morita equivalence.	
\end{proof}

Now let us begin the proof of Lemma \ref{Lemma:TwistedComplex}.
\begin{proof}
	Let us first show that $T$ is a twisted complex. Since $(\theta_1, \dots, \theta_{r-2})$ forms neither a disc sequence nor a composition sequence, the only nonzero terms in the Maurer-Cartan equation are $\fm_2(\Delta(\theta_i), \Delta(\theta_{i+1}))$, for $i=1, \dots, r-2$. However, this also becomes $0$. If $\gamma_{i+1}$ is not folded, then, for any $\sigma_i \in I_i$ and $\sigma_{i+2} \in I_{i+2}$, $$\fm_2(\theta^{\sigma_i}_0, \theta^0_{\sigma_{i+2}}) = 0$$ as they are not concatenable. Now suppose that $\gamma_{i+1}$ is folded so that $\Delta(\gamma_{i+1}) = \gamma_{i+1}^+ \oplus \gamma_{i+1}^-[d_{i+1}]$. Let $\tau_{\pm}$ be the tagging of $\gamma_{i+1}^\pm$ at the interior marking opposite to $\theta_i$. Then, by the full condition, we have $\tau_+ = \tau_- + d_{i+1} + 1$. So, we have $$\fm_2((\theta_i)^{\sigma_i}_+, (\theta_{i+1})^+_{\sigma_{i+2}}) + \fm_2((\theta_i)^{\sigma_i}_-, (\theta_{i+1})^-_{\sigma_{i+2}}) = \frac{1}{2}(-1)^{\tau_+}\theta_i \bullet \theta_{i+1} + \frac{1}{2}(-1)^{\tau_- + d_{i+1}}\theta_i\bullet \theta_{i+1} = 0.$$ This shows the complex $T$ is a twisted complex.
	
	Now let us show the quasi-isomorphicity. It is enough to show that following two morphisms (from the last, and to the first term of $T$)
	\begin{equation}\label{eq:tiso}
	\Delta(\theta_{r-1}): T \rightarrow \Delta(\gamma_r)[-\deg{\theta_r}], \quad \Delta(\theta_r): \Delta(\gamma_r)[-\deg{\theta_r}] \rightarrow T
	\end{equation}
	are quasi-isomorphisms and they are quasi-inverse  to each other (up to scaling).
	
	Denote by $\Sigma$, $\Phi$ and $\nu$ the sign, the weight of the disc sequence $(\theta_1, \dots, \theta_r)$, and the number of folded arcs among $(\gamma_1, \dots, \gamma_r)$, \resp.
	The differentials in the twisted complex $T$ is denoted by $\delta$. Namely, one can think of $\delta$ as a $(r-1) \times (r-1)$ block matrix
	whose $(i,i+1)$st block is $\Delta(\theta_i)$ of size $|I_{i+1}|\times |I_i|$ and 0 matrices on other blocks.
	In this setting, two morphisms in \eqref{eq:tiso} can be written as the following block matrices. We write $\Delta^B(\theta_r)$ the $1 \times (r-1)$  block matrix whose first block is  $\Delta(\theta_r)$ and $\Delta^B(\theta_{r-1})$ the $(r-1) \times 1$ block matrix whose last block is $\Delta(\theta_{r-1})$.
	
	First, let us compute $\fm^{0, \delta, 0}_2(\Delta^B(\theta_r), \Delta^B(\theta_{r-1}))$. We have
	\begin{align*}
	 \fm^{0, \delta, 0}_2(\Delta^B(\theta_r), \Delta^B(\theta_{r-1}))  =&\fm_r(\Delta^B(\theta_r), \delta, \dots, \delta, \Delta^B(\theta_{r-1})) \\
	 =& \fm_r(\Delta(\theta_r), \Delta(\theta_1), \Delta(\theta_2), \dots, \Delta(\theta_{r-1})) \\
	 =& \sum_{( \sigma_1, \dots, \sigma_{r-1}, \sigma_r) \in I_1 \times \dots \times I_{r}}\fm_r((\theta_r)^{\sigma_r}_{\sigma_1}, (\theta_1)^{\sigma_1}_{\sigma_2}, \dots, (\theta_{r-1})^{\sigma_{r-1}}_{\sigma_r}) . 
	\end{align*}	
		When $\gamma_r$ is not folded, this equals
		$$(-1)^{\Sigma}\left(\frac{1}{2}\right)^{\Phi} (2)^\nu e_{\gamma^0_r} =(-1)^{\Sigma}\left(\frac{1}{2}\right)^{\Phi-\nu} e_{\gamma^0_r} .$$ 
		When $\gamma_r$ is folded, this equals
	\begin{align*}
	& \sum_{(\sigma_1, \dots, \sigma_{r-1}) \in I_1 \times \dots \times I_{r-1}}
	\begin{bmatrix} 
		\fm_r((\theta_r)^+_{\sigma_1}, (\theta_1)^{\sigma_1}_{\sigma_2}, \dots, (\theta_{r-1})^{\sigma_{r-1}}_+)  & 0 \\ 
		0 & \fm_r((\theta_r)^-_{\sigma_1}, (\theta_1)^{\sigma_1}_{\sigma_2}, \dots, (\theta_{r-1})^{\sigma_{r-1}}_-) \\
	\end{bmatrix} \\
	=&(-1)^\Sigma\half{\Phi-\nu}\begin{bmatrix}
			e_{\gamma^+_r} & 0 \\ 0 & e_{\gamma^-_r}
		\end{bmatrix} =  (-1)^\Sigma\half{\Phi-\nu}e_{\Delta(\gamma_r)}
	\end{align*}

	Now let us compute the converse composition. Let  $\Delta^B(\theta_i), 1 \leq i \leq r-2$  be the
	$(r-1) \times (r-1)$ block matrix with the only non-zero block is the $(i,i+1)$st block, which is  the matrix $\Delta(\theta_i)$ in \eqref{eq:dtheta}.
	\begin{align*}
		&\fm_2^{\delta, 0, \delta}(\Delta^B(\theta_{r-1}), \Delta^B(\theta_r) ) \\
		&= \fm_2(\Delta^B(\theta_{r-1}), \Delta^B(\theta_r)) + \sum_{i=1}^{r-1}\fm_r(\underbrace{\delta, \dots, \delta}_\text{$(r - i - 1)$ times}, \Delta^B(\theta_{r-1}), \Delta^B(\theta_r), \underbrace{\delta, \dots, \delta}_{\text{$(i - 1)$ times}}) \\
		&=\fm_2(\Delta^B(\theta_{r-1}), \Delta^B(\theta_r)) + \sum_{i=1}^{r-1}\fm_r(\Delta^B(\theta_i), \dots, \Delta^B(\theta_{r-1}), \Delta^B(\theta_r), \Delta^B(\theta_1), \dots, \Delta^B(\theta_{i-1})) \\
		&=\fm_2(\Delta^B(\theta_{r-1}), \Delta^B(\theta_r))  + \sum_{i=1}^{r-1}(-1)^\Sigma\left(\frac{1}{2}\right)^{\Phi-\nu}e_{\Delta(\gamma_i)}.
	\end{align*}
	Here, $\fm_2(\Delta^B(\theta_{r-1}), \Delta^B(\theta_r)) = 0$ by the same reason with the first paragraph.
	Therefore, we know $\Delta^B(\theta_{r-1})$ and $(-1)^\Sigma\half{\nu - \Phi}\Delta^B(\theta_r)$ are quasi-inverse to each other.
\end{proof}

Now let us define a notion of involutivity of a tagged arc system.
\begin{defn}\label{defn:InvolutiveCondition}
	Let $(S, M, O, \eta)$ be a graded marked orbi-surface and $\Gamma$ be a tagged arc system. Then, we say $\Gamma$ is {\em involutive} if the underlying arc system $\underline{\Gamma}$ satisfies the following two conditions.
	\begin{itemize}
		\item For each interior marking $p\in O$, there is at most one arc in $\underline{\Gamma}$ hanging at $p$.
		\item Each arc $\alpha \in \underline{\Gamma}$ has $\nu(\alpha) \leq 1$ and if $\nu(\alpha) = 1$, then it is thick.
	\end{itemize}
\end{defn}
In particular, an involutive tagged arc system has no interior morphisms. Thus it has no composition sequences. Also, the weight of a disc sequence $(\theta_1, \dots, \theta_n)$ is the number of folded pairs $(\theta_i, \theta_{i+1})$ in it.
Using Proposition \ref{prop:MoritaEquivalency}, we get the following result.
\begin{lemma}\label{lemma:MakeInvolutive}
	Let $(S, M, O, \eta)$ be a graded marked orbi-surface and $\Gamma$ be a tagged arc system. Then, there is an involutive tagged arc system $\Gamma'$ which is Morita equivalent to $\Gamma$.
\end{lemma}
\begin{proof}
	We prove the lemma by induction on $\#\{\gamma \in \Gamma : \nu(\gamma) = 2\}$ and $\#\Gamma_p$ for each $p\in O$. In order to prove, we fix some temporal notions. We mean by a disc component the closure of a component of $S \setminus \bigcup_{\gamma \in \Gamma}\gamma$. We say a disc component {\em open} if it has an unmarked boundary component and {\em closed} otherwise. We call tagged arcs, simply, arcs. We say an arc $\gamma$ with $\nu(\gamma) = 2$ is {\em flat} when there is only one disc component containing $\gamma$ as boundary but $\gamma$ is not folded.
	
	\textbf{Step 0.} Suppose that $\Gamma_p$ has less than two arcs. If $\Gamma_p$ is an empty set, then add a thick pair $(\alpha_+, \alpha_-)$ from $p$ to a marked boundary component. If $\Gamma_p$ has only one arc $\alpha_+$, then add a thick pair $\alpha_-$ of $\alpha_+$. By Proposition \ref{prop:MoritaEquivalency}, we have $\Gamma \simeq \Gamma \cup \{\alpha_+, \alpha_-\}$. So we may assume $\Gamma$ has no such an interior marking $p$.
	
	\textbf{Step 1.} Let $\gamma$ be an arc with $\nu(\gamma) = 2$, which is a boundary of a closed disc component $D$. Our goal is to remove $\gamma$ from $\Gamma$. Since $D$ is closed, we can add arcs in $D$ without changing Morita equivalent class of $\Gamma$ as long as it is a tagged arc system. If $\gamma$ is flat, add an arc from an endpoint of $\gamma$ to a marked boundary component of $D$ so that $\gamma$ is not flat anymore. (See Figure \ref{fig:Step1}. In the figure, orange lines, green lines, and $\times$ stand for tagged arcs, boundary markings, and interior markings. Also, $\bullet$ stand for either boundary markings or interior markings.) So we may assume boundaries of disc components are not flat.
	
\begin{figure}[h!]
	\centering
	\begin{subfigure}[b]{0.3\linewidth}
		\includegraphics[width=\linewidth]{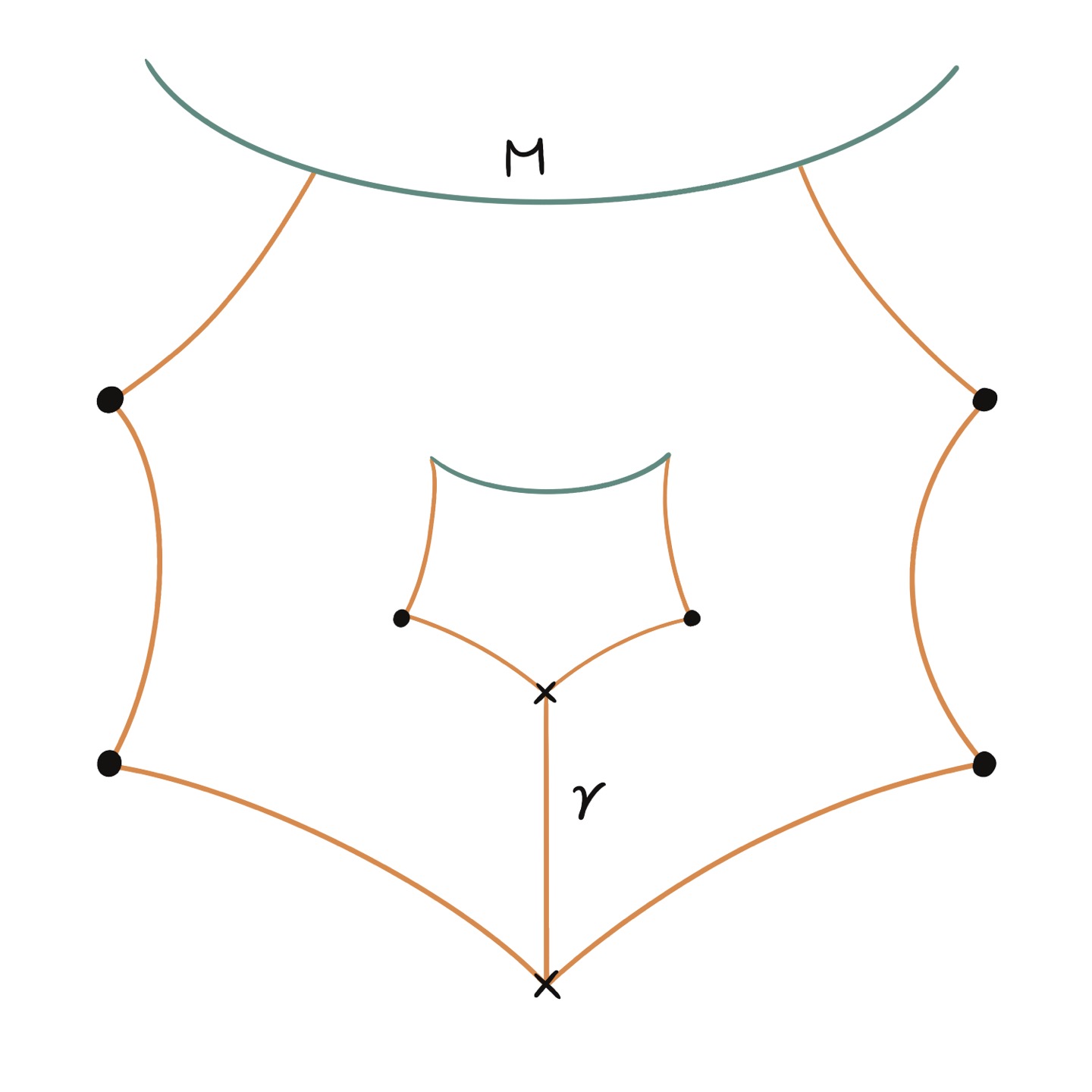}
	\end{subfigure}
	\begin{subfigure}[b]{0.3\linewidth}
		\includegraphics[width=\linewidth]{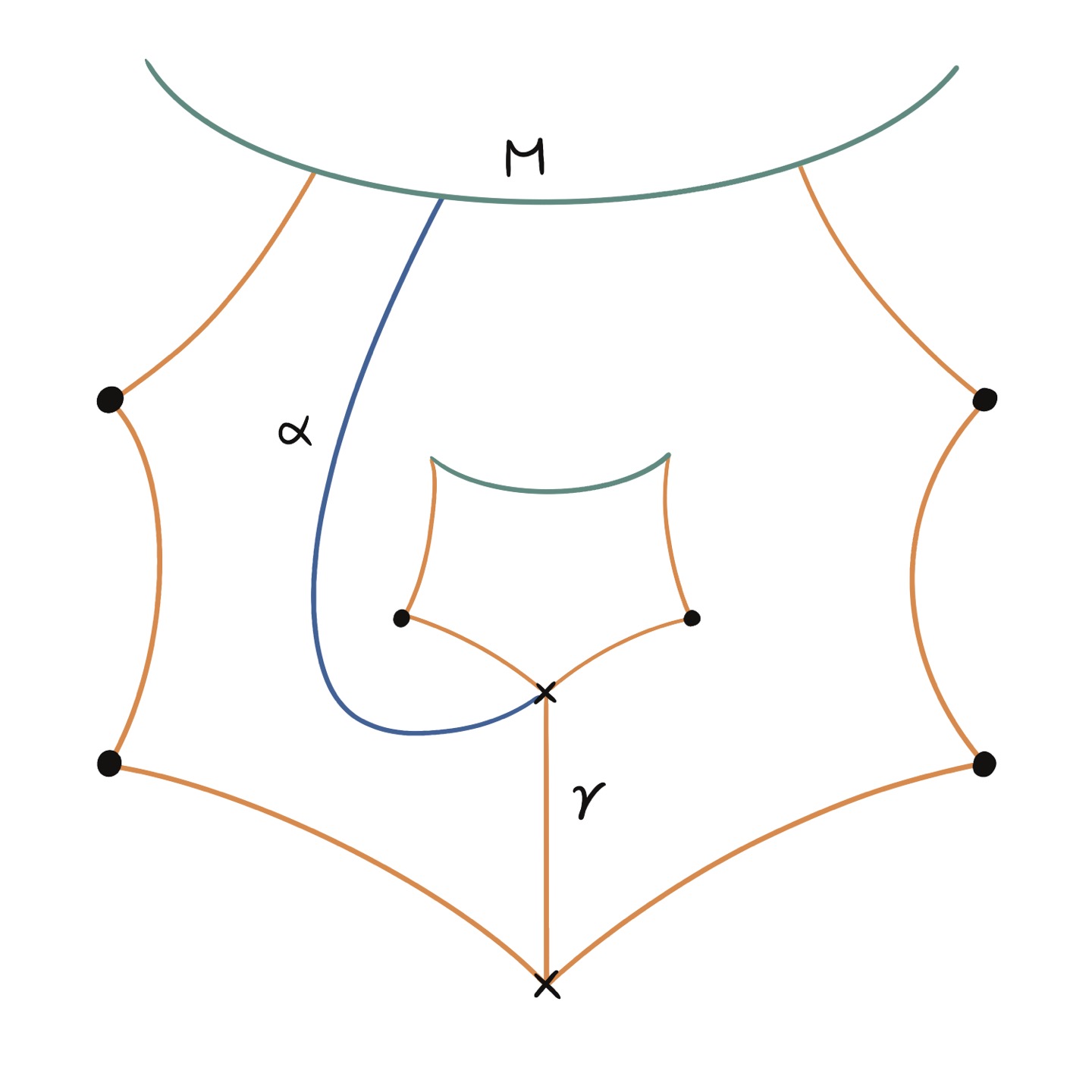}
	\end{subfigure}
	\caption{Step 1}
	\label{fig:Step1}
\end{figure}

	\textbf{Step 2.} The setting is the same with Step 1. If $D$ defines a disc sequence, then we can remove $\gamma$ by using Proposition \ref{prop:MoritaEquivalency}. Now suppose that $D$ has an outward interior marking $p$. Then, $\Gamma_p$ has no thick pair (by the thick condition). Let $\alpha$ and $\beta$ be the last and first arcs hanging at $p$. If one of them is of interior number $1$, then we add its thick pair to $\Gamma$. Otherwise, we add a new thick pair from $p$ to a marked boundary component of $D$. In both cases, we get a new tagged arc system $\tilde{\Gamma}$ such that $\Gamma \simeq \tilde{\Gamma}$ and $\gamma$ is a boundary of a disc component defining a disc sequence. So we can remove $\gamma$ from $\tilde{\Gamma}$. The new system $\tilde{\Gamma}$ has one less arc of interior number $2$ than $\Gamma$ (see Figure \ref{fig:Step2}). In this way, we can remove arcs of interior number $2$ which is a boundary of a closed disc component. So we may assume $\Gamma$ has no such arcs. That is, any disc component whose boundary contains an arc of interior number two is open.
\begin{figure}[h!]
	\centering
	\begin{subfigure}[b]{0.3\linewidth}
		\includegraphics[width=\linewidth]{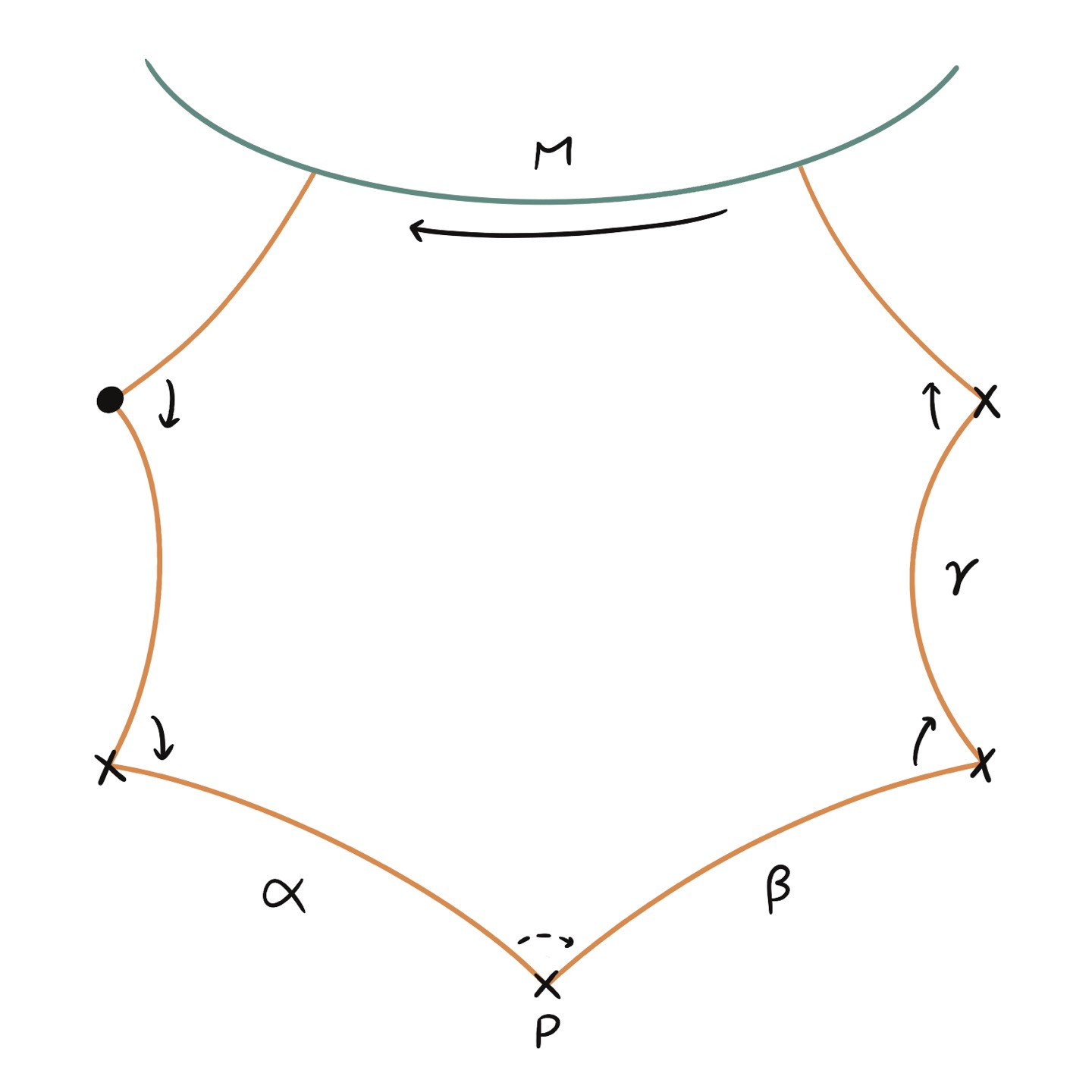}
	\end{subfigure}
	\begin{subfigure}[b]{0.3\linewidth}
		\includegraphics[width=\linewidth]{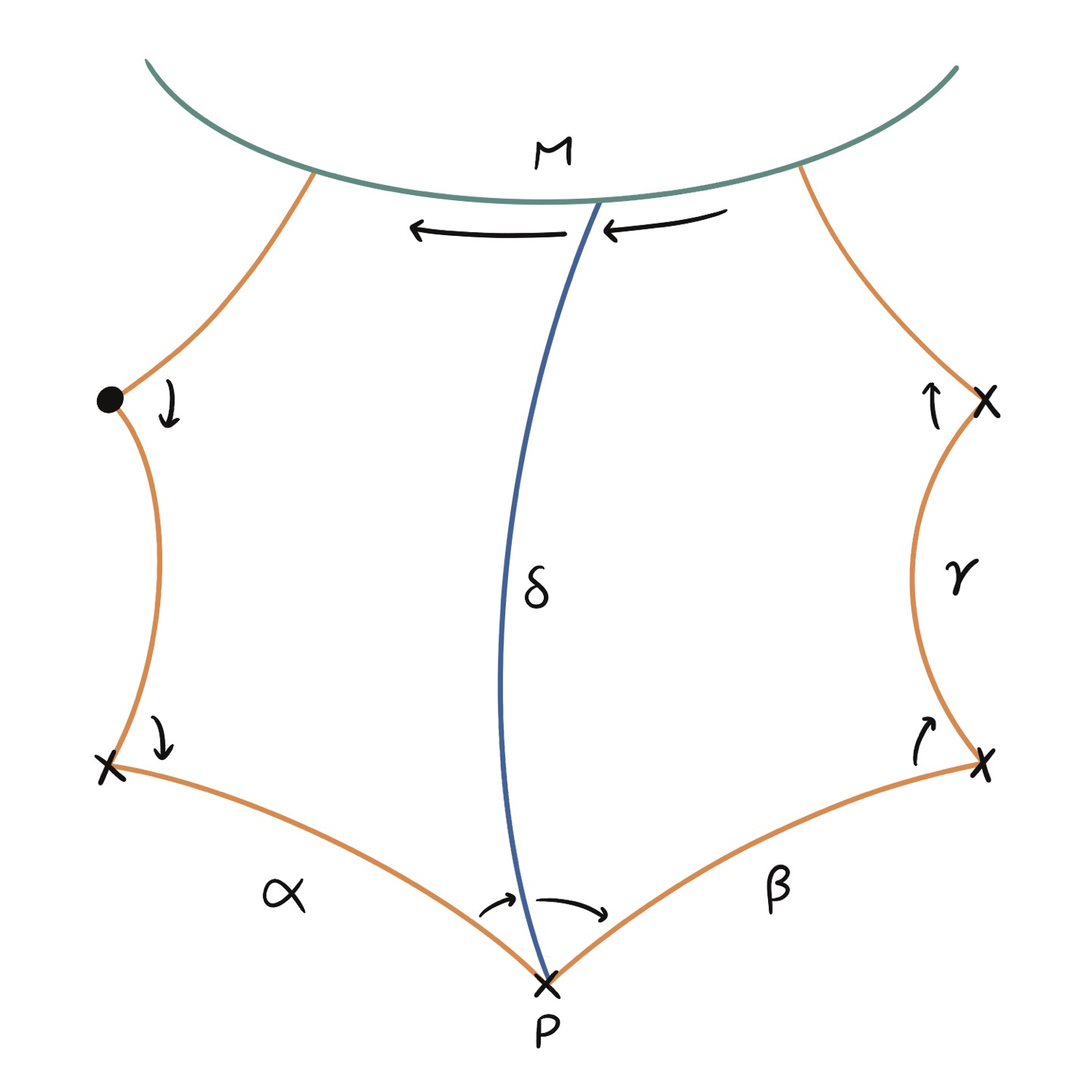}
	\end{subfigure}
	\caption{Step 2}
	\label{fig:Step2}
\end{figure}	

	\textbf{Step 3.} Let $D$ be an open disc component and $\theta_1, \dots, \theta_n$ the basic morphisms associated with $D$ with $\theta_i : \gamma_i \rightarrow \gamma_{i+1}$. Let $\gamma_r$ be the first arc of interior number $2$. Then, we add a new arc from $\theta_{r+1}$ to the marked boundary component where $\gamma_1$ lies (see Figure \ref{fig:Step3}). Then, we can remove $\gamma_r$ as in Step 2. In this way, we can remove all arcs of interior number $2$.
	\begin{figure}[h!]
	\centering
	\begin{subfigure}[b]{0.3\linewidth}
		\includegraphics[width=\linewidth]{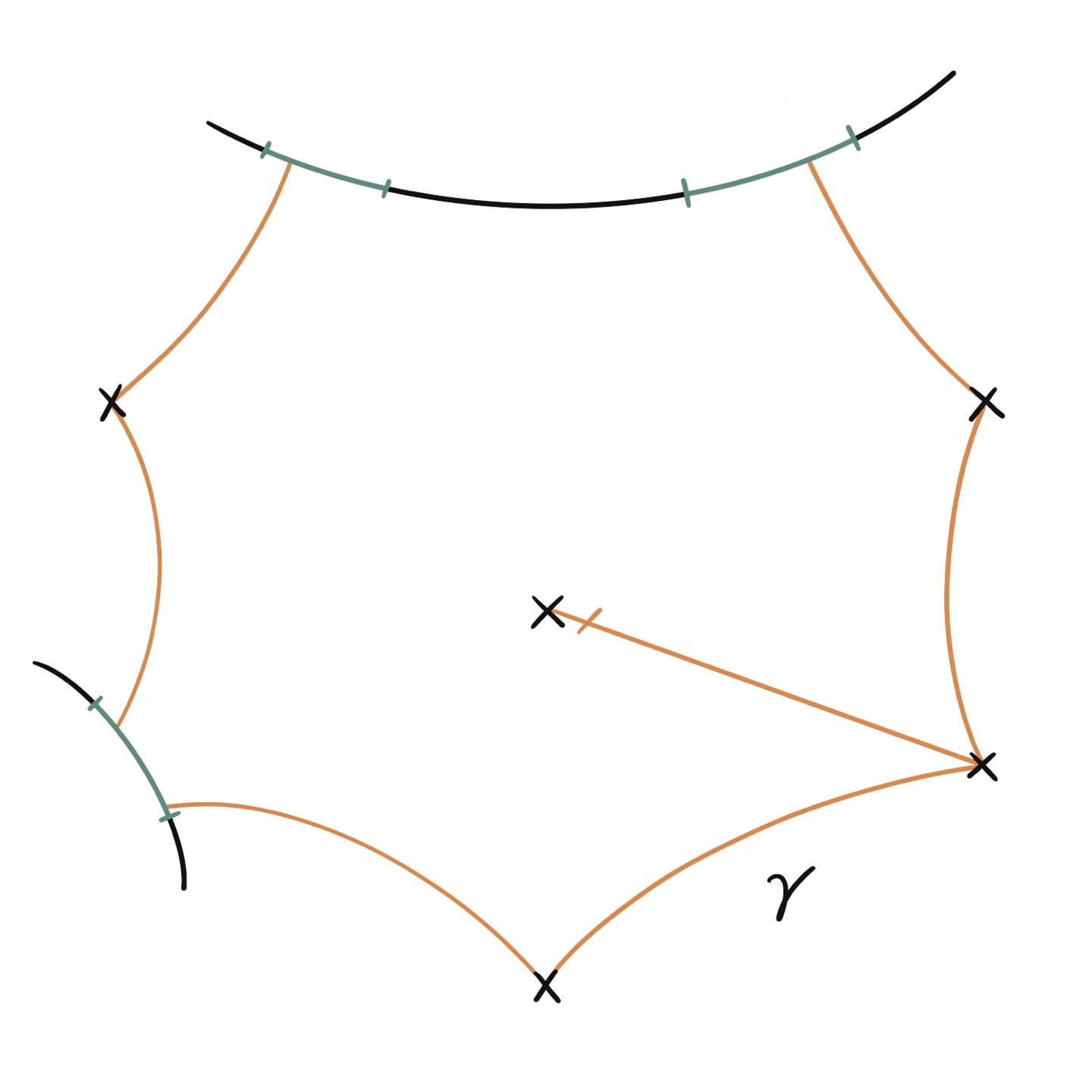}
	\end{subfigure}
	\begin{subfigure}[b]{0.3\linewidth}
		\includegraphics[width=\linewidth]{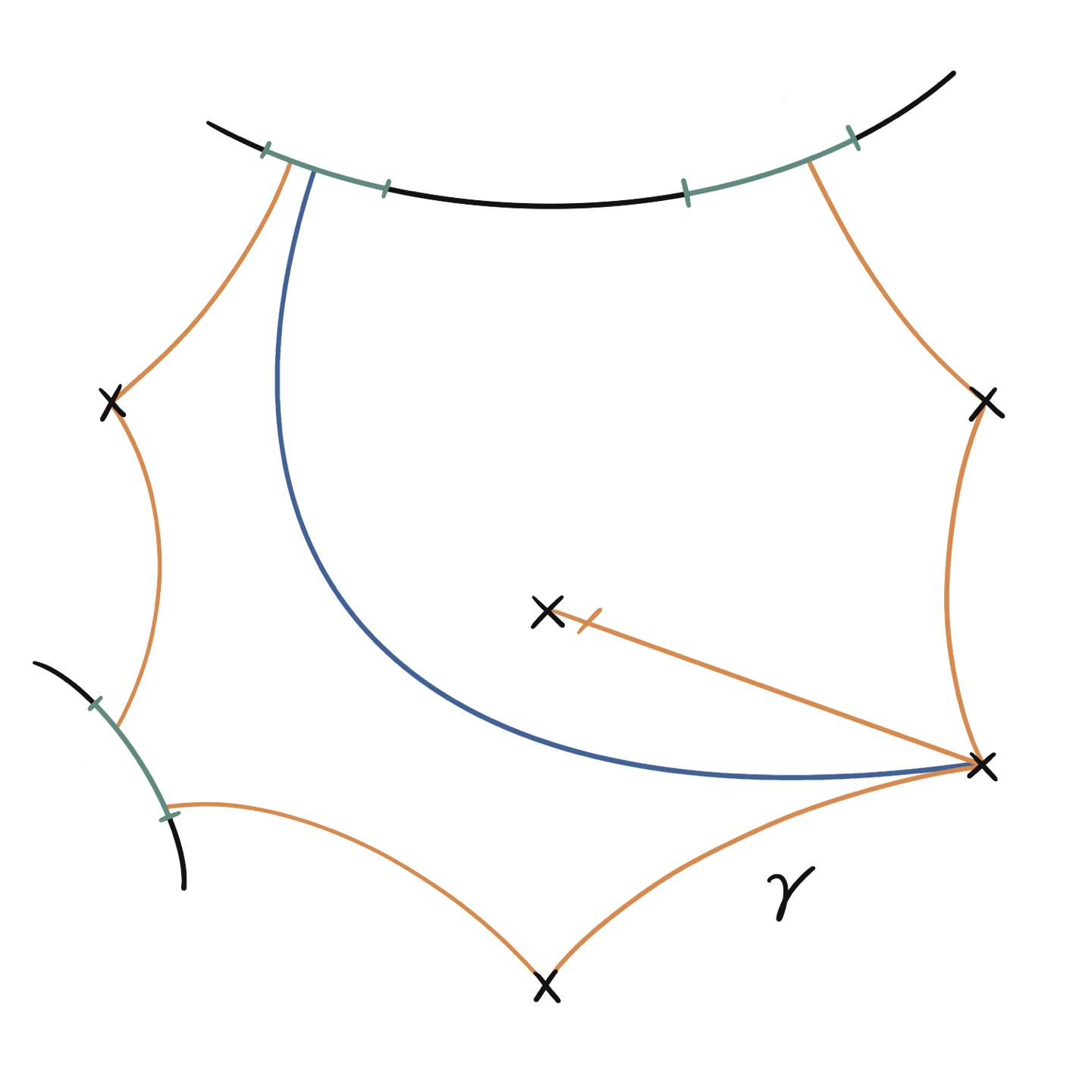}
	\end{subfigure}
	\caption{Step 3}
	\label{fig:Step3}
\end{figure}

	\textbf{Step 4.} Suppose that an interior marking $p$ defines a irreducible interior morphism $p^\alpha_\beta$ from $\alpha$ to $\beta$. Then, there is a disc component $D$ such that $\alpha$ and $\beta$ are its boundary components. If $D$ defines a disc sequence, we can remove $\beta$. If not, $D$ has an unmarked boundary component. After adding an arc isotopic to the unmarked boundary component, we can remove $\beta$. Then, we get a new tagged arc system which has one less arc of interior number $1$ than $\Gamma$. In this way, we can remove all interior morphisms.
	
	\textbf{Step 5.} For $p \in O$ such that $\Gamma_p$ has only one arc $\alpha$, we add the thick pair of $\alpha$. Therefore, we have constructed an involutive arc system $\Gamma'$ from $\Gamma$ without changing its Morita equivalent class.
\end{proof}

From Lemma \ref{lemma:MakeInvolutive} together with theory of involutive surfaces in Section \ref{section:TFCofInvolutiveSurfaces}, we can prove the Morita equivalence class of $\Aoo$-categories associated with tagged arc systems are independent of the choice of the system.
\begin{cor}
	Let $(S, M, O, \eta)$ be a graded marked orbi-surface. Then Morita equivalent class of the $\Aoo$-category $\cff_\Gamma(S, M, O, \eta)$ does not depend on the choice of tagged arc system $\Gamma$.
\end{cor}
\begin{proof}
	Let $\Gamma_1$ and $\Gamma_2$ be tagged arc systems of $(S, M, O, \eta)$. Then, by Lemma \ref{lemma:MakeInvolutive}, there are involutive arc systems $\Gamma'_1$ and $\Gamma'_2$ such that $\Gamma_i \simeq \Gamma'_i$ for $i=1, 2$. Let $(\tilde{S}, \tilde{M}, \eta)$ be the $2$-fold branched covering of $(S, M, O, \eta)$ with the involution $\iota$. Then, by Lemma \ref{lemma:InvolutiveIsInvolutive}, the liftings $\tilde{\Gamma}'_1$ and $\tilde{\Gamma}'_2$ of $\Gamma'_1$ and $\Gamma'_2$ are involutive graded arc systems of $(\tilde{S}, \iota, \tilde{M}, \eta)$. Also, by Theorem \ref{theorem:AooOrbifoldAndInvolutive}, we have, for $i=1, 2$, $$\cff_{\tilde{\Gamma}'_i}(\tilde{S}, \iota, \tilde{M}, \eta) \cong \cff_{\Gamma'_i}(S, M, O, \eta).$$
	However, by Theorem \ref{theorem:MoritaIndependeceForInvolutiveSurfaces}, we know $\cff_{\tilde{\Gamma}'_1}(\tilde{S}, \iota, \tilde{M}, \eta)$ and $\cff_{\tilde{\Gamma}'_2}(\tilde{S}, \iota, \tilde{M}, \eta)$ are Morita equivalent to each other. Therefore, $\cff_{\Gamma'_1}(S, M, O, \eta)$ and $\cff_{\Gamma'_2}(S, M, O, \eta)$ are Morita equivalent, which implies Morita equivalency of $\cff_{\Gamma_1}(S, M, O, \eta)$ and $\cff_{\Gamma_2}(S, M, O, \eta)$.
\end{proof}

Now we define the topological Fukaya category of orbi-surfaces.
\begin{defn}
	The {\em topological Fukaya category} of a graded marked orbi-surface $(S, M, O, \eta)$ is the $\Aoo$-category $$\cff(S, M, O, \eta) \deq \Pi(\Tw(\cff_\Gamma(S, M, O, \eta)))$$ for a tagged arc system $\Gamma$.
\end{defn}

\section{Topological Fukaya categories of involutive surfaces}\label{section:TFCofInvolutiveSurfaces}
Let $(\tilde{S},\tilde{M},\eta)$ be a graded boundary-marked surface and $\Gamma$ be a graded arc system on it.
Let $\iota$ be a non-trivial involution (i.e. $\iota^2 = \id_{\tilde{S}}$) on $\tilde{S}$.
We assume that $\iota$ preserves an orientation of $\tilde{S}$, a boundary marking $\tilde{M}$ and the line field $\eta$. The topological Fukaya category  $\cff_{\Gamma}(\tilde{S}, \tilde{M}, \eta)$  in Section \ref{section:TFCofSurfaces} admits a $\super$-action from $\iota$, and by taking its $\super$-invariant part, we can define the $\Aoo$-category $\cff_\Gamma(\tilde{S}, \iota, \tilde{M}, \eta)$.

In this section, we observe that non-trivial idempotents arise in this setting. We investigate their properties and  find an explicit correspondence between this $\Aoo$-category $\cff_\Gamma(\tilde{S}, \iota, \tilde{M}, \eta)$ and the Fukaya category (constructed in Section \ref{section:TFCofOrbiSurfaces}) of {\em involutive} tagged arc system (defined in Definition \ref{defn:InvolutiveCondition}).

\subsection{Involutive graded marked surfaces}
\begin{defn}
We call the pair $(\tilde{S}, \iota)$ an {\em involutive surface}. Its orbit space $\tilde{S}/\iota$ is a $\super$-orbi-surface and we denote it by $S$ (with a projection $\pi:\tilde{S} \rightarrow S$).
We say a boundary marking $\tilde{M}$ on $\tilde{S}$ is {\em involutive} if $\iota \circ \tilde{M}$ and $\tilde{M}$ are the same map up to permutation. The boundary marking $\tilde{M}$ descends to a boundary marking $M$ of $S$. We say a line field $\eta$ on $\tilde{S}$ is {\em involutive} if $\iota^*(\eta) = \eta$. We denote by $\tilde{O}$ and $O$ the set of $\super$-fixed points in $\tilde{S}$ and its image $\pi(\tilde{O})$, \resp.
We call the tuple $(\tilde{S}, \iota, \tilde{M}, \eta)$ an {\em involutive graded boundary-marked surface}. 
\end{defn}
Note that the set $O$ is finite, and we have the induced graded  marked orbi-surface $(S,O,M, \eta)$.

Let $\alpha$ be a graded curve on an involutive graded boundary-marked surface $(\tilde{S}, \iota, \tilde{M}, \eta)$. The involution $\iota$ sends $\alpha$ to the graded curve $\iota_*\alpha$. For another graded curve $\beta$, intersecting transversely with $\alpha$ at $p$, we have $$i_{\iota(p)}(\iota_*\alpha, \iota_*\beta) = i_p(\alpha, \beta).$$ 

Now let us define the involutive version of graded arc system.
\begin{defn}
	Let $(\tilde{S}, \iota, \tilde{M}, \eta)$ be an involutive graded marked surface. We say a graded arc system $\Gamma = \{\gamma_1, \dots, \gamma_n\}$ on $(\tilde{S}, \tilde{M}, \eta)$ is {\em involutive} if $$\Gamma = \iota_*\Gamma \deq \{\iota_*\gamma_1, \dots, \iota_*\gamma_n\}.$$ We say a graded arc $\gamma_i$ is {\em special} if $\iota_*(\gamma_i) = \gamma_i$, and {\em ordinary} otherwise. We denote by $\Gamma_{\Ord}$ and $\Gamma_{\Sp}$ the set of ordinary and special graded arcs, \resp.
\end{defn}

\begin{lemma}
For an involutive graded arc system $\Gamma = \{\gamma_1, \dots, \gamma_n\}$, we have $$\#\Gamma_{\Sp} \leq \# \tilde{O}.$$
\end{lemma}
\begin{proof}
	As the involution $\iota$ acts as orientation-reversing reparametrization on the domain of special arcs, each special arc passes through precisely one fixed point. This defines a function $\phi : \Gamma_{\Sp} \rightarrow \tilde{O}$. Since two arcs do not intersect, $\phi$ is injective. This shows $\#\Gamma_{\Sp} \leq \#\tilde{O}$.
\end{proof}

\subsection{Topological Fukaya categories of involutive graded marked surfaces}

Let $(\tilde{S}, \iota, \tilde{M}, \eta)$ and $\Gamma$ be an involutive graded boundary-marked surface and an involutive graded arc system.  Since $\iota$ preserves $\eta$, the topological Fukaya category $\cff_\Gamma(\tilde{S}, \tilde{M}, \eta)$ admits a strict $\super$-action induced by $\iota$: 
a boundary morphism from $\gamma_1$ to $\gamma_2$ are
mapped to the boundary morphism from $\iota_*\gamma_1$ to $\iota_*\gamma_2$ with the same sign. 
For a special $\gamma$, two graded arcs $\gamma$ and $\iota_*(\gamma)$ have the same underlying curve (with the same grading) and there is a unique identity morphism between them.
The $\super$-action sends the identity to the identity in this case. Note that a morphism between a special arc $\gamma$ and an ordinary arc $\beta$ is not invariant by itself since the boundary morphism between $\gamma$ and $\beta$ goes to the boundary morphism between $ \iota_*\gamma(=\gamma)$ and $\iota_*(\beta) (\neq \beta)$. 
Here, a strict action means that 
$$\iota_* \fm_k(x_1, \dots, x_k) = \fm_k(\iota_*x_1, \dots, \iota_*x_k).$$
Given a strict $\super$-action on a unital $\Aoo$-category $\caa$, there is an induced strict  $\super$-action on $\Tw\caa$ as well.

For each arc $\alpha$ of $\tilde{S}$, the direct sum object  $\alpha \oplus \iota_*(\alpha)$ is invariant under $\super$-action.
In fact, $\super$-invariant part of twisted complexes $(\Tw\caa)^\super$ (on objects and morphisms) forms a unital $\Aoo$-category as well.

Later, we will describe morphisms in this category explicitly. It is necessary for us to choose a representative in each  $\super$-orbit of arcs of $\tilde{S}$ because some morphisms from/to idempotents might have different sign if we work with the other representative. There is a convenient way to choose representatives.

First, we choose and fix a representing arc (in $\tilde{S}$) among each $\iota$-orbit of arcs. Now, we define the notion of a preferred morphism.
The involution $\iota_*$ acts on the set of markings $\tilde{M}$ freely. Thus for each connected component $C$ of $\tilde{M}$, $\iota(C)$ is disjoint from $C$. 
We choose and fix a representing connected component in $\tilde{M}$ in each $\iota$-orbit.  Recall that boundary morphisms in $\tilde{S}$ are defined when along the boundary marking $\tilde{M}$. A {\em preferred} boundary morphisms are the ones that are defined along the chosen representing boundary markings.
Note that the concatenation of preferred boundary morphisms are preferred ones. From now on, if we write an orbit as $\{\alpha,\iota_*\alpha\}$, $\alpha$ is the chosen representative. Between $\{\theta, \iota_*\theta\}$, $\theta$ is the preferred morphism.

Let us define a strict functor $\Delta : \caa \rightarrow (\Tw\caa)^\super$ as follows.
	\begin{itemize}
		\item For an object $X\in \caa$, $\Delta(X) \deq X \oplus \iota_*(X)$.
		\item For a morphism $x : X_1 \rightarrow X_2$, $\Delta(x) \deq \begin{bmatrix} x & 0 \\ 0 & \iota_*(x) \end{bmatrix}$.
	\end{itemize}
	This is indeed a functor. First of all, $\Delta(x)$ is $\super$-invariant. Also, for any sequence of $n$ composable morphisms $x_1, \dots, x_n$,
	\begin{align*}
		\Delta(\fm_n(x_1, \dots, x_n)) =& \begin{bmatrix} \fm_n(x_1, \dots, x_n) & 0 \\ 0 & \iota_*\fm_n(x_1, \dots, x_n) \end{bmatrix} \\
		=& \begin{bmatrix} \fm_n(x_1, \dots, x_n) & 0 \\ 0 & \fm_n(\iota_*(x_1), \dots, \iota_*(x_n)) \end{bmatrix} \\
		=& \fm_n\left(\begin{bmatrix} x_1 & 0 \\ 0 & \iota_*(x_1) \end{bmatrix}, \dots, \begin{bmatrix} x_n & 0 \\ 0 & \iota_*(x_n) \end{bmatrix}\right) = \fm_n(\Delta(x_1), \dots, \Delta(x_n)).
	\end{align*}

One of the special feature in our setup is that when $\gamma$ is a special arc, $\Delta(\gamma$) has two nontrivial idempotents.
\begin{lemma}
	Let $\gamma$ be a special arc. Then, the following two morphisms are idempotents of the endomorphism algebra of $\Delta(\gamma)$.
	$$p^+_\gamma \deq \frac{1}{2}\begin{bmatrix} e_\gamma & e_\gamma \\ e_\gamma & e_\gamma \end{bmatrix}, \quad p^-_\gamma \deq \frac{1}{2}\begin{bmatrix} e_\gamma & -e_\gamma \\ -e_\gamma & e_\gamma \end{bmatrix}.$$
\end{lemma}
We omit the proof as it is a direct computation.
From these, we get the following objects in the idempotent completion $\Aoo$-category $\Pi(\Tw(\cff_\Gamma(\tilde{S}, \tilde{M}, \eta)))^\iota$, for each special arc $\gamma$,
$$\Delta(\gamma)_+ \deq (\gamma, p^+_\gamma), \quad \Delta(\gamma)_- \deq (\gamma, p^-_\gamma).$$
This is our geometric interpretation of tagging. We remark that the idempotents are not involved in higher compositions as they consist of units.

Now let us define an $\Aoo$-category associated with an involutive graded arc system.

\begin{defn}
	Let $(\tilde{S}, \iota, \tilde{M}, \eta)$ be an involutive graded boundary-marked surface and $\Gamma$ be an involutive graded arc system. Then we define an $\Aoo$-category $\cff_\Gamma(\tilde{S}, \iota, \tilde{M}, \eta)$ as the full sub-$\Aoo$-category of $\Pi(\Tw(\cff_\Gamma(\tilde{S}, \tilde{M}, \eta))^\iota)$ with the set of objects $$\{\Delta(\alpha) : \{\alpha, \iota_*\alpha\} \subset \Gamma_\Ord\} \cup \{\Delta(\beta)_+, \Delta(\beta)_- : \beta \in \Gamma_\Sp\}.$$
\end{defn}
Let us compute its morphism spaces explicitly. First, the following is easy to check.
\begin{lemma}
For any special arc $\gamma$, we have
\begin{align*}
		\ho_{\cff_\Gamma(\tilde{S}, \iota, \tilde{M}, \eta)}(\Delta(\gamma)_\pm, \Delta(\gamma)_\pm) &= \field\langle e_{\Delta(\gamma)_{\pm}}\rangle, \\
		\ho_{\cff_\Gamma(\tilde{S}, \iota, \tilde{M}, \eta)}(\Delta(\gamma)_\pm, \Delta(\gamma)_\mp) &= 0. \\
\end{align*}
\end{lemma}

Let $\alpha_1$ and $\alpha_2$ be ordinary arcs. Then, boundary morphisms
$$\theta_1 : \alpha_1 \rightarrow \alpha_2, \quad \theta_2 : \alpha_1 \rightarrow \iota_*(\alpha_2), \quad \theta_3 : \iota_*(\alpha_1) \rightarrow \alpha_2, \quad \theta_4 : \iota_*(\alpha_1) \rightarrow \iota_*(\alpha_2),$$
give the following morphisms from $\alpha_1 \oplus \iota_*(\alpha_1)$ to $\alpha_2 \oplus \iota_*(\alpha_2)$, under the functor $\Delta$, 
$$\Delta(\theta_1)^0_0 \deq \begin{bmatrix} \theta_1 & 0 \\ 0 & \iota_*(\theta_1) \end{bmatrix}, \quad \Delta(\theta_2)^0_0 \deq \begin{bmatrix} 0 & \iota_*(\theta_2) \\ \theta_2 & 0 \end{bmatrix}, \quad \Delta(\theta_3)^0_0 \deq \begin{bmatrix} 0&\theta_3 \\  \iota_*(\theta_3)  & 0 \end{bmatrix}, \quad \Delta(\theta_4)^0_0 \deq \begin{bmatrix}  \iota_*(\theta_4) & 0 \\0 &\theta_4 \end{bmatrix}.$$

Now suppose $\alpha$ is an ordinary arc, with a unique boundary morphism $\theta:\alpha \rightarrow \gamma$ to a special arc $\gamma$. Then,
$$\ho_{\Tw\cff_\Gamma(\tilde{S}, \tilde{M}, \eta)}(\Delta(\alpha), \Delta(\gamma)) =\left\{\begin{bmatrix} a\theta & b\iota_*(\theta) \\ c\theta & d\iota_*(\theta) \end{bmatrix} \mid a,b,c,d \in k \right\}.$$
The $\iota$-invariant part is given by $a=d, b=c$. By multiplying the idempotents from the right, we have
$$ \ho_{\cff_\Gamma(\tilde{S}, \iota, \tilde{M}, \eta)}(\Delta(\alpha), \Delta(\gamma)_\pm) =\frac{1}{2} \left\{\begin{bmatrix} (a\pm b)\theta & (b\pm a) \iota_*(\theta) \\ (b\pm a)\theta & (a\pm b)\iota_*(\theta) \end{bmatrix} \mid a,b \in k \right\}.$$
We choose a generator of the morphism space for each $\pm$ as 
$$\Delta(\theta)^0_\pm \deq \frac{1}{2}\begin{bmatrix} \theta & \pm\iota_*(\theta) \\ \pm\theta & \iota_*(\theta) \end{bmatrix}.$$
We remark that if we did not fix a preferred morphism (and a representative), then the above choice has an ambiguity of overall sign.

For two special arcs $\beta_1$ and $\beta_2$ with a preferred morphism $\theta : \beta_1 \rightarrow \beta_2$, the morphism $\iota_*(\theta)$ is also a morphism from $\beta_1$ to $\beta_2$. Moreover, these span the morphism space $\ho(\beta_1, \beta_2)$. Thus, the invariant morphism space is given by 
$$\ho(\Delta(\beta_1), \Delta(\beta_2))^{\super} = \field\left<\begin{bmatrix} a\theta + b\iota_*(\theta) & c\theta + d\iota_*(\theta) \\ d\theta + c\iota_*(\theta) & b\theta + a\iota_*(\theta) \end{bmatrix}\right>.$$
Then, by multiplying the idempotents, we get the following.
\begin{align*}
	&\frac{1}{2}\begin{bmatrix} e_{\beta_1} & \pm e_{\beta_1} \\ \pm e_{\beta_1} & e_{\beta_1} \end{bmatrix} \begin{bmatrix} a\theta + b\iota_*(\theta) & c\theta + d\iota_*(\theta) \\ d\theta + c\iota_*(\theta) & b\theta + a\iota_*(\theta) \end{bmatrix} \frac{1}{2}\begin{bmatrix} e_{\beta_2} & \pm' e_{\beta_2} \\ \pm' e_{\beta_2} & e_{\beta_2} \end{bmatrix} \\
	=& \frac{1}{4} \begin{bmatrix} (a\pm'd \pm(c\pm'b))\theta + (b\pm'c \pm(d\pm'a))\iota_*(\theta) & (\pm(a\pm'd) + (c\pm'b))\theta + (\pm(b\pm'c) + (d\pm'a))\iota_*(\theta) \\ (\pm(b\pm'c) + (d\pm'a))\theta + (\pm(a\pm'd) + (c\pm'b))\iota_*(\theta) & (b\pm'c \pm(d\pm'a))\theta + (a\pm'd \pm(c\pm'b))\iota_*(\theta) \end{bmatrix} \\
	=& \frac{1}{4}\begin{bmatrix} (a\pm'd \pm(c\pm'b))\theta \pm\pm' ((a\pm'd) \pm(c\pm'b)\iota_*(\theta) & \pm((a\pm'd)\pm(c\pm'b))\theta \pm'((a\pm'd) \pm(c\pm'b))\iota_*(\theta) \\ \pm'((a\pm'd) \pm(c\pm'b))\theta \pm((a\pm'd) \pm(c\pm'b))\iota_*(\theta) & \pm\pm'((a\pm'd) \pm(c\pm'b)\theta + (a\pm'd \pm(c\pm'b))\iota_*(\theta) \end{bmatrix} \\
	=& \frac{1}{4} (a \pm' d \pm(c\pm'b))\begin{bmatrix} \theta \pm\pm' \iota_*(\theta) & \pm\theta \pm' \iota_*(\theta) \\ \pm'\theta \pm\iota_*(\theta) & \pm\pm'\theta + \iota_*(\theta) \end{bmatrix}.
\end{align*}
Take these together, let us make the following definitions. 
Here, the sign rule will explained in Section \ref{subsection:InvolutiveVSOrbi}.
\begin{defn}\label{defn:delarc}
	Let $\alpha_1, \alpha_2$ be an ordinary arc and $\beta_1, \beta_2$ be special arcs of $\Gamma$.
	\begin{itemize}
		\item Let $\phi : \alpha_1 \rightarrow \alpha_2$ and $\psi : \alpha_1 \rightarrow \iota_*(\alpha_2)$ be preferred morphisms. Then, we define
		$$\Delta(\phi)^0_0 \deq \begin{bmatrix} \phi & 0 \\ 0 & \iota_*(\phi) \end{bmatrix}, \quad \Delta(\psi)^0_0 \deq \begin{bmatrix} 0 & \iota_*(\psi) \\ \psi & 0 \end{bmatrix}.$$
		\item Let $\phi : \alpha_1 \rightarrow \beta_2$ and $\psi : \beta_1 \rightarrow \alpha_2$ be preferred morphisms. Then, we define
		$$\Delta(\phi)^0_\pm \deq \begin{bmatrix} \phi/2 & \pm\iota_*(\phi)/2 \\ \pm\phi/2 & \iota_*(\phi)/2 \end{bmatrix}, \quad \Delta(\psi)^\pm_0 \deq \pm\begin{bmatrix} \psi/2 & \pm\psi/2 \\ \pm\iota_*(\psi)/2 & \iota_*(\psi)/2 \end{bmatrix}.$$
		\item Let $\theta : \beta_1 \rightarrow \beta_2$ be a preferred morphism. Then, we define
		$$\Delta(\theta)^\pm_{\pm'} \deq \pm\begin{bmatrix} (\theta \pm\pm'\iota_*(\theta))/4 & (\pm\theta \pm'\iota_*(\theta))/4 \\ (\pm'\theta \pm\iota_*(\theta))/4 & (\pm\pm'\theta + \iota_*(\theta))/4	\end{bmatrix}.$$
	\end{itemize}
\end{defn}
From the definition, we get the following relation between hom spaces of the original $\Aoo$-category and the invariant $\Aoo$-category.
\begin{lemma}\label{lemma:SkewGentleHom}
	Let $\alpha_1, \alpha_2$ be ordinary arcs and $\beta_1, \beta_2$ be special arcs of $\Gamma$. Then we have the following.
	\begin{align*}
		\ho_{\cff_\Gamma(\tilde{S}, \iota, \tilde{M}, \eta)}(\Delta(\alpha_1), \Delta(\alpha_2)_\pm) &\cong \ho_{\cff_\Gamma(\tilde{S}, \tilde{M}, \eta)}(\alpha_1, \alpha_2)/\left(\super\right), \\
		\ho_{\cff_\Gamma(\tilde{S}, \iota, \tilde{M}, \eta)}(\Delta(\alpha_1), \Delta(\beta_2)_\pm) &\cong \ho_{\cff_\Gamma(\tilde{S}, \tilde{M}, \eta)}(\alpha, \beta_2), \\
		\ho_{\cff_\Gamma(\tilde{S}, \iota, \tilde{M}, \eta)}(\Delta(\beta_1)_\pm, \Delta(\alpha_2)) &\cong \ho_{\cff_\Gamma(\tilde{S}, \tilde{M}, \eta)}(\beta_1, \alpha), \\
		\ho_{\cff_\Gamma(\tilde{S}, \iota, \tilde{M}, \eta)}(\Delta(\beta_1)_\pm, \Delta(\beta_2)_{\pm'}) &\cong \ho_{\cff_\Gamma(\tilde{S}, \tilde{M}, \eta)}(\beta_1, \beta_2)/\iota.
	\end{align*}
\end{lemma}
\begin{proof}
Morphisms defined in Definition \ref{defn:delarc} are basis of morphism spaces of $\cff_\Gamma(\tilde{S}, \iota, \tilde{M}, \eta)$.
\end{proof}

Recall that the idempotent completion may be taken before or after taking the invariant part.
\begin{cor} \cite[Lemma 2.19]{Wu18}
	 $\Aoo$-categories $\Pi(\Tw(\cff_\Gamma(\tilde{S}, \iota, \tilde{M}, \eta)))$ and $\Pi(\Tw(\cff_\Gamma(\tilde{S}, \tilde{M}, \eta))^\iota)$ are quasi-equivalent.
\end{cor}

Now let us show that the Morita equivalence class of $\cff_\Gamma(\tilde{S}, \iota, \tilde{M}, \eta)$ does not depend of the choice of $\Gamma$. Let $\Gamma'$ be another involutive graded arc system. From Theorem \ref{Theorem:MoritaEquivalencyOfTopologicalFukayaCategoryForSurfaces}, there is an $\Aoo$-functor from $\cff_\Gamma(\tilde{S}, \tilde{M}, \eta)$ to $\Tw(\cff_{\Gamma'}(\tilde{S}, \tilde{M}, \eta))$. Following the algebraic construction in \cite[Definition 4.2]{BM03} and \cite[Theorem 2.12]{OPS18Arx}, we can find an explicit twisted complex associated with each arc in $\Gamma$ as follows.

Let $\Gamma'_f$ be a formal graded arc system contained in $\Gamma'$, which always exists. 
Then for a graded arc $\gamma$ in $\Gamma$, there is a unique sequence of graded arcs $(\gamma_1, \dots, \gamma_r)$ in $\Gamma'_f$ with the following properties.
\begin{itemize}
	\item For each $i = 1, \dots, (r-1)$, there is either a boundary morphism $\theta_i^+$ from $\gamma_i$ to $\gamma_{i+1}$ or $\theta_i^-$ from $\gamma_{i+1}$ to $\gamma_i$ of degree $\delta_i$. Let us denote $\epsilon_i$ the sign of the upper index of $\theta_i^\pm$.
	\item The concatenated path $\gamma_1 \bullet \theta_1^\pm \bullet \dots \bullet \theta_{r-1}^\pm \bullet \gamma_r$ is isotopic to $\gamma$.
\end{itemize}
Then let us define $d_k \deq \sum_{i=2}^k \epsilon_{i-1}(\deg{\theta_i^{\epsilon_i}} - 1)$ for $k=1, \dots, r$. In particular, $d_1$ is zero. Then, we define a twisted complex $T_\gamma \deq \bigoplus_{i=1}^r \gamma_i[d_i]$ with the twisting map $\oplus_{i=1}^{r-1}\theta_i^{\epsilon_i}$. Then, $T_\gamma$ is the image of $\gamma$ up to degree shift.

From this construction, we get the following lemma.
\begin{lemma}
	Let $\Gamma_1$ and $\Gamma_2$ be two involutive graded arc systems of $(\tilde{S}, \tilde{M}, \eta)$. Then the canonical functor $$\cff_{\Gamma_1}(\tilde{S}, \tilde{M}, \eta) \rightarrow \Tw\cff_{\Gamma_2}(\tilde{S}, \tilde{M}, \eta)$$ is $\super$-equivariant.
\end{lemma}
\begin{proof}
	For a graded arc $\gamma \in \Gamma_1$, we have constructed the twisted complex $T_\gamma$. Note that this may depend on the choice of a formal arc system $(\Gamma_2)_f$. However, we can take $\iota_*((\Gamma_2)_f)$ to define $T_{\iota_*(\gamma)}$. Then, we get $\iota_*(T_\gamma) = T_{\iota_*(\gamma)}$, which gives $\super$-equivariance of the canonical functor $$\cff_{\Gamma_1}(\tilde{S}, \tilde{M}, \eta) \rightarrow \Tw\cff_{\Gamma_2}(\tilde{S}, \tilde{M}, \eta).$$
\end{proof}
Then, we deduce the independence of the Morita class of $\cff_\Gamma(\tilde{S}, \iota, \tilde{M}, \eta)$ of $\Gamma$.
\begin{thm}\label{theorem:MoritaIndependeceForInvolutiveSurfaces}
	Let $\Gamma_1$ and $\Gamma_2$ be involutive graded arc systems of an involutive graded boundary-marked surface $(\tilde{S}, \iota, \tilde{M}, \eta)$. Then, there is an equivalence of $\Aoo$-categories $$\Pi(\Tw(\cff_{\Gamma_1}(\tilde{S}, \iota, \tilde{M}, \eta))) \simeq \Pi(\Tw(\cff_{\Gamma_2}(\tilde{S}, \iota, \tilde{M}, \eta))).$$
\end{thm}
\begin{proof}
	Since the equivalence between $\Tw(\cff_{\Gamma_1}(\tilde{S}, \tilde{M}, \eta))$ and $\Tw(\cff_{\Gamma_2}(\tilde{S}, \tilde{M}, \eta))$ is $\super$-equivariant, their fixed category are equivalent as well. Thus, we have $$\Pi(\Tw(\cff_{\Gamma_1}(\tilde{S}, \iota, \tilde{M}, \eta))) \simeq \Pi(\Tw(\cff_{\Gamma_1}(\tilde{S}, \tilde{M}, \eta))^\iota) \simeq \Pi(\Tw(\cff_{\Gamma_2}(\tilde{S}, \tilde{M}, \eta))^\iota) \simeq \Pi(\Tw(\cff_{\Gamma_2}(\tilde{S}, \iota, \tilde{M}, \eta))).$$
\end{proof}
As in Definition \ref{definition:TopologicalFukayaCategory}, we define the topological Fukaya category of an involutive surface as follows.
\begin{defn}
	The {\em topological Fukaya category} of an involutive graded boundary-marked surface $(\tilde{S}, \iota, \tilde{M}, \eta)$ is the idempotent completion of the triangulated enhancement of the $\Aoo$-category for an involutive graded arc system $\Gamma$:
	 $$\cff(\tilde{S}, \tilde{M}, \eta) \deq \Pi(\Tw(\cff_\Gamma(\tilde{S}, \iota, \tilde{M}, \eta)).$$	
\end{defn}

\subsection{Fukaya categories of involutive arc systems and tagged arcs}\label{subsection:InvolutiveVSOrbi}
Now, we show that the construction of Fukaya category of involutive arc system in this section is equivalent to that of tagged arc system defined in Definition \ref{defn:FukayaCategoryForMarkedSurface}. The correspondence simplifies a lot because the corresponding tagged arc system is also involutive. In particular, there are no interior morphisms between tagged arcs, because the tagged arcs of the involutive system do not meet at interior markings. (See the next section for the discussion on interior morphisms.)

Let $(\tilde{S}, \iota, \tilde{M}, \eta)$ be an involutive graded boundary-marked surface, and let $(S,M,O,\eta)$ be the graded marked surface obtained by the quotient.
\begin{defn}
For a graded arc system  $\Gamma$ of $\tilde{S}$, we define a tagged arc system $\pi_*\Gamma$ of $S$ as follows:
$\pi_*\Gamma$ consists of the following associated arcs for all representing (among each $\super$-orbit) ordinary and special arcs:
\begin{enumerate}
\item For an ordinary arc $\alpha$, $\pi \circ \alpha$ is a graded arc of $(S,M,O,\eta)$ and we denote it by $\pi_*(\alpha)$.
\item For a special arc $\beta$ (assume that $\beta(\frac{1}{2}) \in O$),  $\pi \circ \beta|_{[0, \frac{1}{2}]}$ is a graded arc of $(S,M,O,\eta)$ and 
for each tagging $\pm$, we have two tagged arcs $\pi_*(\beta)_+$ and $\pi_*(\beta)_-$.
\end{enumerate}
\end{defn}
\begin{lemma}\label{lemma:InvolutiveIsInvolutive}
	The collection $\pi_*(\Gamma)$ is an involutive nice tagged arc system of $(S, M,O, \eta)$. Moreover, any involutive nice tagged arc system of $(S, M,O, \eta)$ is obtained in this way.
\end{lemma}
\begin{proof}
	Let us first show the collection of underlying arcs of $\pi_*(\Gamma)$ is an arc system. Since any two distinct arcs $\alpha$ and $\beta$ in $\Gamma$ are disjoint and non-isotopic, so are $\pi_*(\alpha)$ and $\pi_*(\beta)$. So it satisfies the first and third conditions in Definition \ref{defn:ArcSystemForOrbiSuefaces}. Also, the second condition follows the fact that the restriction of the projection map $$\pi : \tilde{S} \setminus \bigcup_{\alpha \in \Gamma}\alpha \rightarrow S \setminus \bigcup_{\beta \in \pi_*(\beta)}\beta$$ is a covering map. The fourth condition vacuously holds. Thus, $\pi_*(\Gamma)$ is a pre-tagged arc system.
	
	The thick condition and good condition follow the definition. Also, since $\pi_*\Gamma$ has no interior morphisms and any arcs of interior number $1$ are thick, it satisfies the full condition as well. Moreover, it is nice as it has no interior morphism. This proves the first part of the lemma.
	
	Now suppose that $\Gamma$ be an involutive tagged arc system. Let us define the lift of $\Gamma$ as follows. As we have explained in Section \ref{subsection:TaggedArcs}, any $\alpha \in \Gamma$ with $\nu(\alpha) = 0$ has two lifts $\pi^*(\alpha)$ and $\iota_*(\pi^*(\alpha))$, and any $\beta \in \Gamma$ with $\nu(\beta) = 1$ has branched double covering $\pi^*(\beta)$ which is involutive. Then, define $$\pi^*(\Gamma) \deq \{\pi^*(\alpha), \iota_*(\pi^*(\alpha)) : \alpha \in \Gamma, \nu(\alpha) = 0\} \cup \{\pi^*(\beta) : \beta \in \Gamma, \nu*\beta) = 1\}.$$ One can check this is an involutive graded arc system of $(\tilde{S}, \iota, \tilde{M}, \eta)$. Then, we have $\pi_*(\pi^*(\Gamma)) = \Gamma$.
\end{proof}

\begin{thm}\label{theorem:AooOrbifoldAndInvolutive}
	The following two $\Aoo$-categories are isomorphic to each other
	 $$\cff_\Gamma(\tilde{S}, \iota, \tilde{M}, \eta) \cong \cff_{\pi_*(\Gamma)}(S, M,O, \eta).$$
	 \end{thm}
\begin{proof}
	Let us define a strict functor $\Psi : \cff_\Gamma(\tilde{S}, \iota, \tilde{M}, \eta) \rightarrow \cff_{\pi_*(\Gamma)}(S, M,O, \eta)$ as follows.
\begin{itemize}
	\item For $\alpha \in \Gamma_\Ord$, $\Psi$ sends the object $\Delta(\alpha)$ to $\pi_*(\alpha)$.
	\item For $\beta \in \Gamma_\Sp$, $\Psi$ sends the objects $\Delta(\beta)_\pm$ to tagged arcs  $\pi_*(\beta)_\pm$.
	\item For a boundary morphism $\theta$ from $\alpha$ to $\beta$ in $\tilde{S}$, we define
	 $$\Psi_1(\Delta(\theta)^{a}_{b}) \deq \pi_*(\theta) : \pi_*(\alpha)_{a} \rightarrow \pi_*(\beta)_{b},$$ where $a,b \in \{0, +, -\}$.
	 \item Higher functor maps $\Psi_k$ are defined to be zero (for $k \geq 2$).
\end{itemize}
By the construction, the functor $\Psi$ gives a bijection on the object level. Also, by Lemma \ref{lemma:SkewGentleHom}, it gives a bijection on the morphism level. Hence it is enough to check that $\Psi$ is a strict $\Aoo$-functor.
As we have explained in the beginning of this section, there are no interior morphisms as the tagged arc system $\pi_*(\Gamma)$ is involutive.
In particular, the $\Aoo$-operations $\fm^{\comp}, \fm^{\thick}$ vanish. Moreover, the sign $\Sigma(\vec{\phi})$ becomes trivial as well. However, discs on $S$ may be folded along tagged arcs (which lift to honest discs in $\tilde{S}$) and the weight factors $\frac{1}{2}$ for $\fm^{\con}$ and $\fm^{\disc}$ are still non-trivial.
	
Let us first check the case of concatenation $\fm^{\con}$.
Let $\phi : \alpha \rightarrow \beta$ and $\psi : \beta \rightarrow \gamma$ be boundary morphisms in $\Gamma$ on $(\tilde{S}, \tilde{M}, \eta)$.
Assume that  $\alpha, \gamma \in \Gamma_\Ord$ and $\beta \in \Gamma_\Sp$ (all other cases are similar and we leave them as  exercises).
 Then, $(\phi, \psi)$ and $(\iota_*(\phi), \iota_*(\psi))$ are concatenable while $(\phi, \iota(\psi))$ and $(\iota(\phi), \psi)$ are not. So we have 
 	\begin{align*}
		\Psi \left(\fm_2(\Delta(\phi)^0_\pm, \Delta(\psi)^\pm_0)\right) =& \Psi \left( \fm_2 \left(\begin{bmatrix} \phi/2 & \pm\iota_*(\phi)/2 \\ \pm\phi/2 & \iota_*(\phi)/2 \end{bmatrix}, \pm\begin{bmatrix} \psi/2 & \pm\psi/2 \\ \pm\iota_*(\psi)/2 & \iota_*(\psi)/2 \end{bmatrix}\right) \right) \\
		=&\pm(-1)^{\deg{\phi}} \Psi \left( \begin{bmatrix} \phi \bullet \psi /2 & 0 \\ 0 & \iota_*(\phi) \bullet \iota_*(\psi) / 2\end{bmatrix} \right) = \pm(-1)^{\deg{\phi}}\frac{1}{2} \Psi (\Delta(\phi \bullet \psi)^0_0) \\
		=& \pm(-1)^{\deg{\phi}}\frac{1}{2} \pi_*(\phi \bullet \psi) = \pm(-1)^{\deg{\phi}}\frac{1}{2} \pi_*(\phi) \bullet \pi_*(\psi) \\
		=& \fm^\con_2(\pi_*(\phi), \pi_*(\psi)) = \fm_2^{\con}(\Psi(\Delta(\phi)), \Psi(\Delta(\psi))).
	\end{align*}
	Now let us check $\Psi$ is compatible with higher $\Aoo$-structure, namely disc operations $\fm_k$ (on the left) and $\fm_k^{\disc}$ (on the right).  We will show the claim for disc sequences of length $3$.
	(As the proof generalizes in a straightforward way, we will leave the rest as an exercise.) 
	
	 	Let $(\theta_1, \theta_2, \theta_3)$ be a disc sequence in $\Gamma$ with $\theta_i : \gamma_i \rightarrow \gamma_{i+1}$, where $\gamma_4 = \gamma_1$. To show $\Psi$ and $\fm_3$ are compatible, we have to show $$\Psi(\fm_3(\Delta(\theta_1), \Delta(\theta_2), \Delta(\theta_3))) = \fm_3(\pi_*(\theta_1), \pi_*(\theta_2), \pi_*(\theta_3)) = \half{\Phi(\pi_*(\theta_1), \pi_*(\theta_2), \pi_*(\theta_3)) - \left<\pi_*(\theta_3), \pi_*(\theta_1)\right>}e_{\pi_*(\gamma_1)}.$$ However, the weight is the number of special arcs among $\gamma_i$'s. Thus, we need to show $$\fm_3(\Delta(\theta_1), \Delta(\theta_2), \Delta(\theta_3)) = \half{\text{$\#$ of special arcs in $\{\gamma_2, \gamma_3\}$}}e_{\Delta(\gamma_1)}.$$ Also, note that at least one of the arcs has to be ordinary. If $\gamma_i$ is special, then $\theta_{i-1}$ and $\theta_i$ are on different markings which are in the same $\super$-orbit of $\tilde{M}$. Thus, one and only one of them is a preferred morphism. So in this case, we have to act $\iota_*$ in order to compute $\Delta(\theta_{i-1}$) or $\Delta(\theta_i)$. Now, we have the following cases.
		\begin{enumerate}
	 		\item $\gamma_1$ is ordinary. There are four possibilities.
	 		\begin{enumerate}
	 			\item Both $\gamma_2, \gamma_3$ are ordinary. Then $\fm_3(\Delta(\theta_1)^0_0, \Delta(\theta_2)^0_0, \Delta(\theta_3)^0_0) = e_{\Delta(\gamma_1)}$.
	 			\item $\gamma_2$ is special and $\gamma_3$ is ordinary. Then,
	 			\begin{align*}
	 				&\fm_3(\Delta(\theta_1)^0_\pm, \Delta(\iota_*(\theta_2))^\pm_0, \Delta(\iota_*(\theta_3))^0_0) \\
	 				=& \fm_3\left(\begin{bmatrix} \theta_1/2 & \pm\iota_*(\theta_1)/2 \\ \pm\theta_1/2 & \iota_*(\theta_1)/2 \end{bmatrix}, \pm\begin{bmatrix}\iota_*(\theta_2)/2 & \pm\iota_*(\theta_2)/2 \\ \pm\theta_2/2 & \theta_2/2 \end{bmatrix}, \begin{bmatrix} 0 & \theta_3 \\ \iota_*(\theta_3) & 0 \end{bmatrix} \right) \\
	 				=&\pm\frac{1}{4}\begin{bmatrix} \pm2\fm_3(\theta_1, \theta_2, \theta_3) & 0 \\ 0 & \pm2\fm_3(\iota_*(\theta_1), \iota_*(\theta_2), \iota_*(\theta_3)) \end{bmatrix} = \frac{1}{2}e_{\Delta(\gamma_1)}.
	 			\end{align*}
	 			\item $\gamma_2$ is ordinary and $\gamma_3$ is special. This is the same with the previous case.
	 			\item Both $\gamma_2$ and $\gamma_3$ are special. Then, 
	 			\begin{align*}
	 				&\fm_3(\Delta(\theta_1)^0_\pm, \Delta(\iota_*(\theta_2))^\pm_{\pm'}, \Delta((\theta_3))^{\pm'}_0) \\
	 				=& \fm_3\left(\begin{bmatrix} \theta_1/2 & \pm\iota_*(\theta_1)/2 \\ \pm\theta_1/2 & \iota_*(\theta_1)/2 \end{bmatrix}, \pm\begin{bmatrix} (\iota_*(\theta_2) \pm\pm'\theta_2)/4 & (\pm\iota_*(\theta_2) \pm'\theta_2)/4 \\ (\pm'\iota_*(\theta_2) \pm\theta)/4 & (\pm\pm'\iota_*(\theta_2) + \theta)/4	\end{bmatrix}, \begin{bmatrix} \theta_3 & \pm'\theta_3 \\ \pm'\iota_*(\theta_3) & \iota_*(\theta_3) \end{bmatrix} \right) \\
	 				=&\pm\pm'\frac{1}{16}\begin{bmatrix} \pm\pm'4\fm_3(\theta_1, \theta_2, \theta_3) & 0 \\ 0 & \pm\pm'4\fm_3(\iota_*(\theta_1), \iota_*(\theta_2), \iota_*(\theta_3)) \end{bmatrix} = \frac{1}{4}e_{\Delta(\gamma_1)}.
	 			\end{align*}
	 		\end{enumerate}
 		\item $\gamma_1$ is special. Then there are three possibilities.
 		\begin{enumerate}
	 		\item Both $\gamma_2$ and $\gamma_3$ are ordinary. Then,
	 		\begin{align*}
	 			&\fm_3(\Delta(\theta_1)^\pm_0, \Delta(\theta_2)^0_0, \Delta(\iota_*(\theta_3))^0_\pm) \\
	 			=& \fm_3\left(\pm\begin{bmatrix} \theta_1/2 & \pm\theta_1/2 \\ \pm\iota_*(\theta_1)/2 & \iota_*(\theta_1)/2 \end{bmatrix}, \begin{bmatrix} 0 & \iota_*(\theta_2) \\ \theta_2 & 0 \end{bmatrix}, \begin{bmatrix} \iota_*(\theta_3)/2 & \pm\theta_3/2 \\ \pm\iota_*(\theta_3)/2 & \theta_3/2 \end{bmatrix} \right) \\
	 			=&\pm\frac{1}{4}\begin{bmatrix} \pm\fm_3(\theta_1, \theta_2, \theta_3) \pm \fm_3(\iota_*(\theta_1), \iota_*(\theta_2), \iota_*(\theta_3)) & \fm_3(\theta_1, \theta_2, \theta_3) + \fm_3(\iota_*(\theta_1), \iota_*(\theta_2), \iota_*(\theta_3)) \\ \fm_3(\iota_*(\theta_1), \iota_*(\theta_2), \iota_*(\theta_3)) + \fm_3(\theta_1, \theta_2, \theta_3) & \pm\fm_3(\iota_*(\theta_1), \iota_*(\theta_2), \iota_*(\theta_3)) \pm \fm_3(\theta_1, \theta_2, \theta_3) \end{bmatrix} \\
	 			=& e_{\Delta(\gamma_1)_\pm}.
	 		\end{align*}
	 		\item $\gamma_2$ is special and $\gamma_3$ is ordinary. Then,
	 		\begin{align*}
	 			&\fm_3(\Delta(\theta_1)^\pm_{\pm'}, \Delta(\iota_*(\theta_2))^{\pm'}_0, \Delta(\iota_*(\theta_3))^0_\pm) \\
	 			=&\fm_3\left(\pm\begin{bmatrix} (\theta_1 \pm\pm' \iota_*(\theta_1))/4 & (\pm\theta_1 \pm'\iota_*(\theta_1))/4 \\ (\pm\iota_*(\theta_1) \pm' \theta_1)/4 & (\iota_*(\theta_1) \pm\pm' \theta_1)/4 \end{bmatrix}, \pm'\begin{bmatrix} \iota_*(\theta_2)/2 & \pm'\iota_*(\theta_2)/2 \\ \pm'\theta_2/2 & \theta_2/2 \end{bmatrix}, \begin{bmatrix} \iota_*(\theta_3)/2 & \pm\theta_3/2 \\ \pm\iota_*(\theta_3)/2 & \theta_3/2 \end{bmatrix} \right) \\
	 			=& \pm\pm'\frac{1}{16}\begin{bmatrix} \pm\pm'2\fm_3(\theta_1, \theta_2, \theta_3) \pm\pm'2 \fm_3(\iota_*(\theta_1), \iota_*(\theta_2), \iota_*(\theta_3)) & \pm'2\fm_3(\theta_1, \theta_2, \theta_3) \pm'2 \fm_3(\iota_*(\theta_1), \iota_*(\theta_2), \iota_*(\theta_3)) \\ \pm'2\fm_3(\theta_1, \theta_2, \theta_3) \pm'2 \fm_3(\iota_*(\theta_1), \iota_*(\theta_2), \iota_*(\theta_3)) & \pm\pm'2\fm_3(\theta_1, \theta_2, \theta_3) \pm\pm'2 \fm_3(\iota_*(\theta_1), \iota_*(\theta_2), \iota_*(\theta_3)) \end{bmatrix} \\
	 			=& \frac{1}{2}e_{\Delta(\gamma_1)_{\pm'}}.
	 		\end{align*}
 			\item $\gamma_2$ is ordinary and $\gamma_3$ is special. This is the same with the previous case.
	 		\end{enumerate}
		\end{enumerate}
	Thus, $\Psi$ is compatible with $\Aoo$-structure. 
\end{proof}

\subsection{Interior morphisms between tagged arcs and idempotents}\label{sec:tagide}
\begin{figure}[h!]
	\centering
	\includegraphics[width=0.3\linewidth]{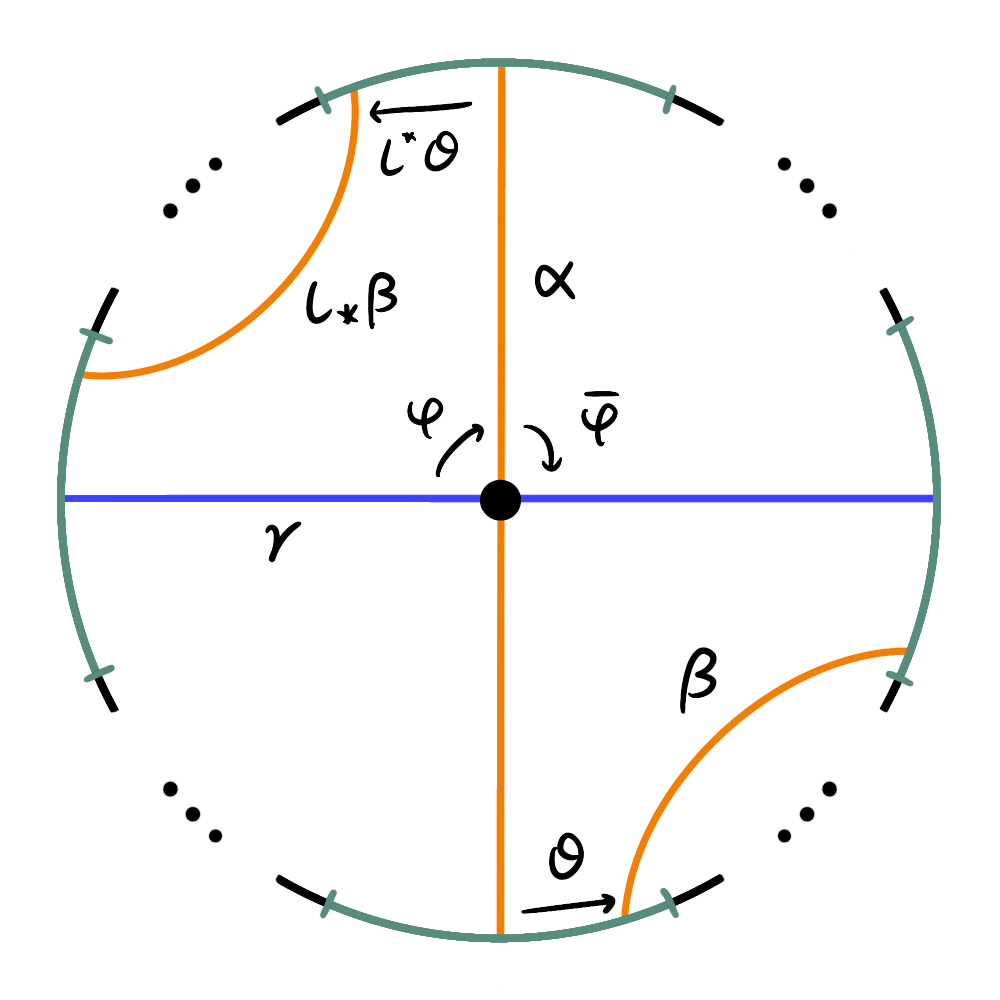}
	\caption{Interior morphism between tagged arcs and idempotents}
	\label{fig:example}
\end{figure}
Let us justify our definition of interior morphism between tagged arcs in Definition \ref{defn:intmor}.
An involutive tagged arc system does not allow two tagged arcs in the arc system to meet at an interior marked point. However, we can add such arcs into the arc system 
in the form of twisted complexes and investigate the morphisms between them. We will consider the following simple example to illustrate ``interior morphisms'' 
between idempotents of two arcs in the setting of Fukaya category.

Suppose we are given a special arc $\alpha$ and an ordinary arc $\beta$ with a boundary morphism $\theta : \alpha \rightarrow \beta$ as in Figure \ref{fig:example}. Then, there is another arc $\iota_*(\beta)$ with the boundary morphism $\iota_*(\theta)$. Let us assume that the boundary morphisms $\theta$ and $\iota_*(\theta)$ are of degree $1$. Then, the blue arc $\gamma$ can be understood as a twisted complex
$$\gamma \deq \left(\alpha[d] \oplus \beta[d] \oplus \iota_*(\beta)[d], \delta \deq \begin{bmatrix} 0 & 0 & 0 \\ \theta & 0 & 0 \\ \iota_*(\theta) & 0 & 0 \end{bmatrix}\right),$$ for some $d \in \bzz$. 
The $\iota$-image of it is the twisted complex 

$$\iota_*(\gamma) =\left(\alpha[d] \oplus  \iota_*(\beta)[d] \oplus \beta[d], \delta \deq \begin{bmatrix} 0 & 0 & 0 \\ \iota_*(\theta) & 0 & 0 \\ \theta & 0 & 0 \end{bmatrix}\right).$$

Then, the hom spaces (in a minimal model) $\ho(\gamma, \alpha)$, $\ho(\iota_*(\gamma), \alpha)$, $\ho(\alpha, \gamma)$, and $\ho(\alpha, \iota_*(\gamma))$ are all one-dimensional and generated by the following morphisms, \resp.
$$\phi \deq \begin{bmatrix} e_\alpha[-d] & 0 & 0 \end{bmatrix}, \quad \phi' \deq \begin{bmatrix} e_\alpha[-d] & 0 & 0 \end{bmatrix}, \quad \overline{\phi} \deq \begin{bmatrix} 0 \\ \theta[d] \\ -\iota_*(\theta)[d] \end{bmatrix}, \quad \overline{\phi}' \deq \begin{bmatrix} 0 \\ \iota_*(\theta)[d] \\ -\theta[d] \end{bmatrix}.$$ \resp. 
Two twisted complexes $\gamma$ and $\iota_*(\gamma)$ are  identified via the morphism $$\epsilon \deq \begin{bmatrix} e_\alpha & 0 & 0 \\ 0 & 0 & e_{\iota_*(\beta)} \\ 0 & e_\beta & 0 \end{bmatrix} : \gamma \rightarrow \iota_*(\gamma).$$
This identifies the generator $\phi$ (which has degree $d$) with $(-1)^d\phi'$ via 
$$\fm_2(\epsilon, \phi) = \fm_2\left(\begin{bmatrix} e_\alpha & 0 & 0 \\ 0 & 0 & e_{\iota_*(\beta)} \\ 0 & e_\beta & 0 \end{bmatrix}, \begin{bmatrix} e_\alpha[-d] & 0 & 0 \end{bmatrix}\right) = (-1)^d\begin{bmatrix} e_\alpha[-d] & 0 & 0 \end{bmatrix} = (-1)^d\phi'.$$
Similarly, we can identify $\overline{\phi}$ with $(-1)^{d+1}\overline{\phi}'$ via
$$\fm_2(\overline{\phi}, \epsilon) = \fm_2\left(\begin{bmatrix} 0 \\ \theta[d] \\ -\iota_*(\theta)[d] \end{bmatrix}, \begin{bmatrix} e_\alpha & 0 & 0 \\ 0 & 0 & e_{\iota_*(\beta)} \\ 0 & e_\beta & 0 \end{bmatrix}\right) = (-1)^d\begin{bmatrix} 0 \\ -\iota_*(\theta)[d] \\ \theta[d] \end{bmatrix} = (-1)^{d+1}\overline{\phi}'.$$
Here the sign comes from the degree shift of $\overline{\phi}$.

Hence, $\iota$ acts on the hom spaces $\ho(\gamma, \alpha)$ and $\ho(\iota_*(\gamma), \alpha)$ of two special arcs as $$\iota_*(\phi) = (-1)^d\phi, \quad \iota_*(\overline{\phi}) = (-1)^{d+1}\overline{\phi}.$$

From this action, we know 
\begin{align*}
	\ho(\Delta(\gamma), \Delta(\alpha))^\iota &\deq \left\{\begin{bmatrix} a\phi & b\phi \\ (-1)^db\phi & (-1)^da\phi \end{bmatrix} : a, b\in \field\right\}, \quad \deg{\begin{bmatrix} a\phi & b\phi \\ (-1)^db\phi & (-1)^da\phi \end{bmatrix}} = d,\\
	\ho(\Delta(\alpha), \Delta(\gamma))^\iota &\deq \left\{\begin{bmatrix} a\overline{\phi} & b\overline{\phi} \\ (-1)^{d+1}b\overline{\phi} & (-1)^{d+1}a\overline{\phi} \end{bmatrix} : a, b\in \field\right\}, \quad \deg{\begin{bmatrix} a\overline{\phi} & b\overline{\phi} \\ (-1)^{d+1}b\overline{\phi} & (-1)^{d+1}a\overline{\phi} \end{bmatrix}} = 1-d.
\end{align*}
Together with idempotents of $\alpha$ and $\gamma$, we get the following results.
\begin{itemize}
	\item When $d$ is even, 
	\begin{alignat*}{3}
		\ho(\Delta(\gamma)_+, \Delta(\alpha)_+) &= \field\left<\frac{1}{2}\begin{bmatrix} \phi & \phi \\ \phi & \phi \end{bmatrix}\right>, \quad 
		&&\ho(\Delta(\alpha)_+, \Delta(\gamma)_+) &&= 0, \\
		\ho(\Delta(\gamma)_+, \Delta(\alpha)_-) &= 0, \quad 
		&&\ho(\Delta(\alpha)_-, \Delta(\gamma)_+) &&= \field\left<\frac{1}{2}\begin{bmatrix} \overline{\phi} & -\overline{\phi} \\ -\overline{\phi} & \overline{\phi}\end{bmatrix}\right>, \\
		\ho(\Delta(\gamma)_-, \Delta(\alpha)_+) &= 0, \quad 
		&&\ho(\Delta(\alpha)_+, \Delta(\gamma)_-) &&= \field\left<\frac{1}{2}\begin{bmatrix} \overline{\phi} & \overline{\phi} \\ \overline{\phi} & \overline{\phi}\end{bmatrix}\right>, \\
		\ho(\Delta(\gamma)_-, \Delta(\alpha)_-) &= \field\left<\frac{1}{2}\begin{bmatrix} \phi & -\phi \\ -\phi & \phi \end{bmatrix}\right>, \quad 
		&&\ho(\Delta(\alpha)_-, \Delta(\gamma)_-) &&= 0.
	\end{alignat*}
	\item When $d$ is odd,
	\begin{alignat*}{3}
		\ho(\Delta(\gamma)_+, \Delta(\alpha)_+) &= 0, \quad 
		&&\ho(\Delta(\alpha)_+, \Delta(\gamma)_+) &&= \field\left<\frac{1}{2}\begin{bmatrix} \overline{\phi} & \overline{\phi} \\ \overline{\phi} & \overline{\phi}\end{bmatrix}\right>, \\
		\ho(\Delta(\gamma)_+, \Delta(\alpha)_-) &= \field\left<\frac{1}{2}\begin{bmatrix} \phi & \phi \\ \phi & \phi \end{bmatrix}\right>, \quad 
		&&\ho(\Delta(\alpha)_-, \Delta(\gamma)_+) &&= 0, \\
		\ho(\Delta(\gamma)_-, \Delta(\alpha)_+) &= \field\left<\frac{1}{2}\begin{bmatrix} \phi & -\phi \\ -\phi & \phi \end{bmatrix}\right>, \quad 
		&&\ho(\Delta(\alpha)_+, \Delta(\gamma)_-) &&= 0, \\
		\ho(\Delta(\gamma)_-, \Delta(\alpha)_-) &= 0, \quad 
		&&\ho(\Delta(\alpha)_-, \Delta(\gamma)_-) &&= \field\left<\frac{1}{2}\begin{bmatrix} \overline{\phi} & -\overline{\phi} \\ -\overline{\phi} & \overline{\phi}\end{bmatrix}\right>.
	\end{alignat*}
\end{itemize}
In summary, there is a morphism from $\Delta(\gamma)_{\sigma_1}$ to $\Delta(\alpha)_{\sigma_2}$ if and only if $d \equiv \sigma_1 - \sigma_2$ modulo under $2$. This explains the algebraic rule in Section \ref{subsection:InteriorMorphisms}.

\section{$\Aoo$-identities}\label{sec:Aooness}

In this section, we prove Theorem \ref{thm:ItIsAooCategory} that the data in Definition \ref{defn:FukayaCategoryForMarkedSurface} indeed form an $\Aoo$-category. In the first subsection, we introduce easy but useful lemmas. Then, we compute $\Aoo$-relations in the next section.

First, we define the following sets.
\begin{itemize}
	\item $\Con$ is the set of concatenable pairs of boundary/interior morphisms.
	\item $\Disc$ is the set of disc sequences.
	\item $\Comp$ is the set of composition sequences. We usually denote a composition sequence as $(\phi_1, \dots, \phi_m ; \psi)$.
	\item $\Thick$ is the set of thick triples. For a thick triple $(\phi_1 ; (\phi_2, \phi_3))$ and its value $\phi_4$, we set $\phi_1^\vee \deq \phi_4$ and $\phi_4^\wedge \deq \phi_1$. Similarly, for a thick triple $((\psi_1, \psi_2) ; \psi_3)$ and its value $\psi_4$, we usually set ${}^\vee\psi_3 \deq \psi_4$ and ${}^\wedge\psi_4 \deq \psi_3$.
\end{itemize}

We denote the domain and codomain of a morphism $\theta$ by $s(\theta)$ and $t(\theta)$, \resp. Let $\alpha$ and $\beta$ be tagged arcs meeting at an interior marking $p$. We say $\alpha$ {\em meets} $\beta$ {\em on the left} or {\em right} if $p$ defines an interior morphism $p^\alpha_\beta$ or $p^\beta_\alpha$, \resp. Let $(\phi, \psi)$ be a concatenable pair. Then, we say $t(\phi) = s(\psi)$ {\em divides} the morphism $\phi \bullet \psi$.

\subsection{Preliminary Lemmas}
There could be multiple $\Aoo$-operations for a given input : some pair can be a concatenable pair and form a composition sequence at the same time. Similarly, some triple can be a thick triple and disc sequence at the same time.  But we can prove the following lemma.
\begin{lemma}
	Let $\phi_1, \dots, \phi_n$ be a sequence of composable basic morphisms. Then, for $1 \leq i < j \leq n$ such that $(i, j) \neq (1, n)$, at most one of the following expression (among  $\circ, \square \in \{\con, \disc, \comp, \thick\}$) is nonzero.
	$$\fm_{n-j+i}^\circ(\phi_1, \dots, \phi_{i-1}, \fm_{j-i+1}^\square(\phi_i, \dots, \phi_j), \phi_{j+1}, \dots, \phi_n).$$
\end{lemma}
\begin{proof}
We show that  values of different $\Aoo$-operations are at different markings, hence one of the outer operation should vanish.

Let $(\theta, \phi)$ is concatenable pair such that $(\theta, \phi ; \psi)$ is also a composition sequence. Suppose that $\theta, \phi$, and $\psi$ are at the same marking. Since $\psi$ is the value of a composition sequence, the marking has to be an interior marking. Also, since $\theta$ and $\phi$ are concatenable, the composition sequence $(\theta, \phi ; \psi)$ is folded. In particular, $\nu(\beta) = 2$. Now let $D$ be a disc associated with the composition sequence. Then, $D$ bounds both endpoints of $\beta$, which violates the fourth condition of Definition \ref{defn:ArcSystemForOrbiSuefaces}. This proves $\psi$ cannot be at the same marking with $\theta$ and $\phi$.

Let $(\phi_1 ; (\phi_2, \phi_3))$ be a thick triple such that $\phi_3$ has a decomposition $(\phi_3^1, \phi_3^2)$ such that $(\phi_1, \phi_2, \phi_3^1)$ is a disc sequence. Then, $\fm_3^\thick(\phi_1, \phi_2, \phi_3)$ is a boundary morphism while $\fm_3^\disc(\phi_1, \phi_2, \phi_3)$ is an interior morphism. The case of a thick triple $((\phi_1, \phi_2) ; \phi_3)$ is the same. Thus we deduce the following result.
\end{proof}

From now on, we look at the splittings of disc and composition sequences and how their weights and sign factors are related.
In the case that output is a result of a disc sequence, we have seen that we have two possible cases as in  Figure \ref{fig:DiscSequenceDecomposition}(a) and (b).
This is when a disc sequence splits into two disc sequences, or one disc and one composition sequence as explained right after Theorem \ref{thm:ItIsAooCategory}.
We compute the relevant weights and sign factors here. Note that the weight and sign factor for the disc sequence is cyclic symmetric. So sometimes we will omit some cases when
they can be obtained by cyclic rotations. 

\begin{lemma}\label{Lemma:PhiSigmaDD} {\bf (Disc splits into two discs)} Let us consider the case of \eqref{eq:dd2} (the computation for \eqref{eq:dd} can be obtained from a cyclic rotation).
	Let $\overrightarrow{\phi} = (\phi_1, \dots, \phi_m)$ and $\overrightarrow{\psi} = (\psi_1, \dots, \psi_n)$ be disc sequences such that $(\phi_m, \psi_1), (\psi_n, \phi_1) \in \Con$ and $\ora{\phi}\bullet\overrightarrow{\psi} \deq (\psi_n \bullet \phi_1, \phi_2, \dots, \phi_{m-1}, \phi_m \bullet \psi_1, \psi_2, \dots, \psi_{n-1}) \in \Disc$. Then,
	\begin{align}
		\Phi(\ora{\phi}\bullet\ora{\psi}) &= \Phi(\ora{\phi}) + \Phi(\ora{\psi}) - \left<\psi_n, \phi_1\right> - \left<\phi_m, \psi_1\right>, \label{AppendixLemma:DiscDiscPhi}\\
		\Sigma(\ora{\phi}\bullet\ora{\psi}) &= \Sigma(\ora{\phi}) + \Sigma(\ora{\psi}) - \sigma(\psi_n, \phi_1) - \sigma(\phi_m, \psi_1). \label{AppendixLemma:DiscDiscSigma}
	\end{align}
\end{lemma}
\begin{proof}
	From
	\begin{align*}
		\Phi(\ora{\phi}) =& \sum_{i=1}^m\left(\left[\phi_i\right] + \left<\phi_i, \phi_{i+1}\right> - \left[\phi_i\right]\left<\phi_i, \phi_{i+1}\right>\right), \\
		\Phi(\ora{\psi}) =& \sum_{j=1}^n\left(\left[\psi_j\right] + \left<\psi_j, \psi_{j+1}\right> - \left[\psi_j\right]\left<\psi_j, \psi_{j+1}\right>\right), \\
		\Phi(\ora{\phi}\bullet\ora{\psi}) =& \sum_{i=1}^{m-1}\left(\left[\phi_i\right] + \left<\phi_i, \phi_{i+1}\right> - \left[\phi_i\right]\left<\phi_i, \phi_{i+1}\right>\right) + \sum_{j=1}^{n-1}\left(\left[\psi_j\right] + \left<\psi_j, \psi_{j+1}\right> - \left[\psi_j\right]\left<\psi_j, \psi_{j+1}\right>\right),
	\end{align*}
	we get
	\begin{align*}
		\Phi(\ora{\phi}) + \Phi(\ora{\psi}) - \Phi(\ora{\phi}\bullet\ora{\psi}) &= \left[\phi_m\right] + \left<\phi_m, \phi_1\right> - \left[\phi_m\right]\left<\phi_m, \phi_1\right> + \left[\psi_n\right] + \left<\psi_n, \psi_1\right> - \left[\psi_n\right]\left<\psi_n, \psi_1\right>.
	\end{align*}
	Note that $\left<\phi_m, \phi_1\right> = 1$ if and only if $\left<\psi_n, \psi_1\right> = 1$. If they are $1$, then $\left<\phi_m, \psi_1\right>$ and $\left<\psi_n, \phi_1\right>$ are also $1$, so (\ref{AppendixLemma:DiscDiscPhi}) holds. If they are zero, then $\left<\phi_m, \psi_1\right> = \left[\psi_n\right]$ and $\left<\psi_n, \phi_1\right> = \left[\phi_m\right]$. Thus (\ref{AppendixLemma:DiscDiscPhi}) holds in this case as well.
	
	Similarly, we have
	\begin{align*}
		\Sigma(\ora{\phi}) + \Sigma(\ora{\psi}) - \Sigma(\ora{\phi}\bullet\ora{\psi}) =& \sum_{i=1}^m\left(\sigma(\phi_i) - \sigma(\phi_i)\left<\phi_i, \phi_{i+1}\right>\right) + \sum_{j=1}^n\left(\sigma(\psi_j) -  \sigma(\psi_j)\left<\psi_j, \psi_{j+1}\right>\right) \\ 
		&- \sum_{i=1}^{m-1}\left(\sigma(\phi_i) - \sigma(\phi_i)\left<\phi_i, \phi_{i+1}\right>\right) - \sum_{j=1}^{n-1}\left(\sigma(\psi_j) - \sigma(\psi_j)\left<\psi_j, \psi_{j+1}\right>\right) \\ 
		=& \sigma(\phi_m) - \sigma(\phi_m)\left<\phi_m, \phi_1\right> + \sigma(\psi_n) - \sigma(\psi_n)\left<\psi_n, \psi_1\right>.
	\end{align*}
	If $\left<\phi_m, \phi_1\right>$ and $\left<\psi_n, \psi_1\right>$ are $1$, then $\sigma(\psi_n, \phi_1) = \sigma(\phi_m, \psi_1)$. If they are zero, then $\sigma(\psi_n, \phi_1) = \sigma(\phi_m)$ and $\sigma(\phi_m, \psi_1) = \sigma(\psi_n)$. In both cases, (\ref{AppendixLemma:DiscDiscSigma}) holds.
\end{proof}

 \begin{lemma}\label{Lemma:PhiSigmaCD}{\bf (Disc splits into composition and disc sequences)}
	Let us consider the case of (\ref{eq:DCComp}). Let $\ora{\phi} = (\phi_1, \dots, \phi_m)$ be a composition sequence with value $\theta$ and $\ora{\psi} = (\psi_1, \dots, \psi_n, \theta)$ be a disc sequence such that $(\phi_m, \psi_1), (\psi_n, \phi_1) \in \Con$ and $\ora{\phi}\bullet\ora{\psi} \deq (\phi_1, \dots, \phi_{m-1}, \phi_m\bullet\psi_1, \psi_2, \dots, \psi_n)$,\\ $\ora{\psi}\bullet\ora{\phi} \deq (\psi_1, \dots, \psi_{n-1}, \psi_n\bullet\phi_1, \phi_2, \cdots, \phi_m) \in \Disc$. Then,
	\begin{align*}
		\Phi(\ora{\phi}\bullet\ora{\psi}) &= \Phi(\ora{\phi} ; \theta) + \Phi(\ora{\psi}), \\
		\Sigma(\ora{\phi}\bullet\ora{\psi}) &= \Sigma(\ora{\phi} ; \theta) + \Sigma(\ora{\psi}) - \sigma(\phi_m, \psi_1), \\
		\Phi(\ora{\psi}\bullet\ora{\phi}) &= \Phi(\ora{\phi} ; \theta) + \Phi(\ora{\psi}), \\
		\Sigma(\ora{\psi}\bullet\ora{\phi}) &= \Sigma(\ora{\phi} ; \theta) + \Sigma(\ora{\psi}) - \sigma(\psi_n, \phi_1).
	\end{align*}
\end{lemma}
 
From now on, let us consider the case when the disc for a composition sequence splits into two parts.
This was divided into several types in Lemma \ref{lem:comdecomp}. We can compute the related weights and sign factors as in the disc cases. 
We will leave the details as an exercise (or see \cite{Kim24}).

\begin{lemma}\label{Lemma:PhiSigmaDCM}{\bf (Figure \ref{fig:CompDecomposition1}(a))}
	Let $\ora{\phi} = (\phi_1, \dots, \phi_m)$ be a composition sequence with value $\theta$ and $\ora{\psi} = (\psi_1, \dots, \psi_n)$ be a disc sequence such that $(\phi_k, \psi_1), (\psi_n, \phi_{k+1}) \in \Con$ for some $k$ and \\$\ora{\phi}\bullet \ora{\psi} \deq (\phi_1, \dots, \phi_{k-1}, \phi_k\bullet\psi_1, \psi_2, \dots, \psi_{n-1}, \psi_n\bullet\phi_{k+1}, \phi_{k+2}, \dots, \phi_m ; \theta) \in \Comp$. Then,
	\begin{align*}
		\Phi(\ora{\phi}\bullet\ora{\psi} ; \theta) =& \Phi(\ora{\phi} ; \theta) + \Phi(\ora{\psi}) -\left<\phi_k, \psi_1\right> - \left<\psi_n, \phi_{k+1}\right>, \\
		\Sigma(\ora{\phi}\bullet \ora{\psi} ; \theta) =& \Sigma(\ora{\phi} ; \theta) + \Sigma(\ora{\psi}) - \sigma(\phi_k, \psi_1) - \sigma(\psi_n, \phi_{k+1}).
	\end{align*}	
\end{lemma}

\begin{lemma}\label{Lemma:PhiSigmaDCF}{\bf (Figure \ref{fig:CompDecomposition1}(b)-1)}
	Let $\ora{\phi} = (\phi_1, \dots, \phi_m)$ be a composition sequence with value $\theta$ and $\ora{\psi} = (\psi_1, \dots, \psi_n)$ be a disc sequence such that $\theta = \psi_1 \bullet \theta'$, $(\psi_n, \phi_1) \in \Con$ and \\$\ora{\psi}\bullet\ora{\phi} \deq (\psi_2, \dots, \psi_{n-1}, \psi_n \bullet \phi_1, \phi_2, \dots, \phi_m ; \theta') \in \Comp$. Then,
	\begin{align*}
		\Phi(\ora{\psi}\bullet\ora{\phi} ; \theta') &= \Phi(\ora{\psi}) + \Phi(\ora{\phi} ; \theta) - 1 - \left<\psi_1, \theta'\right>, \\
		\Sigma(\ora{\psi}\bullet\ora{\phi} ; \theta') &= \Sigma(\ora{\psi}) + \Sigma(\ora{\phi} ; \theta) - \sigma(\psi_n, \phi_1) - \sigma(\psi_1, \theta').
	\end{align*}
\end{lemma}

\begin{lemma}\label{Lemma:PhiSigmaDCL}{\bf (Figure \ref{fig:CompDecomposition1}(b)-2)}
	Let $\ora{\phi} = (\phi_1, \dots, \phi_m)$ be a composition sequence with value $\theta$ and $\ora{\psi} = (\psi_1, \dots, \psi_n)$ be a disc sequence such that $\theta = \theta' \bullet \psi_n$, $(\phi_m, \psi_1) \in \Con$ and \\ $\ora{\phi}\bullet\ora{\psi} \deq (\phi_1, \dots, \phi_{m-1}, \phi_m \bullet \psi_1, \psi_2, \dots, \psi_{n-1} ; \theta') \in \Comp$. Then,
	\begin{align*}
		\Phi(\ora{\phi}\bullet\ora{\psi} ; \theta') &= \Phi(\ora{\phi} ; \theta) + \Phi(\ora{\psi}) - 1 - \left<\theta', \psi_n\right>, \\
		\Sigma(\ora{\phi}\bullet\ora{\psi} ; \theta') &= \Sigma(\ora{\phi} ; \theta) + \Sigma(\ora{\psi}) - \sigma(\phi_m, \psi_1) - \sigma(\theta', \psi_n).
	\end{align*}
\end{lemma}

\begin{lemma}\label{Lemma:PhiSigmaCCM}{\bf (Figure \ref{fig:CompDecomposition1}(c)-1)}
	Let $\ora{\phi} = (\phi_1, \dots, \phi_m)$ be a composition sequence with value $\theta$ and $\ora{\psi} = (\psi_1, \dots, \psi_n)$ be a composition sequence with value $\phi_k$ for some $2 \leq k \leq m-1$ such that $(\phi_{k-1}, \psi_1), (\psi_n, \phi_{k+1}) \in \Con$ and $\ora{\phi} \bullet_{k-1} \ora{\psi} \deq (\phi_1, \dots, \phi_{k-2}, \phi_{k-1} \bullet \psi_1, \psi_2, \dots, \psi_n, \phi_{k+1}, \dots, \phi_m)$ and \\ $\ora{\phi} \bullet_{k+1} \ora{\psi} \deq (\phi_1, \dots, \phi_{k-1}, \psi_1, \psi_2, \dots, \psi_{n-1}, \psi_n \bullet \phi_{k+1}, \phi_{k+2}, \dots, \phi_m)$ are composition sequences with value $\theta$. Then,
	\begin{align*}
		\Phi(\ora{\phi}\bullet_{k-1}\ora{\psi} ; \theta) =& \Phi(\ora{\phi} ; \theta) + \Phi(\ora{\psi} ; \phi_k) - 1, \\
		\Sigma(\ora{\phi}\bullet_{k-1}\ora{\psi} ; \theta) =& \Sigma(\ora{\phi} ; \theta) + \Sigma(\ora{\psi} ; \phi_k) -\sigma(\phi_{k-1}, \psi_1) - \sigma(\psi_n, \phi_{k+1}), \\
		\Phi(\ora{\phi}\bullet_{k+1}\ora{\psi} ; \theta) =& \Phi(\ora{\phi} ; \theta) + \Phi(\ora{\psi} ; \phi_k) - 1, \\
		\Sigma(\ora{\phi}\bullet_{k+1}\ora{\psi} ; \theta) =& \Sigma(\ora{\phi} ; \theta) + \Sigma(\ora{\psi} ; \phi_k) -\sigma(\phi_{k-1}, \psi_1) - \sigma(\psi_n, \phi_{k+1}).
	\end{align*}	
\end{lemma}
\begin{lemma}\label{Lemma:PhiSigmaXDM}{\bf (Figure \ref{fig:CompDecomposition1}(c)-2)}
	Let $\ora{\theta} = (\theta_1, \dots, \theta_l)$, $\ora{\psi} = (\psi_1, \dots, \psi_n)$ be sequences of composable sequences, and $\ora{\phi} = (\phi_1, \dots, \phi_m, \eta)$ be a disc sequence such that $(\theta_l, \phi_1), (\phi_m, \psi_1) \in \Con$ and 
	\begin{align*}
		\ora{\theta} \bullet \ora{\phi} &\deq (\theta_1, \dots, \theta_{l-1}, \theta_l \bullet \phi_1, \phi_2, \dots, \phi_m, \psi_1, \dots, \psi_n), \\
		\ora{\phi} \bullet \ora{\psi} &\deq (\theta_1, \dots, \theta_l, \phi_1, \phi_2, \dots, \phi_{m-1}, \phi_m \bullet \psi_1, \psi_2, \dots, \psi_n)
	\end{align*}
	are composition sequences with value $\xi$. Then we have the following.
	\begin{align*}
		\Phi(\ora{\theta} \bullet \ora{\phi} ; \xi) &= \Phi(\ora{\phi} \bullet \ora{\psi} ; \xi) \\
		\Sigma(\ora{\theta} \bullet \ora{\phi} ; \xi) &= \Sigma(\ora{\phi} \bullet \ora{\psi} ; \xi).
	\end{align*}
\end{lemma}
 
\begin{lemma}\label{Lemma:PhiSigmaXDF}{\bf (Figure \ref{fig:CompDecomposition1}(d)-1)}
	Let $\ora{\phi} = (\phi_1, \dots, \phi_m, \zeta, \eta)$ be a disc sequence, $\ora{\psi} \deq (\psi_1, \dots, \psi_n)$ be a sequence of composable morphism, and $\xi$ be an interior morphism such that $(\phi_m, \psi_1), (\eta, \xi) \in \Con$ and 
	\begin{align*}
		\ora{\phi} \bullet \ora{\psi} &\deq (\eta, \phi_1, \dots, \phi_{m-1}, \phi_m \bullet \psi_1, \psi_2, \dots, \psi_n ; \eta \bullet \xi), \\
		\ora{\phi} \circ \ora{\psi} &\deq (\phi_1, \dots, \phi_m, \psi_1, \dots, \psi_n ; \xi)
	\end{align*}
	are compositions sequences. Then we have the following.
	\begin{align*}
		\Phi(\ora{\phi} \bullet \ora{\psi} ; \eta \bullet \xi) &= \Phi(\ora{\phi} \circ \ora{\psi} ; \xi) + \left<\eta, \xi\right> - 1, \\
		\Sigma(\ora{\phi} \bullet \ora{\psi} ; \eta \bullet \xi) &= \Sigma(\ora{\phi} \circ \ora{\psi} ; \xi) + \sigma(\eta, \xi) - \sigma(\phi_m, \psi_1).
	\end{align*}
\end{lemma}

\begin{lemma}\label{Lemma:PhiSigmaCCL}{\bf (Figure \ref{fig:CompDecomposition1}(d)-2)}
	Let $\ora{\phi} = (\phi_1, \dots, \phi_m)$ be a composition sequence with value $\theta$ and $\ora{\psi} = (\psi_1, \dots, \psi_n)$ be a composition sequence with value $\phi_m$ such that $\theta = \theta' \bullet \psi_n$, $(\phi_{m-1}, \psi_1) \in \Con$ and $\ora{\phi} \bullet_n \ora{\psi} \deq (\phi_1, \dots, \phi_{m-1}, \psi_1, \dots, \psi_{n-1} ; \theta') \in \Comp$. Then, $\ora{\phi} \bullet_1 \ora{\psi} \deq (\phi_1, \dots, \phi_{m-2}, \phi_{m-1} \bullet \psi_1, \psi_2, \dots, \psi_n ; \theta) \in \Comp$ and we have the following.
	\begin{align*}
		\Phi(\ora{\phi}\bullet_n\ora{\psi} ; \theta') =& \Phi(\ora{\phi} ; \theta) + \Phi(\ora{\psi} ; \phi_m) - \left<\theta', \psi_n\right> + 1, \\
		\Sigma(\ora{\phi}\bullet_n\ora{\psi} ; \theta') =& \Sigma(\ora{\phi} ; \theta) + \Sigma(\ora{\psi} ; \phi_m) - \sigma(\theta', \psi_n), \\
		\Phi(\ora{\phi}\bullet_1\ora{\psi} ; \theta) =& \Phi(\ora{\phi} ; \theta) + \Phi(\ora{\psi} ; \phi_m), \\
		\Sigma(\ora{\phi}\bullet_1\ora{\psi} ; \theta) =& \Sigma(\ora{\phi} ; \theta) + \Sigma(\ora{\psi} ; \phi_m) - \sigma(\phi_{m-1}, \psi_1).
	\end{align*}	
\end{lemma}

\begin{lemma}\label{Lemma:PhiSigmaCCF}{\bf (Figure \ref{fig:CompDecomposition1}(d)-3)}
	Let $\ora{\phi} = (\phi_1, \dots, \phi_m)$ be a composition sequence with value $\theta$ and $\ora{\psi} = (\psi_1, \dots, \psi_n)$ be a composition sequence with value $\phi_1$ such that $\theta = \psi_1 \bullet \theta'$, $(\psi_n, \phi_2) \in \Con$ and $\ora{\phi} \bullet_1 \ora{\psi} \deq (\psi_2, \dots, \psi_n, \phi_2, \dots, \phi_m ; \theta') \in \Comp$. Then, $\ora{\phi} \bullet_n \ora{\psi} \deq (\psi_1, \dots, \psi_{n-1}, \psi_n \bullet \phi_2, \phi_3, \dots, \phi_m ; \theta) \in \Comp$ and we have the following.
	\begin{align*}
		\Phi(\ora{\phi}\bullet_1\ora{\psi} ; \theta') =& \Phi(\ora{\phi} ; \theta) + \Phi(\ora{\psi} ; \phi_1) - \left<\psi_1, \theta'\right> + 1, \\
		\Sigma(\ora{\phi}\bullet_1\ora{\psi} ; \theta') =& \Sigma(\ora{\phi} ; \theta) + \Sigma(\ora{\psi} ; \phi_1) - \sigma(\psi_1, \theta'), \\
		\Phi(\ora{\phi}\bullet_n\ora{\psi} ; \theta) =& \Phi(\ora{\phi} ; \theta) + \Phi(\ora{\psi} ; \phi_1), \\
		\Sigma(\ora{\phi}\bullet_n\ora{\psi} ; \theta) =& \Sigma(\ora{\phi} ; \theta) + \Sigma(\ora{\psi} ; \phi_1) - \sigma(\psi_n, \phi_2)
	\end{align*}	
\end{lemma}

\begin{lemma}\label{Lemma:PhiSigmaXDL}{\bf (Figure \ref{fig:CompDecomposition1}(d)-4)}
	Let $\ora{\phi} = (\phi_1, \dots, \phi_m)$ be a sequence of composable morphisms, $\ora{\psi} \deq (\eta, \zeta, \psi_1, \dots, \psi_n)$ be a disc sequence, and $\xi$ be an interior morphism such that $(\phi_m, \psi_1), (\xi, \eta) \in \Con$ and 
	\begin{align*}
		\ora{\phi} \bullet \ora{\psi} &\deq (\phi_1, \dots, \phi_{m-1}, \phi_m \bullet \psi_1, \psi_2, \dots, \psi_n, \eta ; \eta \bullet \xi), \\
		\ora{\phi} \circ \ora{\psi} &\deq (\phi_1, \dots, \phi_m, \psi_1, \dots, \psi_n ; \xi)
	\end{align*}
	are compositions sequences. Then we have the following.
	\begin{align*}
		\Phi(\ora{\phi} \bullet \ora{\psi} ; \eta \bullet \xi) &= \Phi(\ora{\phi} \circ \ora{\psi} ; \xi) + \left<\xi, \eta\right> - 1, \\
		\Sigma(\ora{\phi} \bullet \ora{\psi} ; \eta \bullet \xi) &= \Sigma(\ora{\phi} \circ \ora{\psi} ; \xi) + \sigma(\xi, \eta) - \sigma(\phi_m, \psi_1).
	\end{align*}
\end{lemma}
We have a computation involving thick triples.
\begin{lemma}\label{Lemma:PhiSigmaT}
	Let $\ora{\phi} = (\phi_1, \dots, \phi_n)$ be a disc sequence. If $\psi$ is an interior morphism such that $((\psi, \phi_1) ; \phi_2)$ is a thick triple, then $\ora{\phi}^\vee = (\phi_2, \dots, \phi_n ; \phi_1^\vee)$ is a composition sequence and we have the following.
	\begin{align*}
		\Phi(\ora{\phi}) &= \Phi(\ora{\phi}^\vee ; \psi) + \left<\psi, \phi_1\right> + 1, \\
		\Sigma(\ora{\phi}) &= \Sigma(\ora{\phi}^\vee) + \sigma(\phi_1) + \sigma(\psi, \phi_1).
	\end{align*}
	Similarly, if $\psi$ is an interior morphism such that $(\phi_{n-1}, (\phi_n, \psi))$ is a thick triple, then $\ora{\phi}^\vee = (\phi_1, \dots, \phi_{n-1} ; \phi_n^\vee)$ is a composition sequence and we have the following.
	\begin{align*}
		\Phi(\ora{\phi}) &= \Phi(\ora{\phi}^\vee ; \psi) + \left<\phi_n, \psi\right> + 1, \\
		\Sigma(\ora{\phi}) &= \Sigma(\ora{\phi}^\vee ; \psi) + \sigma(\phi_n) + \sigma(\phi_n, \psi).
	\end{align*}
\end{lemma}

\subsection{Proof of $\Aoo$-identities}

In this section, we prove that the data defined in Section \ref{defn:FukayaCategoryForMarkedSurface} indeed form an $\Aoo$-category. 
We need to prove that for an input $(\phi_1, \dots, \phi_n)$, $$\sum_{k=1}^{n}\sum_{i=0}^{n-k}(-1)^{\deg{\phi_1} + \dots \deg{\phi_i} + i}\fm_{n-k+1}(\phi_1, \dots, \phi_i, \fm_k(\phi_{i+1}, \dots, \phi_{i+k}), \phi_{i+k+1}, \dots, \phi_n) = 0.$$ We will prove this case by case.

Suppose that we are given two pairs $\ora{\phi} = (\phi_1, \dots, \phi_m), \ora{\psi} = (\psi_1, \dots, \psi_n)$ such that $\fm_m(\phi_1, \dots, \phi_m) \neq 0$ and $\fm_n(\psi_1, \dots, \psi_n) = c\phi_k$, for some nonzero $c \in \field$. Then, we will check the $\Aoo$-identity for the input  $(\phi_1, \dots, \phi_{k-1}, \psi_1, \dots, \psi_n, \phi_{k+1}, \dots, \phi_m)$. Let us call $\ora{\phi}$ and $\ora{\psi}$ the {\em outer-input} and {\em inner-input}, \resp. Also we call $\phi_k$ the {\em value} of $\ora{\psi}$. An outer-input is one of the following.
\begin{enumerate}
	\item Unit $(e, \phi)$, $(\phi, e)$.
	\item Concatenable pair $(\phi_1, \phi_2)$.
	\item Disc sequence $(\phi_1, \dots, \phi_n)$.
	\item Composition sequence $(\phi_1, \dots, \phi_n)$.
	\item Thick triple $(\phi_1 ; (\phi_2, \phi_3))$ and $((\phi_1, \phi_2) ; \phi_3)$.
\end{enumerate}

Since computations are similar, we will give the full detail only for the first case and only the sketches for the other cases. 
(rest of the details will be available at \cite{Kim24}).  
\subsubsection{Unit}\label{Appendix:Unit}
Here, we deal with the outer-input $\ora{\phi} = (e, \phi)$. 
If $\phi$ is the value of the inner input, then unital property can be used to show an $\Aoo$-identity. 
So let us assume that $e$ is the value of the inner-input $\ora{\psi}$. 
Thus $\ora{\psi} = (\psi_1, \dots, \psi_n)$ has to be a disc sequence. We divide them as follows.
\begin{enumerate}
	\item $(\psi_n, \psi_1) \notin \Con$. (Hence $\left<\psi_n, \psi_1\right>=0$.) This further has the following three subcases.
	\begin{enumerate}
		\item $(\psi_n, \phi) \in \Con$. 
		\begin{align*}
			&\fm_2(\fm_n(\psi_1, \dots, \psi_n), \phi) + (-1)^{\deg{\psi_1} + \dots + \deg{\psi_{n-1}} + n-1}\fm_n(\psi_1, \dots, \psi_{n-1}, \fm_2(\psi_n, \phi)) \\
			=&(-1)^{\Sigma(\ora{\psi})}\half{\Phi(\ora{\psi})}\fm_2(e, \phi) - (-1)^{\sigma(\psi_n, \phi)}\half{\left<\psi_n, \phi\right>}\fm_n(\psi_1, \dots, \psi_{n-1}, \psi_n \bullet \phi)\\
			=&(-1)^{\Sigma(\ora{\psi})}\half{\Phi(\ora{\psi})}\phi - (-1)^{\Sigma(\ora{\psi})}\half{\Phi(\ora{\psi})}\phi = 0.
		\end{align*}
		Let us explain the computation in more detail. From Definition \ref{defn:FukayaCategoryForMarkedSurface}, we first have $\fm_2(\psi_n, \phi) = (-1)^{\deg{\psi_n} + \sigma(\psi_n, \phi)}\half{\left<\psi_n, \phi\right>}\psi_n \bullet \phi$. So, the sign becomes $\deg{\psi_1} + \dots + \deg{\psi_n} + n-1 + \sigma(\psi_n, \phi)$. Since $(\psi_1, \dots, \psi_n)$ is a disc sequence, by Lemma \ref{Lemma:DegreeFormulae}, $\deg{\psi_1} + \dots + \deg{\psi_n} + n$ is even. So the sign is equivalent to $\sigma(\psi_n, \phi) + 1$ and we get the first equality. Then, from the definition of $\fm_m(\psi_1, \dots, \psi_{n-1}, \psi_n \bullet \phi)$, the sign $\sigma(\psi_n, \phi)$ and weight $\left<\psi_n, \phi\right>$ are canceled. Thus we get the second equality. In the following computations, we will use Lemma \ref{Lemma:DegreeFormulae} repeatedly.
		\item $(\psi_n, \phi) \notin \Con$. This further has the following three subcases.
		\begin{enumerate}
			\item $\phi = \psi_1\bullet \theta$ for some nonzero $\theta$.
			\begin{align*}
				&\fm_2(\fm_n(\psi_1, \dots, \psi_n), \psi_1\bullet\theta) + (-1)^{\deg{\psi_1} + 1}\fm_2(\psi_1, \fm_n(\psi_2, \dots, \psi_n, \psi_1 \bullet \theta)) \\
				=&(-1)^{\Sigma(\ora{\psi})}\half{\Phi(\ora{\psi})}\fm_2(e, \psi_1 \bullet \theta) + (-1)^{\deg{\psi_1} + 1 + \Sigma(\ora{\psi}) - \sigma(\psi_1, \theta)}\half{\Phi(\ora{\psi}) - \left<\psi_1, \theta\right>}\fm_2(\psi_1, \theta) \\
				=&(-1)^{\Sigma(\ora{\psi})}\half{\Phi(\ora{\psi})}\psi_1 \bullet \theta - (-1)^{\Sigma(\ora{\psi})}\half{\Phi(\ora{\psi})}\psi_1 \bullet \theta =0.
			\end{align*}
			\item $\phi = \psi_1$. This further has the following two subcases.
			\begin{enumerate}
				\item $(\psi_1, \psi_2) \notin \Con$.
				\begin{align*}
					&\fm_2(\fm_n(\psi_1, \dots, \psi_n), \psi_1) + (-1)^{\deg{\psi_1} + 1}\fm_2(\psi_1, \fm_n(\psi_2, \dots, \psi_n, \psi_1)) \\
					=&(-1)^{\Sigma(\ora{\psi})}\half{\Phi(\ora{\psi})}\fm_2(e, \psi_1) + (-1)^{\deg{\psi_1} + 1 + \Sigma(\ora{\psi})}\half{\Phi(\ora{\psi})}\fm_2(\psi_1, e) \\
					=&(-1)^{\Sigma(\ora{\psi})}\half{\Phi(\ora{\psi})}\psi_1 - (-1)^{\Sigma(\ora{\psi})}\half{\Phi(\ora{\psi})}\psi_1 =0.
				\end{align*}
				\item $(\psi_1, \psi_2) \in \Con$.
				\begin{align*}
					&\fm_2(\fm_n(\psi_1, \dots, \psi_n), \psi_1) + \fm_n(\fm_2(\psi_1, \psi_2), \psi_3, \dots, \psi_n, \psi_1) \\
					&+ (-1)^{\deg{\psi_1} + 1}\fm_2(\psi_1, \fm_n(\psi_2, \dots, \psi_n, \psi_1))  \\
					=&(-1)^{\Sigma(\ora{\psi})}\half{\Phi(\ora{\psi})}\fm_2(e, \psi_1) + (-1)^{\deg{\psi_1} + \sigma(\psi_1, \psi_2)}\frac{1}{2}\fm_n(\psi_1 \bullet \psi_2, \psi_3, \dots, \psi_n, \psi_1) \\
					&+ (-1)^{\deg{\psi_1} + 1 + \Sigma(\ora{\psi})}\half{\Phi(\ora{\psi}) - 1}\fm_2(\psi_1, e) \\
					=&(-1)^{\Sigma(\ora{\psi})}\half{\Phi(\ora{\psi})}\psi_1 + (-1)^{\Sigma(\ora{\psi})}\half{\Phi(\ora{\psi})}\psi_1 - (-1)^{\Sigma(\ora{\psi})}\half{\Phi(\ora{\psi})-1}\psi_1=0.
				\end{align*}
			\end{enumerate}
			\item $\psi_1 = \phi \bullet \theta$, for some nonzero $\theta$. This further has the following three subcases.
			\begin{enumerate}
				\item $t(\phi)$ divides $\psi_k$ into $\psi_k^1 \bullet \psi_k^2$ for some $2 < k < n$ so that $\ora{\psi}^1 \deq (\theta, \psi_2, \dots, \psi_{k-1}, \psi_k^1)$, $\ora{\psi}^2 \deq (\psi_k^2, \psi_{k+1}, \dots, \psi_n, \phi) \in \Disc$.
				\begin{align*}
					&\fm_2(\fm_n(\psi_1, \dots, \psi_n), \phi) + (-1)^{\deg{\psi_1} + \dots + \deg{\psi_{k-1}} + k-1}\fm_k(\psi_1, \dots, \psi_{k-1}, \fm_{n-k+2}(\psi_k, \dots, \psi_n, \phi)) \\
					=&(-1)^{\Sigma(\ora{\psi})}\half{\Phi(\ora{\psi})}\fm_2(e, \phi) -(-1)^{\deg{\phi} + \Sigma(\ora{\psi}^2) - \sigma(\psi_k^1, \psi_k^2)}\half{\Phi(\ora{\psi}^2) - \left<\psi_k^1, \psi_k^2\right>}\fm_k(\psi_1, \dots, \psi_{k-1}, \psi_k^1) \\
					=&(-1)^{\Sigma(\ora{\psi})}\half{\Phi(\ora{\psi})}\phi -(-1)^{\Sigma(\ora{\psi}^2) + \Sigma(\ora{\psi}^1) - \sigma(\psi_k^1, \psi_k^2) - \sigma(\phi, \theta)}\half{\Phi(\ora{\psi}^1) + \Phi(\ora{\psi}^2) - \left<\psi_k^1, \psi_k^2\right> - \left<\phi, \theta\right>}\phi \\
					=& 0 \quad \text{(by Lemma \ref{Lemma:PhiSigmaDD})}.
				\end{align*}
				\item $t(\phi)$ meets $t(\psi_k)$ on the left for some $1 < k < n$ with an interior morphism $\xi$ so that $(\theta, \psi_2, \dots, \psi_k ; \xi) \in \Comp$ and $(\xi, \psi_{k+1}, \dots, \psi_n, \phi) \in \Disc$. In this case, $(\psi_k, \psi_{k+1}) \in \Con$ and $\ora{\psi}' \deq (\phi, \theta, \psi_2, \dots, \psi_{k-1}, \psi_k \bullet \psi_{k+1}, \psi_{k+2}, \dots, \psi_n) \in \Disc$. Also, $\sigma(\phi, \theta) = \sigma(\psi_k, \psi_{k+1})$ and $\left<\phi, \theta\right> = \left<\psi_k, \psi_{k+1}\right>$.
				\begin{align*}
					&\fm_2(\fm_n(\psi_1, \dots, \psi_n), \phi) + (-1)^{\deg{\psi_1} + \dots + \deg{\psi_{k-1}} + k-1}\fm_n(\psi_1, \dots, \psi_{k-1}, \fm_2(\psi_k, \psi_{k+1}), \psi_{k+2}, \dots, \psi_n, \phi) \\
					=& (-1)^{\Sigma(\ora{\psi})}\half{\Phi(\ora{\psi})}\fm_2(e, \phi) - (-1)^{\deg{\phi} + \sigma(\psi_k, \psi_{k+1})}\frac{1}{2}\fm_n(\psi_1, \dots, \psi_{k-1}, \psi_k \bullet \psi_{k+1}, \psi_{k+2}, \dots, \psi_n, \phi) \\
					=&(-1)^{\Sigma(\ora{\psi})}\half{\Phi(\ora{\psi})}\phi - (-1)^{\Sigma(\ora{\psi}')}\half{\Phi(\ora{\psi}')}\phi \\
					=& 0 \quad \text{(by Lemma \ref{Lemma:PhiSigmaCD})}.
				\end{align*}
				\item $t(\phi)$ meets $t(\psi_k)$ on the right for some $1 < k < n$ with an interior morphism $\xi$ so that $\ora{\psi}^1 \deq (\psi_{k+1}, \dots \psi_n, \phi ; \xi) \in \Comp$ and $\ora{\psi}^2 \deq (\xi, \theta, \psi_2, \dots, \psi_k) \in \Disc$. In this case, $(\psi_k, \psi_{k+1}) \in \Con$.  Also, $\sigma(\phi, \theta) = \sigma(\psi_k, \psi_{k+1})$ and $\left<\phi, \theta\right> = \left<\psi_k, \psi_{k+1}\right>$.
				\begin{align*}
					&\fm_2(\fm_n(\psi_1, \dots, \psi_n), \phi) + (-1)^{\deg{\psi_1} + \dots + \deg{\psi_{k-1}} + k-1}\fm_n(\psi_1, \dots, \psi_{k-1}, \fm_2(\psi_k, \psi_{k+1}), \psi_{k+2}, \dots, \psi_n, \phi) \\
					&+(-1)^{\deg{\psi_1} + \dots + \deg{\psi_k} + k}\fm_{k+1}(\psi_1, \dots, \psi_k, \fm_{n-k+1}(\psi_{k+1}, \dots, \psi_n, \phi)) \\
					=& (-1)^{\Sigma(\ora{\psi})}\half{\Phi(\ora{\psi})}\fm_2(e, \phi) + (-1)^{\deg{\phi} + \sigma(\psi_k, \psi_{k+1})}\frac{1}{2}\fm_n(\psi_1, \dots, \psi_{k-1}, \psi_k \bullet \psi_{k+1}, \psi_{k+2}, \dots, \psi_n, \phi) \\
					&+ (-1)^{\Sigma(\ora{\psi}^1 ; \xi)}\half{\Phi(\ora{\psi}^1 ; \xi)}\fm_{k+1}(\psi_1, \dots, \psi_k, \xi) \\ 
					=& (-1)^{\Sigma(\ora{\psi})}\half{\Phi(\ora{\psi})}\phi + (-1)^{\Sigma(\ora{\psi}')}\half{\Phi(\ora{\psi}')}\phi  - (-1)^{\Sigma(\ora{\psi}^1 ; \xi) + \Sigma(\ora{\psi}^2) - \sigma(\phi, \theta)}\half{\Phi(\ora{\psi}^1 ; \xi) + \Phi(\ora{\psi}^2) - 1}\phi \\
					=& 0 \quad \text{(by Lemma \ref{Lemma:PhiSigmaCD})}.
				\end{align*}
			\end{enumerate}
		\end{enumerate}
	\item $(\psi_n, \phi)$ is a thick pair so that $(\psi_{n-1} ; (\psi_n, \phi)) \in \Thick$. Then, $\ora{\psi}^\vee \deq (\psi_1, \dots, \psi_{n-2}, \psi_{n-1}^\vee ; \phi)$ is a composition sequence.
	\begin{align*}
		&\fm_2(\fm_n(\psi_1, \dots, \psi_n), \phi) + (-1)^{\deg{\psi_1} + \dots + \deg{\psi_{n-2}} +  n-2}\fm_{n-1}(\psi_1, \dots, \psi_{n-2}, \fm_3(\psi_{n-1}, \psi_n, \phi)) \\
		=& (-1)^{\Sigma(\ora{\psi})}\half{\Phi(\ora{\psi})}\fm_2(e, \phi) - (-1)^{\sigma(\psi_n) +  \sigma(\psi_n, \phi)}\half{1 + \left<\psi_n, \phi\right>}\fm_{n-1}(\psi_1, \dots, \psi_{n-2}, \psi_{n-1}^\vee) \\
		=& (-1)^{\Sigma(\ora{\psi})}\half{\Phi(\ora{\psi})}\phi - (-1)^{\sigma(\psi_n, \phi) +  \sigma(\psi_n) + \Sigma(\ora{\psi}^\vee)}\half{1 + \left<\psi_n, \phi\right> + \Phi(\ora{\psi}^\vee)}\phi\\
    	=& 0 \quad \text{(by Lemma \ref{Lemma:PhiSigmaT})}.
	\end{align*}
	\end{enumerate}
	\item $(\psi_n, \psi_1) \in \Con$. This further has the following three subcases.
	\begin{enumerate}
		\item $(\psi_n, \phi) \in \Con$. This further has the following three subcases.
		\begin{enumerate}
			\item $\phi = \psi_1\bullet \theta$ for some nonzero $\theta$.
			\begin{align*}
				&\fm_2(\fm_n(\psi_1, \dots, \psi_n), \psi_1\bullet\theta) + (-1)^{\deg{\psi_1} + 1}\fm_2(\psi_1, \fm_n(\psi_2, \dots, \psi_n, \psi_1 \bullet \theta)) \\
				&+ (-1)^{\deg{\psi_1} + \dots, \deg{\psi_{n-1}} + n-1}\fm_n(\psi_1, \dots, \psi_{n-1}, \fm_2(\psi_n, \psi_1\bullet\theta))\\			
				=&(-1)^{\Sigma(\ora{\psi})}\half{\Phi(\ora{\psi})-1}\fm_2(e, \psi_1 \bullet \theta) + (-1)^{\deg{\psi_1} + 1 + \Sigma(\ora{\psi}) - \sigma(\psi_1, \theta)}\half{\Phi(\ora{\psi}) - \left<\psi_1, \theta\right>}\fm_2(\psi_1, \theta) \\
				&- (-1)^{\sigma(\psi_n, \psi_1)}\frac{1}{2}\fm_n(\psi_1, \dots, \psi_{n-1}, \psi_n \bullet \psi_1\bullet\theta)\\			
				=&(-1)^{\Sigma(\ora{\psi})}\half{\Phi(\ora{\psi})-1}\psi_1 \bullet \theta - (-1)^{\Sigma(\ora{\psi})}\half{\Phi(\ora{\psi})}\psi_1 \bullet \theta -(-1)^{\Sigma(\ora{\psi})}\half{\Phi(\ora{\psi})}\psi_1 \bullet \theta =0.
			\end{align*}
			\item $\phi = \psi_1$. This further has the following two subcases.
			\begin{enumerate}
				\item $(\psi_1, \psi_2) \notin \Con$.
				\begin{align*}
					&\fm_2(\fm_n(\psi_1, \dots, \psi_n), \psi_1) + (-1)^{\deg{\psi_1} + 1}\fm_2(\psi_1, \fm_n(\psi_2, \dots, \psi_n, \psi_1)) \\
					&+ (-1)^{\deg{\psi_1} + \dots + \deg{\psi_{n-1}} + n-1}\fm_n(\psi_1, \dots, \psi_{n-1}, \fm_2(\psi_n, \psi_1))\\
					=&(-1)^{\Sigma(\ora{\psi})}\half{\Phi(\ora{\psi})-1}\fm_2(e, \psi_1) + (-1)^{\deg{\psi_1} + 1 + \Sigma(\ora{\psi})}\half{\Phi(\ora{\psi})}\fm_2(\psi_1, e) \\
					&- (-1)^{\sigma(\psi_n, \psi_1)}\frac{1}{2}\fm_n(\psi_1, \dots, \psi_{n-1}, \psi_n \bullet \psi_1)\\					=&(-1)^{\Sigma(\ora{\psi})}\half{\Phi(\ora{\psi})-1}\psi_1 -(-1)^{\Sigma(\ora{\psi})}\half{\Phi(\ora{\psi})}\psi_1 -(-1)^{\Sigma(\ora{\psi})}\half{\Phi(\ora{\psi})}\psi_1 = 0.
				\end{align*}
				\item $(\psi_1, \psi_2) \in \Con$.
				\begin{align*}
					&\fm_2(\fm_n(\psi_1, \dots, \psi_n), \psi_1) + \fm_n(\fm_2(\psi_1, \psi_2), \psi_3, \dots, \psi_n, \psi_1) \\
					&+(-1)^{\deg{\psi_1} + 1}\fm_2(\psi_1, \fm_n(\psi_2, \dots, \psi_n, \psi_1)) + (-1)^{\deg{\psi_1} + \dots + \deg{\psi_{n-1}} + n-1}\fm_n(\psi_1, \dots, \psi_{n-1}, \fm_2(\psi_n, \psi_1))\\					=&(-1)^{\Sigma(\ora{\psi})}\half{\Phi(\ora{\psi})-1}\fm_2(e, \psi_1) + (-1)^{\deg{\psi_1} + \sigma(\psi_1, \psi_2)}\frac{1}{2}\fm_n(\psi_1 \bullet \psi_2, \psi_3, \dots, \psi_n, \psi_1) \\
					&+ (-1)^{\deg{\psi_1} + 1 + \Sigma(\ora{\psi})}\half{\Phi(\ora{\psi}) - 1}\fm_2(\psi_1, e) -(-1)^{\sigma(\psi_n, \psi_1)}\frac{1}{2}\fm_n(\psi_1, \dots, \psi_{n-1}, \psi_n \bullet \psi_1)\\					
					=&(-1)^{\Sigma(\ora{\psi})}\half{\Phi(\ora{\psi})-1}\psi_1 +(-1)^{\Sigma(\ora{\psi})}\half{\Phi(\ora{\psi})}\psi_1 \\
					&-(-1)^{\Sigma(\ora{\psi})}\half{\Phi(\ora{\psi})-1}\psi_1 - (-1)^{\Sigma(\ora{\psi})}\half{\Phi(\ora{\psi})}\psi_1=0.
				\end{align*}
			\end{enumerate}
			\item $\psi_1 = \phi \bullet \theta$ for some nonzero $\theta$. This further has the following three subcases.
			\begin{enumerate}
				\item $t(\phi)$ divides $\psi_k$ into $\psi_k^1 \bullet \psi_k^2$ for some $2 < k < n-1$ so that $\ora{\psi}^1 \deq (\theta, \psi_2, \dots, \psi_{k-1}, \psi_k^1)$, $\ora{\psi}^2 \deq (\psi_k^2, \psi_{k+1}, \dots, \psi_n, \phi) \in \Disc$.
				\begin{align*}
					&\fm_2(\fm_n(\psi_1, \dots, \psi_n), \phi) + (-1)^{\deg{\psi_1} + \dots + \deg{\psi_{n-1}} + n-1}\fm_n(\psi_1, \dots, \psi_{n-1}, \fm_2(\psi_n, \phi)) \\
					&+(-1)^{\deg{\psi_1} + \dots + \deg{\psi_{k-1}} + k-1}\fm_k(\psi_1, \dots, \psi_{k-1}, \fm_{n-k+2}(\psi_k, \dots, \psi_n, \phi)) \\
					=&(-1)^{\Sigma(\ora{\psi})}\half{\Phi(\ora{\psi}) - 1}\fm_2(e, \phi) -(-1)^{\sigma(\psi_n, \phi)}\half{\left<\psi_n, \phi\right>}\fm_n(\psi_1, \dots, \psi_{n-1}, \psi_n \bullet \phi) \\
					&-(-1)^{\deg{\phi} + \Sigma(\ora{\psi}^2) - \sigma(\psi_k^1, \psi_k^2)}\half{\Phi(\ora{\psi}^2) - \left<\psi_k^1, \psi_k^2\right>}\fm_k(\psi_1, \dots, \psi_{k-1}, \psi_k^1) \\
					=&(-1)^{\Sigma(\ora{\psi})}\half{\Phi(\ora{\psi}) - 1}\phi - (-1)^{\Sigma(\ora{\psi})}\half{\Phi(\ora{\psi})}\phi \\ 
					&-(-1)^{\Sigma(\ora{\psi}^2) + \Sigma(\ora{\psi}^1) - \sigma(\psi_k^1, \psi_k^2) - \sigma(\phi, \theta)}\half{\Phi(\ora{\psi}^1) + \Phi(\ora{\psi}^2) - \left<\psi_k^1, \psi_k^2\right> - \left<\phi, \theta\right>}\phi \\
					=& 0 \quad \text{(by Lemma \ref{Lemma:PhiSigmaDD})}.
				\end{align*}
				\item $t(\phi)$ meets $t(\psi_k)$ on the left for some $1 < k < n$ with an interior morphism $\xi$ so that $(\theta, \psi_2, \dots, \psi_k ; \xi) \in \Comp$ and $(\xi, \psi_{k+1}, \dots, \psi_n, \phi) \in \Disc$. Then, $(\psi_k, \psi_{k+1}) \in \Con$ and $\ora{\psi}' \deq (\phi, \theta, \psi_2, \dots, \psi_{k-1}, \psi_k \bullet \psi_{k+1}, \psi_{k+2}, \dots, \psi_n) \in \Disc$. Also, $\sigma(\phi, \theta) = \sigma(\psi_k, \psi_{k+1})$ and $\left<\phi, \theta\right> = \left<\psi_k, \psi_{k+1}\right>$.
				\begin{align*}
					&\fm_2(\fm_n(\psi_1, \dots, \psi_n), \phi) +  (-1)^{\deg{\psi_1} + \dots + \deg{\psi_{n-1}} + n-1}\fm_n(\psi_1, \dots, \psi_{n-1}, \fm_2(\psi_n, \phi)) \\
					&+ (-1)^{\deg{\psi_1} + \dots + \deg{\psi_{k-1}} + k-1}\fm_n(\psi_1, \dots, \psi_{k-1}, \fm_2(\psi_k, \psi_{k+1}), \psi_{k+2}, \dots, \psi_n, \phi) \\
					=& (-1)^{\Sigma(\ora{\psi})}\half{\Phi(\ora{\psi})-1}\fm_2(e, \phi) -(-1)^{\sigma(\psi_n, \phi)}\half{\left<\psi_n, \phi\right>}\fm_n(\psi_1, \dots, \psi_{n-1}, \psi_n \bullet \phi)\\
					&- (-1)^{\deg{\phi} + \sigma(\psi_k, \psi_{k+1})}\frac{1}{2}\fm_n(\psi_1, \dots, \psi_{k-1}, \psi_k \bullet \psi_{k+1}, \psi_{k+2}, \dots, \psi_n, \phi) \\
					=&(-1)^{\Sigma(\ora{\psi})}\half{\Phi(\ora{\psi}) - 1}\phi - (-1)^{\Sigma(\ora{\psi})}\half{\Phi(\ora{\psi}) }\phi - (-1)^{\Sigma(\ora{\psi}')}\half{\Phi(\ora{\psi}')}\phi \\
					=& 0 \quad \text{(by Lemma \ref{Lemma:PhiSigmaCD})}.
				\end{align*}
				\item $t(\phi)$ meets $t(\psi_k)$ on the right for some $1 < k < n$ with an interior morphism $\xi$ so that $\ora{\psi}^1 \deq (\psi_{k+1}, \dots \psi_n, \phi ; \xi) \in \Comp$ and $\ora{\psi}^2 \deq (\xi, \theta, \psi_2, \dots, \psi_k) \in \Disc$. Then, $(\psi_k, \psi_{k+1}) \in \Con$. Also, $\sigma(\phi, \theta) = \sigma(\psi_k, \psi_{k+1})$ and $\left<\phi, \theta\right> = \left<\psi_k, \psi_{k+1}\right>$.
				\begin{align*}
					&\fm_2(\fm_n(\psi_1, \dots, \psi_n), \phi) + (-1)^{\deg{\psi_1} + \dots + \deg{\psi_{n-1}} + n-1}\fm_n(\psi_1, \dots, \psi_{n-1}, \fm_2(\psi_n, \phi)) \\
					&+(-1)^{\deg{\psi_1} + \dots + \deg{\psi_{k-1}} + k-1}\fm_n(\psi_1, \dots, \psi_{k-1}, \fm_2(\psi_k, \psi_{k+1}), \psi_{k+2}, \dots, \psi_n, \phi) \\
					&+(-1)^{\deg{\psi_1} + \dots + \deg{\psi_k} + k}\fm_{k+1}(\psi_1, \dots, \psi_k, \fm_{n-k+1}(\psi_{k+1}, \dots, \psi_n, \phi)) \\
					=& (-1)^{\Sigma(\ora{\psi})}\half{\Phi(\ora{\psi}) - 1}\fm_2(e, \phi) -(-1)^{\sigma(\psi_n, \phi)}\half{\left<\psi_n, \phi\right>}\fm_n(\psi_1, \dots, \psi_{n-1}, \psi_n \bullet \phi) \\
					&+(-1)^{\deg{\phi} + \sigma(\psi_k, \psi_{k+1})}\frac{1}{2}\fm_n(\psi_1, \dots, \psi_{k-1}, \psi_k \bullet \psi_{k+1}, \psi_{k+2}, \dots, \psi_n, \phi) \\
					& - (-1)^{\deg{\phi} + \Sigma(\ora{\psi}^1 ; \xi)}\half{\Phi(\ora{\psi}^1 ; \xi)}\fm_{k+1}(\psi_1, \dots, \psi_k, \xi) \\ 
					=& (-1)^{\Sigma(\ora{\psi})}\half{\Phi(\ora{\psi}) - 1}\phi  - (-1)^{\Sigma(\ora{\psi})}\half{\Phi(\ora{\psi})}\phi \\
					&+ (-1)^{\Sigma(\ora{\psi}')}\half{\Phi(\ora{\psi}')}\phi - (-1)^{\Sigma(\ora{\psi}^1 ; \xi) + \Sigma(\ora{\psi}^2) - \sigma(\phi, \theta)}\half{\Phi(\ora{\psi}^1 ; \xi) + \Phi(\ora{\psi}^2) - 1}\phi \\
					=& 0 \quad \text{(by Lemma \ref{Lemma:PhiSigmaCD})}.
				\end{align*}
			\end{enumerate}
		\end{enumerate}
		\item $(\psi_n, \phi) \notin \Con$. Then, $t(\phi)$ divides $\psi_k$ into $\psi_k^1 \bullet \psi_k^2$ for some $1 < k < n$ so that $\ora{\psi}^1 \deq (\psi_1, \dots, \psi_{k-1}, \psi_k^1 ; \phi) \in \Comp$ and $\ora{\psi}^2 \deq (\phi, \psi_k^2, \psi_{k+1}, \dots, \psi_n) \in \Disc$.
		\begin{align*}
			&\fm_2(\fm_n(\psi_1, \dots, \psi_n), \phi) + (-1)^{\deg{\psi_1} + \dots + \deg{\psi_{n-1}} + n - 1}\fm_n(\psi_1, \dots, \psi_{n-1}, \fm_2(\psi_n, \phi)) \\
			&+ (-1)^{\deg{\psi_1} + \dots + \deg{\psi_{k-1}} + k -1}\fm_k(\psi_1, \dots, \psi_{k-1}, \fm_{n-k+2}(\psi_k, \dots, \psi_n, \phi)) \\
			=&(-1)^{\Sigma(\ora{\psi})}\half{\Phi(\ora{\psi}) - 1}\fm_2(e, \phi)  -(-1)^{\sigma(\psi_n, \phi)}\half{\left<\psi_n, \phi\right>}\fm_n(\psi_1, \dots, \psi_{n-1}, \psi_n \bullet \phi) \\
			&-(-1)^{\Sigma(\ora{\psi}^2) - \sigma(\psi_k^1, \psi_k^2)}\half{\Phi(\ora{\psi}^2) - \left<\psi_k^1, \psi_k^2\right>}\fm_k(\psi_1, \dots, \psi_{k-1}, \psi_k^1)\\
			=&(-1)^{\Sigma(\ora{\psi})}\half{\Phi(\ora{\psi}) - 1}\phi  -(-1)^{\Sigma(\ora{\psi})}\half{\Phi(\ora{\psi})}\phi \\
			&-(-1)^{\Sigma(\ora{\psi}^1 ; \phi) + \Sigma(\ora{\psi}^2) - \sigma(\psi_k^1, \psi_k^2)}\half{\Phi(\ora{\psi}^1 ; \phi) + \Phi(\ora{\psi}^2) - \left<\psi_k^1, \psi_k^2\right>}\phi \\
			=& 0 \quad \text{(by Lemma \ref{Lemma:PhiSigmaCD})}.
		\end{align*}
	\item $(\psi_n, \phi)$ is a thick pair and $(\psi_{n-1} ; (\psi_n, \phi)) \in \Thick$. Then, $\phi = \psi_1 \bullet \theta$ for some nonzero $\theta$ and$\ora{\psi}^\vee \deq (\psi_1, \dots, \psi_{n-2}, \psi_{n-1}^\vee ; \phi)$ is a composition sequence. Note that $\left<\psi_n, \phi\right> = 1$.
\begin{align*}
&\fm_2(\fm_n(\psi_1, \dots, \psi_n), \phi) + (-1)^{\deg{\psi_1} + 1}\fm_2(\psi_1, \fm_n(\psi_2, \dots, \psi_n, \phi)) \\
&+ (-1)^{\deg{\psi_1} + \dots + \deg{\psi_{n-2}} + n-2}\fm_{n-1}(\psi_1, \dots, \psi_{n-2}, \fm_3(\psi_{n-1}, \psi_n, \phi)) \\
=& (-1)^{\Sigma(\ora{\psi})}\half{\Phi(\ora{\psi}) - 1}\fm_2(e, \phi) + (-1)^{\deg{\psi_1} + 1 + \Sigma(\ora{\psi}) - \sigma(\psi_1, \theta)}\half{\Phi(\ora{\psi}) - \left<\psi_1, \theta\right>}\fm_2(\psi_1, \theta) \\
&- (-1)^{\sigma(\psi_n) + \sigma(\psi_n, \phi)}\half{1 + \left<\psi_n, \phi\right>}\fm_{n-1}(\psi_1, \dots, \psi_{n-2}, \psi_{n-1}^\vee) \\
=& (-1)^{\Sigma(\ora{\psi})}\half{\Phi(\ora{\psi}) - 1}\phi - (-1)^{\Sigma(\ora{\psi})}\half{\Phi(\ora{\psi})}\phi - (-1)^{\sigma(\psi_n) + \sigma(\psi_n, \phi) + \Sigma(\ora{\psi}^\vee ; \phi)}\half{1 + \left<\psi_n, \phi\right> + \Phi(\ora{\psi}^\vee ; \phi)}\phi \\
=& 0 \quad \text{(by Lemma \ref{Lemma:PhiSigmaT})}.
\end{align*}
	\end{enumerate}
\end{enumerate}

The case of $(e, \phi)$ is symmetric and omitted.

\subsubsection{Concatenable pair}\label{Appendix:ConcatenablePair}

Here, we deal with the case when the outer-input is a concatenable pair $\ora{\phi} = (\phi_1, \phi_2)$. 
From symmetry, we may assume that the value of  the inner input $\ora{\psi}$ is $\phi_1$.

\paragraph{Concatenable pair}
Suppose that $\ora{\psi} = (\psi_1, \psi_2)$ is a concatenable pair with $\phi_1 = \psi_1 \bullet \psi_2$.
\begin{align*}
	&\fm_2(\fm_2(\psi_1, \psi_2), \phi_2) + (-1)^{\deg{\psi_1} + 1}\fm_2(\psi_1, \fm_2(\psi_2, \phi_2)) = 0.
\end{align*}

\paragraph{Disc sequence}
Suppose that $\ora{\psi} = (\psi_1, \dots, \psi_n)$ is a disc sequence with $(\phi_1, \psi_1) \in \Con$. Then, there are following cases.
\begin{enumerate}
	\item $(\psi_n, \psi_1) \notin \Con$. Then, $(\psi_n, \phi_2) \notin \Con$. In particular, $(\psi_n, \phi_2)$ is not a thick pair. This further has the following three subcases.
	\begin{enumerate}
		\item $\phi_2 = \psi_1 \bullet \theta$, for some nonzero $\theta$.
		$$\fm_2(\fm_n(\phi_1\bullet\psi_1, \psi_2, \dots, \psi_n), \psi_1 \bullet \theta) + (-1)^{\deg{\phi_1} + \deg{\psi_1} + 1}\fm_2(\phi_1\bullet \psi_1, \fm_n(\psi_2, \dots, \psi_n, \psi_1 \bullet \theta)) = 0.$$
		\item $\phi_2 = \psi_1$. This further has the following two subcases.
		\begin{enumerate}
			\item $(\psi_1, \psi_2) \notin \Con$.
			$$\fm_2(\fm_n(\phi_1 \bullet \psi_1, \psi_2, \dots, \psi_n), \psi_1) + (-1)^{\deg{\phi_1} + \deg{\psi_1} + 1}\fm_2(\phi_1 \bullet \psi_1, \fm_n(\psi_2, \dots, \psi_n, \psi_1)) = 0.$$
			\item $(\psi_1, \psi_2) \in \Con$.
			\begin{align*}
				&\fm_2(\fm_n(\phi_1 \bullet \psi_1, \psi_2, \dots, \psi_n), \psi_1) + (-1)^{\deg{\phi_1} + \deg{\psi_1} + 1}\fm_2(\phi_1 \bullet \psi_1, \fm_2(\psi_2, \dots, \psi_n, \psi_1)) \\
				&+ \fm_n(\fm_2(\phi_1 \bullet \psi_1, \psi_2), \psi_3, \dots, \psi_n, \psi_1) = 0.
			\end{align*}
		\end{enumerate}
		\item $\psi_1 = \phi_2 \bullet \theta$, for some nonzero $\theta$. This further has the following three subcases.
		\begin{enumerate}
			\item $t(\phi_2)$ divides $\psi_k$ into $\psi_k^1 \bullet \psi_k^2$ for some $2 < k < n$ so that $(\theta, \psi_2, \dots, \psi_k^1)$, $(\psi_k^2, \dots, \psi_n, \phi_2) \in \Disc$.
			\begin{align*}
				&\fm_2(\fm_n(\phi_1 \bullet \psi_1, \psi_2, \dots, \psi_n), \phi_2) \\
				&+ (-1)^{\deg{\phi_1} + \deg{\psi_1} + \deg{\psi_2} + \dots + \deg{\psi_{k-1}} + k-1}\fm_k(\phi_1 \bullet \psi_1, \psi_2, \dots, \psi_{k-1}, \fm_{n-k+2}(\psi_k, \dots, \psi_n, \phi_2)) \\
				=& 0 \quad \text{(by Lemma \ref{Lemma:PhiSigmaDD})}.
			\end{align*}
			\item $t(\phi_2)$ meets $t(\psi_k)$ on the left for some $1 < k < n$ with an interior morphism $\xi$ so that $(\theta, \psi_2, \dots, \psi_k ; \xi) \in \Comp$ and $(\xi, \psi_{k+1}, \dots, \psi_k, \phi_2) \in \Disc$. Then, $(\psi_k, \psi_{k+1}) \in \Con$ and \\$(\theta, \psi_2, \dots, \psi_{k-1}, \psi_k\bullet \psi_{k+1}, \psi_{k+2}, \dots, \psi_n, \phi_2) \in \Disc$.
			\begin{align*}
				&\fm_2(\fm_n(\phi_1 \bullet \psi_1, \psi_2, \dots, \psi_n), \phi_2) \\
				&+ (-1)^{\deg{\phi_1} + \deg{\psi_1} + \deg{\psi_2} + \dots + \deg{\psi_{k-1}} + k-1}\fm_n(\phi_1 \bullet \psi_1, \psi_2, \dots, \psi_{k-1}, \fm_2(\psi_k, \psi_{k+1}), \psi_{k+2}, \dots, \psi_n, \phi_2) \\
				=& 0 \quad \text{(by Lemma \ref{Lemma:PhiSigmaCD})}.
			\end{align*}
			\item $t(\phi_2)$ meets $t(\psi_k)$ on the right for some $1 < k < n$ with an interior morphism $\xi$ so that $(\theta, \psi_2, \dots, \psi_k, \xi) \in \Disc$ and $(\psi_{k+1}, \dots, \psi_k, \phi_2 ; \xi) \in \Comp$. Then, $(\psi_k, \psi_{k+1}) \in \Con$ and \\$(\theta, \psi_2, \dots, \psi_{k-1}, \psi_k\bullet \psi_{k+1}, \psi_{k+2}, \dots, \psi_n, \phi_2) \in \Disc$.
			\begin{align*}
				&\fm_2(\fm_n(\phi_1 \bullet \psi_1, \psi_2, \dots, \psi_n), \phi_2) \\
				&+ (-1)^{\deg{\phi_1} + \deg{\psi_1} + \deg{\psi_2} + \dots + \deg{\psi_{k-1}} + k-1}\fm_n(\phi_1 \bullet \psi_1, \psi_2, \dots, \psi_{k-1}, \fm_2(\psi_k, \psi_{k+1}), \psi_{k+2}, \dots, \psi_n, \phi_2) \\
				&+ (-1)^{\deg{\phi_1} + \deg{\psi_1} + \deg{\psi_2} + \dots + \deg{\psi_{k-1}} + \deg{\psi_k} + k}\fm_{k+1}(\phi_1 \bullet \psi_1, \psi_2, \dots, \psi_k, \fm_{n - k + 1}(\psi_{k+1}, \psi_{k+2}, \dots, \psi_n, \phi_2)) \\
				=& 0 \quad \text{(by Lemma \ref{Lemma:PhiSigmaCD})}.
			\end{align*}
		\end{enumerate}
	\end{enumerate}
	\item $(\psi_n, \psi_1) \in \Con$. This case is the same with the previous case except the following one case.
	\begin{enumerate}
		\item $(\psi_n, \phi_2)$ is a thick pair so that $(\psi_{n-1} ; (\psi_n, \phi_2)) \in \Thick$. Then, $\phi_2 = \psi_1 \bullet \theta$ for some nonzero $\theta$.
		\begin{align*}
			&\fm_2(\fm_n(\phi_1 \bullet \psi_1, \psi_2, \dots, \psi_n), \phi_2) + (-1)^{\deg{\phi_1} + \deg{\psi_1} + 1}\fm_2(\phi_1 \bullet \psi_1, \fm_n(\psi_2, \dots, \psi_n, \phi_2)) = 0.
		\end{align*}
	\end{enumerate}
\end{enumerate}

Now suppose that $\ora{\psi} = (\psi_1, \dots, \psi_n)$ is a disc sequence with $(\psi_n, \phi_1) \in \Con$. Then, there is only one case.
\begin{align*}
	&\fm_2(\fm_n(\psi_1, \dots, \psi_n \bullet \phi_1), \phi_2) + (-1)^{\deg{\psi_1} + \dots + \deg{\psi_{n-1}} + n-1}\fm_n(\psi_1, \dots, \psi_{n-1}, \fm_2(\psi_n \bullet \phi_1, \phi_2)) = 0.
\end{align*}

\paragraph{Composition sequence}

Suppose that $\ora{\psi} =(\psi_1, \dots, \psi_n ; \phi_1)$ is a composition sequence. Then, there are the following cases.
\begin{enumerate}
	\item $t(\phi_2)$ divides $\psi_k$ into $\psi_k^1 \bullet \psi_k^2$ for some $1 < k < n$ so that $(\psi_1, \dots, \psi_{k-1}, \psi_k^1 ; \phi_1 \bullet \phi_2) \in \Comp$ and $(\psi_k^2, \psi_{k+1}, \dots, \psi_n, \phi_2) \in \Disc$.
	\begin{align*}
		&\fm_2(\fm_n(\psi_1, \dots, \psi_n), \phi_2) + (-1)^{\deg{\psi_1} + \dots + \deg{\psi_{k-1}} + k-1}\fm_k(\psi_1, \dots, \psi_{k-1}, \fm_{n-k+2}(\psi_k, \dots, \psi_n, \phi_2)) \\
		&= 0 \quad \text{(by Lemma \ref{Lemma:PhiSigmaDCL})}.
	\end{align*}
	\item $t(\phi_2)$ meets $t(\psi_k)$ on the left for some $1 \leq k <n$ with an interior morphism $\xi$ so that $(\psi_{k+1}, \dots, \psi_n, \phi_2, \xi) \in \Disc$. In this case, $(\psi_k, \psi_{k+1}) \in \Con$. Also, $(\psi_1, \dots, \psi_{k-1}, \psi_k \bullet \psi_{k+1}, \psi_{k+2}, \dots, \psi_n, \phi_2 ; \phi_1 \bullet \phi_2) \in \Comp$.
	\begin{align*}
		&\fm_2(\fm_n(\psi_1, \dots, \psi_n), \phi_2) + (-1)^{\deg{\psi_1} + \dots + \deg{\psi_{k-1}} + k - 1}\fm_n(\psi_1, \dots, \psi_{k-1}, \fm_2(\psi_k, \psi_{k+1}), \psi_{k+2}, \dots, \psi_n, \phi_2) \\
		&= 0 \quad \text{(by Lemma \ref{Lemma:PhiSigmaXDL})}.
	\end{align*}
	\item $t(\phi_2)$ meets $t(\psi_k)$ on the right for some $1 \leq k < n$ with an interior morphism $\xi$ so that $(\psi_1, \dots, \psi_k, \xi) \in \Disc$ and $(\psi_{k+1}, \dots, \psi_n, \phi_2 ; \xi) \in \Comp$. In this case, $(\psi_k, \psi_{k+1}) \in \Con$.\\ Also, $(\psi_1, \dots, \psi_{k-1}, \psi_k \bullet \psi_{k+1}, \psi_{k+2}, \dots, \phi_2 ; \phi_1 \bullet \phi_2) \in \Comp$.
	\begin{align*}
		&\fm_2(\fm_n(\psi_1, \dots, \psi_n), \phi_2) + (-1)^{\deg{\psi_1} + \dots + \deg{\psi_{k-1}} + k - 1}\fm_n(\psi_1, \dots, \psi_{k-1}, \fm_2(\psi_k, \psi_{k+1}), \psi_{k+2}, \dots, \psi_n, \phi_2) \\
		&+ (-1)^{\deg{\psi_1} + \dots + \deg{\psi_k} + k}\fm_{k+1}(\psi_1, \dots, \psi_k, \fm_{n-k+1}(\psi_{k+1}, \dots, \psi_n, \phi_2)) \\
		=& 0 \quad \text{(by Lemma \ref{Lemma:PhiSigmaCCL})}.
	\end{align*}
\end{enumerate}

\paragraph{Thick triple}
Suppose that $\ora{\psi} = (\psi_1, \psi_2)$ is a thick pair. Then, there are the following cases.
\begin{enumerate}
	\item $(\phi_1^\wedge ; (\psi_1, \psi_2))$ is a thick triple. Then, $((\psi_1, \psi_2) ; \phi_2)$ is also a thick triple.
	\begin{align*}
		&\fm_2(\fm_3(\phi_1^\wedge, \psi_1, \psi_2), \phi_2) + (-1)^{\deg{\phi_1^\wedge} + 1}\fm_2(\phi_1^\wedge, \fm_3(\psi_1, \psi_2, \phi_2)) = 0.
	\end{align*}
	\item $((\psi_1, \psi_2) ; {}^\wedge\phi_1)$ is a thick triple. Then, $((\psi_1, \psi_2) ; {}^\wedge\phi_1 \bullet \phi_2)$ is also a thick triple. Note that ${}^\wedge \phi_1 \bullet \phi_2 = {}^\wedge(\phi_1 \bullet \phi_2)$.
	\begin{align*}
		&\fm_2(\fm_3(\psi_1, \psi_2, {}^\vee\phi_1), \phi_2) + \fm_3(\psi_1, \psi_2, \fm_2({}^\vee\phi_1, \phi_2)) = 0.
	\end{align*}
\end{enumerate}

\subsubsection{Disc sequence}\label{Appendix:DiscSequence}

Here, we deal with the case when the outer-input is a disc sequence $\ora{\phi} = (\phi_1, \dots, \phi_m)$. The cases of \\$(\theta \bullet \phi_1, \phi_2, \dots, \phi_m)$ and $(\phi_1, \dots, \phi_{m-1}, \phi_m \bullet \theta)$ are similar, so we only list the cases for $\ora{\phi}$.

\paragraph{Concatenable pair}

Suppose that $\phi_i$ has a decomposition $\phi_i^1 \bullet \phi_i^2$. As we have already seen the cases of $$(\phi_1^1, \phi_1^2, \phi_2, \dots, \phi_m), \quad (\phi_1, \dots, \phi_{m-1}, \phi_m^1, \phi_m^2),$$ in Section \ref{Appendix:ConcatenablePair}, we may assume $1 < i < m$. Then, there are the following cases.
\begin{enumerate}
	\item $t(\phi_i^1)$ divides $\phi_j$ into $\phi_j^1 \bullet \phi_j^2$ for some $j$ with $\deg{i-j} > 1$ so that $(\phi_i^2, \phi_{i+1}, \dots, \phi_{j-1}, \phi_j^1)$,\\ $(\phi_j^2, \phi_{j+1}, \dots, \phi_{i-1}, \phi_i^1) \in \Disc$. We may assume $i+1 < j < m$.
	\begin{align*}
		&\fm_m(\phi_1, \dots, \phi_{i-1}, \fm_2(\phi_i^1, \phi_i^2), \phi_{i+1}, \dots, \phi_m) \\
		&+ (-1)^{\deg{\phi_i^1} + 1}\fm_{m-j+i+1}(\phi_1, \dots, \phi_{i-1}, \phi_i^1, \fm_{j-i+1}(\phi_i^2, \phi_{i+1}, \dots, \phi_j), \phi_{j+1}, \dots, \phi_m) \\
		=& 0 \quad \text{(by Lemma \ref{Lemma:PhiSigmaDD})}.
	\end{align*}
	\item $t(\phi_i^1)$ meets $t(\phi_j)$ on the left, for some $j$ such that $j < i-1$ or $j > i$, with an interior morphism $\xi$ so that $(\phi_{j+1}, \dots, \phi_{i-1}, \phi_i^1, \xi) \in \Disc$ and $(\phi_i^2, \phi_{i+1}, \dots, \phi_j ; \xi) \in \Comp$. In this case, $(\phi_j, \phi_{j+1}) \in \Con$. Also, $(\phi_j \bullet \phi_{j+1}, \phi_{j+2}, \dots, \phi_{i-1}, \phi_i^1, \phi_i^2, \phi_{i+1}, \dots, \phi_{j-1}) \in \Disc$. We may assume $1 < j < m$.
	\begin{enumerate}
		\item $1 < j < i - 1$.
		\begin{align*}
			&\fm_m(\phi_1, \dots, \phi_{j-1}, \fm_2(\phi_j, \phi_{j+1}), \phi_{j+2}, \dots, \phi_{i-1}, \phi_i^1, \phi_i^2, \phi_{i+1}, \dots, \phi_m) \\
			&+(-1)^{\deg{\phi_j} + \dots + \deg{\phi_{i-1}} + i - j}\fm_m(\phi_1, \dots, \phi_{j-1}, \phi_j, \phi_{j+1}, \phi_{j+2}, \dots, \phi_{i-1}, \fm_2(\phi_i^1, \phi_i^2), \phi_{i+1}, \dots, \phi_m) \\
			=& 0 \quad \text{(by Lemma \ref{Lemma:PhiSigmaCD})}.
		\end{align*}		
		\item $i < j < m$.
		\begin{align*}
			&\fm_n(\phi_1, \dots, \phi_{i-1}, \fm_2(\phi_i^1, \phi_i^2), \phi_{i+1}, \dots, \phi_m) \\
			&+ (-1)^{\deg{\phi_i^1} + 1}\fm_{n-j+i+1}(\phi_1, \dots, \phi_{i-1}, \phi_i^1, \fm_{j - i +1}(\phi_i^2, \phi_{i+1}, \dots, \phi_j), \phi_{j+1}, \dots, \phi_m) \\
			&+(-1)^{\deg{\phi_i^1} + \deg{\phi_i^2} + \deg{\phi_{i+1}} + \dots + \deg{\phi_{j-1}} + j - i + 1}\fm_m(\phi_1, \dots, \phi_{i-1}, \phi_i^1, \phi_i^2, \phi_{i+1}, \dots, \phi_{j-1}, \fm_2(\phi_j, \phi_{j+1}), \phi_{j+2}, \dots, \phi_m) \\
			=& 0 \quad \text{(by Lemma \ref{Lemma:PhiSigmaCD})}.
		\end{align*}
	\end{enumerate} 
	\item $t(\phi_i^1)$ meets $t(\phi_j)$ on the right, for some $j$ such that $j<i-1$ or $j>i$, with an interior morphism $\xi$ so that $(\phi_{j+1}, \dots, \phi_{i-1}, \phi_i^1 ; \xi) \in \Comp$ and $(\phi_i^2, \phi_{i+1}, \dots, \phi_j, \xi) \in \Disc$. In this case, $(\phi_j, \phi_{j+1}) \in \Con$. Also, $(\phi_j \bullet \phi_{j+1}, \phi_{j+2}, \dots, \phi_{i-1}, \phi_i^1, \phi_i^2, \phi_{i+1}, \dots, \phi_{j-1}) \in \Disc$. We may assume $1 < j < m$.
	\begin{enumerate}
		\item $1 < j < i - 1$.
		\begin{align*}
			&\fm_m(\phi_1, \dots, \phi_{j-1}, \fm_2(\phi_j, \phi_{j+1}), \phi_{j+2}, \dots, \phi_{i-1}, \phi_i^1, \phi_i^2, \phi_{i+1}, \dots, \phi_m) \\
			&+(-1)^{\deg{\phi_j} + \dots + \deg{\phi_{i-1}} + i - j}\fm_m(\phi_1, \dots, \phi_{j-1}, \phi_j, \phi_{j+1}, \phi_{j+2}, \dots, \phi_{i-1}, \fm_2(\phi_i^1, \phi_i^2), \phi_{i+1}, \dots, \phi_m) \\
			&+(-1)^{\deg{\phi_j} + 1}\fm_{n-i+j+1}(\phi_1, \dots, \phi_j, \fm_{i-j+1}(\phi_{j+1}, \dots, \phi_{i-1}, \phi_i^1), \phi_i^2, \phi_{i+1}, \dots, \phi_m) \\
			=& 0 \quad \text{(by Lemma \ref{Lemma:PhiSigmaCD})}.
		\end{align*}		
		\item $i < j < m$.
		\begin{align*}
			&\fm_n(\phi_1, \dots, \phi_{i-1}, \fm_2(\phi_i^1, \phi_i^2), \phi_{i+1}, \dots, \phi_m) \\
			&+(-1)^{\deg{\phi_i^1} + \deg{\phi_i^2} + \deg{\phi_{i+1}} + \dots + \deg{\phi_{j-1}} + j - i + 1}\fm_m(\phi_1, \dots, \phi_{i-1}, \phi_i^1, \phi_i^2, \phi_{i+1}, \dots, \phi_{j-1}, \fm_2(\phi_j, \phi_{j+1}), \phi_{j+2}, \dots, \phi_m) \\
			=& 0 \quad \text{(by Lemma \ref{Lemma:PhiSigmaCD})}.
		\end{align*}
	\end{enumerate}
\end{enumerate}

\paragraph{Disc sequence}
Suppose that $\ora{\psi} = (\psi_1, \dots, \psi_n)$ is a disc sequence with $(\psi_n, \phi_i) \in \Con$. The case of $(\phi_i, \psi_1) \in \Con$ is symmetric. There are the following cases.
\begin{enumerate}
	\item $(\phi_{i-1}, \psi_1) \in \Con$. In this case, $(\phi_1, \dots, \phi_{i-2}, \phi_{i-1} \bullet \psi_1, \psi_2, \dots, \psi_{n-1}, \psi_n \bullet \phi_i, \phi_{i+1}, \dots, \phi_m) \in \Disc$. This further has the following two subcases.
	\begin{enumerate}
		\item $i=1$.
		\begin{align*}
			&\fm_m(\fm_n(\psi_1, \dots, \psi_{n-1}, \psi_n \bullet \phi_1), \phi_2, \dots, \phi_m) \\
			&+ (-1)^{\deg{\psi_1} + \dots + \deg{\psi_{n-1}} + n-1}\fm_n(\psi_1, \dots, \psi_{n-1}, \fm_m(\psi_n \bullet \phi_1, \phi_2, \dots, \phi_m)) \\
			=& 0 \quad \text{(by Lemma \ref{Lemma:PhiSigmaDD})}.
		\end{align*}
		\item $i\neq 1$.
		\begin{align*}
			&\fm_{n+m-2}(\phi_1, \dots, \phi_{i-2}, \fm_2(\phi_{i-1}, \psi_1), \psi_2, \dots, \psi_{n-1}, \psi_n \bullet \phi_i, \phi_{i+1}, \dots \phi_n) \\
			&+ (-1)^{\deg{\phi_{i-1}} + 1}\fm_n(\phi_1, \dots, \phi_{i-1}, \fm_m(\psi_1, \dots, \psi_{m-1}, \psi_m \bullet \phi_i), \phi_{i+1}, \dots, \phi_n) \\
			=& 0 \quad \text{(by Lemma \ref{Lemma:PhiSigmaDD})}.
		\end{align*}
	\end{enumerate}
	\item $(\phi_{i-1}, \psi_1) \notin \Con$. This further has the following three subcases.
	\begin{enumerate}
		\item $\phi_{i-1}$ is an interior morphism from $t(\phi_{i-1})$ to $s(\phi_i)$. Then, $t(\phi_{i-2})$ divides $\psi_k$ into $\psi_k^1 \bullet \psi_k^2$ for some $1 < k < n$ so that $(\psi_1, \dots, \psi_{k-1}, \psi_k^1, \phi_{i-1}), (\psi_k^2, \psi_{k+1}, \dots, \psi_{n-1}, \psi_n \bullet \phi_i, \phi_{i+1}, \dots, \phi_{i-2}) \in \Disc$ and $(\psi_k^2, \psi_{k+1}, \dots, \psi_n ; \phi_{i-1}) \in \Comp$.
		\begin{enumerate}
			\item $i=1$.
			\begin{align*}
				&\fm_m(\fm_n(\psi_1, \dots, \psi_{n-1}, \psi_n \bullet \phi_1), \phi_2, \dots, \phi_m) \\
				&+(-1)^{\deg{\psi_1} + \dots  + \deg{\psi_{k-1}} + k-1}\fm_{k+1}(\psi_1, \dots, \psi_{k-1}, \fm_{n + m - k - 1}(\psi_k, \dots, \psi_{n-1}, \psi_n \bullet \phi_1, \phi_2, \dots, \phi_{m-1}), \phi_m) \\
				=& 0 \quad \text{(by Lemma \ref{Lemma:PhiSigmaCD})}.
			\end{align*}
			\item $i=2$.
			\begin{align*}
				&\fm_{n+m-k-1}(\fm_{k+1}(\phi_1, \psi_1, \dots, \psi_k), \psi_{k+1}, \dots, \psi_{n-1}, \psi_n \bullet \phi_2, \phi_3, \dots, \phi_m) \\
				&+(-1)^{\deg{\phi_1} + 1}\fm_m(\phi_1, \fm_n(\psi_1, \dots, \psi_{n-1}, \psi_n \bullet \phi_2), \phi_3, \dots, \phi_m) \\
				&+(-1)^{\deg{\phi_1} + \deg{\psi_1} + \dots  + \deg{\psi_{k-1}} + k}\fm_{k+1}(\phi_1, \psi_1, \dots, \psi_{k-1}, \fm_{n + m - k - 1}(\psi_k, \dots, \psi_{n-1}, \psi_n \bullet \phi_2, \phi_3, \dots, \phi_m)) \\
				=& 0 \quad \text{(by Lemma \ref{Lemma:PhiSigmaCD})}.
			\end{align*}
			\item $3 \leq i \leq m$.
			\begin{align*}
				&\fm_{n+m-k-1}(\phi_1, \dots, \phi_{i-2}, \fm_{k+1}(\phi_{i-1}, \psi_1, \dots, \psi_k), \psi_{k+1}, \dots, \psi_{n-1}, \psi_n \bullet \phi_i, \phi_{i+1}, \dots, \phi_m) \\
				&+(-1)^{\deg{\phi_{i-1}} + 1}\fm_m(\phi_1, \dots, \phi_{i-2}, \phi_{i-1}, \fm_n(\psi_1, \dots, \psi_k, \psi_{k+1}, \dots, \psi_{n-1}, \psi_n \bullet \phi_i), \phi_{i+1}, \dots, \phi_m) \\
				=& 0 \quad \text{(by Lemma \ref{Lemma:PhiSigmaCD})}.
			\end{align*}
		\end{enumerate}
		\item $\psi_1$ is an interior morphism from $t(\psi_n)$ to $s(\psi_2)$. Then, $s(\psi_2)$ divides $\phi_k$ into $\phi_k^1 \bullet \phi_k^2$ for $k$ such that $k < i-1$ or $k > i$ so that $(\phi_k^1, \psi_2, \dots, \psi_{n-1}, \psi_n \bullet \phi_i, \phi_{i+1}, \dots, \phi_{k-1}) \in \Disc$, $(\phi_i, \dots, \phi_{k-1}, \phi_k^1 ; \phi_{i-1}) \in \Comp$, and $(\phi_k^2, \phi_{k+1}, \dots, \phi_{i-1}, \psi_1) \in \Disc$.
	\begin{enumerate}
	\item $i < k$.
	\begin{align*}
		&\fm_m(\phi_1, \dots, \phi_{i-2}, \phi_{i-1}, \fm_n(\psi_1, \dots, \psi_{n-1}, \psi_n \bullet \phi_i), \phi_{i+1}, \dots, \phi_m) \\
		&+(-1)^{\deg{\psi_1} + 1}\fm_{m-k+i+1}(\phi_1, \dots, \phi_{i-1}, \psi_1, \fm_{n+k-i-1}(\psi_2, \dots, \psi_{n-1}, \psi_n \bullet \phi_i, \phi_{i+1}, \dots, \phi_k), \phi_{k+1}, \dots, \phi_m) \\
		=& 0 \quad \text{(by Lemma \ref{Lemma:PhiSigmaCD})}.
	\end{align*}
	\item $k < i - 1$. Here, $\dagger \deq \deg{\phi_k} + \dots + \deg{\phi_{i-1}} + i - k$.
	\begin{align*}
		&(-1)^{\dagger}\fm_m(\phi_1, \dots, \phi_{i-2}, \phi_{i-1}, \fm_n(\psi_1, \dots, \psi_{n-1}, \psi_n \bullet \phi_i), \phi_{i+1}, \dots, \phi_m) \\
		&+\fm_{n+m-i+k-1}(\phi_1, \dots, \phi_{k-1}, \fm_{i-k+1}(\phi_k, \dots, \phi_{i-1}, \psi_1), \psi_2, \dots, \psi_{n-1}, \psi_n \bullet \phi_i, \phi_{i+1}, \dots, \phi_m) \\
		&+(-1)^{\dagger + \deg{\psi_1} + 1}\fm_{m-k+i+1}(\phi_1, \dots, \phi_{i-1}, \psi_1, \fm_{n+k-i-1}(\psi_2, \dots, \psi_{n-1}, \psi_n \bullet \phi_i, \phi_{i+1}, \dots, \phi_k), \phi_{k+1}, \dots, \phi_m) \\
		=& 0 \quad \text{(by Lemma \ref{Lemma:PhiSigmaCD})}.
	\end{align*}
	\end{enumerate}
		\item $(\phi_{i-2} ; (\phi_{i-1}, \psi_1)) \in \Thick$. Then, $((\phi_{i-1}, \psi_1) ; \psi_2) \in \Thick$.
		\begin{enumerate}
			\item $i=1$. This is the same with $1.(a)$.
			\item $i=2$. Then, $(\psi_2^\vee, \psi_3, \dots, \psi_{n-1}, \psi_n \bullet \phi_2, \phi_3, \dots, \phi_m) \in \Disc$.
			\begin{align*}
				&(-1)^{\deg{\phi_1} + 1}\fm_m(\phi_1, \fm_n(\psi_1, \dots, \psi_{n-1}, \psi_n \bullet \phi_2), \phi_3, \dots, \phi_m) \\
				&+\fm_{n+m-3}(\fm_3(\phi_1, \psi_1, \psi_2), \psi_3, \dots, \psi_{n-1}, \psi_n \bullet \phi_2, \phi_3, \dots, \phi_m) \\
				=& 0 \quad \text{(by Lemma \ref{Lemma:PhiSigmaDD})}.
			\end{align*}
			\item $3 \leq i \leq m$. Then, $(\phi_1, \dots, \phi_{i-2}, \psi_2^\vee, \dots, \psi_{n-1}, \psi_n \bullet \phi_i, \phi_{i+1}, \dots, \phi_m)$,\\ $(\phi_1, \dots, \phi_{i-3}, \phi_{i-2}^\vee, \psi_2, \dots, \psi_{n-1}, \psi_n \bullet \phi_i, \phi_{i+1}, \dots, \phi_m) \in \Disc$.
			\begin{align*}
				&(-1)^{\deg{\phi_{i-2}} + \deg{\phi_{i-1}}}\fm_m(\phi_1, \dots, \phi_{i-1}, \fm_n(\psi_1, \dots, \psi_{n-1}, \psi_n \bullet \phi_i), \phi_{i+1}, \dots, \phi_m) \\
				&+\fm_{n+m-3}(\phi_1, \dots, \phi_{i-3}, \fm_3(\phi_{i-2}, \phi_{i-1}, \psi_1), \psi_2, \dots, \psi_{n-1}, \psi_n \bullet \phi_i, \phi_{i+1}, \dots, \phi_m) \\
				&+(-1)^{\deg{\phi_{i-2}} + 1}\fm_{n+m-3}(\phi_1, \dots, \phi_{i-2}, \fm_3(\phi_{i-1}, \psi_1, \psi_2), \psi_3, \dots, \psi_{n-1}, \psi_n \bullet \phi_i, \phi_{i+1}, \dots, \phi_m) \\
				=& 0 \quad \text{(by Lemma \ref{Lemma:PhiSigmaDD})}.
			\end{align*}
		\end{enumerate}
	\end{enumerate}
\end{enumerate}

\paragraph{Composition sequence}
Suppose that $\ora{\psi} = (\psi_1, \dots, \psi_n ; \phi_i)$ is a composition sequence. Then, there are the following cases.
\begin{enumerate}
	\item $1 < i < m$. Then, $(\phi_{i-1}, \psi_1), (\psi_n, \phi_{i+1}) \in \Con$ and $(\phi_1, \dots, \phi_{i-2}, \phi_{i-1} \bullet \psi_1, \psi_2, \dots, \psi_n, \phi_{i+1}, \dots, \phi_m)$, $(\phi_1, \dots, \phi_{i-1}, \psi_1, \dots, \psi_{n-1}, \psi_n \bullet \phi_{i+1}, \phi_{i+2}, \dots, \phi_m) \in \Disc$.
	\begin{align*}
		&(-1)^{\deg{\phi_{i-1}} + 1}\fm_m(\phi_1, \dots, \phi_{i-1}, \fm_n(\psi_1, \dots, \psi_n), \phi_{i+1}, \dots, \phi_m) \\
		&+\fm_{m+n-2}(\phi_1, \dots, \fm_2(\phi_{i-1}, \psi_1), \psi_2, \dots, \psi_n, \phi_{i+1}, \dots, \phi_m) \\
		&+(-1)^{\deg{\phi_{i-1}} + \deg{\psi_1} + \dots + \deg{\psi_{n-1}} + n}\fm_{m+n-2}(\phi_1, \dots, \phi_{i-1}, \psi_1, \dots, \psi_{n-1}, \fm_2(\psi_n, \phi_{i+1}), \phi_{i+2}, \dots, \phi_m) \\
		=& 0 \quad \text{(by Lemma \ref{Lemma:PhiSigmaCD})}.
	\end{align*}
	\item $i=1$. Then, $(\psi_1, \dots, \psi_{n-1}, \psi_n \bullet \phi_1, \phi_2, \dots, \phi_m) \in \Disc$.
	\begin{align*}
		&\fm_m(\fm_n(\psi_1, \dots, \psi_n), \phi_2, \dots, \phi_m) +(-1)^{\deg{\psi_1} + \dots + \deg{\psi_{n-1}} + n-1}\fm_{n+m-2}(\psi_1, \dots, \psi_{n-1}, \fm_2(\psi_n, \phi_2), \phi_3, \dots, \phi_m) \\
		=& 0 \quad \text{(by Lemma \ref{Lemma:PhiSigmaCD})}.
	\end{align*}
	\item $i=m$. Then, $(\phi_1, \dots, \phi_{m-1}, \phi_{m-1} \bullet \psi_1, \psi_2, \dots, \psi_n) \in \Disc$.
	\begin{align*}
		&(-1)^{\deg{\phi_{m-1}} + 1}\fm_m(\phi_1, \dots, \phi_{m-1}, \fm_n(\psi_1, \dots, \psi_n)) +\fm_{n+m-2}(\phi_1, \dots, \phi_{m-2}, \fm_2(\phi_{m-1}, \psi_1), \psi_2, \dots, \psi_n) \\
		=& 0 \quad \text{(by Lemma \ref{Lemma:PhiSigmaCD})}.
	\end{align*}
\end{enumerate}

\paragraph{Thick triple}
Suppose that $(\phi_i^\wedge ; (\psi_1, \psi_2))$ is a thick triple. Then, $(\phi_i, \phi_{i+1}) \in \Con$. If not, $\phi_{i+1}$, $\psi_1$, and $\psi_2$ are all defined by the same interior marking, say $p$. Since $(\psi_1, \psi_2)$ is a thick pair, $p$ cannot define an interior morphism from $t(\psi_2)$. This contradicts the fact that $s(\phi_{i+1}) = t(\phi_i) = t(\psi_2)$. This proves that $\phi_i$ and $\phi_{i+1}$ are concatenable. Also, we have $((\psi_1, \psi_2) ; \phi_{i+1}) \in \Thick$. Moreover, $t(\psi_1)$ divides $\phi_k$ into $\phi_k^1 \bullet \phi_k^2$ for some $k\neq i, i+1$ so that $(\psi_2, \phi_{i+1}, \dots, \phi_{k-1}, \phi_k^1)$, $(\psi_1, \phi_k^2, \phi_{k+1}, \dots, \phi_{i-1}, \phi_i^\wedge) \in \Disc$. Then, there are the following cases.
\begin{enumerate}
	\item $i + 1< k$
	\begin{align*}
		&\fm_m(\phi_1, \dots, \phi_{i-1}, \fm_3(\phi_i^\wedge, \psi_1, \psi_2), \phi_{i+1}, \dots, \phi_m) \\
		&+(-1)^{\deg{\phi_i^\wedge} + 1}\fm_m(\phi_1, \dots, \phi_{i-1}, \phi_i^\wedge, \fm_3(\psi_1, \psi_2, \phi_{i+1}), \phi_{i+2}, \dots, \phi_m) \\
		&+(-1)^{\deg{\phi_i^\wedge} + \deg{\psi_1}}\fm_{m-k+i+2}(\phi_1, \dots, \phi_{i-1}, \phi_i^\wedge, \psi_1, \fm_{k-i+1}(\psi_2, \phi_{i+1}, \dots, \phi_k), \phi_{k+1}, \dots, \phi_m) \\
		=& 0 \quad \text{(by Lemma \ref{Lemma:PhiSigmaDD})}.
	\end{align*}
	\item $k < i$. This further has the following two subcases.
	\begin{enumerate}
		\item $i < m$.
		\begin{align*}
			&(-1)^{\deg{\phi_k} + \dots + \deg{\phi_{i-1}} + k - i + 1}\fm_m(\phi_1, \dots, \phi_{i-1}, \fm_3(\phi_i^\wedge, \psi_1, \psi_2), \phi_{i+1}, \dots, \phi_m) \\
			&+(-1)^{\deg{\phi_k} + \dots + \deg{\phi_{i-1}} + \deg{\phi_i^\wedge} + k - i + 2}\fm_m(\phi_1, \dots, \phi_{i-1}, \phi_i^\wedge, \fm_3(\psi_1, \psi_2, \phi_{i+1}), \ \dots, \phi_m) \\
			&+\fm_{m+k-i+1}(\phi_1, \dots, \phi_{k-1}, \fm_{i-k+2}(\phi_k, \dots, \phi_{i-1}, \phi_i^\wedge, \psi_1), \psi_2, \phi_{i+1}, \dots, \phi_m) \\
			=& 0 \quad \text{(by Lemma \ref{Lemma:PhiSigmaDD})}.
		\end{align*}
		\item $i = m$.
		\begin{align*}
			&(-1)^{\deg{\phi_k} + \dots + \deg{\phi_{m-1}} + k - m + 1}\fm_m(\phi_1, \dots, \phi_{m-1}, \fm_3(\phi_m^\wedge, \psi_1, \psi_2))\\
			&+\fm_{k+1}(\phi_1, \dots, \phi_{k-1}, \fm_{m-k+2}(\phi_k, \dots, \phi_{m-1}, \phi_m^\wedge, \psi_1), \psi_2) \\
			=& 0 \quad \text{(by Lemma \ref{Lemma:PhiSigmaDD})}.
		\end{align*}
	\end{enumerate}
\end{enumerate}
The case of $((\psi_1, \psi_2) ; {}^\wedge\phi_i)$ is symmetric. 

\subsubsection{Composition sequence}\label{Appendix:CompositionSequence}

Here, we deal with the case when the outer-input is a composition sequence $\ora{\phi} = (\phi_1, \dots, \phi_m ; \theta)$.

\paragraph{Concatenable pair}

Suppose that $\phi_i$ has a decomposition $\phi_i^1 \bullet \phi_i^2$. Then, there are the following cases.
\begin{enumerate}
	\item $1 < i < m$. This further has the following six subcases.
	\begin{enumerate}
		\item $t(\phi_i^1)$ divides $\phi_j$ into $\phi_j^1 \bullet \phi_j^2$ for some $j$ with $\deg{i-j} > 1$.
		\begin{enumerate}
			\item $j < i-1$. Then, $(\phi_j^2, \phi_{i+1}, \dots, \phi_{j-1}, \phi_i^1) \in \Disc$ and $(\phi_1, \dots, \phi_{i-1}, \phi_j^1, \phi_i^2, \phi_{j+1}, \dots, \phi_m ; \theta) \in \Comp$.
			\begin{align*}
				&(-1)^{\deg{\phi_j} + \dots + \deg{\phi_{i-1}} + i - j}\fm_m(\phi_1, \dots, \phi_{i-1}, \fm_2(\phi_i^1, \phi_i^2), \phi_{i+1}, \dots, \phi_m) \\
				&+\fm_{m-i+j+1}(\phi_1, \dots, \phi_{j-1}, \fm_{i-j+1}(\phi_j, \dots, \phi_{i-1}, \phi_i^1), \phi_i^2, \phi_{i+1}, \dots, \phi_m) \\
				=& 0 \quad \text{(by Lemma \ref{Lemma:PhiSigmaDCM})}.
			\end{align*}
			\item $j > i+1$. This is symmetric to the previous case.
		\end{enumerate}
		\item $t(\phi_i^1)$ meets $t(\phi_j)$ with an interior morphism $\xi$ so that $(\phi_j, \phi_{j+1}) \in \Con$ and the following hold.
		\begin{enumerate}
			\item $j < i-1$. In this case, $(\phi_1, \dots, \phi_{j-1}, \phi_j \bullet \phi_{j+1}, \phi_{j+2}, \dots, \phi_{i-1}, \phi_i^1, \phi_i^2, \phi_{i+1}, \dots, \phi_m ; \theta) \in \Comp$.
			\begin{enumerate}
				\item $\xi$ is from $t(\phi_i^1)$ to $t(\phi_j)$. Then, $(\phi_{j+1}, \dots, \phi_{i-1}, \phi_i^1, \xi) \in \Disc$.
				\begin{align*}
					&\fm_m(\phi_1, \dots, \phi_{j-1}, \fm_2(\phi_j, \phi_{j+1}), \phi_{j+2}, \dots, \phi_m) \\
					&+ (-1)^{\deg{\phi_j} + \deg{\phi_{j+1}} + \dots + \deg{\phi_{i-1}} + i - j}\fm_m(\phi_1, \phi_2, \dots, \phi_{i-1}, \fm_2(\phi_i^1, \phi_i^2), \phi_{i+1} \dots, \phi_m) \\
					=& 0 \quad \text{(by Lemma \ref{Lemma:PhiSigmaXDM})}.
				\end{align*}
				\item $\xi$ is from $t(\phi_j)$ to $t(\phi_i^1)$. Then, $(\phi_{j+1}, \dots, \phi_{i-1}, \phi_i^1 ; \xi) \in \Comp$ and \\$(\phi_1, \dots, \phi_j, \xi, \phi_i^2, \phi_{i+1}, \dots, \phi_m ; \theta) \in \Comp$.
				\begin{align*}
					&\fm_m(\phi_1, \dots, \phi_{j-1}, \fm_2(\phi_j, \phi_{j+1}), \phi_{j+2}, \dots, \phi_m) \\
					&+(-1)^{\deg{\phi_j} + 1}\fm_{m-i+j+2}(\phi_1, \dots, \phi_j, \fm_{i-j}(\phi_{j+1}, \dots, \phi_{i-1}, \phi_i^1), \phi_i^2, \phi_{i+1}, \dots, \phi_m) \\
					&+ (-1)^{\deg{\phi_j} + \deg{\phi_{j+1}} + \dots + \deg{\phi_{i-1}} + i - j}\fm_m(\phi_1, \phi_2, \dots, \phi_{i-1}, \fm_2(\phi_i^1, \phi_i^2), \phi_{i+1} \dots, \phi_m) \\
					=& 0 \quad \text{(by Lemma \ref{Lemma:PhiSigmaCCM})}.		
				\end{align*}
			\end{enumerate}
			\item $j > i$. This case is symmetric to the previous case.
		\end{enumerate}
		\item $t(\phi^1_i)$ meets $s(\theta$) on the left with an interior morphism $\xi$ so that $(\xi, \theta) \in \Con$. Then, $(\xi, \phi_1, \dots, \phi_{i-1}, \phi_i^1) \in \Disc$ and $(\phi_i^2, \phi_{i+1}, \dots, \phi_m ; \xi \bullet \theta) \in \Comp$.
		\begin{align*}
			&\fm_m(\phi_1, \dots, \phi_{i-1}, \fm_2(\phi_i^1, \phi_i^2), \phi_{i+1}, \dots, \phi_m) \\ 
			&+(-1)^{\deg{\phi_i^1} + 1}\fm_{i+1}(\phi_1, \dots, \phi_{i-1}, \phi_i^1, \fm_{m-i+1}(\phi_i^2, \phi_{i+1}, \dots, \phi_m)) \\ 
			=& 0  \quad \text{(by Lemma \ref{Lemma:PhiSigmaDCF})}.
		\end{align*}
		\item $t(\phi^1_i)$ meets $t(\theta)$ on the right with an interior morphism $\xi$ so that $(\theta, \xi) \in \Con$. This is symmetric to the previous case.
		\item $\theta$ has a decomposition $\phi_1 \bullet \eta$ and $t(\phi_i^1)$ meets $s(\phi_1)$ with an interior morphism $\xi$. \\Then, $(\phi_2, \dots, \phi_{i-1}, \phi_i^1 \bullet \phi_i^2, \phi_{i+1}, \dots, \phi_m ; \eta) \in \Comp$. This further has the following two subcases.
		\begin{enumerate}
			\item $\xi$ is from $t(\phi_i^1)$ to $s(\phi_1)$.
			\begin{align*}
				&(-1)^{\deg{\phi_2} + \dots + \deg{\phi_{i-1}} + i-2}\fm_m(\phi_1, \dots, \phi_{i-1}, \fm_2(\phi_i^1, \phi_i^2), \phi_{i+1}, \dots, \phi_m) \\
				&+\fm_2(\phi_1, \fm_m(\phi_2, \dots, \phi_{i-1}, \phi_i^1, \phi_i^2, \phi_{i+1}, \dots, \phi_m)) \\
				=& 0 \quad \text{(by Lemma \ref{Lemma:PhiSigmaXDF})}.
			\end{align*}
			\item $\xi$ is from $s(\phi_1)$ to $t(\phi_i^1)$. Then, $(\phi_1, \dots, \phi_{j-1}, \phi_i^1 ; \xi), (\xi, \phi_i^2, \phi_{i+1}, \dots, \phi_m ; \theta) \in \Comp$.
			\begin{align*}
				&(-1)^{\deg{\phi_1} + \dots + \deg{\phi_{i-1}} +  i - 1}\fm_m(\phi_1, \dots, \phi_{i-1}, \fm_2(\phi_i^1, \phi_i^2), \phi_{i+1}, \dots, \phi_m) \\
				&+\fm_{m-i+2}(\fm_i(\phi_1, \dots, \phi_{i-1}, \phi_i^1), \phi_i^2, \phi_{i+1}, \dots, \phi_m)\\
				&+(-1)^{\deg{\phi_1} + 1}\fm_2(\phi_1, \fm_m(\phi_2, \dots, \phi_{i-1}, \phi_i^1, \phi_i^2, \phi_{i+1}, \dots, \phi_m)) \\
				=& 0 \quad \text{(by Lemma \ref{Lemma:PhiSigmaCCF})}.
			\end{align*}
		\end{enumerate}
		\item $\theta$ has a decomposition $\eta \bullet \phi_m$. This case is symmetric to the previous case.
	\end{enumerate}
	\item $i=1$. This has three subcases.
	\begin{enumerate}
		\item $t(\phi_1^1)$ divides $\phi_k$ into $\phi_k^1 \bullet \phi_k^2$ for some $k>2$ so that $(\phi_1^2, \phi_2, \dots, \phi_{k-1}, \phi_k^1) \in \Disc$ and \\ $(\phi_1^1, \phi_k^2, \phi_{k+1}, \dots, \phi_m ; \theta) \in \Comp$. This is the same with the case 1.(a).ii.
		\item $t(\phi_1^1)$ meets $t(\phi_k)$ for some $1 < k < m$. This is the same with the case 1.(b).(ii).
		\item $\theta$ has a decomposition $\eta \bullet \phi_m$ and $t(\phi_1^1)$ meets $t(\phi_m)$. This is the same with the case 1.(f).
	\end{enumerate}
	\item $i=m$. This is symmetric to the previous case.
\end{enumerate}

\paragraph{Disc sequence}
Suppose that $\ora{\psi} = (\psi_1, \dots, \psi_n)$ is a disc sequence such that $(\psi_n, \phi_i) \in \Con$. Then, we have the following cases. The case of $(\phi_i, \psi_1) \in \Con$ is symmetric.
\begin{enumerate}
	\item $i=1$.
	\begin{enumerate}
		\item $\psi_1$ and $\theta$ are defined by the same interior marking. This further has the following three subcases.
		\begin{enumerate}
			\item $\theta = \psi_1 \bullet \xi$ for some nonzero $\xi$. Then, $(\psi_2, \dots, \psi_{n-1}, \psi_n \bullet \phi_1, \phi_2, \dots, \phi_m ; \xi) \in \Comp$.
			\begin{align*}
				&\fm_m(\fm_n(\psi_1, \dots, \psi_{n-1}, \psi_n \bullet \phi_1), \phi_2, \dots, \phi_m) \\
				&+(-1)^{\deg{\psi_1} + 1}\fm_2(\psi_1, \fm_{m+n-2}(\psi_2, \dots, \psi_{n-1}, \psi_n \bullet \phi_1, \phi_2, \dots, \phi_m)) \\
				=& 0 \quad \text{(by Lemma \ref{Lemma:PhiSigmaDCF})}.
			\end{align*}
			\item $\theta = \psi_1$. Then, $(\phi_m, \psi_2) \in \Con$ and $(\psi_2, \dots, \psi_{n-1}, \psi_n \bullet \phi_1, \phi_2, \dots, \phi_m) \in \Disc$.
			\begin{align*}
				&\fm_m(\fm_n(\psi_1, \dots, \psi_{n-1}, \psi_n \bullet \phi_1), \phi_2, \dots, \phi_m) \\
				&+(-1)^{\deg{\psi_1} + 1}\fm_2(\psi_1, \fm_{m+n-2}(\psi_2, \dots, \psi_{n-1}, \psi_n \bullet \phi_1, \phi_2, \dots, \phi_m)) \\
				=& 0 \quad \text{(by Lemma \ref{Lemma:PhiSigmaCD})}.
			\end{align*}
			\item $\psi_1 = \theta \bullet \xi$ for some nonzero $\xi$. Then, $t(\xi)$ divides $\phi_k$ into $\phi_k^1 \bullet \phi_k^2$ for some $1 < k < m$ so that $(\psi_2, \dots, \psi_{n-1}, \psi_n \bullet \phi_1, \phi_2, \dots, \phi_{k-1}, \phi_k^1) \in \Disc$ and $(\xi, \phi_k^2, \phi_{k+1}, \dots, \phi_m) \in \Disc$.
			\begin{align*}
				&\fm_m(\fm_n(\psi_1, \dots, \psi_{n-1}, \psi_n \bullet \phi_1), \phi_2, \dots, \phi_m) \\
				&+(-1)^{\deg{\psi_1} + 1}\fm_{m-k+2}(\psi_1, \fm_{n+k-2}(\psi_2, \dots, \psi_{n-1}, \psi_n \bullet \phi_1, \phi_2, \dots, \phi_k), \phi_{k+1}, \dots, \phi_m) \\
				=& 0 \quad \text{(by Lemma \ref{Lemma:PhiSigmaDCL}, \ref{Lemma:PhiSigmaCD})}.
			\end{align*}
		\end{enumerate}
		\item $s(\psi_2)$ meets $t(\psi_n)$ on the right with the interior morphism $\psi_1$. Then, $s(\psi_2)$ divides $\phi_k$ into $\phi_k^1 \bullet \phi_k^2$ for some $1 < k < m$ so that $(\psi_2, \dots, \psi_{n-1}, \psi_n \bullet \phi_1, \phi_2, \dots, \phi_{k-1}, \phi_k^1) \in \Disc$ and $(\psi_1, \phi_k^2, \phi_{k+1}, \dots, \phi_m ; \theta) \in \Comp$.
		\begin{align*}
			&\fm_m(\fm_n(\psi_1, \dots, \psi_{n-1}, \psi_n \bullet \phi_1), \phi_2, \dots, \phi_m) \\
			&+(-1)^{\deg{\psi_1} + 1}\fm_{m-k+2}(\psi_1, \fm_{n+k-2}(\psi_2, \dots, \psi_{n-1}, \psi_n \bullet \phi_1, \phi_2, \dots, \phi_k), \phi_{k+1}, \dots, \phi_m) \\
			=& 0 \quad \text{(by Lemma \ref{Lemma:PhiSigmaCD}, \ref{Lemma:PhiSigmaCCF})}.
		\end{align*}
		\item $t(\phi_m)$ meets $t(\psi_n)$ on the right with the interior morphism $\theta$. Then, $t(\phi_m)$ divides $\psi_k$ into $\psi_k^1 \bullet \psi_k^2$ for some $1 < k < n$ so that $(\psi_k^2, \psi_{k+1}, \dots, \psi_{n-1}, \psi_n \bullet \phi_1, \phi_2, \dots, \phi_m) \in \Disc$ and $(\psi_1, \dots, \psi_{k-1}, \psi_k^1 ; \theta) \in \Comp$.
		\begin{align*}
			&\fm_m(\fm_n(\psi_1, \dots, \psi_{n-1}, \psi_n \bullet \phi_1), \phi_2, \dots, \phi_m) \\
			&+(-1)^{\deg{\psi_1} + \dots + \deg{\psi_{k-1}} + k-1}\fm_k(\psi_1, \dots, \psi_{k-1}, \fm_{m+n-k}(\psi_k, \dots, \psi_{n-1}, \psi_n \bullet \phi_1, \phi_2, \dots, \phi_m)) \\
			=& 0 \quad \text{(by Lemma \ref{Lemma:PhiSigmaCD})}.
		\end{align*}
	\end{enumerate}
	\item $i > 1$. This further has the following four subcases.
	\begin{enumerate}
		\item $(\phi_{i-1}, \psi_1) \in \Con$. Then, $(\phi_1, \dots, \phi_{i-2}, \phi_{i-1} \bullet \psi_1, \psi_2, \dots, \psi_{n-1}, \psi_n \bullet \phi_i, \phi_{i+1}, \dots, \phi_m ; \theta) \in \Comp$.
		\begin{align*}
			&(-1)^{\deg{\phi_{i-1}} + 1}\fm_m(\phi_1, \dots, \phi_{i-1}, \fm_n(\psi_1, \dots, \psi_{n-1}, \psi_n \bullet \phi_i), \phi_{i+1}, \dots, \phi_m) \\
			&+\fm_{m+n-2}(\phi_1, \dots, \phi_{i-2}, \fm_2(\phi_{i-1}, \psi_1), \psi_2, \dots, \psi_{n-1}, \psi_n \bullet \phi_i, \phi_{i+1}, \dots, \phi_m) \\
			=& 0 \quad \text{(by Lemma \ref{Lemma:PhiSigmaDCM})}.
		\end{align*}
		\item $s(\psi_2)$ meets $t(\psi_n)$ on the right with the interior morphism $\psi_1$. This further has the following three subcases.
		\begin{enumerate}
			\item $s(\psi_2)$ divides $\phi_k$ into $\phi_k^1 \bullet \phi_k^2$ for some $k$ such that $k < i-1$ or $k> i$.
			\begin{enumerate}
				\item $k < i-1$. Then, $(\phi_1, \dots, \phi_{k-1}, \phi_k^1, \psi_2, \dots, \psi_{n-1}, \psi_n \bullet \phi_i, \phi_{i+1}, \dots, \phi_m ; \theta) \in \Comp$ and 
				
				$(\phi_k^2, \phi_{k+2}, \dots, \phi_{i-1}, \psi_1) \in \Disc$.
				\begin{align*}
					&(-1)^{\deg{\phi_k} + \dots + \deg{\phi_{i-1}} + i-k}\fm_m(\phi_1, \dots, \phi_{i-1}, \fm_n(\psi_1, \dots, \psi_{n-1}, \psi_n \bullet \phi_i), \phi_{i+1}, \dots, \phi_m) \\
					&+\fm_{m+n-i+k+1}(\phi_1, \dots, \phi_{k-1}, \fm_{i-k+1}(\phi_k, \dots, \phi_{i-1}, \psi_1), \psi_2, \dots, \psi_{n-1}, \psi_n \bullet \phi_i, \phi_{i+1}, \dots, \phi_m) \\
					=& 0 \quad \text{(by Lemma \ref{Lemma:PhiSigmaDCM})}.
				\end{align*}
				\item $k > i$. Then, $(\psi_2, \dots, \psi_{n-1}, \psi_n \bullet \phi_i, \phi_{i+1}, \dots, \phi_{k-1}, \phi_k^1) \in \Disc$ and 
				
				$(\phi_1, \dots, \phi_{i-1}, \psi_1, \phi_k^2, \phi_{k+1}, \dots, \phi_m ; \theta) \in \Comp$
				\begin{align*}
					&\fm_m(\phi_1, \dots, \phi_{i-1}, \fm_n(\psi_1, \dots, \psi_{n-1}, \psi_n \bullet \phi_i), \phi_{i+1}, \dots, \phi_m) \\
					&+(-1)^{\deg{\psi_1} + 1}\fm_{m - k + i + 1}(\phi_1, \dots, \phi_{i-1}, \psi_1, \fm_{n + k - i - 1}(\psi_2, \dots, \psi_{n-1}, \psi_n \bullet \phi_i, \phi_{i+1}, \dots, \phi_k), \phi_{k+1}, \dots, \phi_m) \\
					=& 0 \quad \text{(by Lemma \ref{Lemma:PhiSigmaDD}, \ref{Lemma:PhiSigmaDCM})}.
				\end{align*}
			\end{enumerate}
			\item $t(\psi_1)$ meets $s(\theta)$ on the left with an interior morphism $\xi$ so that $(\xi, \theta) \in \Con$. \\Then, $(\psi_2, \dots, \psi_{n-1}, \psi_n \bullet \phi_i, \phi_{i+1}, \dots, \phi_m ; \xi \bullet \theta) \in \Comp$ and $(\xi, \phi_1, \dots, \phi_{i-1}, \psi_1) \in \Disc$.
			\begin{align*}
				&\fm_m(\phi_1, \dots, \phi_{i-1}, \fm_n(\psi_1, \dots, \psi_{n-1}, \psi_n \bullet \phi_i, \phi_{i+1}, \dots, \phi_m)) \\
				&+(-1)^{\deg{\psi_1} + 1}\fm_{i+1}(\phi_1, \dots, \phi_{i-1}, \psi_1, \fm_{m+n-i-1}(\psi_2, \dots, \psi_{n-1}, \psi_n \bullet \phi_i, \phi_{i+1}, \dots, \phi_m)) \\
				=&0 \quad \text{(by Lemma \ref{Lemma:PhiSigmaXDF})}.
			\end{align*}
			\item $t(\psi_1)$ meets $t(\theta)$ on the right with an interior morphism $\xi$ so that $(\theta, \xi) \in \Con$. \\Then, $(\phi_1, \dots, \phi_{i-1}, \psi_1 ; \theta \bullet \xi) \in \Comp$ and $(\xi, \psi_2, \dots, \psi_{n-1}, \psi_n \bullet \phi_i, \phi_{i+1}, \dots, \phi_m) \in \Disc$.
			\begin{align*}
				&(-1)^{\deg{\phi_1} + \dots + \deg{\phi_{i-1}} + i - 1}\fm_m(\phi_1, \dots, \phi_{i-1}, \fm_n(\psi_1, \dots, \psi_{n-1}, \psi_n \bullet \phi_i), \phi_{i+1}, \dots, \phi_m) \\
				+&\fm_{m+n-i}(\fm_i(\phi_1, \dots, \phi_{i-1}, \psi_1), \psi_2, \dots, \psi_{n-1}, \psi_n \bullet \phi_i, \phi_{i+1}, \dots, \phi_m) \\
				=& 0 \quad \text{(by Lemma \ref{Lemma:PhiSigmaCCL})}.				
			\end{align*}
		\end{enumerate}
		\item $i=2$ and $s(\theta)$ meets $t(\psi_n)$ on the left with the interior morphism $\phi_1$. Then, $s(\theta)$ divides $\psi_k$ into $\psi_k^1 \bullet \psi_k^2$ for some $2 < k < n$ so that $(\phi_1, \psi_1, \dots, \psi_{k-1}, \psi_k^1) \in \Disc$ and\\ $(\psi_k^2, \psi_{k+1}, \dots, \psi_{n-1}, \psi_n \bullet \phi_2, \phi_3, \dots, \phi_m ; \theta) \in \Comp$.
		\begin{align*}
			&(-1)^{\deg{\phi_1} + 1}\fm_m(\phi_1, \fm_n(\psi_1, \dots, \psi_{n-1}, \psi_n \bullet \phi_2), \phi_3, \dots, \phi_m) \\
			+&\fm_{n+m-k-1}(\fm_{k+1}(\phi_1, \psi_1, \dots, \psi_k), \psi_{k+1}, \dots, \psi_{n-1}, \psi_n \bullet \phi_2, \phi_3, \dots, \phi_m) \\
			=& 0 \quad \text{(by Lemma \ref{Lemma:PhiSigmaCD}, \ref{Lemma:PhiSigmaCCL})}.
		\end{align*}
		\item $i > 2$ and $s(\phi_{i-2})$ meets $s(\phi_i)$ on the left with the interior morphism $\phi_{i-1}$ Then, $s(\phi_{i-2})$ divides $\psi_k$ into $\psi_k^1 \bullet \psi_k^2$ for some $2 < k < n$ so that $(\phi_{i-1}, \psi_1, \dots, \psi_{k-1}, \psi_k^1) \in \Disc$ and\\ $(\phi_1, \dots, \phi_{i-2}, \psi_k^2, \psi_{k+1}, \dots, \psi_{n-1}, \psi_n \bullet \phi_i, \phi_{i+1}, \dots, \phi_m ; \theta) \in \Comp$.
		\begin{align*}
			&(-1)^{\deg{\phi_{i-1}} + 1}\fm_m(\phi_1, \dots, \phi_{i-1}, \fm_n(\psi_1, \dots, \psi_{n-1}, \psi_n \bullet \phi_i), \phi_{i+1}, \dots, \phi_m) \\
			&+\fm_{n+m-k-1}(\phi_1, \dots, \phi_{i-2}, \fm_{k+1}(\phi_{i-1}, \psi_1, \dots, \psi_k), \psi_{k+1}, \dots, \psi_{n-1}, \psi_n \bullet \phi_i, \phi_{i+1}, \dots, \phi_m) \\
			=& 0 \quad \text{(by Lemma \ref{Lemma:PhiSigmaCD}, \ref{Lemma:PhiSigmaCCM})}.
		\end{align*}
		\end{enumerate}
		\item $(\phi_{i-1}, \psi_1)$ is a thick pair so that $((\phi_{i-1}, \psi_1) ; \psi_2) \in \Thick$. Note that $i \neq 2$ by thick condition. Then, $(\phi_1, \dots, \phi_{i-2}, \phi_{i-1}^\vee, \psi_1, \dots, \psi_{n-1}, \psi_n \bullet \phi_i, \phi_{i+1}, \dots, \phi_m ; \theta)$,\\ $(\phi_1, \dots, \phi_{i-1}, \psi_1^\vee, \dots, \psi_{n-1}, \psi_n \bullet \phi_i, \phi_{i+1}, \dots, \phi_m ; \theta) \in \Comp$.
		\begin{align*}
			&(-1)^{\deg{\phi_{i-2}} + \deg{\phi_{i-1}} + 2}\fm_m(\phi_1, \dots, \phi_{i-1}, \fm_n(\psi_1, \dots, \psi_{n-1}, \psi_n \bullet \phi_i), \phi_{i+1}, \dots, \phi_m) \\
			&+\fm_{m+n-3}(\phi_1, \dots, \phi_{i-3}, \fm_3(\phi_{i-2}, \phi_{i-1}, \psi_1), \psi_2, \dots, \psi_{n-1}, \psi_n \bullet \phi_i, \phi_{i+1}, \dots, \phi_m) \\
			&+(-1)^{\deg{\phi_{i-2}} +1}\fm_{m+n-3}(\phi_1, \dots, \phi_{i-2}, \fm_3(\phi_{i-1}, \psi_1, \psi_2), \psi_3, \dots, \psi_{n-1}, \psi_n \bullet \phi_i, \phi_{i+1}, \dots, \phi_m) \\
			=& 0 \quad \text{(by Lemma \ref{Lemma:PhiSigmaCCM}, \ref{Lemma:PhiSigmaT})}.
		\end{align*}
\end{enumerate}

\paragraph{Composition sequence}

Suppose that $\ora{\psi} = (\psi_1, \dots, \psi_n ; \phi_i)$ is a composition sequence. Then, there are the following cases.
\begin{enumerate}
	\item $1 < i < m$. Then, $(\phi_{i-1}, \psi_1), (\psi_n, \phi_{i+1}) \in \Con$ and $(\phi_1, \dots, \phi_{i-2}, \phi_{i-1} \bullet \psi_1, \psi_2, \dots, \psi_n, \phi_{i+1}, \dots, \phi_m ; \theta)$, $(\phi_1, \dots, \phi_{i-1}, \psi_1, \dots, \psi_{n-1}, \psi_n \bullet \phi_{i+1}, \phi_{i+2}, \dots, \phi_m ; \theta) \in \Comp$.
	\begin{align*}
		&(-1)^{\deg{\phi_{i-1}} + 1}\fm_m(\phi_1, \dots, \phi_{i-1}, \fm_n(\psi_1, \dots, \psi_n), \phi_{i+1}, \dots, \phi_m) \\
		&+\fm_{m+n-2}(\phi_1, \dots, \fm_2(\phi_{i-1}, \psi_1), \psi_2, \dots, \psi_n, \phi_{i+1}, \dots, \phi_m) \\
		&+(-1)^{\deg{\phi_{i-1}} + \deg{\psi_1} + \dots + \deg{\psi_{n-1}} + n}\fm_{m+n-2}(\phi_1, \dots, \phi_{i-1}, \psi_1, \dots, \psi_{n-1}, \fm_2(\psi_n, \phi_{i+1}), \phi_{i+2}, \dots, \phi_m) \\
		=& 0 \quad \text{(by Lemma \ref{Lemma:PhiSigmaCCM})}.
	\end{align*}
	\item $i=1$. Then, $\theta = \psi_1 \bullet \theta'$ for some nonzero $\theta'$. Then, $(\psi_1, \dots, \psi_{n-1}, \psi_n \bullet \phi_2, \phi_3, \dots, \phi_m ; \theta)$,\\ $(\psi_2, \dots, \psi_n, \phi_2, \dots, \phi_m ; \theta') \in \Comp$.
	\begin{align*}
		&\fm_m(\fm_n(\psi_1, \dots, \psi_n), \phi_2, \dots, \phi_m) \\
		&+(-1)^{\deg{\psi_1} + \dots + \deg{\psi_{n-1}} + n-1}\fm_{m+n-2}(\psi_1, \dots, \psi_{n-1}, \fm_2(\psi_n, \phi_2), \phi_3, \dots, \phi_m) \\
		&+(-1)^{\deg{\psi_1} + 1}\fm_2(\psi_1, \fm_{n+m-2}(\psi_2, \dots, \psi_n, \phi_2, \dots, \phi_m)) \\
		=& 0 \quad \text{(by Lemma \ref{Lemma:PhiSigmaCCF})}.
	\end{align*}
	\item $i=m$. This is symmetric to the previous case.
\end{enumerate}

\paragraph{Thick triple}

Suppose that $((\psi_1, \psi_2) ; {}^\wedge\phi_i)$ is a thick triple. Then, as we have shown in the disc sequence case, $(\phi_{i-1}, \phi_i) \in \Con$. Then, there are the following cases.
\begin{enumerate}
	\item $i=1$. This further has the following three subcases.
	\begin{enumerate}
		\item $\psi_1 = \theta \bullet \theta'$. Then, $t(\psi_1)$ divides $\phi_k$ into $\phi_k^1 \bullet \phi_k^2$ for some $1 < k < m$ so that $(\psi_2, {}^\wedge\phi_1, \phi_2, \dots, \phi_{k-1}, \phi_k^1)$, $(\theta', \phi_k^2, \phi_{k+1}, \dots, \phi_m) \in \Disc$.
		\begin{align*}
			&\fm_m(\fm_3(\psi_1, \psi_2, {}^\wedge\phi_1), \phi_2, \dots, \phi_m) \\
			&+(-1)^{\deg{\psi_1} + 1}\fm_{m-k+2}(\psi_1, \fm_{k+1}(\psi_2, {}^\wedge\phi_1, \phi_2, \dots, \phi_{k-1}, \phi_k), \phi_{k+1}, \dots, \phi_m) \\
			=& 0 \quad \text{(by Lemma \ref{Lemma:PhiSigmaDCF}, \ref{Lemma:PhiSigmaT})}.
		\end{align*}
		\item $\psi_1 = \theta$. Then, $(\psi_2, {}^\wedge\phi_1, \phi_2, \dots, \phi_m) \in \Disc$.
		\begin{align*}
			&\fm_m(\fm_3(\psi_1, \psi_2, {}^\wedge\phi_1), \phi_2, \dots, \phi_m) \\
			&+(-1)^{\deg{\psi_1} + 1}\fm_2(\psi_1, \fm_{m+1}(\psi_2, {}^\wedge\phi_1, \phi_2, \dots, \phi_m)) \\
			=& 0 \quad \text{(by Lemma \ref{Lemma:PhiSigmaT})}.
		\end{align*}
		\item $\theta = \psi_1 \bullet \theta'$. Then, $\psi_2 = \theta' \bullet \xi$ and $(\xi, {}^\wedge\phi_1, \phi_2, \dots, \phi_m) \in \Disc$.
		\begin{align*}
			&\fm_m(\fm_3(\psi_1, \psi_2, {}^\wedge\phi_1), \phi_2, \dots, \phi_m) \\
			&+(-1)^{\deg{\psi_1} + 1}\fm_2(\psi_1, \fm_{m+1}(\psi_2, {}^\wedge\phi_1, \phi_2, \dots, \phi_m)) \\
			=& 0 \quad \text{(by Lemma \ref{Lemma:PhiSigmaT})}.
		\end{align*}
	\end{enumerate}
	\item $i > 1$. Then, $(\phi_{i-1}, \phi_i) \in \Con$, $(\phi_{i-1} ; (\psi_1, \psi_2)) \in \Thick$, and $(\phi_1, \dots, \phi_{i-2}, \phi_{i-1}^\vee, {}^\wedge\phi_i, \phi_{i+1}, \dots, \phi_m ; \theta) \in \Comp$. This further has the following three subcases.
	\begin{enumerate}
		\item $t(\psi_1)$ meets $s(\theta)$ on the left with an interior morphism $\xi$ so that $(\xi, \theta) \in \Con$, $(\xi, \phi_1, \dots, \phi_{i-1}, \psi_1) \in \Disc$, and $(\psi_2, \phi_i^\wedge, \phi_{i+1}, \dots, \phi_m ; \xi \bullet \theta) \in \Comp$.
		\begin{align*}
			&(-1)^{\deg{\phi_1} + \dots + \deg{\phi_{i-2}} + i-2}\fm_m(\phi_1, \dots, \phi_{i-2}, \fm_3(\phi_{i-1}, \psi_1, \psi_2), \phi_i^\wedge, \phi_{i+1}, \dots, \phi_m) \\
			&+(-1)^{\deg{\phi_1} + \dots + \deg{\phi_{i-1}} + i-1}\fm_m(\phi_1, \dots, \phi_{i-1}, \fm_3(\psi_1, \psi_2, \phi_i^\wedge), \phi_{i+1}, \dots, \phi_m) \\
			&+\fm_{m-i+3}(\fm_i(\phi_1, \dots, \phi_{i-1}, \psi_1), \psi_2, \phi_i^\wedge, \phi_{i+1}, \dots, \phi_m) \\
			=& 0 \quad \text{(by Lemma \ref{Lemma:PhiSigmaDCF}, \ref{Lemma:PhiSigmaT})}.
		\end{align*}
		\item $t(\psi_1)$ meets $t(\theta)$ on the right with an interior morphism $\xi$ so that $(\theta, \xi) \in \Con$, $(\phi_1, \dots, \phi_{i-1}, \psi_1 ; \theta \bullet \xi)$, and $(\xi, \psi_2, \phi_i^\wedge, \phi_{i+1}, \dots, \phi_m) \in \Disc$. This case is symmetric to the previous case.
		\item $t(\psi_1)$ divides $\phi_k$ into $\phi_k^1 \bullet \phi_k^2$ with $k < i-1$ or $k > i$. 
		\begin{enumerate}
			\item $k < i-1$. Then, $(\phi_1, \dots, \phi_{k-1}, \phi_k^1, \psi_2, \phi_i^\wedge, \phi_{i+1}, \dots, \phi_m ; \theta) \in \Comp$, and $(\phi_k^2, \phi_{k+1}, \dots, \phi_{i-1}, \psi_1) \in \Disc$.
			\begin{align*}
				&(-1)^{\deg{\phi_k} + \dots + \deg{\phi_{i-2}} + i - k-1}\fm_m(\phi_1, \dots, \phi_{i-2}, \fm_3(\phi_{i-1}, \psi_1, \psi_2), \phi_i^\wedge, \phi_{i+1}, \dots, \phi_m) \\
				&+(-1)^{\deg{\phi_k} + \dots + \deg{\phi_{i-1}} + i-k}\fm_m(\phi_1, \dots, \phi_{i-1}, \fm_3(\psi_1, \psi_2, \phi_i^\wedge), \phi_{i+1}, \dots, \phi_m) \\
				&+\fm_{m-i+k+2}(\phi_1, \dots, \phi_{k-1}, \fm_{i-k+1}(\phi_k, \dots, \phi_{i-1}, \psi_1), \psi_2, \phi_i^\wedge, \phi_{i+1}, \dots, \phi_m) \\
				=& 0 \quad \text{(by Lemma \ref{Lemma:PhiSigmaDCM}, \ref{Lemma:PhiSigmaT})}.
			\end{align*}
			\item $k > i$. This case is symmetric to the previous case.
		\end{enumerate}
	\end{enumerate}
\end{enumerate}

\subsubsection{Thick triple}\label{Appendix:ThickTriple}

Here, we deal with the case when the outer-input is a thick triple $\ora{\phi} = (\phi_1 ; (\phi_2, \phi_3))$. The case for a thick triple $((\phi_1, \phi_2) ; \phi_3)$ is symmetric.

\paragraph{Concatenable pair}

Suppose that $\phi_i$ has a decomposition $\phi_i^1 \bullet \phi_i^2$. Then, there are the following cases.
\begin{enumerate}
	\item $i=1$. Then, $(\phi_1^2 ; (\phi_2, \phi_3)) \in \Thick$ and $\phi_1^1 \bullet (\phi_1^2)^\vee = \phi_1^\vee$.
	\begin{align*}
		&\fm_3(\fm_2(\phi_1^1, \phi_1^2), \phi_2, \phi_3) +(-1)^{\deg{\phi_1^1} + 1}\fm_2(\phi_1^1, \fm_3(\phi_1^2, \phi_2, \phi_3)) = 0.
	\end{align*}
	\item $i=2$. Then, $(\phi_2^2, \phi_3) \in \Con$ and $(\phi_1 ; (\phi_2^1, \phi_2^2 \bullet \phi_3)) \in \Thick$.
	\begin{align*}
		&\fm_3(\phi_1, \fm_2(\phi_2^1, \phi_2^2), \phi_3) + (-1)^{\deg{\phi_2^1} + 1}\fm_3(\phi_1, \phi_2^1, \fm_2(\phi_2^2, \phi_3)) = 0.
	\end{align*}
	\item $i=3.$ The same with the previous case.
\end{enumerate}

\paragraph{Disc sequence}
Suppose that $\ora{\psi} = (\psi_1, \dots, \psi_n)$ is a disc sequence with $(\phi_i, \psi_1) \in \Con$ or $(\psi_n, \phi_i) \in \Con$. There are the following cases.
\begin{enumerate}
	\item $(\phi_1, \psi_1) \in \Con$. Then, $((\phi_2, \phi_3) ; {}^\wedge\psi_1) \in \Thick$. Also, $(\psi_n, \psi_1) \in \Con$ and $t(\phi_2)$ divides $\psi_k$ into $\psi_k^1 \bullet \psi_k^2$ for some $1 < k < n$ so that $({}^\wedge\psi_1, \psi_2, \dots, \psi_{k-1}, \psi_k^1, \phi_3)$, $(\psi_k^2, \psi_{k+1}, \dots, \psi_n, \phi_2) \in \Disc$.
	\begin{align*}
		&\fm_3(\fm_n(\phi_1 \bullet \psi_1, \psi_2, \dots, \psi_n), \phi_2, \phi_3) \\
		&+(-1)^{\deg{\phi_1} + \deg{\psi_1} + \dots + \deg{\psi_{k-1}} + k-1}\fm_{k+1}(\phi_1 \bullet \psi_1, \psi_2, \dots, \psi_{k-1}, \fm_{n-k+2}(\psi_k, \dots, \psi_n, \phi_2), \phi_3) \\
		=& 0 \quad \text{(by Lemma \ref{Lemma:PhiSigmaCD}, \ref{Lemma:PhiSigmaT})}.
	\end{align*}
	\item $(\psi_n, \phi_1) \in \Con$. Then, $(\psi_n \bullet \phi_1 ; (\phi_2, \phi_3)) \in \Thick$.
	\begin{align*}
		&\fm_3(\fm_n(\psi_1, \dots, \psi_{n-1}, \psi_n \bullet \phi_1), \phi_2, \phi_3) \\
		&+(-1)^{\deg{\psi_1} + \dots + \deg{\psi_{n-1}} + n-1}\fm_n(\psi_1, \dots, \psi_{n-1}, \fm_3(\psi_n \bullet \phi_1, \phi_2, \phi_3)) = 0.
	\end{align*}
	\item $(\phi_2, \psi_1) \in \Con$. Then, $\phi_3 = \psi_1 \bullet \theta$ for some nonzero $\theta$ and $(\phi_1 ; (\phi_2 \bullet \psi_1, \theta)) \in \Thick$.
	\begin{align*}
		&\fm_3(\phi_1, \fm_n(\phi_2 \bullet \psi_1, \psi_2, \dots, \psi_n), \phi_3) + (-1)^{\deg{\phi_2} + \deg{\psi_1} + 1}\fm_3(\phi_1, \phi_2 \bullet \psi_1, \fm_n(\psi_2, \dots, \psi_n, \phi_3)) = 0.
	\end{align*}
	\item $(\psi_n, \phi_2) \in \Con$. As $(\phi_2, \phi_3)$ is a thick pair, this case is not valid.
	\item $(\phi_3, \psi_1) \in \Con$. As $(\phi_2, \phi_3)$ is a thick pair, this case is not valid.
	\item $(\psi_n, \phi_3) \in \Con$. Then, $\phi_2 = \theta \bullet \psi_n$ for some nonzero $\theta$ and $(\phi_1 ; (\theta, \psi_n \bullet \phi_3)) \in \Thick$.
	\begin{align*}
		&(-1)^{\deg{\phi_2} + 1}\fm_3(\phi_1, \phi_2, \fm_n(\psi_1, \dots, \psi_{n-1}, \psi_n \bullet \phi_3)) + \fm_3(\phi_1, \fm_n(\phi_2, \psi_1, \dots, \psi_{n-1}), \psi_n \bullet \phi_3) = 0.
	\end{align*}
\end{enumerate}

\paragraph{Composition sequence}

Suppose that $\ora{\psi} = (\psi_1, \dots, \psi_n)$ is a composition sequence with value $\phi_i$. Since the value of a composition sequence is an interior morphism, $i=2$ or $i=3$.
\begin{enumerate}
	\item $i=2$. Then, $(\phi_1, \psi_1) \in \Con$ and $(\psi_1^\wedge, \psi_2, \dots, \psi_n, \phi_3) \in \Disc$. Note that $\phi_1 \bullet \psi_1 = \phi_1^\vee \bullet \psi_1^\wedge$.
	\begin{align*}
		&(-1)^{\deg{\phi_1} + 1}\fm_3(\phi_1, \fm_n(\psi_1, \dots, \psi_n), \phi_3) + \fm_{n+1}(\fm_2(\phi_1, \psi_1), \psi_2, \dots, \psi_n, \phi_3) = 0.
	\end{align*}
	\item $i=3$. Then, $t(\psi_n) = t(\phi_1^\vee)$. This has three subcases.
	\begin{enumerate}
		\item $\psi_n = \theta \bullet \phi_1^\vee$ for some nonzero $\theta$. This further has the following three subcases.
		\begin{enumerate}
			\item $t(\theta)$ divides $\psi_k$ into $\psi_k^1 \bullet \psi_k^2$ for some $1 \leq k < n-1$ so that $(\phi_1, \phi_2, \psi_1, \dots, \psi_{k-1}, \psi_k^1)$ and
			 $(\psi_k^2, \psi_{k+1}, \dots, \psi_{n-1}, \theta) \in \Disc$.
			\begin{align*}
				&(-1)^{\deg{\phi_1} + \deg{\phi_2} + 2}\fm_3(\phi_1, \phi_2, \fm_n(\psi_1, \dots, \psi_n)) \\
				&+\fm_{n-k+1}(\fm_{k+2}(\phi_1, \phi_2, \psi_1, \dots, \psi_k), \psi_{k+1}, \dots, \psi_n) \\
				=& 0 \quad \text{(by Lemma \ref{Lemma:PhiSigmaDD}, \ref{Lemma:PhiSigmaT})}.
			\end{align*}
			\item $t(\theta)$ meets $t(\psi_k)$ on the left for some $1 \leq k < n-1$ with an interior morphism $\xi$ so that $(\phi_1, \phi_2, \psi_1, \dots, \psi_k ; \xi) \in \Comp$ and $(\xi, \psi_{k+1}, \dots, \psi_{n-1}, \theta) \in \Disc$. In this case, $(\psi_k, \psi_{k+1}) \in \Con$ and $(\phi_1, \phi_2, \psi_1, \dots, \psi_{k-1}, \psi_k \bullet \psi_{k+1}, \psi_{k+2}, \dots, \psi_{n-1}, \theta) \in \Disc$.
			\begin{align*}
				&(-1)^{\deg{\phi_1} + \deg{\phi_2} + 2}\fm_3(\phi_1, \phi_2, \fm_n(\psi_1, \dots, \psi_n)) \\
				&+\fm_{n-k+1}(\fm_{k+2}(\phi_1, \phi_2, \psi_1, \dots, \psi_k), \psi_{k+1}, \dots, \psi_n) \\
				&+(-1)^{\deg{\phi_1} + \deg{\phi_2} + \deg{\psi_1} + \dots + \deg{\psi_{k-1}} + k + 1}\fm_{n+1}(\phi_1, \phi_2, \psi_1, \dots, \psi_{k-1}, \fm_2(\psi_k, \psi_{k+1}), \psi_{k+2}, \dots, \psi_n) \\
				=& 0 \quad \text{(by Lemma \ref{Lemma:PhiSigmaCD}, \ref{Lemma:PhiSigmaT})}.
			\end{align*}
			\item $t(\theta)$ meets $t(\psi_k)$ on the right for some $1 \leq k < n-1$ with an interior morphism $\xi$ so that $(\phi_1, \phi_2, \psi_1, \dots, \psi_k, \xi) \in \Disc$ and $(\psi_{k+1}, \dots, \psi_{n-1}, \theta ; \xi) \in \Comp$. In this case, $(\psi_k, \psi_{k+1}) \in \Con$ and $(\phi_1, \phi_2, \psi_1, \dots, \psi_{k-1}, \psi_k \bullet \psi_{k+1}, \psi_{k+2}, \dots, \psi_{n-1}, \theta) \in \Disc$.
			\begin{align*}
				&\fm_3(\phi_1, \phi_2, \fm_n(\psi_1, \dots, \psi_n)) \\
				&+(-1)^{\deg{\psi_1} + \dots + \deg{\psi_{k-1}} + k - 1}\fm_{n+1}(\phi_1, \phi_2, \psi_1, \dots, \psi_{k-1}, \fm_2(\psi_k, \psi_{k+1}), \psi_{k+2}, \dots, \psi_n) \\
				=& 0 \quad \text{(by Lemma \ref{Lemma:PhiSigmaCD}, \ref{Lemma:PhiSigmaT})}.				
			\end{align*}
		\end{enumerate}
		\item $\psi_n = \phi_1^\vee$. Then, $(\phi_1, \phi_2, \psi_1, \dots, \psi_{n-1}) \in \Disc$.
		\begin{align*}
			&(-1)^{\deg{\phi_1} + \deg{\phi_2} + 2}\fm_3(\phi_1, \phi_2, \fm_n(\psi_1, \dots, \psi_n)) + \fm_2(\fm_{n+1}(\phi_1, \phi_2, \psi_1, \dots, \psi_{n-1}), \psi_n) \\
			=& 0 \quad \text{(by Lemma \ref{Lemma:PhiSigmaT})}.
		\end{align*}
		\item $\phi_1^\vee = \theta \bullet \psi_n$. Then, $\phi_1 = \theta \bullet \psi_n^\wedge$ and $(\phi_2, \psi_1, \dots, \psi_n^\wedge) \in \Disc$.
		\begin{align*}
			&(-1)^{\deg{\phi_1} + \deg{\phi_2} + 2}\fm_3(\phi_1, \phi_2, \fm_n(\psi_1, \dots, \psi_n)) + \fm_2(\fm_{n+1}(\phi_1, \phi_2, \psi_1, \dots, \psi_{n-1}), \psi_n) \\
			=& 0 \quad \text{(by Lemma \ref{Lemma:PhiSigmaT})}.
		\end{align*}
	\end{enumerate}
\end{enumerate}

\paragraph{Thick triple}
The last case is that $\ora{\psi}$ is a thick triple whose value is $\phi_i$. The only possible case is $\ora{\psi} = ((\psi_1, \psi_2) ; \psi_3)$ with ${}^\vee\psi_3 = \phi_1$. Then, $(\psi_3 ; (\phi_2, \phi_3)) \in \Thick$. 
\begin{align*}
	&\fm_3(\psi_1, \psi_2, \fm_3(\psi_3, \phi_2, \phi_3)) + \fm_3(\fm_3(\psi_1, \psi_2, \psi_3), \phi_2, \phi_3) = 0.
\end{align*}

\bibliographystyle{amsplain}
\bibliography{OnTopologicalOrbifoldFukayaCategories}

\end{document}